\documentclass{amsart}

\usepackage[a4paper,margin=3cm]{geometry}
\usepackage{amsmath}
\usepackage{amssymb}
\usepackage{amsthm}
\usepackage{booktabs}
\usepackage{caption}
\usepackage[shortlabels]{enumitem}
\usepackage{mathtools}
\usepackage{multirow}
\usepackage{threeparttable}
\usepackage[all]{xy}
\usepackage{subfiles}
\usepackage{mathrsfs}
\usepackage{longtable}
\usepackage{todonotes}
\usepackage{hyperref}
\usepackage{pdflscape}
\usepackage{siunitx}
\usepackage{ulem}

\newtheorem{theorem}{Theorem}
\newtheorem{lemma}[theorem]{Lemma}
\newtheorem{prop}[theorem]{Proposition}
\newtheorem{corollary}[theorem]{Corollary}
\theoremstyle{definition}
\newtheorem{definition}[theorem]{Definition}
\newtheorem{example}[theorem]{Example}
\theoremstyle{remark}
\newtheorem{remark}[theorem]{Remark}

\DeclareMathOperator{\sgn}{sgn}
\DeclareMathOperator{\diag}{diag}
\DeclareMathOperator{\del}{del}
\DeclareMathOperator{\Aut}{Aut}

\DeclareMathOperator{\Out}{Out}
\DeclareMathOperator{\Hom}{Hom}
\DeclareMathOperator{\soc}{soc}
\DeclareMathOperator{\Core}{Core}

\DeclareMathOperator{\Rad}{Rad}
\DeclareMathOperator{\End}{End}

\newcommand{\boldL}{\mathbf{L}}

\DeclareMathOperator{\GL}{GL}
\DeclareMathOperator{\SL}{SL}

\DeclareMathOperator{\PSL}{PSL}
\let\L\relax
\DeclareMathOperator{\L}{L}

\newcommand{\boldU}{\mathbf{U}}
\DeclareMathOperator{\GU}{GU}
\DeclareMathOperator{\SU}{SU}
\DeclareMathOperator{\PSU}{PSU}
\DeclareMathOperator{\U}{U}

\newcommand{\boldS}{\mathbf{S}}
\DeclareMathOperator{\GSp}{GSp}
\DeclareMathOperator{\Sp}{Sp}
\DeclareMathOperator{\PSp}{PSp}
\let\S\relax
\DeclareMathOperator{\S}{S}

\newcommand{\boldO}{\mathbf{O}}

\DeclareMathOperator{\SO}{SO}
\DeclareMathOperator{\POmega}{P\Omega}

\let\O\relax
\DeclareMathOperator{\O}{O}

\DeclareMathOperator{\AGL}{AGL}
\DeclareMathOperator{\Alt}{Alt}
\DeclareMathOperator{\Sym}{Sym}

\newcommand{\mleq}{<_{\max}}

\DeclareMathOperator{\M}{M}
\DeclareMathOperator{\HS}{HS}
\DeclareMathOperator{\J}{J}
\DeclareMathOperator{\Co}{Co}
\DeclareMathOperator{\McL}{McL}
\DeclareMathOperator{\Suz}{Suz}
\DeclareMathOperator{\He}{He}
\DeclareMathOperator{\Fi}{Fi}
\DeclareMathOperator{\Ru}{Ru}
\DeclareMathOperator{\ON}{O'N}
\DeclareMathOperator{\Th}{Th}
\DeclareMathOperator{\Ly}{Ly}
\DeclareMathOperator{\HN}{HN}

\newcommand{\Atlas}{\mathbb{A}\mathbb{T}\mathbb{L}\mathbb{A}\mathbb{S}}

\title[Generation of second maximal subgroups of almost simple groups]{The structure and generation of the second maximal subgroups of the almost simple groups with alternating, classical or sporadic socle}

\author{Patricia Medina Capilla}
\address{Patricia Medina Capilla, Mathematics Institute, University of Warwick, Coventry, CV4 7AL. ORCID: https://orcid.org/0009-0005-6183-5954}
\email{patricia.medina-capilla@warwick.ac.uk}

\begin{document}

\begin{abstract}
	Let $G$ be an almost simple group whose socle is an alternating, classical, or sporadic group, and let $H$ be a non-parabolic maximal subgroup of $G$. We prove that any maximal subgroup $M$ of $H$ can be generated by at most $7$ elements, and that this bound is sharp when the socle of $G$ is alternating or classical; this improves the bound of $12$ due to Burness, Liebeck and Shalev. When the socle is sporadic, at most $5$ generators suffice, and this is again best possible. The proof relies upon a detailed structural analysis of $H$ and $M$, especially when $H$ is a subgroup of a wreath product. In particular, we determine the chief factors of $M$ and subsequently bound its number of generators using the theory of crowns.

    We also correct the classification of maximal subgroups $H$ of almost simple groups requiring more than three generators, established by Lucchini, Marion and Tracey. Their bound of five generators remains valid, but their classification is missing several pairs $(G,H)$, including cases in which $G$ has socle $\mathrm{PSU}_n(q)$.
\end{abstract}

\maketitle

\section{Introduction}\label{section:intro}

The generation of finite simple groups has long been a central topic of interest in group theory. A key result, established in \cite{AG, S}, is that every finite simple group can be generated by two elements. This result has led to extensive work on the 2-generation of these groups, such as in \cite{MQR}, which demonstrates that a finite simple group is highly likely to be generated by two randomly selected elements. Moreover, analogous questions have been studied for related classes of groups. For instance, \cite{DVL1995} shows that the minimal number of generators of an almost simple group is also uniformly bounded, in this case by three.

Naturally, our attention turns to the subgroups of almost simple groups. In \cite{BLS2013}, Burness, Liebeck, and Shalev showed that maximal subgroups of almost simple groups can be generated by six elements, although it was not known whether this bound was best possible. Following this work, in \cite{LMT}, Lucchini, Marion and Tracey investigated whether this bound could be lowered using the so-called theory of crowns. Originally developed by Gaschütz in \cite{G} for solvable groups and later extended to all finite groups by Dalla Volta and Lucchini in \cite{DVL1998}, the theory of crowns provides a method for bounding the number of generators of a finite group based on information about its chief factors. Using this approach, they lowered the bound on the number of generators of these groups to five and showed that this result is best possible. Moreover, they provided a classification of the maximal subgroups that require four or five generators.

Pursuing this further, one can investigate the next level of the subgroup lattice by studying the maximal subgroups of the maximal subgroups of an almost simple group $G$. We call such a subgroup a \textit{second maximal subgroup} of $G$. Burness, Liebeck and Shalev studied these groups in \cite{BLS}, and proved the following: if $G$ is an almost simple group whose socle is alternating, sporadic, or a classical group of Lie type, and $M$ is a second maximal subgroup which is not maximal in a parabolic subgroup of $G$, then $M$ can be generated by 12 elements. In particular, we have $d(M) \leq 12$ for any such second maximal subgroup, where $d(M)$ denotes the minimal number of elements required to generate $M$. The main aim of this paper is to establish the following improvement to this result.

\begin{theorem}\label{theorem:bigresult}
	Let $M$ be a second maximal subgroup of an almost simple group $G$. Suppose that the socle of $G$ is an alternating group, a classical group, or a sporadic group. Furthermore, in the case that $G$ is a classical group, assume that $M$ is not maximal in a parabolic subgroup. Then $d(M) \leq 7$. Moreover, there exist second maximal subgroups requiring $7$ generators in the alternating and classical cases, whilst $d(M) \leq 5$ in all the sporadic cases. 
\end{theorem}

When $G$ is an exceptional group of Lie type, or $M$ is contained in a parabolic subgroup of $G$, Burness, Liebeck and Shalev show in the same paper that $M$ is generated by 70 elements. This requires different techniques, including methods from the representation theory of algebraic groups. Moreover, the existing literature on the parabolic case is incomplete, due to an error in the analysis of these second maximal subgroups. These cases will be treated in a forthcoming paper, joint with Adam Thomas and Gareth Tracey.

The starting point of the proof of Theorem \ref{theorem:bigresult} is the extensive literature on the maximal subgroups of the almost simple groups, which provides the essential foundation for our approach. Building on these results, we undertake a systematic classification of the second maximal subgroups of almost simple groups. Once sufficiently detailed structural information about these subgroups has been established, we use a combination of methods, most notably consisting of the theory of crowns, as well as some ad hoc techniques, in order to bound $d(M)$.

When the socle of $G$ is a sporadic group, the situation is comparatively straightforward, since the maximal subgroups of $G$ are explicitly known. This allows us to use computational methods to classify the relevant second maximal subgroups and to bound the number of generators of $M$. If the socle of $G$ is instead a classical or alternating group, the maximal subgroups of $G$ have been broadly classified into geometric classes, together with a class consisting of almost simple groups. For most of these classes, it is sufficient to obtain a rough classification for $M$ in terms of unspecified extensions of groups which are known. This level of information is adequate for our purposes, and we then use ad hoc methods to bound $d(M)$.

The remaining cases primarily correspond to classes of maximal subgroups that are described as direct products or wreath products of almost simple groups. These classes constitute the main difficulty of this paper, necessitating the theory of crowns to be applied to its fullest potential. This, in turn, requires detailed information about $M$ and its chief factors. To this end, we give a complete classification of these second maximal subgroups, together with a corresponding description of their chief factors.

An additional strength of the theory of crowns is that it often yields not only upper bounds but also matching lower bounds for $d(M)$, thereby showing that the results obtained are best possible. This phenomenon occurs for second maximal subgroups arising as maximal subgroups of wreath products, enabling us to prove that the given bound on $d(M)$ is tight both when the socle of $G$ is an alternating group and a classical group. If the socle of $G$ is sporadic we instead obtain a tight bound of 5 on $d(M)$.

The aforementioned classification results for $M$ depend on the classification of maximal subgroups of almost simple groups. Consequently, bounding $d(M)$ relies on existing information about the chief series of such maximal subgroups, as developed in \cite{LMT}. Upon a detailed analysis of these results, we identified a gap in the case of maximal subgroups of classical groups lying in the Aschbacher classes $\mathscr{C}_2$ and $\mathscr{C}_7$. This mistake leads to \cite[Table 1]{LMT}, which lists all maximal subgroups that are 4- or 5-generated, being incomplete. For example, it is asserted that, for a maximal subgroup of a classical almost simple group to be 4- or 5-generated, the almost simple group must have socle $\PSL_n(q)$ or $\POmega_{2n}^\pm(q)$, for some $n, q$. However, we exhibit in this paper a maximal subgroup of an almost simple group with socle $\PSU_n(q)$ that is 4-generated. In Section \ref{section:LMTcorrections}, we identify the mistakes in \cite{LMT}, and in Section \ref{section:LMTcorrectionproof} we prove the following corrected version of their result.

\begin{theorem}\label{theorem:LMTcorrection}
    Let $H$ be a maximal subgroup of $G$, for some almost simple group $G$. Then $d(H) \leq 5$, and there is a classification of the pairs $(G, H)$ for which $d(H) > 3$, see Table \ref{table:LMTtable} in Section \ref{section:LMTcorrections}.
\end{theorem}

We conclude by outlining the structure of this paper. In Section \ref{section:theoryofcrowns}, we introduce the theory of crowns, following \cite[Section 2]{LMT}. In Section \ref{section:prelimsmethod}, we describe several methods for computing maximal subgroups of finite groups, beginning with general results and progressing to more specific cases, including a classification of the maximal subgroups of certain wreath products. Section \ref{section:wreathproducts} is devoted to a more detailed study of wreath products, as the cases when $M$ is contained in a wreath product constitute the main difficulty of this paper. In particular, we determine the chief factors of certain subgroups of wreath products, extending \cite[Lemma 3.9]{LMT}. In Section \ref{section:classicalprelims}, we establish several results concerning the chief factors of maximal subgroups of classical groups of Lie type, including the aforementioned correction to \cite[Theorem 1]{LMT}, along with further refinements. Finally, in Sections \ref{section:classicals}, \ref{section:alternating} and \ref{section:sporadics} we apply these results to the almost simple groups with socle equal to a classical group of Lie type, an alternating group, and a sporadic group respectively, completing the proof of Theorem \ref{theorem:bigresult}.

\subsection*{Acknowledgments.}
The author thanks Adam Thomas and Gareth Tracey for their guidance and support, as well as their helpful comments and feedback.
This work was supported by the Additional Funding Programme for Mathematical Sciences, delivered by EPSRC (EP/V521917/1) and the Heilbronn Institute for Mathematical Research.

\numberwithin{theorem}{section}

\section{Theory of crowns} \label{section:theoryofcrowns}
The theory of crowns, developed by Dalla Volta and Lucchini, provides a powerful framework for relating the generation of a finite group to the structure of its chief factors. In the first part of this section, we introduce their work, culminating in a result relating the number of generators of a group to its chief factors. For a more thorough treatment of the theory, including proofs of the results presented here, see \cite[Section 2]{LMT}. We then establish several additional results that will be useful throughout the remainder of the paper when applying the theory of crowns.

We say a group $A$ is a $G$-group if $G$ acts on $A$ by automorphisms. Two $G$-groups $A, B$ are $G$-isomorphic, $A \cong_G B$, if the group actions of $G$ on $A$ and $B$ are equivalent. That is, there exists an isomorphism $\phi \colon A \to B$ such that $\phi(a^g) = \phi(a)^g$ for every $a \in A$ and $g \in G$.

\begin{definition}\label{def:CFsEquivalence}
	Let $A, B$ be two $G$-groups. We say $A$ is \emph{$G$-equivalent} to $B$, and write $A \equiv_G B$, if there exist isomorphisms $\phi \colon A \to B$ and $\Phi \colon A \rtimes G \to B \rtimes G$ such that the following diagram commutes
	\[
	\xymatrix{
		1 \ar[r] & A \ar[r] \ar[d]^\phi & A \rtimes G \ar[r] \ar[d]^\Phi & G \ar[r] \ar[d] ^{\text{id}}& 1\\
		1 \ar[r] & B \ar[r] & B \rtimes G \ar[r] & G \ar[r] & 1.\\
	}
	\]
\end{definition}

\begin{prop}
	Let $A, B$ be two $G$-groups. Then $A$ is $G$-equivalent to $B$ if $A$ is $G$-isomorphic to $B$. If additionally $A$ and $B$ are abelian, then $A$ is $G$-equivalent to $B$ if and only if $A$ is $G$-isomorphic to $B$.
\end{prop}

We apply these equivalence relations in the following setting. Let $G$ be a finite group, and consider a chief series
\[
1 = N_0 < N_1 < \cdots < N_k = G
\]
of $G$. The group $G$ acts by conjugation on each chief factor, and the multiset
\[
\{ N_1/N_0, \dots, N_k/N_{k-1} \}
\]
of $G$-groups is independent of the choice of chief series of $G$. Our interest lies in the $G$-equivalence classes of a certain subset of chief factors of $G$. Let $\Phi(G) = \bigcap_{M <_{\max} G} M$ denote the Frattini subgroup of a group $G$.

\begin{definition}
	We say a chief factor $H / K$ of $G$ is \emph{non-Frattini} if
	\[
	H / K \nleq \Phi(G / K).
	\]
\end{definition}

For the statement of the following proposition, we say that a section $H / K$ of a group $G$ is complemented if there exists some subgroup $U$ of $G$ such that $UH = G$ and $U \cap H = K$.

\begin{prop} [{\cite[Lemma 2.3]{LMT}}]
	Let $H / K$ be a chief factor of $G$.
	\begin{enumerate}[\upshape(1)]
		\item If $H / K$ is non-abelian, then it is non-Frattini.
		\item If $H / K$ is abelian, then it is non-Frattini if and only if it is complemented.
	\end{enumerate}
\end{prop}

We introduce the last concept we require in order to present the main result of this section.

\begin{definition}
	A group $L$ is said to be \emph{monolithic} if it has a unique minimal normal subgroup $A$. If, in addition, $A \nleq \Phi(L)$, then $L$ is called a \emph{monolithic primitive group}.
\end{definition}

\begin{definition}
	Let $L$ be a monolithic primitive group, and let $A$ be its unique minimal normal subgroup. The \emph{crown-based power} of $L$ of length $k$ is
	\[
	L_k := \left\lbrace (\ell_1, \dots, \ell_k) : \ell_1 A = \ell_i A \quad \forall 1 \leq i \leq k \right\rbrace \leq L^k,
	\]
	where $L_0 := 1$.
\end{definition}

For an irreducible $G$-group $A$, let $\delta_G(A)$ denote the number of non-Frattini chief factors of $G$ which are $G$-equivalent to $A$. It can be shown that this value does not depend on the choice of chief series of $G$. Additionally, let
\[
L_A := \begin{cases*}
	A \rtimes (G / C_G(A)) \qquad & if $A$ is abelian,\\
	G / C_G(A) \qquad & otherwise,
\end{cases*}
\]
which is a monolithic primitive group with unique minimal normal subgroup $A$. If $A$ is a non-Frattini chief factor of $G$, then there exists a normal subgroup $N$ of $G$ such that $G / N \cong L_A$. We define
\[
R_G(A) := \bigcap_{\substack{{N \lhd G,}\\ { G/N \cong L_A}}} N.
\]
Then, for any non-Frattini chief factor $A$ of $G$,
\[
G / R_G(A) \cong (L_A)_{\delta_G(A)}.
\]

Suppose that $A$ is an abelian chief factor of $G$. Then $L_A/A$ is an irreducible subgroup of $\GL(A)$, and so $\mathbb F:=\End_{L_A/A}(A)$ is a field by Schur’s lemma, and $A$ is an $\mathbb F[L_A/A]$-module. 

\begin{theorem} [{\cite[Proposition 2.6]{LMT}}] \label{theorem:crowns}
	Let $G$ be a non-cyclic finite group. Then the following holds:
	\begin{enumerate}[\upshape(1)]
		\item We have
		\[
		d(G) = \max_{\substack{A \text{ a non-Frattini} \\ \text{chief factor of } G}} d((L_A)_{\delta_G(A)}).
		\]
		
		\item Suppose that for every non-abelian chief factor $A$ of $G$ we have
		\[
		\delta_G(A) \leq \frac{|A|}{2n |\text{Out}(S)|},
		\]
		where $S$ is the simple group such that $A \cong S^n$. Then either $d(G) = 2$ or
		\[
		d(G) = \max_{\substack{A \text{ an abelian, non-Frattini} \\ \text{chief factor of } G}} d((L_A)_{\delta_G(A)}).
		\]
		
		\item Let $A$ be an abelian, non-Frattini chief factor of $G$. Then,
		\begin{align*}
			d((L_A)_{\delta_G(A)}) & = h(A)\\
			& \leq \theta(A) + \delta_G(A),
		\end{align*}
		where
		\begin{align*}
			r(A) & = \dim_{\mathbb F}(A),\\
			s(A) & = \dim_{\mathbb F}(H^1(L_A/A,A)),\\ 
			\theta(A) & = \begin{cases*}
				0 \qquad & if $A$ is central,\\
				1 \qquad & otherwise.
			\end{cases*}\\
			h(A) & = \theta(A) + \left\lceil \frac{\delta_G(A) + s(A)}{r(A)} \right\rceil
		\end{align*}
		Additionally, if $A$ is central, then $d((L_A)_{\delta_G(A)}) = \delta_G(A)$.
	\end{enumerate}
\end{theorem}

In order to avoid repeated verification of the bound in (2), we give the following useful lemma. Note that all logarithms here and in the following discussion are base 2.

\begin{lemma}
	Suppose that $\delta_G(A) \leq 5$ for all non-abelian chief factors $A$ of $G$. Then the condition in (2) of Theorem \ref{theorem:crowns} holds.
\end{lemma}

\begin{proof}
	By \cite[Lemma 2.1]{MQR}, we have that
	\[
	\left| \text{Out}(S) \right| \leq \frac{6}{7} \log \left| S \right|,
	\]
	The function $2^z / z$ is monotonic for $z \geq 1$. Hence,
	\[
	2^z \geq \frac{60}{\log 60} z \qquad\qquad \forall z \geq \log 60.
	\]
	Since $S$ is a non-abelian simple group, we have that $\log (\left| S \right| ^n ) \geq \log 60$. Thus, letting $z = \log (\left| S \right| ^n )$, we obtain that
	\[
	\left| S \right| ^n \geq \frac{60}{\log 60} \log (\left| S \right| ^n ) \geq \frac{70}{\log 60} n \left| \text{Out} (S) \right|.
	\]
	As such,
	\[
	\frac{|S^n|}{2n |\text{Out}(S)|} \geq \frac{35}{\log 60} \approx 5.92 \geq 5
	\]
	for all non-abelian simple groups $S$, as required.
\end{proof}

We now establish several auxiliary results that are used throughout the remainder of this paper. In order to prove these results, we first require a slight generalisation of a result in \cite{G}. In the following, the abelian socle of $G$, written $\soc_{\operatorname{ab}}(G)$, is the product of the abelian minimal normal subgroups of $G$.

\begin{lemma}\label{lemma:gaschutz}
	Let $G$ be a group with trivial Frattini subgroup. Then $G$ splits over $\soc_{\operatorname{ab}}(G)$.
\end{lemma}

\begin{proof}
	Let $\{M_1, \dots, M_d\}$ be a set of maximal subgroups of $G$ of minimal size such that
	\[
	\bigcap_i \Core_G(M_i) = 1.
	\]
	Then $G$ embeds as a subdirect subgroup of $G_1 \times \cdots \times G_d$, where $G_i := G/\Core_G(M_i)$. Identifying $G$ with its image under this embedding, the minimality of $d$ ensures that $G\cap G_i>1$ for all $i$. Thus, each $G_i$ has a minimal normal subgroup which is contained in $G$. Without loss of generality, let $G_1, \dots, G_r$ be the $G_i$ for which this minimal normal subgroup is abelian. Since $G_i$ is a primitive permutation group, we deduce that $G_i$ is of affine type when $1 \leq i \leq r$. Hence, in this case there exist subgroups $V_i$, $L_i$ such that $G_i = V_i : L_i$, where $V_i$ is the unique minimal normal subgroup of $G_i$. Thus $G$ contains $V_1 \times \cdots \times V_r$, so, by Dedekind’s law, $G=(V_1 \times \cdots \times V_r):H$, where 
	\[
	H:=G\cap (L_1 \times \cdots \times L_r \times G_{r+1} \times \cdots \times G_d).
	\]
	By construction, $V_1 \times \cdots \times V_r$ is contained in the abelian socle of $G$. Additionally, by Goursat's Lemma, any minimal normal subgroup of $G$ which is abelian is contained in $G_1 \times \cdots \times G_r$. Therefore, since 
	\[
	\soc_{\operatorname{ab}}(G) \cap (G_1 \times \cdots \times G_r) = \soc_{\operatorname{ab}}(G \cap (G_1 \times \cdots \times G_r)) = V_1 \times \cdots \times V_r,
	\]
	we obtain that $V_1 \times \cdots \times V_r$ is the abelian socle of $G$.
\end{proof}

We now introduce a useful concept which allows us to compute $\delta_G(A)$ from the chief factors of a normal subgroup $N$ of $G$ and those of $G/N$. In practice, this significantly simplifies arguments involving the chief factors of $G$.

\begin{definition}
	Let $G$ be a group, with two normal subgroups $Y < X$. Then there exists a chief series of $G$ through both $X$ and $Y$, 
	\[
	1 < N_1 < \cdots < N_s = Y < N_{s+1} < \cdots < N_l = X < \cdots < G.
	\] 
	We say the chief factors $N_i/N_{i-1}$ for $s+1 \leq i \leq l$  are the \emph{chief factors of $G$ contained in $X/Y$}. We let $\delta_{G,X/Y}(A)$ be the number of non-Frattini chief factors of $G$ in $X/Y$ which are $G$-equivalent to $A$.
\end{definition}

This definition implies that, for any normal subgroup of $G$, we have $\delta_G(A) = \delta_{G,N}(A) + \delta_{G/N}(A)$. The following results provide an upper bound for $\delta_{G,N}(A)$ in terms of the non-Frattini chief factors of $N$.

\begin{prop}\label{prop:CFsFromNormalSubgp}
	Let $G$ be a group, $N$ a normal subgroup of $G$, and $A$ a chief factor of $G$ contained in $N$. Then:
	\begin{enumerate}[\upshape(1)] 
		\item The chief factors of $N$ in $A$ are all isomorphic;
		\item The chief factors of $N$ in $A$ are either all central or all non-central in $N$;
		\item The chief factors of $N$ in $A$ are either all Frattini or all non-Frattini in $N$; and
		\item If all of the chief factors of $N$ in $A$ are Frattini in $N$ then $A$ is Frattini in $G$.
	\end{enumerate}
\end{prop}

\begin{proof}
	Assume without loss of generality that $A$ is a minimal normal subgroup of $G$ in $N$. If $A$ is abelian, then, by Clifford's Theorem, there exist some minimal $N$-normal subgroups $X_i$ of $A$ such that $A = X_1 \times \cdots \times X_s$. Moreover for each $i$ there exists some $g \in G$ such that $X_1^g = X_i$. We show the same is true if $A$ is non-abelian. Let $A = S^b$ for some non-abelian simple group $S$, and write $A = S_1 \times \cdots \times S_b$, where each group $S_i$ is isomorphic to $S$. Thus, any normal subgroup of $A$ is equal to $\prod_{i \in I} S_i$ for some subset $I$ of $\{1, \dots, b\}$. Let $X_1$ be a minimal $N$-normal subgroup of $A$. Since $\langle X_1^g : g \in G \rangle = A$, and $X_1 \cap X_1^g$ is either trivial or equal to $X_1$, we may choose $g_1, \dots, g_s$ in $G$ such that 
	\[
	X_j := X_1^{g_j} = \prod_{i \in I_j} S_i
	\]
	for some subsets $I_j$ which partition $\{1, \dots, b\}$. 
	
	The sections 
	\[
	A_1 := X_1, A_2 := \frac{X_1 \times X_2}{X_1}, \dots, A_s := \frac{X_1 \times \cdots \times X_s}{X_1 \times \cdots \times X_{s-1}}
	\]
	form a set of chief factors of $N$ in $A$. Hence, we have proven (1) and (2), since the action of $N$ on $X_1$ is equivalent to the action of $N^g$ on $X_1^g$. 
	
	We next prove (3). Suppose that $A$ is an abelian chief factor of $G$, as otherwise the result is immediate. Since $\Phi(N)$ is characteristic in $N$, and $A$ is a minimal normal subgroup of $G$, the subgroup $\Phi(N)\cap A$ is either $A$ or $1$. If $\Phi(N)\cap A=A$, then $A_1,\dots,A_s$ are all Frattini in $N$. We may therefore assume that $\Phi(N)\cap A=1$, and we show that $A_1,\dots,A_s$ are all non-Frattini in $N$. 
	
	By the Correspondence Theorem, it is sufficient to prove that the chief factors of $N/\Phi(N)$ contained in $\Phi(N)A/\Phi(N) \cong A$ are non-Frattini. Thus, we may assume that $\Phi(N) = 1$. By Lemma \ref{lemma:gaschutz}, $N = \soc_{\operatorname{ab}}(N) : L$ for some subgroup $L$. Moreover, $A$ is an $N$-normal subgroup of $\soc_{\operatorname{ab}}(N)$, so $A$ is complemented in $\soc_{\operatorname{ab}}(N)$ by some $N$-normal subgroup $J$. Hence $N = A : JL$.  Therefore, for any $i$,
	\[
	\left( \prod_{j \neq i} X_j \right) : JL
	\]
	is a maximal subgroup of $N$ which contains $X_1 \times \cdots \times X_{i-1}$ but does not contain $X_1 \times \cdots \times X_i$. This implies that $A_i$ is non-Frattini in $N$, as required.
	
	If all of the chief factors of $N$ in $A$ are Frattini, then
	\[
	A \leq \Phi(N) \leq \Phi(G),
	\]
	proving (4).
\end{proof}

\begin{remark}
	Let $G$, $N$, and $A$ be as in Proposition \ref{prop:CFsFromNormalSubgp}, and let $A_1, \dots, A_s$ be the chief factors of $N$ in $A$. Note that it is not necessary for $A_i$ to be $N$-equivalent to $A_j$ for all $i,j$, unless $A_i$ is central. Indeed, consider the group $G = \Alt(4) \wr \Sym(5)$, with normal subgroup $N = \Alt(4)^5$. In $N$, each copy of the Klein 4-group $K$ is not $N$-equivalent to the rest, despite $K^5$ being a minimal normal subgroup of $G$.
\end{remark}

\begin{prop}\label{prop:CFequivalenceinN}
	Let $G$ be a group and $N$ a normal subgroup of $G$. Let $A$ and $B$ be two chief factors of $G$ contained in $N$, which are $G$-equivalent. If $\{A_1,..,A_s\}$ and $\{B_1,..,B_t\}$ are the $N$-chief factors contained in $A$ and $B$ respectively, then $s=t$ and, up to reordering, $A_i \equiv_N B_i$ for each $i$.
\end{prop}

\begin{proof}
	There exist isomorphisms $\phi \colon A \to B$ and $\Phi \colon A \rtimes G \to B \rtimes G$ such that
	\[
	\xymatrix{
		1 \ar[r] & A \ar[r]^-{\iota_1} \ar[d]^\phi & A \rtimes G \ar[r]^-{\pi_1} \ar[d]^\Phi & G \ar[r] \ar[d] ^{\text{id}}& 1\\
		1 \ar[r] & B \ar[r]^-{\iota_2} & B \rtimes G \ar[r]^-{\pi_2} & G \ar[r] & 1\\
	}
	\]
	is a commutative diagram. Thus,
	\[
	\Phi(A \rtimes N) = B \rtimes N,
	\]
	and, as such, $A \equiv_N B$. Let $X$ be an $N$-invariant subgroup of $A$, and let $Y := \phi(X)$. We show that $Y$ is an $N$-invariant subgroup of $B$. Let $a \in X$ and $n \in N$. Note that in $B \rtimes G$ we have that $(\phi(a)^n, 1)$ is equal to $(\phi(a),1)^{(1,n)}$. Hence, to show $Y$ is $N$-invariant, it is sufficient to show that $\iota_2(\phi(a))^{(1,n)}$ lies in $\iota_2(Y)$. By the commutative diagram,
	\[
	\iota_2(\phi(a))^{(1,n)} = \Phi((a, 1))^{(1,n)} \in \iota_2(B).
	\]
	Additionally, $\pi_2\Phi((y,m)) = m$, for any $m \in N$ and $y \in A$, so $\Phi^{-1}(1,n) = (x,n)$ for some $x \in A$. Thus, 
	\[
	\iota_2(\phi(a))^{(1,n)} = \Phi((a,1)^{(x,n)}).
	\]
	Now, $(a, 1)^{(x,n)} \in \iota_1(X)$, by definition, and hence $\Phi((a,1)^{(x,n)}) \in \iota_2(Y)$. Thus, $Y = \phi(X)$ is indeed an $N$-invariant subgroup of $B$. Similarly, if $Y$ is an $N$-invariant subgroup of $B$, then $\phi^{-1}(Y)$ is an $N$-invariant subgroup of $A$. Hence, given a maximal chain
	\[
	1 < X_1 < \cdots < X_s = A
	\]
	of $N$-invariant subgroups of $A$, we obtain a maximal chain
	\[
	1 < \phi(X_1) < \cdots < \phi(X_s) = B
	\]
	of $N$-invariant subgroups of $B$, and $X_i/X_{i-1} \equiv_N \phi(X_i)/\phi(X_{i-1})$.
\end{proof}

\begin{corollary}\label{cor:upperbounddeltaGN}
	Let $G$ be a group and let $N$ be a normal subgroup of $G$. Given a chief factor $A$ of $G$, let $A \cong S^b$ for some simple group $S$. Then,
	\[
	\delta_{G,N}(A) \leq \max_{\{N\text{-group }B \cong S^c, \, c \mid b\}} \delta_N(B).
	\]
\end{corollary}

\begin{proof}
	Let $A_1, \dots, A_d$ be the non-Frattini chief factors of $G$ in $N$ which are $G$-equivalent to $A$, so $d = \delta_{G,N}(A)$.  For each $i$, let $\{A_{i,1}, \dots, A_{i,k_i}\}$ be the set of chief factors of $N$ in $A_i$. Each $A_i$ is non-Frattini in $G$, so, by Proposition \ref{prop:CFsFromNormalSubgp}(3) and (4), each $A_{i,j}$ is non-Frattini in $N$. Additionally, by Proposition \ref{prop:CFequivalenceinN}, for each $i$ there exists some $j_i$ such that $A_{1,1} \equiv_N A_{i, j_i}$. Thus, $N$ contains at least $d$ non-Frattini chief factors $N$-equivalent to $A_{1,1}$, implying that
	\[
	\delta_{G, N}(A) \leq \delta_N(A_{1,1}) .
	\]
	Additionally, since $A_{1,1} \cong A_{1,i}$ for all $i$ by Proposition \ref{prop:CFsFromNormalSubgp}(1), we have that $|A_{1,1}|^{k_1} = |A|$. Hence $A_{1,1} \cong S^c$ for some $c \mid b$.
\end{proof}

\section{Maximal subgroups of finite groups}\label{section:prelimsmethod}
In this section, we describe the maximal subgroups of a finite group $G$, focusing in particular on the case where $G$ is a wreath product. We additionally describe the non-Frattini chief factors of certain maximal subgroups of such wreath products.

Suppose that $G$ is a finite group, with chief series
\[
1 = G_0 < G_1 < \cdots < G_m = G.
\]
For each $1 \leq i \leq m$, let $X_i = G_i/G_{i-1}$. Let $M$ be a maximal subgroup of $G$, and let $i$ be chosen minimally so that $M$ does not contain $G_i$. Hence $M$ contains $G_{i-1}$, and $M/G_{i-1}$ is a maximal subgroup of $G/G_{i-1}$ which does not contain $G_i/G_{i-1}$. Thus $MG_i=G$ and so, in particular, 
\begin{align*}
	\frac{M\cap G_j}{M\cap G_{j-1}} & \cong X_j \qquad \forall j \neq i,\\
	\frac{M\cap G_i}{M\cap G_{i-1}} & < X_i.
\end{align*}

Since each chief factor is a minimal normal subgroup in a quotient of $G$, each $X_i$ is isomorphic to $S^t$ for some simple group $S$ and some $t \in \mathbb{N}$. To classify the maximal subgroups of a general finite group, we are therefore led to consider the following problem: given any simple group $S$, an extension $S^t.C$ where $S^t$ is a minimal normal subgroup of this extension, and a maximal subgroup $M \mleq S^t.C$ which does not contain $S^t$, what are the possibilities for $M \cap S^t$? The main objective of this section is to answer this question.

If $S$ is abelian then the group $M \cap S^t$ must be a submodule of the irreducible $C$-module $S^t$, and hence $M \cap S^t$ is trivial. The situation is more complicated when $S$ is non-abelian simple. In this case, a complete classification is provided by Corollary \ref{cor:maximalsubgpextension} when $t=1$, and by Proposition \ref{prop:XSimpleMaxSubgrpXkC} when $t>1$. These results are summarised in Proposition \ref{prop:maximalsubgroupgeneralmethod}.

However, these structural descriptions are not by themselves sufficient to determine a chief series for $M$, since the sections $X_1,\dots,X_{i-1}$ need not remain chief factors of $M$. Computing the chief factors of $M$ is particularly difficult if $i < m$ and $G_i$ does not act trivially on $G_{i-1}$; see the discussion after Proposition \ref{prop:maximalsubgroupgeneralmethod} for more details. To compute these chief factors we require an understanding of the extension 
\[
	\frac{M}{G_{i-1}} = \left( \frac{M \cap G_i}{G_{i-1}} \right) . \left( \frac{G}{G_i} \right).
\]
Accordingly, throughout this section we prove supplementary results that provide information about such extensions when $G$ belongs to certain families of groups which will appear in future sections. For example, Proposition \ref{prop:maxsubgpdirectproduct} fully classifies $M$ when $G$ is a subgroup of a direct product. In Section \ref{section:wreathproducts} we then see how these results can be used to classify the chief factors of maximal subgroups of wreath products.

\subsection{Maximal subgroups of an extension \texorpdfstring{$S.C$}{S.C}}

We begin by classifying the maximal subgroups of an extension $S.C$, where $S$ is a non-abelian simple group. To do so, we first establish a preliminary result regarding the structure of the maximal subgroups of subgroups of direct products.

\begin{prop}\label{prop:maxsubgpdirectproduct}
	Let $X_1, X_2$ be two groups, and let $G$ be a subgroup of $X_1 \times X_2$. For each $i$, let $\pi_i$ be the quotient map $X_1 \times X_2 \to X_i$, and
	\[
	P_i := \pi_i(G), \quad N_i := G \cap X_{i}.
	\]
	Then, for any maximal subgroup $M$ of $G$, one of the following must hold:
	\begin{enumerate}[\upshape(1)]
		\item $M$ contains $N_i$ for some $i$, and 
		\[
		M = (P_i \times \pi_{3-i}(M)) \cap G,
		\]
		where $\pi_{3-i}(M)$ is maximal in $P_{3-i}$; or
		\item $M$ is subdirect in $P_1 \times P_2$, and
		\[
		M = (M_1 \times M_2) . C
		\]
		where $M_i$ are maximal $P_i$-normal subgroups of $N_i$, and $M_i . C \cong P_i$, for each $i$.
	\end{enumerate}
\end{prop}

\begin{proof}
	Assume that $M$ is subdirect in $P_1 \times P_2$, and let $M_i$ denote $M \cap N_i$ for each $i$. Fix $i \in \{1, 2\}$. If $M$ contains $N_i$, then, by Goursat's Lemma, $M/N_i$ and $G/N_i$ are both isomorphic to $P_{3-i}$. Hence, the order of $M$ is equal to that of $G$, which is a contradiction. Therefore, $M_i < N_i$, so $M_i$ is a maximal $M$-invariant subgroup of $N_i$. As $M$ is subdirect, this implies that $M_i$ is a maximal $P_i$-normal subgroup of $N_i$, and hence $M$ lies in (2) by Goursat's Lemma.
	
	If $M$ is not subdirect then we can assume without loss of generality that $\pi_1(M) < P_1$. As such, 
	\[
	\frac{M N_2}{N_2} \cong \frac{M}{N_2 \cap M} \cong  \pi_1(M) < P_1.
	\]
	In particular, $MN_2 < G$ and thus $M$ contains $N_2$. Hence, $\pi_1(M) = M/N_2$ is maximal in $P_1$. Additionally,
	\[
	M \leq (\pi_1(M) \times P_2) \cap G < G,
	\]
	so $M = (\pi_1(M) \times P_2) \cap G$.
\end{proof}

\begin{prop} \label{prop:maxsubgpsextensioncenter1}
	Let $G$ be a group of shape $X.C$, where the normal subgroup $X$ of $G$ has trivial centre. Let $M$ be a maximal subgroup of $G$, not containing $X$. Then, the following hold:
	\begin{enumerate}[\upshape(1)]
		\item $G$ is isomorphic to a subdirect subgroup of $\widetilde{X} \times C$, where $\widetilde{X} = G/C_G(X) \leq \Aut(X)$; and
		\item Identifying $G$ with its image under the embedding in (1), either: $M \cap X = \widetilde{M} \cap X$ for some maximal  subgroup $\widetilde{M}$ of $\widetilde{X}$; or $M\cap X$ is a maximal $\widetilde{X}$-normal subgroup of $X$, and $\widetilde X/(M \cap X)$ is isomorphic to a quotient of $C$.
	\end{enumerate}
\end{prop}

\begin{proof}
	Let $\pi$ be the quotient map $X.C \to C$. Consider the map
	\begin{align*}
		\phi \colon X.C & \to \Aut(X) \times C\\
		g & \mapsto ((x \mapsto x^g), \pi(g)).
	\end{align*}
	This map is injective, since $Z(X) = 1$. By definition the projection of $\phi(G)$ to each coordinate is $\widetilde X$ and $C$ respectively, thus proving (1). Additionally, we have $\phi(G) \cap \widetilde{X} = \phi(X)$ and $\phi(G) \cap C = \phi(C_G(X))$. Hence we can apply Proposition \ref{prop:maxsubgpdirectproduct}, where $P_1 = \widetilde X$, $P_2 = C$, $N_1 = \phi(X)$, and $N_2 = \phi(C_G(X))$. If $\phi(M)$ is in case (1) of Proposition \ref{prop:maxsubgpdirectproduct}, then, using the fact that $\phi(M)$ does not contain $\phi(X)$, we have that $\phi(M) = (\widetilde M \times C) \cap \phi(G)$ for some maximal subgroup $\widetilde M$ of $\widetilde X$. Hence $\phi(M) \cap \phi(X) = \widetilde M \cap \phi(X)$, as required. Similarly, if $\phi(M)$ is in case (2), then $\phi(M) \cap \phi(X)$ is a maximal $\widetilde X$-normal subgroup of $\phi(X)$. By the final condition of (2) we must have that $\widetilde X/(M \cap X)$ is isomorphic to some quotient of $C$.
\end{proof}

To summarise what we require from the result above for our classification result, we introduce the following definitions.

\begin{definition}
	Let $G$ be an almost simple group, with socle $S$. We say a proper, non-maximal subgroup $M$ of $G$ is a \textit{novelty maximal subgroup} of $G$ if there exists some almost simple group $K$ with $G \leq K \leq \Aut(S)$ and a maximal subgroup $\widetilde M \mleq K$ such that $\widetilde M \cap G = M$.
\end{definition}

\begin{definition}
	Let $X$ be a quasisimple group. We say a proper, non-maximal subgroup $M$ of $X$ is a \textit{novelty maximal subgroup} of $X$ if $M$ contains $Z(X)$ and $M/Z(X)$ is a novelty maximal subgroup of $X/Z(X)$.
\end{definition}

\begin{corollary}\label{cor:maximalsubgpextension}
	Let $X$ be an almost simple or quasisimple group, and $X.C$ an extension of $X$. Suppose that $M$ is a maximal subgroup of $X.C$. Assume that $M$ does not contain $\soc(X)$ when $X$ is almost simple, and does not contain $X$ when $X$ is quasisimple. Then $M \cap X$ is either a maximal subgroup of $X$, a novelty maximal subgroup of $X$, or is central in $X$. Moreover, if $C$ has no quotients isomorphic to $X/Z(X)$, then $M \cap X$ is not central.
\end{corollary}

\begin{proof}
	Suppose that $X$ is an almost simple group. If $M \cap X$ is not normal in $G$, then, by Proposition \ref{prop:maxsubgpsextensioncenter1}, $M \cap X$ is equal to $\widetilde M \cap X$, where $\widetilde M$ is a maximal subgroup of $\widetilde X$. Since $\widetilde X$ is also an almost simple group, $M$ is a maximal or novelty maximal subgroup of $X$. Otherwise, if $M \cap X$ is $G$-normal, then $M \cap X$ is trivial so the result follows directly by the statement of Proposition \ref{prop:maxsubgpsextensioncenter1}.
	
	If $X$ is quasisimple, then $Z(X)=\Phi(X)$, so $Z(X) \leq \Phi(X.C)$. Hence, any maximal subgroup of $X.C$ contains the centre of $X$. Thus, we project onto $\overline{X} = X/Z(X)$, and apply Proposition \ref{prop:maxsubgpsextensioncenter1} to this non-abelian simple group, giving the desired result.
\end{proof}

\subsection{Maximal subgroups of an extension \texorpdfstring{$S^t.C$}{St.C} with \texorpdfstring{$t \geq 2$}{t >= 2}}\label{section:StC}

We next consider extensions of the form $S^t.C$ with $t \geq 2$, where $S$ is a non-abelian simple group, and $S^t$ is a minimal normal subgroup of $S^t.C$. Since there exists a (possibly non-injective) homomorphism from $S^t . C$ into $\Aut(S^t) = \Aut(S) \wr \Sym(t)$, we begin by considering wreath products more generally. 

Let $G$ be a subgroup of a wreath product $K \wr \Sym(t)$ for some group $K$ and some $t \in \mathbb{N}$. Let $X$ be a normal subgroup of $K$ such that $X^t \leq G$. For each $i$, let $\iota_i$ denote the natural injection map $K \to K^t$ given by
\[
\iota_i \colon k \mapsto (1, \dots, k, \dots, 1),
\]
where $k$ is in the $i$th position. Additionally, let $\rho_i$ denote the projection map $K^t \to K$ given by 
\[
\rho_i \colon (k_1, \dots, k_t) \mapsto k_i.
\]
Let $\pi$ denote the quotient map $K \wr \Sym(t) \to \Sym(t)$.

Fix $i \in \{1, \cdots, t\}$. The subgroup $N_G(\iota_i(K))$ contains $X^t$, and $\prod_{j \neq i} \iota_j(X)$ is normal in $N_G(\iota_i(K))$. Let
\[
R_i := N_G(\iota_i(K)) \Big/ \prod_{j \neq i} \iota_j(X),
\]
and let $\beta_i$ denote the map $\beta_i \colon N_G(\iota_i(K)) \to R_i$. Note that $\beta_i(\iota_i(X))$ is isomorphic to $X$, so we write $X$ for the normal subgroup $\beta_i(\iota_i(X)) $ of $ R_i$.

\begin{prop}\label{prop:XNonsimpleMaxSubgpXkC}
	Let $G$, $X$ and $K$ be as above. Assume that $\pi(G)$ is transitive, and that $M$ is a maximal subgroup of $G$ with $M \cap X^t$ non-subdirect. Then,
	\[
	M \cap X^t = M_1 \times \cdots \times M_t
	\]
	for some subgroups $M_i$ of $X$ which are pairwise conjugate in $K$. Additionally, there exist maximal subgroups $\widetilde M_i$ of $R_i$ such that $M_i = \widetilde M_i \cap X$.
\end{prop}

\begin{proof}
	Note that $M \cap X^t$ is a proper subgroup of $X^t$, and thus
	\[
	M = (M \cap X^t) . (G/X^t).
	\]
	Since $\pi(G)$ is a transitive subgroup of $\Sym(t)$, the groups $\rho_i(M \cap X^t)$ are all conjugate in $K$, and the same is true for $M \cap \iota_i(X)$.
	
	As $M \cap X^t$ is not subdirect,
	\[
	M \cap X^t \leq \rho_1(M \cap X^t) \times \cdots \times \rho_t(M \cap X^t) < X^t.
	\]
	Hence, since $M \cap X^t$ is a maximal $M$-normal subgroup of $X^t$, it follows that $M \cap X^t = \prod \iota_i\rho_i(M \cap X^t)$. For each $i$, let $M_i=\rho_i(M \cap X^t) < X$ and let $\widetilde M_i = \beta_i(N_M(\iota_i(K)))$. Recall that we write $X$ for the image of $\iota_i(X)$ under $\beta_i$, and so
	\begin{align*}
		\widetilde M_i \cap X & = \beta_i(N_M( \iota_i(K))) \cap X\\
		& = \beta_i(N_M(\iota_i(K)) \cap \beta_i^{-1}(X))\\
		& = \beta_i(N_M(\iota_i(K)) \cap X^t)\\
		& = \beta_i(M \cap X^t) \\
		& = M_i.
	\end{align*} 
	In particular, $\widetilde M_i$ is a proper subgroup of $R_i$. It remains to show, without loss of generality, that $\widetilde{M}_1$ is a maximal subgroup of $R_1$. Let $D := \langle \prod_{j \neq 1} \iota_j(X) \rangle$, and let $J$ be a group with $N_M(\iota_1(K))D \leq J < N_G(\iota_1(K))$. We show that $J$ is equal to $N_M(\iota_1(K))D$, proving the desired result. The group $J$ does not contain $X^t$, since $X^t N_M(\iota_1(K))=N_G(\iota_1(K))$, so
	\[
	J \cap X^t = J_1 \times X \cdots \times X
	\]
	for some proper subgroup $J_1$ of $X$ which contains $M_1$. Additionally, $J \cap X^t$ is a normal subgroup of $J$, and thus is normalised by $N_M(\iota_1(K))$. Since $\pi(M)$ is a transitive subgroup of $\Sym(t)$, there exists a right transversal $\{m_1 = 1, m_2, \dots, m_t \}$ of $N_M(\iota_1(K))$ in $M$ such that $\iota_1(K)^{m_i} = \iota_i(K)$ for each $i$. Let $m_i = k_i \sigma_i$ for each $i$, for some $k_i \in K^t$ and $\sigma_i \in \Sym(t)$. Thus, the normal closure of $J \cap X^t$ in $M$ is
	\[
	(J \cap X^t)^M = \prod_i \iota_1(J_1)^{m_i} = J_1 \times J_1^{\rho_1(k_2)} \times \cdots \times J_1^{\rho_1(k_t)}.
	\]
	Moreover, $(J \cap X^t)^M$ is a proper $M$-normal subgroup of $X^t$, so it follows that $(J \cap X^t)^M = M_1 \times \cdots \times M_t$. However, this implies that $J_1 = M_1$, and hence $J = N_M(K_1)D$, as required.
\end{proof}

Having treated the non-subdirect case, we now consider the case where $G \cap X^t$ is subdirect. Here we assume that $X$ is simple and apply the following result.

\begin{lemma}[{\cite[Theorem 4.16(iii)]{CS}}]\label{lemma:subdirectsubgpXk}
	Let $S$ be a non-abelian simple group, and $t \geq 1$. Let $L$ be a subdirect subgroup of $S^t$. Then $L \cong S^l$ for some $l \leq t$.
\end{lemma}

The preceding result, together with Proposition \ref{prop:XNonsimpleMaxSubgpXkC}, yields the following proposition.

\begin{prop}\label{prop:XSimpleMaxSubgrpXkC}
	Let $S$ be a non-abelian simple group. Let $G=S^t.C$, with $S^t$ a minimal normal subgroup of $G$, and let $M$ be a maximal subgroup of G. Then one of the following holds:
	\begin{enumerate}[\upshape(1)]
		\item $M = S^t . J$ for some maximal subgroup $J$ of $C$;
		\item $M = (M_1 \times \cdots \times M_t) . C$, for some $M_i$ which are maximal or novelty maximal subgroups of $S$, and are all conjugate in $\Aut(S)$; or
		\item $M \cong S^\ell . C$, $0 \leq \ell < t$.
	\end{enumerate}
	
	Additionally, case (3) with $\ell > 0$ occurs only if there exists a partition $\Phi$ of $\{1, \dots, t\}$ of size $\ell$ which is normalised by $\pi(S^t.C)$.
\end{prop}

\begin{proof}
	If $M$ contains $S^t$, then $M$ is in (1). Hence, we assume that $M \cap S^t < S^t$. Additionally, the case where $M \cap S^t$ is trivial is covered by case (3) with $\ell = 0$, so we assume that $M \cap S^t$ is non-trivial for the remainder of this proof.
	
	By Proposition \ref{prop:maxsubgpsextensioncenter1}, we can embed $G$ as a subgroup of $\Aut(S^t) \times C$. Moreover, $M \cap S^t$ is either a $G$-normal subgroup of $S^t$, or is equal to $\widetilde M \cap S^t$, for some maximal subgroup $\widetilde M$ of the projection of $G$ to $\Aut(S^t)$. Since the only proper $G$-normal subgroup of $S^t$ is the trivial group, we may assume that $M$ is of the latter form. Thus, to determine the structure of $M \cap S^t$, we may assume that $G \leq \Aut(S^t) = \Aut(S) \wr \Sym(t)$. Since $S^t$ is a minimal normal subgroup of $G$, $\pi(G)$ is a transitive subgroup of $\Sym(t)$.
	
	Suppose that $M \cap S^t$ is non-subdirect. For each $i$, let $R_i$ be defined as in Proposition \ref{prop:XNonsimpleMaxSubgpXkC}, and recall that $S$ is a normal subgroup of $R_i$. For each $i$, let $\widetilde M_i$ be a maximal subgroup of $R_i$ as defined previously, and let $M_i = \widetilde M_i \cap S$ so that
	\[
	M = \left( M_1 \times \cdots \times M_t \right) . C.
	\]
	Additionally, recall that the subgroups $M_i$ are pairwise conjugate in $\Aut(S)$. Since $M \cap S^t$ is a non-trivial subgroup, each group $M_i$ is either a maximal subgroup, or a novelty maximal subgroup of $S$, by Corollary \ref{cor:maximalsubgpextension}.
	
	If $M\cap S^t$ is instead subdirect, then Lemma \ref{lemma:subdirectsubgpXk} implies that $M  \cap S^t =  S^\ell$ for some $\ell > 0$. Then, by \cite[Theorem 4.16]{CS}, there exists a partition $\Phi$ of $\{1, \dots, t\}$ of size $\ell$ such that
	\[
	M \cap S^t = \prod_{B \in \Phi} \diag \left( \prod_{i \in B} \iota_i(S) \right),
	\]
	where
	\[
	\diag \left( \prod_{i \in B} \iota_i(S) \right) = \left\{ \prod_{i \in B} \iota_i(x^{\phi_i}) : x \in S \right\}
	\]
	for some automorphisms $\phi_i \in \Aut(S)$. Since this group is $M$-normal, the subgroup $\pi(M) = \pi(G)$ of $\Sym(t)$ must stabilise $\Phi$.
\end{proof}

This concludes the analysis of the maximal subgroups of $S^t.C$ which do not contain $S^t$, in the case that $S$ is a non-abelian simple group, and $S^t$ is a minimal normal subgroup of $S^t.C$. We summarise these results in the following proposition.

\begin{prop}\label{prop:maximalsubgroupgeneralmethod}
	Suppose that $G$ is a group with chief series
	\[
	1 = G_0 < G_1 < \cdots < G_l = G.
	\]
	For any maximal subgroup $M$ of $G$ there exists some $i$ such that $M$ contains $G_{i-1}$ but does not contain $G_i$. Let $G_i/G_{i-1} \cong S^t$ for some simple group $S$ and $t \in \mathbb N$. Then $M$ has a normal series
	\[
	1 = G_0 < G_1 < \cdots < G_{i-1} < M \cap G_i < \cdots < M \cap G_l = M,
	\]
	where $(M \cap G_j)/(M \cap G_{j-1}) \cong G_j/G_{j-1}$ for all $j > i$, and one of the following holds:
	\begin{enumerate}[\upshape(1)]
		\item $(M \cap G_{i})/G_{i-1}$ is trivial;
		\item $S$ is non-abelian and $(M\cap G_{i})/G_{i-1} \cong L^t$ for some maximal or novelty maximal subgroup $L$ of $S$; or
		\item $S$ is non-abelian and $(M \cap G_{i})/G_{i-1} \cong S^d$ for some $d \mid t$.
	\end{enumerate}
\end{prop}

In the remainder of this paper, we determine the chief factors of such maximal subgroups $M$, which is essential for applying the theory of crowns. Following the notation of the above proposition, it is fairly straightforward to compute a chief series of $M/G_{i-1}$, but determining the chief factors of $M$ contained in $G_{i-1}$ is more challenging. These chief factors depend on the structure of the group $G$, as well as on the group $\frac{M \cap G_i}{G_{i-1}}$. However, they also depend on the extension $(M \cap G_i).(G/G_i)$, which in general can be difficult to study.

To remedy this, we consider the maximal subgroups of $G$ where $G$ belongs to a specific family of groups. For instance, we apply Proposition \ref{prop:maxsubgpdirectproduct} when $G$ is a subgroup of a direct product, and apply Proposition \ref{prop:XNonsimpleMaxSubgpXkC} when $G$ is a subgroup of a wreath product. In the latter setting, however, we obtain more precise results when considering only those groups $G$ whose intersection with the base group is a subdirect subgroup.

\begin{prop} \label{prop:subdirectwreathcase}
	Let $G \leq K \wr \Sym(t)$, and let $X$ be a normal subgroup of $K$ such that $X^t \leq G$. Let $\{N_1, \dots, N_d\}$ be the set of maximal $K$-normal subgroups of $X$. Assume that $\pi(G)$ is transitive, and that $G \cap K^t$ is subdirect. Any maximal subgroup $M$ of $G$ with $M \cap X^t$ proper and subdirect in $X^t$ contains $(\bigcap_{i=1}^d N_i)^t$.
\end{prop}

\begin{proof}
	The group $M \cap X^t$ is subdirect, and $M \cap K^t$ projects onto the subdirect subgroup $\frac{G \cap K^t}{X^t}$ of $(K/X)^t$, so $M \cap K^t$ is subdirect in $K^t$. Since $M \cap K^t$ is a proper subgroup of $G \cap K^t$, there exists a maximal subgroup $\widetilde{M}$ of $G \cap K^t$ containing $M \cap K^t$. Note that $\widetilde M$ is subdirect in $K^t$, and does not contain $X^t$, since $M X^t = G$.
	
	Let $L_i := \rho_i(\widetilde M \cap \iota_i(X))$ for each $i$. If $L_i$ is a proper subgroup of $X$, then $L_i$ is a maximal $K$-normal subgroup of $X$, since $\widetilde M$ is a maximal, subdirect subgroup of $K^t$. Otherwise, $L_i = X$, and thus contains all maximal $K$-normal subgroups of $X$. Hence, each $L_i$ contains some maximal $K$-normal subgroup $N_{j_i}$ of $X$.
	
	Now, since $M$ is maximal, and
	\[
	M \cap X^t \leq \bigcap_{m \in M} (\widetilde M \cap X^t) ^m < X^t,
	\]
	we have
	\[
	M \cap X^t = \bigcap_{m \in M} (\widetilde M \cap X^t) ^m.
	\]
	Let $m$ be an element of $M$, and write $m = k \sigma$ for some $k \in K$ and $\sigma \in \Sym(t)$. Since $\widetilde M \cap X^t$ contains $N_{j_1} \times \cdots \times N_{j_t}$, we have
	\[
	(\widetilde M \cap X^t)^m \geq N_{j_{\sigma^{-1}(1)}} \times \cdots \times N_{j_{\sigma^{-1}(t)}}.
	\]
	Thus, $M \cap X^t$ contains $(\bigcap_{k=1}^t N_{j_k})^t$, which contains $(\bigcap_{j=1}^d N_j)^t$, as required.
\end{proof}

\begin{prop} \label{prop:nonsubdirectwreathcase}
	Let $G, X$ and $K$ be defined as previously, with $\pi(G)$ transitive and $G \cap K^t$ subdirect. Let $M$ be a maximal subgroup of $G$ with $M \cap X^t$ not subdirect, and not normal in $K^t$. Then there exists some $g \in G$  and some maximal subgroup $M_1$ of $K$ such that
	\[
	M^g = ( M_1 \wr \Sym(t)) \cap G.
	\]
\end{prop}

\begin{proof}
	As in the proof of Proposition \ref{prop:XNonsimpleMaxSubgpXkC}, $M \cap X^t$ is equal to $\rho_1(M \cap X^t) \times \cdots \times \rho_t(M \cap X^t)$. Additionally, the groups $\rho_i(M \cap X^t)$ are pairwise conjugate in $K$, and are not $K$-normal, since $M \cap X^t$ is not normal in $K^t$.
	
	Suppose that $M \cap K^t$ is subdirect in $K^t$. Then, by Goursat's Lemma, $\rho_1(M \cap \iota_1(K))$ is a normal subgroup of $K$, and its intersection with $X$ is equal to $\rho_1(M \cap X^t)$. However, $\rho_1(M \cap X^t)^K$ is contained in $X$, since $X \unlhd K$, and is not equal to $\rho_1(M \cap X^t)$. Thus, any normal subgroup $N$ of $K$ which contains $\rho_1(M \cap X^t)$ has $N \cap X > \rho_1(M \cap X^t)$, contradicting the statement above. Hence $M \cap K^t$ is not subdirect in $K^t$.
	
	Let $M_i := \rho_i(M \cap K^t)$, so
	\[
	M \cap K^t = \left( M_1 \times \cdots \times M_t \right) \cap G.
	\]
	Note that $M_i \cap X = \rho_i(M \cap X^t)$ and hence each $M_i$ is a non-normal subgroup of $K$. Since $M/(M \cap X^t ) \cong G/X^t$ and $G/X^t \cap (K/X)^t$ is a subdirect subgroup of $(K/X)^t$, we have that $M_iX = K$.
	
	Let $\sigma_i$ be an element of $\pi(G)$ which maps $1$ to $i$. Then, there exists some $h_i \in K^t$ such that $h_i \sigma_i \in G$. Additionally, since $G \cap K^t$ is subdirect, we may assume that $\rho_1(h_i)$ is trivial for each $i$. Since $X^t M = G$, there exists some $m_i \in X^t$ with $m_i h_i \sigma_i \in M$, and hence $ M_1 ^{\rho_1(m_i)}$ is equal to $ M_i$. Let
	\[
	g := (1, \rho_1(m_2 ^{-1}), \dots, \rho_1(m_t ^{-1})) \in X^t \leq G.
	\]
	Note that $M^g \cap K^t = M_1^t \cap G$. We claim that $M^g = ( M_1 \wr \Sym(t)) \cap G$. Indeed, consider any element $h$ of $M^g$. We may write $h = k \sigma$ for some $k \in K^t$ and $\sigma \in \Sym(t)$. Then, for any $i$, we have that $ M_1^{\rho_i(k)} =  M_1$, and so $\rho_i(k) \in N_K( M_1)$. Thus $M^g$ is contained in $(N_K(M_1) \wr \Sym(t)) \cap G$. Since $N_K(M_1) < K$, and $M_1X = K$, we have that $N_K(M_1) \cap X < X$, so
	\[
	M^g \leq (N_K(M_1) \wr \Sym(t)) \cap G < G.
	\]
	This implies that $M^g = (N_K(M_1) \wr \Sym(t)) \cap G$. If $M_1 < N_K(M_1)$, then, by the same reasoning, we have that $M^g \cap X^t < (N_K(M_1) \cap X)^t$, which is a contradiction. So $M_1 = N_K(M_1)$, and hence $M^g = ( M_1 \wr \Sym(t)) \cap G$.
	
	Finally, we show that $M_1$ is a maximal subgroup of $K$. Suppose otherwise, so there exists a subgroup $J$ of $K$ with $M_1 < J < K$. This implies that
	\[
	M^g < (J \wr \Sym(t)) \cap G < G,
	\]
	as above, which is a contradiction.
\end{proof}

\subsection{Non-Frattini chief factors of wreath products}\label{section:wreathproducts}

The aim of this section is to determine the chief factors of certain families of subgroups of wreath products, extending \cite[Lemma 3.9]{LMT}. This will be of particular use when considering second maximal subgroups which are maximal in a subgroup of a wreath product; see for instance Sections \ref{section:classicalsC2} and \ref{section:altcase2}.

Retaining the notation of \cite[Lemma 3.9]{LMT}, let $E$ be a group and let $t \geq 2$ be an integer. Consider a subgroup $H$ of $E \wr \Sym(t)$, and let $F$ be a normal subgroup of $E$ such that $F^t \leq H$. Lemma 3.9 of \cite{LMT} can then be applied, for example, in order to compute the chief factors of $H$ when $E = F$ and $H$ contains $\Alt(t)$.

However, there are subgroups of wreath products to which this lemma does not apply, as it requires conditions (b)(ii) and (b)(iii) of \cite[Definition 3.8]{LMT} to hold. These state, respectively, that $\Phi(E) \cap F = \Phi(F)$ and $\delta_{E,F}(W) \leq 1$ for all non-Frattini chief factors $W$ of $E$. These properties frequently fail in the groups we consider in this paper. In this section, we impose no conditions on $E$ and $F$. Instead, we restrict our attention to certain families of subgroups $H$ of $E \wr \Sym(t)$. We then classify the chief factors of $H$ contained in $F^t$, and give conditions under which such chief factors are Frattini in $H$. Additionally, we determine which chief factors of $H$ contained in $F^t$ can be $H$-equivalent.

Consider a chief series
\[
1 = F_0 < F_1 < \cdots < F_r = F = E_0 < E_1 < \cdots < E_s = E
\]
of $E$, through $F$. We can obtain a normal series of $H$ given by
\[
1 = F_0^t < F_1^t < \cdots < F_r^t = F^t = E_0^t \cap H < E_1^t \cap H < \cdots < E_s^t \cap H = E^t \cap H.
\]
Hence, each chief factor of $H$ in $F^t$ is contained in $F_i^t/F_{i-1}^t$ for some $i$. In the following discussion we determine the chief factors of $H$ in each $F_i^t/F_{i-1}^t$, and determine conditions for such chief factors to be $H$-equivalent.

Let $A = F_i/F_{i-1}$. When $A$ is abelian, $A$ is isomorphic to $\mathbb F_p^b$ for some prime $p$ and some $b \in \mathbb N$. In the remainder of this section we use additive notation for $A$ and $A^t$ when $A$ is abelian.

\begin{definition}
	Let $X$ be an abelian group, and $t \geq 2$. The \textit{deleted subgroup} of $X^t$ is
	\[
	\del(X^t) := \{ (x_1, \dots, x_t) \in X^t : \sum_{i=1}^t x_i = 1 \}.
	\]
\end{definition}

Additionally, in the remainder of this section, when $\diag(X^t)$ is referenced for any group $X$ and $t>2$, we specifically refer to the diagonal subgroup
\[
\diag(X^t) = \{ (x, \dots, x) : x \in X \}
\]
of $X^t$.

Note that the exponent of a group $G$, written $\exp(G)$, is the smallest $n$ such that $g^n = 1$ for all $g \in G$.

\begin{prop}\label{prop:t5CFsinwreathproduct}
	Let $t \geq 5$. Assume that $H \cap E^t$ is subdirect, and that $H$ contains $\Alt(t)$. Let $A, B$ be chief factors of $E$ in $F$. Then either:
	\begin{enumerate}[\upshape(1)]
		\item $A^t$ is a chief factor of $H$; or
		\item $A$ is central in $F$, and either:
		\begin{enumerate}[\upshape(i)]
			\item $\exp(A) \nmid t$ and $A^t = \diag(A^t) \times \del(A^t)$, where $\diag(A^t)$ and $\del(A^t)$ are minimal $H$-normal subgroups; or
			\item $\exp(A) \mid t$ and the chief factors of $H$ in $A^t$ are $\diag(A^t)$, $\del(A^t)/\diag(A^t)$, and $A^t/\del(A^t)$, the first two of which are Frattini in $H$.
		\end{enumerate}
	\end{enumerate}
	Additionally, if two chief factors of $H$ contained in $A^t$ and $B^t$ respectively are $H$-equivalent, then $A \equiv_E B$.
\end{prop}

Which chief factors can arise depends on how large $H \cap E^t$ is. For example, if $E/F$ is abelian and $(H \cap E^t)/F^t$ is equal to $\del((E/F)^t)$, then the first case occurs if and only if $F_i/F_{i-1}$ is non-central in $E$. On the other hand, if $(H \cap E^t)/F^t$ is equal to $\diag((E/F)^t)$ and $H$ contains $\Sym(t)$, then the first case occurs if and only if $F_i/F_{i-1}$ is non-central in $F$.

\begin{proof}
	Let $A$ be a minimal normal subgroup of $E$ contained in $F$. If $A$ is nonabelian, then $A^t$ is a minimal normal subgroup of H. So assume that $A$ is abelian, and let $K$ be an $H$-normal, non-trivial, proper subgroup of $A^t$. We show that $K$ is one of $\diag(A^t)$ or $\del(A^t)$. Let $a = (a_1, \dots, a_t)$ be a non-trivial element of $K$. Thus, $K$ contains
	\[
	a - a^{(1, 2, 3)} - (a - a^{(1, 2, 3)})^{(1,4,5)} = (a_1 - a_2, 0 , 0, a_2 - a_1, 0, \dots, 0),
	\]
	and hence contains $(a_1 - a_2, a_2 - a_1, 0, \dots, 0)$. Suppose that $a_1 - a_2 = 0$ for all $a \in K$. Then $K$ is contained in $\diag(A^t)$, since $\Alt(t)$ is $2$-transitive. Let $(a, \dots, a)$ be a non-trivial element of $K$ for some $a \in A$. For any $e_1 \in E$, there exist some elements $e_2, \dots, e_t$ in $E$ such that $(e_1, \dots, e_t)$ is in $H \cap E^t$. Thus $K$ contains $(a^{e_1}, \dots, a^{e_t})$. Since $A$ is a minimal normal subgroup of $E$, this implies that $|K| \geq |A|$, so $K = \diag(A^t)$.
	
	On the other hand, if there exists some $a$ for which $a_1 - a_2 \neq 0$, then we have
	\[
	((a_1 - a_2)^{e_1}, (a_2 - a_1)^{e_2}, 0, \dots, 0)
	\]
	contained in $K$, where $e_i$ are defined as above. Hence, since $t \geq 5$,
	\[
	((a_1 - a_2)^{e_1}, 0, (a_2 - a_1)^{e_1}, 0, \dots, 0)
	\]
	is also contained in $K$. Thus $K$ contains $\del(A^t)$. If $K$ strictly contains $\del(A^t)$ then one can similarly show that $K$ is equal to $A^t$, so $K$ is one of $\diag(A^t)$ or $\del(A^t)$.
	
	The groups $\diag(A^t)$ and $\del(A^t)$ are $H$-normal if and only if the following holds: for any $k \sigma \in H$ with $k \in E^t$ and $\sigma \in \Sym(t)$, the image of $k$ in $\Aut(A)^t$ is equal to $(\phi, \dots, \phi)$ for some isomorphism $\phi$ of $A$. If neither of these groups is $H$-normal, case (1) holds. Thus, suppose that both $\diag(A^t)$ and $\del(A^t)$ are normal in $H$. In particular, $A$ is a central section of $F$. Note that $\diag(A^t) \leq \del(A^t)$ if and only if $\exp(A) \mid t$. Furthermore, if $\diag(A^t) \leq \del(A^t)$ then $\diag(A^t)$ and $\del(A^t)/\diag(A^t)$ are not complemented in $A^t$ by an $H$-normal subgroup, and so are Frattini chief factors of $H$. Hence, case (2)(ii) holds. On the other hand, if $\diag(A^t)$ is not contained in $\del(A^t)$ then $A^t = \diag(A^t) \times \del(A^t)$, so case (2)(i) holds.
	
	If case (1) holds, there exists an element $k \sigma$ of $H$ such that $k = (k_1, \dots, k_t) \in E^t$ where some $k_i, k_j$ have different images in $\Aut(A)$. By conjugating by an element of $\Alt(t)$, we can assume that $k_1, k_2$ have different images in $\Aut(A)$. Additionally, by multiplying on the right by some $\tau \in \Alt(t)$ we can assume that $\sigma$ is equal to $1$ or $(4,5)$. Hence,
	\[
	[k \sigma, (1,2,3)] = (k_1^{-1} k_2, k_2^{-1} k_3, k_3^{-1} k_1, 1, \dots, 1) \in H.
	\]
	Thus, for any $i \neq j$, the group $H \cap E^t$ contains an element $(e_1, \dots, e_t)$ where $e_i$ acts non-trivially on $A$ and $e_j = 1$.
	
	Let $A = F_i/F_{i-1}$ and $B = F_j / F_{j-1}$ and assume that $A^t$ and $B^t$ contain $H$-chief factors $\widetilde A, \widetilde B$ which are $H$-equivalent. We show that this implies that $A$ and $B$ are $E$-equivalent. Let $\phi, \Phi$ be isomorphisms as given in Definition \ref{def:CFsEquivalence}, and write $A^t = A_1 \times \cdots A_t$ and $B^t = B_1 \times \cdots \times B_t$ for some groups $A_i$, $B_i$ isomorphic to $A$, $B$ respectively.
	
	To begin with, we show that $\widetilde A = A^t$ if and only if $\widetilde B = B^t$. Suppose for a contradiction that $\widetilde A$ is not equal to $A^t$ while $\widetilde B = B^t$. Note that this implies that $A$ and $B$ are abelian. For any element $e = (e_1, \dots, e_t)$ of $H \cap E^t$, if $e_1$ centralises $A$ then $e$ centralises $A^t$ and thus centralises $\widetilde A$ and $\widetilde B$. However, since $\widetilde B = B^t$, there exists some element $e \in H \cap E^t$ for which $e_1 = 1$ and $e_2$ does not centralise $B$. Hence $e$ centralises $\widetilde A$ but not $\widetilde B$, which is a contradiction.
	
	Assume that $\widetilde A = A^t$ and $\widetilde B = B^t$. We show that $\phi(A_1) = B_1$, and so, without loss of generality, $\phi(A_i) = B_i$ for all $i$.
	
	If $A$ is abelian then $\phi$ is an $(H \cap E^t)$-invariant isomorphism. Since $\widetilde A = A^t$, the group $H \cap E^t$ contains some element $e$ for which $e_1 = 1$ and $e_t$ acts non-trivially on $A$. Since $A_1$ is centralised by $(2, i, j)$ for any $i, j > 2$, we have that $\phi(a_1, 0, \dots, 0)$ is equal to $(b_1, b_2, b_2, \dots, b_2)$ for some $b_1, b_2 \in B$. Assume that there exists some $a_1 \in A$ such that the second coordinate of $\phi(a_1, 0, \dots, 0) $ is non-zero. The group $\phi(A_1)$ is $(H \cap E^t)$-normal, so this implies that $\phi(A_1)$ projects onto $B_i$ for any $i > 1$. However, $A_1$ is centralised by $e$, so $e_i$ acts trivially on $B$ for any $i > 1$. Hence, $e$ centralises $B^t$, as $e_1 = 1$, but does not centralise $A^t$, so we have a contradiction. Thus, $\phi(A_1) = B_1$.
	
	If $A$ is non-abelian, then it is a product of non-abelian simple groups. Hence, $\phi$ must map each $A_i$ to some $B_j$. Suppose that $A_1$ maps to $B_i$ where $i \neq 1$. Let $j, k$ be distinct elements of $\{1, \dots, t\} \setminus \{1, i\}$. Then, by the commutative diagram in Definition \ref{def:CFsEquivalence}, we have that
	\[
	(\phi(A_1^{(1,j,k)}),1) = (\phi(A_1),1)^{\Phi((1,(1,j,k)))}.
	\]
	However, $\phi(A_1^{(1,j,k)})$ is equal to $\phi(A_j)$, while $(B_i,1)^{\Phi((1,(1,j,k)))}$ is equal to $(B_i, 1)$. Thus $\phi$ maps $A_1$ and $A_j$ to $B_i$, which is a contradiction.
	
	We have shown that, if $\widetilde A = A^t$, then $\phi(A_i) = B_i$ for all $i$. Let $N = (1 \times E^{t-1}) \cap H$. Thus, 
	\[
	\phi\left((N \cap F_i^t)/F_{i-1}^t \right) = \phi(A_2 \times \cdots \times A_t) = B_2 \times \cdots \times B_t = (N \cap F_j^t)/F_{j-1}^t,
	\]
	and so
	\[
	\frac{A^t}{(N \cap F_i^t)/F_{i-1}^t} \equiv_{H \cap E^t} \frac{B^t}{(N \cap F_j^t)/F_{j-1}^t}.
	\]
	Additionally, $N$ centralises these sections, so
	\[
	\frac{A^t}{(N \cap F_i^t)/F_{i-1}^t} \equiv_{(H \cap E^t)/N} \frac{B^t}{(N \cap F_j^t)/F_{j-1}^t}.
	\]
	Since $(H \cap E^t)/N \cong E$ and the images of $A^t$ and $B^t$ under this quotient map are $A$ and $B$ respectively, this proves the desired result.
	
	Now assume that $\widetilde A$ and $\widetilde B$ are not equal to $A^t$ and $B^t$ respectively. This implies that the chief factors of $F$ in $A$ and $B$ are central. Let $\widetilde A = \widetilde A_1/\widetilde A_2$ for some subgroups $\widetilde A_i$ of $A^t$. There exists an $(H \cap E^t)$-invariant subgroup of $A^t$ which complements $\widetilde A_2$ in $\widetilde A_1$. For instance, if $\widetilde A = \del(A^t)/\diag(A^t)$ we can take this normal subgroup to be $\langle (a, -a, 0, \dots, 0), \dots, (0, \dots, 0, a, -a, 0) : a \in A \rangle$. Hence, $H \cap E^t$ acts the same on $\widetilde A$ as on this subgroup of $A^t$, so we can view $\widetilde A$ as a subgroup of $A^t$, and $\widetilde B$ as a subgroup of $B^t$. Let $X$ be a minimal $(H \cap E^t)$-normal subgroup of $\widetilde A$, and let $Y = \phi(X)$. The group $H \cap E^t$ acts diagonally on $A^t$, by above, so $X$ is isomorphic to $A$. Let $(a_1, \dots, a_t)$ be a non-zero element of $X$ so that $\langle (a_1, \dots, a_t) \rangle ^{H\cap E^t} = X$. Choose $i$ with $a_i \neq 0$. Note that this implies that $\rho_i$, the projection map from $A^t$ to $A_i$, is an isomorphism when restricted to $X$. Let $\phi(a_1, \dots, a_t) = (b_1, \dots, b_t)$ and assume that $b_j \neq 0$. Let $\pi_j$ be the natural projection map from $B^t$ to $B_j$. Then the map $\pi_j \circ \phi \circ \rho_i^{-1}$ is an $E$-invariant isomorphism from $A$ to $B$ since $H \cap E^t$ acts diagonally on $A^t$ and $B^t$. Thus $A \equiv_E B$, as required.
\end{proof}

\begin{remark}
	If $t < 5$ then $\diag(A^t)$ and $\del(A^t)$ may not be the only $H$-normal, proper, non-trivial subgroups of $A^t$. For instance, let $E = F = C_7$, and let $H = E \wr \Alt(3)$. Then, the group $\{(a, a^2, a^4) : a \in C_7\}$ is $H$-normal.
\end{remark}

\begin{prop}\label{prop:t3CFsinwreathproduct}
	Let $t \geq 3$. Assume that $H \cap E^t$ is subdirect and that $H$ contains $\Alt(t)$. Suppose that there exists some $\alpha \in E$ such that $H$ contains $(1, \dots, 1, \alpha) (1,2)$. Let $A, B$ be chief factors of $E$ in $F$. Then the chief factors of $H$ in $A^t$ are as given in Proposition \ref{prop:t5CFsinwreathproduct}, and are all non-equivalent when non-Frattini. 
	
	Moreover, if two chief factors of $H$ contained in $A^t$ and $B^t$ are $H$-equivalent, then $A \equiv_E B$.
\end{prop}

\begin{proof}
	For any $i$, we can conjugate $(1, \dots, 1, \alpha) (1,2)$ by an element of $\Alt(t) \leq H$ to obtain $(1, \dots, 1, \alpha, 1, \dots, 1) \sigma \in H$ with $\alpha$ in the $i$th coordinate, for some $\sigma \in \Sym(t) \setminus \Alt(t)$. Thus we can multiply by an element of $\Alt(t)$ to obtain $(1, \dots, 1, \alpha, 1, \dots, 1) \tau \in H$ for any $\tau \in \Sym(t) \setminus \Alt(t)$.
	
	Let $A$ be a minimal normal subgroup of $E$ contained in $F$. Let $K$ be an $H$-normal, non-trivial, proper subgroup of $A^t$. We show that $K$ is one of $\diag(A^t)$ or $\del(A^t)$ when $A$ is abelian. Let $a = (a_1, \dots, a_t)$ be a non-trivial element of $K$. Thus, $K$ contains
	\[
	a - a^{(\alpha, 1, \dots, 1) (1,2)} - (a - a^{(\alpha, 1, \dots, 1) (1,2)})^{(1,1,\alpha,1,\dots, 1)(1,3)} = (a_1 - a_2, 0 , a_2 - a_1, 0, \dots, 0).
	\]
	Hence, either $K$ is contained in $\diag(A^t)$ or contains $\del(A^t)$, and so $K$ is equal to $\diag(A^t)$ or $\del(A^t)$ as previously. The argument then follows as in Proposition \ref{prop:t5CFsinwreathproduct}.
	
	Since $t \geq 3$, two non-Frattini chief factors of $H$ in $A^t$ cannot have the same order, so in particular these chief factors are not $H$-equivalent.
\end{proof}

\begin{prop}\label{prop:t2CFsinwreathproduct}
	Let $t = 2$, and assume that $H$ is as in Proposition \ref{prop:t3CFsinwreathproduct}. Let $A, B$ be chief factors of $E$ in $F$. If there exists a chief factor of $F$ in $A$ which is non-central, then $A^t$ is a chief factor of $H$. If $A^t$ is not a chief factor of $H$ then it contains two chief factors isomorphic to $A$. If both are non-Frattini, then they are not $H$-equivalent.
	
	Moreover, if two chief factors of $H$ contained in $A^t$ and $B^t$ respectively are $H$-equivalent, then $A \equiv_F B$.
\end{prop}

\begin{proof}
	If some chief factor of $F$ in $A$ is non-central, then, by Proposition \ref{prop:CFsFromNormalSubgp}, any chief factor of $F$ in $A$ is non-central. Let $X$ be a minimal normal subgroup of $F$ in $A$, on which $F$ acts non-centrally. Then any minimal normal subgroup $K$ of $H$ in $A^t$ contains $X^t$. However, this implies that $K$ contains $A^t$, so $A^t$ is a chief factor of $H$.
	
	By Goursat's Lemma, any $H$-normal subgroup of $A^t$ which is non-trivial and proper is of the form
	\[
	\{(a,a^\phi) : a \in A\}
	\]
	for some isomorphism $\phi$ in $A$. For such a group to be normal, we require $a^{g_1\phi} = a^{\phi g_2}$ for any $(g_1, g_2) \in H \cap E^t$. Thus, if $f_{g}$ is the image of $g \in E$ in $\Aut(A)$, we require $f_{g_1}^{\phi} = f_{g_2}$. As such, the image of $H \cap E^t$ in $\Aut(A)^t$ is equal to the diagonal subgroup $\{(f_g, f_g^\phi) : g \in E\}$.
	
	Suppose that $A^t$ contains two non-Frattini, $H$-equivalent chief factors. Thus, there exist distinct isomorphisms $\psi_1, \psi_2$ of $A$ such that $X_i := \{(a, a^{\psi_i}) : a \in A \}$ are normal subgroups of $H$. This implies that the image of $H \cap E^t$ in $\Aut(A)^t$ is contained in $\{(f_e, f_e^{\psi_i}) : e \in E \}$ for $i = 1,2$. Thus, $\psi_1 \psi_2^{-1}$ commutes with the image of $E$ in $\Aut(A) = \GL(A)$. However, this image is an irreducible subgroup of $\GL(A)$, so by Schur's Lemma $\psi_1 \psi_2^{-1}$ is equal to a scalar $\lambda$ of $\GL(A)$. Using the fact that $X_1, X_2$ are both normalised by $(1, \alpha) (1,2)$, we can show that $\lambda = -1$. Hence, we let $\psi = \psi_1$ and $-\psi = \psi_2$. Let $\phi \colon X_1 \to X_2$ be an $H$-isomorphism and note that there exists some isomorphism $\widetilde \phi$ of $A$ such that  
	\[
	\phi(a, a^{\psi}) = (a^{\widetilde \phi}, - a^{\widetilde \phi \psi}) \qquad \forall a \in A.
	\]
	For any $e_1 \in E$, let $e = (e_1, e_2) \in H \cap E^t$. Then $\phi((a, a^{\psi_1}))^e = \phi((a, a^{\psi_1})^e)$ implies that $a^{e_1 \widetilde \phi} = a^{\widetilde \phi e_1}$. Thus, $\widetilde \phi$ is a scalar of $\GL(A)$, which we denote by $\mu$. Applying the same argument with $(1, \alpha) (1,2)$ gives $2 \mu a^{\psi \alpha} = 0$ for all $a \in A$. So the exponent of $A$ is 2, but then $\psi = - \psi$, so $X_1 = X_2$.
	
	Let $A = F_i/F_{i-1}$ and $B = F_j/F_{j-1}$. Assume that $A^t$ and $B^t$ contain $H$-chief factors $\widetilde A$ and $\widetilde B$, respectively, which are $H$-equivalent. If there exists a chief factor $X$ of $F$ in $A$ which is central then, by Proposition \ref{prop:CFsFromNormalSubgp}, all of the chief factors of $F$ in $A$ are central. Thus $\widetilde A$ and $\widetilde B$ are centralised by $F^t$. It follows that the chief factors of $F$ contained in $B$ are also central. By the proof of Proposition \ref{prop:CFsFromNormalSubgp}, in fact $F$ centralises $A$ and $B$. Consequently, it suffices to prove that $A \cong B$ in order to deduce that $A \equiv_F B$. Suppose for a contradiction that $A \not\cong B$. Without loss of generality, we may assume that: $A$ is isomorphic to $B^2$, $\widetilde B = B^2$, and $\widetilde A$ is one of $\diag(A^2)$ or $A^2/\diag(A^2)$. Hence, the image of $E$ in $\GL(A)$ is conjugate to the image of $H \cap E^2$ in $\GL(B^2) = \GL(A)$. Thus $E$ does not act irreducibly on $A$, which is a contradiction.
	
	If the chief factors of $F$ in $A$ are non-central, then the chief factors of $F$ in $B$ are also non-central, so $\widetilde A = A^t$ and $\widetilde B = B^t$. Let $\phi, \Phi$ be given as in Definition \ref{def:CFsEquivalence}. If $A$ is abelian, since $F$ does not centralise any subgroup of $A$, we can show that $\phi$ maps $A_k$ to $B_k$ for any $k$, and so as previously we obtain that $A \equiv_E B$. In particular, $A \equiv_F B$.
	
	Suppose instead that $A$ is non-abelian. Recall that $A = F_i/F_{i-1}$ and $B = F_j/F_{j-1}$. We assume that $i < j$ without loss of generality. If $\phi(A_k) = B_k$ for every $k$ then we are done as previously, so we assume that $\phi(A_1) = B_2$. By the definition of $\Phi$, this implies that the image of $F$ in $\Aut(A)$ and $\Aut(B)$ is inner. Hence, $F/F_{i-1}$ is equal to $A \times C_{F/F_{i-1}}(A)$. Similarly, there is a quotient of $C_{F/F_{i-1}}(A)$ which is equal to a direct product of $B$ and a group which centralises $B$. Thus there exists a quotient of $F$ which is isomorphic to $A \times B$. It is now simple to show that $A \equiv_F B$, since $A \cong B$.
\end{proof}

When $t=2$, the last statement of Proposition \ref{prop:t3CFsinwreathproduct} no longer holds, as the following example shows.

\begin{example}
	Let $S = \Alt(6)$, and let $g_1, g_2$ be elements of $\Aut(S)$ such that $\langle S, g_1, g_2 \rangle = \Aut(S)$. Consider $F = S \times S$ and 
	\[
	E = \langle F, (g_1, g_2), (g_2, g_1) \rangle \leq \Aut(S)^2.
	\]
	The chief factors $A = S \times 1$ and $B = S^2/(S \times 1)$ of $E$ are not $E$-equivalent. Indeed, these chief factors are not $E$-isomorphic since their centralisers are $S \times 1$ and $1 \times S$ respectively. As such, the two non-abelian chief factors $A$ and $B$ are $E$-equivalent if and only if there exists a maximal subgroup $M$ of $E$ such that $E/\operatorname{Core}_E(M)$ has two chief factors $E$-isomorphic to $A$ and $B$ respectively, see \cite[Proposition 1.4]{JP}. However, there exists no isomorphism of $\Aut(S)$ which maps $g_1$ to $g_2$, so, by Goursat's Lemma and Proposition \ref{prop:maxsubgpdirectproduct}, any maximal subgroup $M$ of $E$ contains one of $S \times 1$ or $1 \times S$. As such, $E/\operatorname{Core}_E(M)$ contains at most one chief factor isomorphic to $S$.
	
	Consider the group
	\[
	H = \langle F^2, ((g_1, g_2), (g_2, g_1)), (1,2) \rangle \leq E \wr \Sym(2).
	\]
	The chief factors $A^2$ and $B^2$ are $H$-equivalent. Indeed, consider the maximal subgroup
	\[
	M = \langle ((s_1, s_2), (s_2, s_1)), ((g_1, g_2), (g_2, g_1)), (1,2) : s_i \in S \rangle
	\]
	of $H$, for which $H / \operatorname{Core}_H(M) \cong H$ contains two chief factors $H$-isomorphic to $A^2$ and $B^2$ respectively. Thus, by \cite[Proposition 1.4]{JP}, $A^2 \equiv_H B^2$.
	
	We have shown that there exists a group $H$ satisfying the conditions of Proposition \ref{prop:t2CFsinwreathproduct} but for which $A^2 \equiv_H B^2$ does not imply $A \equiv_E B$. Such a group exists also in the case that $A, B$ are abelian. For instance, if $n > 2$ and $\iota$ denotes the inverse transpose map in $\Aut(\GL_n(p))$, one can take $V = C_p^n$, $F = V \times V$, and $E$ a subgroup of $(V : \GL_n(p))^2$ given by
	\[
	E = \langle F, (g, g^{\iota}) : g \in \GL_n(p) \rangle.
	\]
	In $E$, the chief factors $A = V \times 1$ and $B = (V \times V)/(V \times 1)$ are not $E$-equivalent, while $A^2 \equiv_H B^2$ if 
	\[
	H := \langle F^2, ((g, g^{\iota}), (g^{\iota}, g)), (1,2) \rangle \leq E \wr \Sym(2).
	\]
\end{example}

We summarise the above results in the following corollary.

\begin{corollary}\label{cor:CFsinwreathproduct}
	Let $t \geq 2$ and let $H$ be a subgroup of $E \wr \Sym(t)$ containing $F^t$ and $\Alt(t)$, where $H \cap E^t$ is subdirect in $E^t$. If $t \leq 4$ assume additionally that $H$ contains $(1, \dots, 1, \alpha) (1,2)$ for some $\alpha \in E$. Let $A, B$ be two chief factors of $E$ in $F$. Then either:
	\begin{enumerate}[\upshape(1)]
		\item $A^t$ is a chief factor of $H$; or
		\item $A$ is a central section of $F$, and either:
		\begin{enumerate}[\upshape(i)]
			\item $\exp(A) \nmid t$ and $A^t$ contains $H$-chief factors isomorphic to $A$ and $A^{t-1}$; or
			\item $\exp(A) \mid t$ and $A^t$ contains $H$-chief factors isomorphic to $A$, $A^{t-2}$ and $A$, the first two of which are Frattini.
		\end{enumerate}
	\end{enumerate}
	The non-Frattini chief factors of $H$ in $A^t$ are non-equivalent. If $A^t$ and $B^t$ contain $H$-equivalent chief factors, then $A \equiv_F B$. If in addition $t \geq 3$, then $A \equiv_E B$.
\end{corollary}

We next determine which chief factors can be non-Frattini in $H$. Since $\Phi(X^t)$ is equal to $\Phi(X)^t$, any chief factor $A$ which is Frattini in $F$ must give rise to a section $A^t$ of $H$ which is also Frattini. However, if $A$ is non-Frattini in $F$, then it is not immediately clear whether $A^t$ is Frattini in $H$ or not.

\begin{lemma}\label{lemma:jumpingbabyresultfull}
	Let $t \geq 2$, and assume that $H \cap E^t$ is subdirect. If $F_i/F_{i-1}$ is a central chief factor of $E$ and is contained in $[E/F_{i-1}, F/F_{i-1}]$, then $F_i^t/F_{i-1}^t$ is Frattini in $H$.
\end{lemma}

\begin{proof}
	Since $\frac{H \cap E^t}{F_{i-1}^t}$ is a subdirect subgroup of $(E/F_{i-1})^t$, we have
	\begin{align*}
		\frac{F_i^t}{F_{i-1}^t} \leq \left[ \frac{E}{F_{i-1}}, \frac{F}{F_{i-1}} \right]^t  = \left[\frac{H \cap E^t}{F_{i-1}^t}, \frac{F^t}{F_{i-1}^t} \right] \leq \left[\frac{H \cap E^t}{F_{i-1}^t}, \frac{H\cap E^t}{F_{i-1}^t}\right].
	\end{align*}
	Thus, since $F_i/F_{i-1}$ is a central chief factor of $E$,
	\begin{align*}
		F_i^t/F_{i-1}^t \leq \left[\frac{H \cap E^t}{F_{i-1}^t}, \frac{H\cap E^t}{F_{i-1}^t}\right] \cap Z \! \left( \frac{H \cap E^t}{F_{i-1}^t} \right)  \leq \Phi \! \left( \frac{H \cap E^t}{F_{i-1}^t} \right)  \leq \Phi( H/F_{i-1}^t ),
	\end{align*}
	as required.
\end{proof}

We show that a similar result holds if $H$ contains $\del(D^t)$ for some $D \unlhd E$ instead.

\begin{lemma}\label{lemma:jumpingbabyresultdel}
	Let $D \unlhd E$ be an abelian subgroup such that $H \cap D^t = \del(D^t)$. Assume that $H \cap E^t$ is subdirect. If $D_i/D_{i-1}$ is a central chief factor of $E$ and is contained in $[E/D_{i-1}, D/D_{i-1}]$, then $\del(D_i^t/D_{i-1}^t)$ is Frattini in $H$.
\end{lemma}

\begin{proof}
	The proof of this result is entirely similar, with the only difficulty lying in proving that 
	\[
	\del \left( \left[ \frac{E}{D_{i-1}}, \frac{D}{D_{i-1}} \right]^t \right) \leq \left[ \frac{H \cap E^t}{\del(D_{i-1}^t)}, \frac{\del(D^t)}{\del(D_{i-1}^t)} \right].
	\]
	To prove this inclusion, we may assume that $D_{i-1} = 1$, and show that
	\[
	\del \left( \left[ E, D \right]^t \right) \leq \left[ H \cap E^t, \del(D^t) \right].
	\]
	Consider any $e \in E$. There exist some elements $e_i$ of $E$ such that $(e, e_2, \dots, e_t) \in H$, since $H \cap E^t$ is subdirect. Additionally, given any $d \in D$, the element $(d, d^{-1}, 1, \dots, 1)$ lies in $H \cap D^t = \del(D^t)$. Thus, $[(e, e_2, \dots, e_t), (d, d^{-1}, 1, \dots, 1)]$ is contained in $\del(D^t)$, as $\del(D^t)$ is a normal subgroup of $H \cap E^t$. Hence,  
	\[
	([e,d], [e_2, d^{-1}], 1, \dots, 1) \in \del(D^t),
	\]
	which implies that $[e,d] = [e_2, d^{-1}]^{-1}$. Therefore, for any $e \in E$ and $d \in D$, we have
	\[
	([e,d], [e,d]^{-1}, 1, \dots, 1) \in [H \cap E^t, \del(D^t)].
	\]
	The proof now concludes as in Lemma \ref{lemma:jumpingbabyresultfull}.
\end{proof}

The above results together with Corollary \ref{cor:upperbounddeltaGN} yield the following theorem.

\begin{theorem} \label{theorem:mainjumping}
	Let $H$ be a subgroup of $E \wr \Sym(t)$ as in Corollary \ref{cor:CFsinwreathproduct}. Then, for any central chief factor $W$ of $H$, we have
	\begin{align*}
		\delta_{H, F^t}(W) \leq {} & \delta_{\frac{F}{[E,F]}}(W).
	\end{align*}
	When $W$ is non-central, if $W \cong S^b$ for some simple group $S$, we have that
	\[
	\delta_{H, F^t}(W) \leq \max_{\{ F\text{-group } B \cong S^c, \, c \mid b\}} \delta_{F}(B).
	\]
\end{theorem}

\begin{proof}
	Let $W$ be a non-Frattini chief factor of $H$ in $F^t$. By Corollary \ref{cor:CFsinwreathproduct}, for any chief factor $A$ of $E$ in $F$ we have $\delta_{H, A^t}(W) \leq 1$, with equality only if $A$ is non-Frattini in $F$.
	
	Suppose that $W$ is central. If $\delta_{H, A^t}(W) = 1$, then $A$ is a central chief factor of $E$ which is not contained in $[E,F]$, by Lemma \ref{lemma:jumpingbabyresultfull}, and is non-Frattini in $F$. Hence, $\delta_{H, F^t}(W)$ is bounded above by the number of central chief factors of $E$ in $F/[E,F]$ which are non-Frattini in $F$ and isomorphic to $W$. This is bounded above by $\delta_{\frac{F}{[E,F]}}(W)$, by Corollary \ref{cor:upperbounddeltaGN}.
	
	Suppose that $W$ is non-central. By the final statement of Corollary \ref{cor:CFsinwreathproduct}, $\delta_{H, F^t}(W)$ is at most the number of chief factors of $E$ in $F$ which are $F$-equivalent to $A$ and non-Frattini in $F$. Thus, by Corollary \ref{cor:upperbounddeltaGN}, we obtain the desired result.
\end{proof}

\section{Chief factors of maximal subgroups of classical groups}\label{section:classicalprelims}

As shown in Section \ref{section:prelimsmethod}, many of the second maximal subgroups we consider in this paper are constructed from maximal subgroups of almost simple groups. Consequently, in order to bound the number of non-Frattini chief factors of these second maximal subgroups, it is essential to have a detailed understanding of the chief factors of the maximal subgroups themselves. This section establishes several results describing the chief factors of such groups, with particular attention to the case in which the ambient almost simple group is classical. We begin by introducing the notation and definitions for classical groups that are used throughout this section, and Section \ref{section:classicals}.

Following \cite{KL}, let $V$ be a vector space over a finite field $\mathbb{F}$, and let $\kappa$ be either the zero form, or a non-degenerate unitary, symplectic or quadratic form on $V$. We say the form $\kappa$ is of type $\boldL, \boldU, \boldS$ or $\boldO$, respectively. We may refer to type $\boldU$ as $\boldL^-$. If $\kappa$ is of type $\boldO$, and $\dim(V)$ is even, then there are two similarity classes of quadratic form $\kappa$, which we refer to as having type $\boldO^+$ and $\boldO^-$. With this notation, there is only one similarity class per type of $\kappa$. We refer to any such triple $(V, \mathbb{F}, \kappa)$ as a classical geometry.

We define the following groups:
\begin{enumerate}[(1)]
	\item Let $S = S(V, \mathbb{F}, \kappa)$ be the group of special $\kappa$-isometries of $V$;
	\item Let $I = I(V, \mathbb{F}, \kappa)$ be the group of $\kappa$-isometries of $V$;
	\item Let $\Delta = \Delta(V, \mathbb{F}, \kappa)$ be the group of $\kappa$-similarities of $V$;
	\item Let $\Gamma = \Gamma(V, \mathbb{F}, \kappa)$ be the group of $\kappa$-semisimilarities of $V$;
	\item If $\kappa$ is of type $\boldO$, we let $\Omega = \Omega(V, \mathbb{F}, \kappa)$ be the index two subgroup of $S$ equal to the kernel of the spinor norm $\theta$ (see \cite[Proposition 2.5.7]{KL}). Otherwise, we let $\Omega = S$;
	\item If $\kappa$ is of type $\boldL$ with $\dim(V) > 2$, we let $\Sigma = \Sigma(V, \mathbb{F}, \kappa) = \langle \Gamma, \iota \rangle$, where $\iota$ is the inverse transpose map on $\Gamma$. Otherwise we let $\Sigma = \Gamma$. 
\end{enumerate}
Since there is only one similarity class per type of $\kappa$, its type determines the isomorphism class of the above groups.

We have defined the following chain of groups 
\begin{align}\label{eq:chainofsubgps}
	\Omega \leq  S \leq  I \leq  \Delta \leq  \Gamma \leq  \Sigma,
\end{align}
with $\Gamma$ acting on the vector space $V$.

For any subgroup $K$ of $\Sigma$, we write $\overline K$ to denote the image of $K$ under the quotient by the scalars of $\Sigma$. Additionally, we write $\ddot K$ to denote $\overline K / \overline \Omega$. A group $G$ is a classical group if there exists a classical geometry $(V, \mathbb F, \kappa)$ such that either $\Omega \leq G \leq \Sigma$, or $\overline \Omega \leq G \leq \overline \Sigma$.

\begin{remark}
	If $\overline \Omega$ is a simple group, then the group $\overline \Sigma$ is contained in $\Aut(\overline \Omega)$. In almost all cases $\overline \Sigma = \Aut(\overline \Omega)$, unless $\Omega$ is equal to $\Sp_4(q)$ with $q$ even, or $\Omega_8^+(q)$ for any $q$. Then, $\Aut(\overline \Omega)$ is equal to $\langle \overline \Sigma, \overline \gamma \rangle$, for some graph automorphism $\gamma$.
\end{remark}

\subsection{Aschbacher's Theorem}\label{section:Aschbacher}

Let $(V, \mathbb{F}, \kappa)$ be a classical geometry, with $n = \dim(V)$. Let $p$ be a prime, and let $q = p^f$ be a prime power with $\mathbb{F} = \mathbb F_{q^u}$, where $u = 2$ when $\kappa$ is of type $\boldU$, and $u = 1$ otherwise. Let $G$ be a finite classical group with $\Omega \leq G \leq \Sigma$, and let $H$ be a maximal subgroup of $G$ which does not contain $\Omega$. Then $H$ is classified by Aschbacher's Theorem. This theorem states that either $H$ is in one of eight classes of groups $\mathscr{C}_1$ to $\mathscr{C}_8$, or $\overline H$ is almost simple and we say $H$ belongs to $\mathscr{S}$.

In the case that $H$ does not contain $\Omega$, or is in $\mathscr{S}$, it is easy to compute the chief factors of $H$. Thus, most of our work is dedicated to classes $\mathscr{C}_1$ to $\mathscr{C}_8$. As such, we dedicate this section to discussing the structure of the groups in each of these classes. Particularly, we treat the maximal subgroups $H$ whose structures we require in this paper, and do not detail the group structures of any $H$ which we do not require. For instance, this may be the case if the number of chief factors of any maximal subgroup $M$ of $H$ is very small, and so we do not need detailed information about the group structure of $H$ to determine this.

The classes $\mathscr{C}_i$ split further into types, which generally indicate more information about the group structure of $H$. For instance, if $H \in \mathscr{C}_1$, then $H$ can be of type $\GL_m(q) \oplus \GL_{n-m}(q)$, in which case $H \cap I$ is a subgroup of $\GL_m(q) \times \GL_{n-m}(q)$. For any classical group $\Omega \leq L \leq \Sigma$, we let $H_L = H_\Sigma \cap L$, where $H_\Sigma$ is the largest subgroup of $\Sigma$ of the same type as $H$, see \cite[Section 3.1]{KL}. Similarly, if $H_{\overline \Sigma}$ denotes the image of $H_\Sigma$ under the quotient map by the scalars of $\Sigma$, we let $H_{\overline L} = H_{\overline\Sigma} \cap \overline L$. For example, $H = H_G$ and $\overline H = H_{\overline G}$. We remark that this differs from the notation of \cite[Section 6]{LMT}, where $H_L = H \cap L$, and so $H_L$ depends on both $G$ and the type of $H$, unlike in our notation where it only depends on the type of $H$. Finally, we let $c$ be the index of $H_{\overline \Sigma} \overline \Omega/\overline \Omega$ in $\overline \Sigma / \overline \Omega$.

For more details regarding each class $\mathscr{C}_i$, see Section 4.$i$ of \cite{KL}.

\subsubsection{Class $\mathscr{C}_1$}\label{section:C1explanation}

If $H$ is a parabolic subgroup, or of type $\Sp_{n-2}(q)$, we do not need detailed information about the structure of $H$, and so we do not treat these maximal subgroups in this section.

Hence, we may assume that $H$ is of type $\GL_m(q) \oplus \GL_{n-m}(q) $, $\GU_m(q) \perp \GU_{n-m}(q)$, $\Sp_m(q) \perp \Sp_{n-m}(q)$ or $\O_m^{\epsilon_1}(q) \perp \O_{n-m}^{\epsilon_2}(q)$. Then, $H$ stabilises a vector subspace decomposition $V = V_1 \perp V_2$, where each $V_i$ has a bilinear form $\kappa_i$ of the same type as $G$. For $X \in \{\Omega, S, I, \Delta, \Gamma, \Sigma\}$, we can consider $X(V_i, \mathbb{F}, \kappa_i)$. For instance, $H_I$ is equal to $I(V_1, \kappa_1) \times I(V_2, \kappa_2)$. One might expect that $H$ is contained in $\Sigma(V_1, \kappa_1) \times \Sigma(V_2, \kappa_2)$, however in certain cases $H$ is instead contained in the direct product of two slightly larger groups, which we now define.

\begin{definition}\label{def:sigma*}
	Let $(V, \mathbb F, \kappa)$ be a classical geometry. Let
	\[
	\Sigma^*(V, \mathbb F, \kappa) := \begin{cases*}
		\Sigma(V, \mathbb F, \kappa) : \langle \iota \rangle \qquad & in case $\boldL$, with $n \leq 2$,\\
		\Sigma(V, \mathbb F, \kappa) & otherwise,
	\end{cases*}
	\]
	where $\iota$ is an element of order 2 which acts as the inverse transpose map on $\Gamma(V, \mathbb F, \kappa)$.
\end{definition}

Note that, if $\dim(V) \leq 2$ and we are in case $\boldL$, then the inverse transpose automorphism is inner in $\Aut(\PSL_n(q))$. Consequently, $\Sigma^* = \Sigma : C_2$, and $\overline \Sigma^* \cong \overline \Sigma \times C_2$.

In the following statement, for $X \in \{\Omega, I, \Sigma^*\}$, let $X_i := X(V_i, \kappa_i)$.

\begin{prop}\label{prop:prelimsC1}
	Let $H$ be of type $\GL_m(q) \oplus \GL_{n-m}(q) $, $\GU_m(q) \perp \GU_{n-m}(q)$, $\Sp_m(q) \perp \Sp_{n-m}(q)$ or $\O_m^{\epsilon_1}(q) \perp \O_{n-m}^{\epsilon_2}(q)$. Then
	\[
	H_\Omega = (\Omega_1 \times \Omega_2) : \del(I_1/\Omega_1 \times I_2 / \Omega_2),
	\]
	and
	\[
	H_\Sigma = (I_1 \times I_2):\diag(\Sigma^*_1/I_1 \times \Sigma^*_2/I_2) \leq \Sigma_1^* \times \Sigma_2^*.
	\]
\end{prop}

\begin{proof}
	The result for $H_{\Omega}$ follows by the statements in \cite[Section 4.1]{KL}. Additionally, by \cite[Lemma 4.1.1]{KL}, $H_I = I_1 \times I_2$. In fact, since $H_\Gamma$ stabilises $V_1 \perp V_2$, $H_\Gamma$ is contained in $\Gamma_1 \times \Gamma_2$. An element $(g_1, g_2) \in \Gamma_1 \times \Gamma_2$ is in $H_\Gamma$ if and only if $\tau(g_1) = \tau(g_2)$ and $\sigma(g_1) = \sigma(g_2)$ (see \cite[Lemma 2.1.2]{KL}), so $H_\Gamma = (I_1 \times I_2) . \diag(\Gamma_1/I_1 \times \Gamma_2/I_2)$. If $\Gamma = \Sigma$, the result is proven, so we assume that we are in case $\boldL$ with $n > 2$. 
	
	By definition, $H_\Sigma = N_{\Sigma}(H_\Gamma)$. The inverse transpose of a block diagonal matrix is the block diagonal matrix whose blocks are the inverse transposes of the corresponding blocks. Thus, $\iota$ normalises $H_\Gamma$, and $\iota$ acts on each $\Gamma_i$ as the inverse transpose automorphism.
\end{proof}

\subsubsection{Class $\mathscr{C}_2$}\label{section:C2explanation}

In this section, $H$ stabilises a vector subspace decomposition of the form $V = V_1 \oplus \cdots \oplus V_t$, where $m = \dim(V_i)$ and $n = mt$. Each $V_i$ is additionally either non-degenerate or a totally singular $m$-space, and $V_i$ is orthogonal to $V_j$ for any $i \neq j$. Let $\kappa_i$ be a bilinear form on each $V_i$ as defined in \cite[Section 4.2]{KL}.

As in Section \ref{section:C1explanation}, for each $X \in \{\Omega, S, I, \Sigma, \Sigma ^*\}$, we define $X_i := X(V_i, \kappa_i)$. Note that, by definition of $\mathscr C_2$, the group $X_1$ is isomorphic to $X_i$ for any $X$ as above and any $i$.

Assume first that $H$ is of type $\GL_m(q) \wr \Sym(t)$, $\GU_m(q)\wr \Sym(t)$, $\Sp_m(q)\wr \Sym(t)$, or $\O_m^\epsilon(q) \wr \Sym(t)$. In these cases, the vector spaces $V_i$ are pairwise isometric.

\begin{prop}\label{prop:C2wreathstructure}
	Suppose that $H$ is of type $\GL_m(q) \wr \Sym(t)$, $\GU_m(q)\wr \Sym(t)$, $\Sp_m(q)\wr \Sym(t)$, or $\O_m^\epsilon(q) \wr \Sym(t)$. If $H$ is not of type $\O_1(q) \wr \Sym(t)$, then
	\[
	H_{\Omega} = \Omega_1^t . \del((I_1 / \Omega_1)^t) . \Sym(t) \leq I_1 \wr \Sym(t).
	\]
	On the other hand, if $H$ is of type $\O_1(q) \wr \Sym(t)$, then
	\[
	2^{t-1}.\Alt(t) \leq H_{\Omega} \leq 2^{t-1} . \Sym(t).
	\]
	If $H$ is of type $\GL_1(q) \wr \Sym(2)$, then
	\[
		H_\Sigma = (I_1^t : \diag((\Sigma_1/I_1)^t)) : \Sym(t) \leq \Sigma_1 \wr \Sym(t).
	\]
	Otherwise,
	\[
	H_\Sigma = (I_1^t : \diag((\Sigma^*_1/I_1)^t)) : \Sym(t) \leq \Sigma^*_1 \wr \Sym(t).
	\]
\end{prop}

\begin{proof}
	The results for $H_\Omega$ are given by \cite[Section 4.2]{KL}. Additionally, $H_I = I_1 \wr \Sym(t)$, and $H_\Gamma \leq \Gamma_1 \wr \Sym(t)$. The subgroup $\Sym(t)$ belongs to $H_\Gamma$, while an element $(g_1, \dots, g_t) \in \Gamma_1^t$ is contained in $H_{\Gamma}$ if and only if $\tau(g_1) = \tau(g_i)$ and $\sigma(g_1) = \sigma(g_i)$ for all $i$. If $G$ is in case $\boldL$ with $n  >2$ then $\iota$ normalises $H_\Gamma$, commutes with $\Sym(t)$, and acts as the inverse transpose automorphism on each $\Gamma_i$. We conclude as in Proposition \ref{prop:prelimsC1}.
\end{proof}

\begin{prop}
	Suppose that $H$ is of type $\GL_{n/2}(q^u).2$. Then,
	\[
	\Omega_1 \leq H_\Omega \leq I_1.2 = H_I.
	\]
\end{prop}

\begin{proof}
	By \cite[Lemma 4.2.3]{KL} we have that $H_I = I_1.2$. For the result for $H_\Omega$ see the proofs of \cite[Propositions 4.2.4, 4.2.5, 4.2.7]{KL}.
\end{proof}

\begin{prop}\label{prop:prelimC2Omq2}
	Suppose that $H$ is of type $\O_{n/2}(q)^2$. We have
	\[
	H_{\Omega} = S_1 \times S_2,
	\]
	and
	\[
	H_\Sigma = (I_1^2 : \diag((\Sigma_1/I_1)^2)) : \Sym(2) \leq \Sigma_1 \wr \Sym(2).
	\]
\end{prop}

\begin{proof}
	The proof follows by the same argument as that of Proposition \ref{prop:C2wreathstructure} by noting that $H_I = I_1 \times I_2$ and $H_{\Delta} = (I_1^2.\diag((\Delta_1/I_1)^2)):\Sym(2)$.
\end{proof}

\subsubsection{Class $\mathscr{C}_3$}\label{section:C3explanation}

Let $\mathbb{F}_{\sharp}$ be a finite extension of $\mathbb{F}$ of degree $r$, which is prime. Let $V_{\sharp}$ be the vector space $V \otimes_{\mathbb{F}} \mathbb{F}_{\sharp}$ over $\mathbb{F}_{\sharp}$, and $\kappa_{\sharp}$ the form on this vector space as defined in \cite[Section 4.3]{KL}. Then we define
\[
X_{\sharp} := X_{\sharp}(V_{\sharp}, \mathbb{F}_{\sharp}, \kappa_{\sharp})
\]
for $X$ equal to $\Omega$ or $\Gamma$.

\begin{prop}
	Suppose that $H$ is in $\mathscr{C}_3$. Then, 
	\[
	\Omega_\sharp \leq H_\Omega.
	\]
	If $G$ is in case $\boldL$, then
	\[
	H_\Sigma = \Gamma_\sharp . 2,
	\]
	and otherwise
	\[
	H_\Sigma = \Gamma_\sharp.
	\]
\end{prop}

\begin{proof}
	By definition, $H_\Gamma = \Gamma_\sharp$. Since $\Omega_\sharp$ is perfect by \cite[Section 4.3]{KL}, $H_\Omega \geq \Omega_\sharp$.
\end{proof}

\subsubsection{Class $\mathscr{C}_4$}\label{section:C4explanation}

If $H \in \mathscr{C}_4$, then $H$ stabilises the decomposition $V = V_1 \otimes V_2$. Let $\kappa_i$ be a form on $V_i$ as defined in \cite[Section 4.4]{KL}. For each $X \in \{\Omega, I, \Delta, \Gamma, \Sigma^*\}$, define $X_i$ for each $V_i$ as previously.

\begin{prop}
	Suppose that $H$ is in $\mathscr{C}_4$. Then
	\[
	\overline \Omega_1 \times \overline \Omega_2 \leq H_{\overline \Omega} \leq \overline \Delta_1 \times \overline \Delta_2,
	\]
	and
	\[
	H_{\overline \Sigma} = (\overline \Delta_1 \times \overline \Delta_2) : \diag((\overline \Sigma^*_1/\overline \Delta_1) \times (\overline \Sigma^*_2/\overline \Delta_2))
	\]
\end{prop}

\begin{proof}
	By \cite[(4.4.8)]{KL}, $H_{\overline \Delta} = \overline \Delta_1 \times \overline \Delta_2$. The proof then proceeds as in Proposition \ref{prop:prelimsC1}.
\end{proof}

\subsubsection{Class $\mathscr{C}_5$}\label{section:C5explanation}

Let $\mathbb{F}_\sharp$ be a subfield of $\mathbb{F}$ of prime index $r$, and $V_\sharp$ an $n$-dimensional vector space over $\mathbb{F}_\sharp$, so that $V = V_\sharp \otimes \mathbb{F}$. Let $\kappa_\sharp$ be a form on $V_\sharp$ as described in \cite[Section 4.5]{KL}. Finally, let $X_\sharp := X(V_\sharp, \mathbb F_\sharp, \kappa_\sharp)$ for $X \in \{ \Omega, S, I, \Delta , \Gamma, \Sigma\}$.

\begin{prop}
	Let $H$ be in $\mathscr{C}_5$. Then,
	\[
	\overline \Omega_\sharp \leq H_{\overline \Omega}
	\]
	and
	\[
	H_{\overline \Sigma} = \overline \Delta_\sharp . \left( \frac{\overline \Sigma}{\overline \Delta} \right).
	\]
\end{prop}

\begin{proof}
	By \cite[(4.5.5)]{KL}, $H_{\overline \Delta} = \overline \Delta_\sharp$. If $\overline \Omega_\sharp$ is perfect, then $\overline \Omega_\sharp$ is contained in $H_{\overline \Omega}$. Otherwise, by \cite[Tables 8.1, 8.5, 8.7, 8.12]{BHRD}, we have that $\overline \Omega_\sharp \leq H_{\overline \Omega}$.
\end{proof}

\subsubsection{Class $\mathscr{C}_6$}\label{section:C6explanation}

\begin{prop}\label{prop:C6prelims1}
	Let $G$ be of type $\boldL$ or $\boldU$, and assume that $H$ is of type $r^{2m}.\Sp_{2m}(r)$, with $n = r^m$ and $r$ an odd prime. If $n=3$ and $c=1$, then
	\[
	H_{\overline \Omega} = 3^2 . Q_8,
	\]
	and otherwise
	\[
	H_{\overline \Omega} = r^{2m} : \Sp_{2m}(r).
	\]
	In both cases,
	\[
	H_{\overline \Sigma} = r^{2m} : \GSp_{2m}(r).
	\]
\end{prop}

\begin{proof}
	By \cite[Proposition 4.6.5]{KL} and the discussion preceding it, we have that $V = W \otimes \cdots \otimes W$ for some $r$-dimensional vector space $W$ over $\mathbb{F}$. Let $\{w_1, \dots, w_r\}$ be a basis for $W$, and let $\beta$ be the corresponding basis for $V$ given by $\{w_{i_1} \otimes \cdots \otimes w_{i_{n/r}} : 1 \leq i_j \leq r\}$. Then, if $\omega$ is a primitive $r$th root of unity in $\mathbb{F}^*$, we set
	\[
	R = \langle x_1, y_1, z_1 \rangle \otimes \cdots \otimes \langle x_r, y_r, z_r \rangle
	\]
	where
	\[
	x_i = \diag(1, \omega, \dots, \omega^{r-1})
	\]
	and $y_i$ maps $w_j$ to $w_{j+1}$ with the indices modulo $r$. Finally, $z_i$ is any scalar of order $r$ in the $i$th copy of $W$. Hence, $\overline{R} = \langle x_1, y_1 \rangle \times \cdots \times \langle x_r, y_r \rangle = r^{2m}$, and can be seen as a $2m$-dimensional vector space over $\mathbb F_r$ with basis $\beta' := \{x_1, \dots, x_r, y_1, \dots, y_r\}$. This group is normalised by a subgroup isomorphic to $\Sp_{2m}(r)$, by the definition of $\mathscr{C}_6$ in \cite[Section 4.6]{KL}. Additionally, it is also normalised by $ \phi_\beta $ always, and by $\iota$ in the $\boldL$ case. Now, $\phi_\beta$ acts on $R$ by fixing $y_i$ and by mapping $x_i$ to $x_i^p$. Hence, $\phi_\beta$ acts like $\delta_{\beta'}(p)$ does on $\mathbb{F}_r^{2m}$ as described in \cite[Section 2.4]{KL}. Similarly, $\iota$ acts like $\delta_{\beta'}(-1)$. So we indeed have that
	\[
	H_{\overline \Sigma} = r^{2m} : \GSp_{2m}(r).\qedhere
	\]
\end{proof}

\begin{prop}\label{prop:C6prelims2}
	Let $H$ be of type $(4 \circ 2 ^{1+2m}). \Sp_{2m}(2)$. Then, if $n=4$ and $c=2$, we have
	\[
	H_{\overline \Omega} = 2^4 . \Alt(6),
	\]
	and otherwise
	\[
	H_{\overline \Omega} = 2^{2m} . \Sp_{2m}(2).
	\]
	In both cases,
	\[
	H_{\overline \Sigma} = (2^{2m} . \Sp_{2m}(2)) \times 2.
	\]
\end{prop}

\begin{proof}
	This can be deduced from \cite[Proposition 4.6.6]{KL} as in the case above.
\end{proof}

\begin{prop}\label{prop:C6prelims3}
	Suppose that we are in case $\boldL$, and $H$ is of type $2_-^{1+2}.\O_2^-(2)$ with $q=p \geq 3$. Then 
	\[
	H_{\overline \Omega} = \Alt(4),
	\]
	and
	\[
	H_{\overline \Sigma} = \Sym(4).
	\]
\end{prop}

\begin{proof}
	See \cite[Proposition 4.6.7]{KL}.
\end{proof}

\begin{prop}\label{prop:C6prelims4}
	Suppose that $G$ is in case $\boldO$ or $\boldS$, and $H$ is of type $2_\pm^{1+2m}.\O_{2m}^\pm(2)$. Then,
	\[
	H_{\overline \Omega} = 2^{2m} . \Omega^+_{2m}(2),
	\]
	and
	\[
	H_{\overline \Sigma} = 2^{2m} . \O^+_{2m}(2).
	\]
\end{prop}

\begin{proof}
	See \cite[Proposition 4.6.8]{KL}.
\end{proof}

\subsubsection{Class $\mathscr{C}_7$}\label{section:C7explanation}

In this case, $H$ stabilises a tensor decomposition $V_1 \otimes \cdots \otimes V_t$ of $V$. Let $\kappa_i$ be forms on each $V_i$, as defined at the beginning of \cite[Section 4.7]{KL}. Note that the vector spaces $V_i$ are pairwise similar. For each $X \in \{\Omega, I, \Delta, \Gamma, \Sigma^*\}$, let $X_i = X(V_i, \mathbb F, \kappa_i)$ for each $i$.

\begin{prop}\label{prop:C7structure}
	Suppose that $H$ is in $\mathscr{C}_7$. Then,
	\[
	(\overline \Omega_1^t : \del((\overline I_1/\overline \Omega_1)^t)) . \Alt(t) \leq H_{\overline \Omega} \leq (\overline I_1 ^t : \del((\overline \Delta_1/\overline I_1)^t)) : \Sym(t),
	\]
	and
	\[
	H_{\overline \Sigma} = (\overline \Delta_1^t : \diag((\overline \Sigma^*_1/\overline \Delta_1)^t)):\Sym(t) \leq \overline \Sigma^*_1 \wr \Sym(t).
	\]
\end{prop}

\begin{proof}
	This proof follows the same argument as the proof of Proposition \ref{prop:C2wreathstructure}, using the fact that $H_{\overline \Delta} = \overline \Delta_1 \wr \Sym(t)$ by \cite[(4.7.3)]{KL}.
\end{proof}

\subsubsection{Class $\mathscr{C}_8$}\label{section:C8explanation}

By \cite[(4.8.1)]{KL}, we have
\[
H_{\Gamma} = \Gamma(V, \mathbb{F}, \kappa_\sharp)
\]
where $\kappa_\sharp$ is one of the non-degenerate forms indicated in \cite[Table 4.8.A]{KL}. We note that the scalars of $\Gamma(V, \mathbb{F}, \kappa_\sharp)$ are equal to those of $\Gamma$. Hence,
\[
H_{\overline \Gamma} = \overline \Gamma(V, \mathbb{F}, \kappa_\sharp).
\]
As previously, for each $X \in \{\Omega, S, I, \Delta, \Gamma, \Sigma\}$, let $X_\sharp$ be $X(V, \mathbb{F}, \kappa_\sharp)$.

\begin{prop}
	Suppose that $H$ is in $\mathscr{C}_8$. Then,
	\[
	H_{\overline \Omega} \geq \overline\Omega_{\sharp},
	\]
	and
	\[
	H_{\overline \Sigma} = \overline\Gamma_\sharp . C,
	\]
	where $C$ is a subgroup of $\overline \Sigma / \overline \Gamma$.
\end{prop}

\begin{proof}
	This follows immediately by the results in \cite[Section 4.8]{KL}, since $H_{\overline \Gamma} = \overline \Gamma_\sharp$ by definition.
\end{proof}

\subsection{Errors in the proof of \texorpdfstring{\cite[Theorem 1]{LMT}}{Theorem 1}}\label{section:LMTcorrections}

Using the descriptions of the group structure of each maximal subgroup $H$ from the previous section, we are now in a position to establish results concerning the chief factors of these groups.

To begin with, we describe the non-Frattini chief factors of the maximal subgroups of the almost simple groups. As mentioned in the introduction, this was treated in \cite[Theorem 1]{LMT}, although there are some errors in the proof of this result, which we wish to identify in this section. We then rectify these mistakes in Section \ref{section:LMTcorrectionproof}. 

The proof of \cite[Theorem 1]{LMT} relies on the theory of crowns by proving that, for any maximal subgroup $H$ of an almost simple group $G$, and for any non-Frattini chief factor $A \not\cong C_2$ of $H$, we have
\[
d((L_A)_{\delta_H(A)}) \leq 3.
\]
Consequently, if $d(H) > 3$, then $\delta_H(C_2) = d(H)$, reducing the problem of bounding $d(H)$ to bounding $\delta_H(C_2)$. 

The mistakes found in \cite{LMT} all occur within the case where $G$ is an almost simple classical group. Hence, to ease notation, we let $(V, \mathbb F, \kappa)$ be a classical geometry with $\overline \Omega$ simple, and let $G$ be a group such that $\Omega \leq G \leq \Sigma$. Thus, $\overline G$ is an almost simple classical group. Assume that $\overline H$ is a maximal subgroup of $\overline G$ which does not contain $\overline \Omega$, and let $H$ denote the preimage of this subgroup in $G$.

We now describe the three mistakes found. Firstly, in the introduction of \cite{LMT}, the properties $(\gamma)_i$, $(F)_i$ and $(D)_i$ are defined as follows: $(\gamma)_1$ holds if $\overline G/(\overline G \cap \overline \Gamma) \cong C_2$, and  $(\gamma)_0$ holds otherwise; $(F)_1$ holds if $\ddot G \cap \langle \ddot \phi \rangle $ has even order (where $\phi$ is as defined in \cite[Chapter 2]{KL}), and $(F)_0$ holds otherwise; and $(D)_i$ holds if $\delta_{\ddot G \cap \ddot \Delta}(C_2) = i$. It is then claimed that $\delta_{\overline G}(C_2) = 3$ if and only if $(F)_1$ and $(D)_2$ hold in the $\boldO$ case, or $(\gamma)_1$, $(F)_1$, and $(D)_1$ hold in the $\boldL$ case. We argue that this is not the case, with the following examples.

\begin{example}\label{example:FDcounterexample}
	Let $(V, \mathbb F, \kappa)$ be in case $\boldO^+$ with $q$ odd and both $f$ and $n$ divisible by 4. Then $\kappa$ has square discriminant, so $\ddot \Sigma = \ddot \Delta \times \langle \ddot \phi \rangle \cong D_8 \times C_f $, by \cite[Proposition 2.7.3]{KL}. Let 
	\[
	G = \langle I, \delta^2, \phi^2, \delta\phi\rangle,
	\]
	for which $\ddot G/ \langle \ddot \phi \rangle \cong D_8$ and $\ddot G \cap \ddot \Delta \cong C_2^2$. Thus one can determine that $\delta_{\overline G}(C_2) = \delta_{\ddot G}(C_2) = 2$, despite $(F)_1$ and $(D)_2$ holding for $G$.
	
	Similarly, let $(V, \mathbb F, \kappa)$ be in case $\boldL$ with $q = 9$ and $n = 4$, so $\ddot \Sigma = \langle \ddot \delta \rangle : \langle \ddot \phi, \ddot \iota \rangle$ is isomorphic to $C_4 : (C_2 \times C_2)$. Note that $\ddot \delta^{\ddot \phi} = \ddot \delta^3$, so $\ddot \phi$ acts by inversion on $\langle \ddot \delta \rangle$, as does $\ddot \iota$. Let $G$ be a subgroup of $\Sigma$ such that $\ddot G = \langle \ddot\delta^2, \ddot\phi, \ddot\delta \ddot\iota \rangle$. Then $(\gamma)_1$, $(F)_1$ and $(D)_1$ hold for this group. However, $\ddot \delta^2 = (\ddot \delta \ddot \iota \ddot \phi)^2$, implying that $\langle \ddot \delta^2 \rangle$ is a Frattini chief factor of $\ddot G$, and so $\delta_{\overline G}(C_2) \leq 2$. Finally, let $G$ be a subgroup of $\Sigma$ such that $\ddot G = \langle \ddot \delta^2, \ddot \delta \ddot \phi, \ddot \delta \ddot \iota \rangle$. Then $\ddot G \cong C_2^3$, despite $(F)_0$ holding for this group.
\end{example}

We correct this error by providing the following new definition of $(\gamma)_i$, $(F)_i$ and $(D)_i$. We additionally define a new property $(H)_i$, which we remark on later in this section.

\begin{definition}
	Let $\Omega \leq G \leq \Sigma$ be a classical group. We say that:
	\begin{enumerate}[(1)]
		\item $(\gamma)_1$ holds if and only if $|\overline G/(\overline G \cap \overline \Gamma)|$ is even, and $(\gamma)_0$ holds otherwise;
		\item $(F)_i$ holds if $\delta_{\overline G, (\overline G \cap \overline \Gamma)/(\overline G \cap \overline \Delta)}(C_2) = i$;
		\item $(D)_i$ holds if $\delta_{\overline G, \overline G \cap \overline \Delta}(C_2) = i$;
		\item $(H)_i$ holds if $\delta_{H_{\overline G}, H_{\overline \Omega}}(C_2) = i$.
	\end{enumerate}
\end{definition}

Note that, if $(\gamma)_i$, $(F)_j$ and $(D)_k$ hold for $G$, then $\delta_{\overline G}(C_2) = i + j + k$ by definition, so the claim in the introduction of \cite{LMT} holds for this new definition. We provide the following lemma to aid the reader in computing $i$, $j$ and $k$.

\begin{lemma}\label{lemma:GammaiFiDi}
	Suppose first that $\overline G$ is not in case $\boldL$. Then $(F)_1$ holds if and only if $|\overline G/(\overline G \cap \overline \Delta)|$ is even, and $(\gamma)_0$ holds. Moreover:
	\begin{enumerate}[\upshape(1)]
		\item If $\overline G$ is in cases $\boldU$ or $\boldS$, then $(D)_1$ holds if and only if $|\ddot G \cap \ddot \Delta|$ is even;
		\item If $\overline G$ is in case $\boldO$, then $(D)_i$ holds if $\delta_{\overline G \cap \overline \Delta}(C_2) = i$, unless $\overline G/\overline \Omega$ projects onto $D_8$ and $(\overline G \cap \overline \Delta)/\overline \Omega \cong C_2^2$. In this case, $(D)_1$ holds, despite $\delta_{\overline G \cap \overline \Delta}(C_2) = 2$.
	\end{enumerate}
	If $\overline G$ is in case $\boldL$ with $n > 2$, then $(F)_1$ holds if and only if $(\overline G \cap \overline \Gamma)/(\overline G \cap \overline \Delta)$ has the same order as $\overline G/(\overline G \cap \langle \overline \Delta, \overline \iota \rangle)$, and is even. If $\overline G$ is in case $\boldL$ with $n = 2$ then $(F)_1$ holds if and only if $|\overline G/(\overline G \cap \overline \Delta)|$ is even.
\end{lemma}

\begin{proof}
	If $G$ is not in case $\boldL$ with $n > 2$, then $\ddot G/(\ddot G \cap \ddot \Delta)$ is a cyclic group, so the first statement of the result is clear.
	
	Suppose that $G$ is in case $\boldS$. When $q$ is even, $\ddot \Delta$ is trivial, so $(D)_0$ holds. Hence, we assume that $q$ is odd. Then $\ddot \Sigma = \ddot \Delta \times \langle \ddot \phi \rangle \cong C_2 \times C_f$. If $|\ddot G \cap \ddot \Delta|$ is odd then $(D)_0$ holds. Otherwise, if $|\ddot G \cap \ddot \Delta|$ is even, then $\ddot G$ is a direct product, and thus $(D)_1$ holds.
	
	If instead $\overline G$ is in case $\boldO$, then $\Out(\overline \Omega)$ is isomorphic to $D \times C_{lf}$, where $D$ is one of $C_2, C_2^2$ or $D_8$, and $l$ is $1$ or $2$. It is clear that $\delta_{\overline G, \overline G \cap \overline \Delta}(C_2)$ depends only on the intersection of $\ddot G = \overline G / \overline \Omega$ with $D$, and its projection to $D$. As such, we can prove the desired result using Magma \cite{magma}.
	
	Assume that $\overline G$ is in case $\boldL$, with $n > 2$. Then $\overline \Sigma / \overline \Delta $ is isomorphic to $ C_f \times C_2$. If the projection of $\overline G$ to $C_f$ is not equal to the intersection of $\overline G / (\overline G \cap \overline \Delta)$ with $C_f$ then this implies that $\overline G / (\overline G \cap \overline \Delta)$ is cyclic, and equal to $\langle \overline \phi^i \overline \iota \rangle(\overline G \cap \overline \Delta)$ for some $i$. Hence, $\overline \phi^{2i}$ is a square in this group. Additionally, $\langle \overline \phi(\overline G \cap \overline \Delta) \rangle \cap \overline G / (\overline G \cap \overline \Delta)$ is equal to $\langle \overline \phi^{2i} \rangle(\overline G \cap \overline \Delta)$, so $(F)_0$ holds.
	
	On the other hand, suppose that the projection of $\overline G$ to $C_f$ is equal to the intersection of $\overline G/(\overline G\cap\overline\Delta)$ with $C_f$. Then $\overline G/(\overline G\cap\overline\Delta)$ is a direct product, and hence $(F)_1$ holds if and only if the projection of $\overline G$ to $C_f$ has even order.
	
	Finally, assume that $\overline G$ is in case $\boldU$. Since $\ddot \Sigma = \langle \ddot \delta \rangle : \langle \ddot \phi \rangle$, there exist some $j, \lambda \in \mathbb N$ such that $\ddot G = \langle \ddot \delta^j, \ddot \delta^\lambda \ddot \phi^i \rangle$, where $(q+1,n)/j = |\ddot G \cap \ddot \Delta|$. We may assume that $2 \mid (n, q+1)/j$, since otherwise it is clear that $\ddot G \cap \ddot \Delta$ contains no chief factors isomorphic to $C_2$. 
	
	The group $N = \langle \ddot \delta^{2j} \rangle$ contains no non-Frattini chief factors isomorphic to $C_2$, so we may consider $\ddot G / N$. This group contains a normal subgroup $\langle \ddot \delta^j \rangle / N$, which is isomorphic to $C_2$. Additionally, $\ddot G/\langle \ddot \delta^j \rangle \cong C_{2f/i}$. As the automorphism group of $C_2$ is trivial, this implies that $\ddot G/N$ is an abelian group. Hence, $\ddot G/N$ is isomorphic to $C_2 \times C_{2f/i}$ or $C_{4f/i}$.
	
	If $(F)_0$ holds then $2f/i$ is odd, and so $\ddot G/N \cong C_2 \times C_{2f/i}$. Thus, $(D)_1$ holds in this case. So we assume that $(F)_1$ holds. For any $l$, we have
	\[
	(\ddot \delta^\lambda \ddot \phi^i)^l = \ddot \delta^{\lambda(1 + p^i + \cdots + p^{i(l-1)})} \ddot \phi^{li}.
	\]
	In particular, 
	\[
	(\ddot \delta^\lambda \ddot \phi^i)^{2f/i} = \ddot \delta^{\lambda \frac{q^2-1}{p^i-1}}.
	\]
	Therefore, $\langle \ddot \delta^\lambda \ddot \phi^i \rangle N $ complements $\langle \ddot \delta^j \rangle/N$ if and only if
	\[
	\lambda \frac{q^2-1}{p^i-1} \equiv 0 \pmod {2j}.
	\]
	As $2 \mid (2f/i)$, in particular $i$ divides $f$. Since $2j$ divides $q+1$, it also divides $\frac{q^2-1}{p^i-1} = (q+1) \frac{q-1}{p^i-1}$. Hence the above always holds, and so $\ddot G/N \cong C_2 \times C_{2f/i}$, implying that $(D)_1$ holds.
\end{proof}

In the lemma above, we do not mention when $(D)_1$ holds in the $\boldL$ case. We defer the explanation for this omission to Remark \ref{remark:Hl}.

The properties $(\gamma)_i$, $(F)_i$ and $(D)_i$ are used in \cite{LMT} as follows. Recall that, in order to bound $d(\overline H)$, our main goal is to bound $\delta_{\overline H}(C_2)$. To do so, we use the following normal series of $\overline H$
\[
H_{\overline \Omega} \leq H_{\overline \Delta \cap \overline G} \leq H_{\overline \Gamma \cap \overline G} \leq \overline H,
\]
from which we obtain
\[
\delta_{\overline H}(C_2) = \delta_{\overline H, H_{\overline \Omega}}(C_2) + \delta_{\overline H, H_{\overline \Delta \cap \overline G}/H_{\overline \Omega}}(C_2) + \delta_{\overline H, H_{\overline \Gamma \cap \overline G}/H_{\overline \Delta \cap \overline G}}(C_2) + \delta_{\overline H / H_{\overline \Gamma \cap \overline G}}(C_2).
\]
However, since $\overline H$ is maximal in $\overline G$, this implies that $\overline H/H_{\overline \Omega} \cong (\overline H \overline \Omega)/\overline \Omega = \overline G/\overline \Omega$. As such, the above is equivalent to
\[
\delta_{\overline H}(C_2) = \delta_{\overline H, H_{\overline \Omega}}(C_2) + \delta_{\overline G, (\overline \Delta \cap \overline G)/\overline \Omega}(C_2) + \delta_{\overline G, (\overline \Gamma \cap \overline G)/(\overline \Delta \cap \overline G)}(C_2) + \delta_{\overline G / (\overline \Gamma \cap \overline G)}(C_2).
\]
Hence, if $(\gamma)_i$, $(F)_j$, $(D)_k$ and $(H)_l$ hold, then we have that $\delta_{\overline H}(C_2) = i + j + k + l$. Thus, our goal when correcting \cite[Theorem 1]{LMT} is to determine conditions on $\overline G$ and $\overline H$ under which $i+j+k+l > 3$. The second mistake we identify in \cite{LMT} concerns the value of $l$ when $H$ is of type $\mathscr C_2$ and $\mathscr C_7$, which we demonstrate with the following example.

\begin{example}
	Let $\overline \Omega = \PSU_4(3)$ and $\ddot G = \langle \ddot \delta^2, \ddot \phi\rangle$. Let $H \in \mathscr C_2$ be of type $\GU_2(3) \wr \Sym(2)$. Then $(\gamma)_0$, $(F)_1$ and $(D)_1$ hold. In \cite{LMT}, it is claimed that $(H)_1$ holds, and hence that $d(\overline H) = 3$. However, by the proof of Proposition \ref{prop:LMTcorrectionLUSC2}, $(H)_2$ holds and $d(\overline H) = 4$.
	
	Similarly, let $\overline \Omega = \PSL_4(9)$ and $\ddot G = \langle \ddot \delta^2, \ddot \phi, \ddot \iota \rangle$, for which $(\gamma)_1$, $(F)_1$ and $(D)_1$ hold. Consider $H \in \mathscr C_2$ of type $\GL_2(9) \wr \Sym(2)$. In \cite{LMT}, it is claimed that $(H)_1$ holds, implying that $d(\overline H) = 4$. However, by the proof of Proposition \ref{prop:LMTcorrectionLUSC2}, $(H)_2$ holds and $d(\overline H) = 5$.
\end{example}

This error leads to several missing rows in \cite[Table 1]{LMT}. For instance, there are no rows in this table in which $\overline G$ is of type $\boldU$.

The final mistake we have identified in \cite{LMT} concerns the number of central non-Frattini chief factors of $\overline H$. In \cite[Theorem 2.7]{LMT} it is claimed that $\delta_{\overline H}(A) \leq 2$ when $A$ is a non-central chief factor, and hence $d((L_A)_{\delta_{\overline H}}(A)) \leq 3$ follows immediately by Theorem \ref{theorem:crowns}. However, we show in the proof of Proposition \ref{prop:LMTcorrectionOC2} that we may have $\delta_{\overline H}(A) = 3$ when $A \cong C_2^2$. Using the stronger bound in Theorem \ref{theorem:crowns} it is possible to show that $d((L_A)_{\delta_{\overline H}}(A)) \leq 3$ still holds in this case. As such, $d(\overline H) > 3$ implies that $d(\overline H) = \delta_{\overline H}(C_2)$ as before, and so the proof strategy of \cite{LMT} remains the same.

In light of the previous mistakes, we give here a corrected version of \cite[Theorem 2.7]{LMT}, for which we require the following notation. Suppose that $\overline G$ is in case $\boldO^\pm$ with $q$ odd, and $\overline H $ in $\mathscr C_2$ or $\mathscr C_7$ is of type $\O_m^\pm(q) \wr \Sym(2)$. Then the socle of $\overline G$ is equal to $\POmega^+_{2m}(q)$ or $\POmega^+_{m^2}(q)$ respectively. The discriminant of the quadratic form of $G$, which we denote by $D(Q)$, must be square, and as such $\ddot \Gamma \cong D_8 \times C_f$. Since $2m$ and $m^2$ are divisible by 4, by \cite[Proposition 2.7.3]{KL} we have $\ddot \Gamma = \ddot \Delta \times \langle \ddot \phi \rangle$. Let $r_\Box$ and $r_\boxtimes$ be two reflections with square and non-square spinor norm respectively, and let $\delta$ be defined as in \cite[Section 2]{KL}. Then, $\ddot r_\Box \ddot r_\boxtimes$ is a rotation by $180$ degrees in $D_8$, $\ddot r_\Box$, $\ddot r_\boxtimes$ are reflections parallel to the square's sides in $D_8$, and $\ddot \delta$ is a reflection through two of the square's vertices. Let $\rho = r_\Box \delta$, so that $\ddot \rho$ is a rotation by $90$ degrees. We let $\pi$ be the quotient map $\pi \colon \overline G \to \overline \Sigma/\langle \overline \Omega, \overline \phi \rangle \cong D_8$.

Additionally, we say $(D^{\mathsf{c}})_1$ holds if $2 \mid |H_{\overline I}\overline \Omega|/|\overline G \cap H_{\overline I}\overline \Omega|$, noting that we can compute $H_{\overline I}\overline \Omega$ using the information detailed in \cite[Section 4]{KL}, specifically the value of $c = [\overline \Sigma : H_{\overline \Sigma} \overline \Omega ]$. Finally, following \cite{LMT}, we say that $(\star)$ holds if $(\gamma)_i$, $(F)_j$ and $(D)_k$ hold with $i + j + k = 3$.

\begin{theorem}\label{theorem:LMTcrowns}
	Let $G$ be a classical group, with $\Omega \leq G \leq \Sigma$, and let $H$ be a maximal subgroup of $G$. Fix a non-Frattini chief factor $A$ of $\overline H$. 
	\begin{enumerate}[\upshape(1)]
		\item If $A$ is non-abelian, $\delta_{\overline H}(A) \leq 2$.
		\item If $A$ is abelian but non-central, then $\delta_{\overline H}(A) \leq 2$, unless $A \cong C_2^2$ and $H$ is in $\mathscr{C}_2$ or $\mathscr{C}_7$ of type $\O_m^\pm(q) \wr \Sym(4)$. In this case $\delta_{\overline H}(A) \leq 3$, and $d((L_A)_{\delta_{\overline H}(A)}) \leq 3$, where $L_A$ is as defined in Section \ref{section:theoryofcrowns}.
		\item If $A$ is central, then $\delta_{\overline H}(A) \leq 3$, unless $A \cong C_2$ and $(\overline G, \overline H)$ are one of the pairs indicated in Table \ref{table:LMTtable}. In this case, $\delta_{\overline H}(A) = d$, where $d$ is as in the last column of Table \ref{table:LMTtable}.
	\end{enumerate}
	The rows of Table \ref{table:LMTtable} marked with $\ast$ indicate rows which were not present in \cite[Table 1]{LMT}. A row marked with $\dagger$ indicates that the row is present in \cite[Table 1]{LMT} but has been modified.
\end{theorem}

\begin{landscape}
	\begin{longtable}{cllllc}
		\caption{The exceptional cases $(\overline G, \overline H)$ in part 2(ii) of Theorem \ref{theorem:LMTgenerators}} \label{table:LMTtable}\\
		
		& $\overline G_0$ & Class & Type & Conditions & $d(H)$ \\\toprule \endfirsthead
		& $\overline G_0$ & Class & Type & Conditions & $d(H)$ \\\toprule \endhead
		& $\L_n(q)$ & $\mathscr{C}_1$ & $P_{m, n-m}$ & $m$ and $n-m$ are even; $(\star)$ holds & $4$\\
		
		& $\L_n(q)$ & $\mathscr C_2$ & $\GL_m(q) \wr \Sym(t)$ & $t > 2$; $(\star)$ holds& $4$\\
		
		$\ast$ & $\L_n(q)$ & $\mathscr{C}_2$ & $\GL_m(q) \wr \Sym(2)$ & $(H)_2$, $(D^{\mathsf{c}})_1$, $(D)_i$, $(F)_j$, $(\gamma)_k$ hold, with $i + j + k \geq 2$ & $2 + i +j + k$\\
		
		$\ast$ & $\L_n(q)$ & $\mathscr{C}_2$ & $\GL_m(q) \wr \Sym(2)$ & $(H)_1$, $(D)_1$, $(F)_1$, $(\gamma)_1$ hold & $4$\\
		
		& $\L_n(q)$ & $\mathscr C_4$ & $\GL_{n_1} (q) \otimes \GL_{n_2}(q)$ & $n_1, n_2$ are even; $q$ is odd; $d(G \langle \delta \rangle/\langle\delta\rangle) = 2$ & $4$\\
		
		& $\L_n(q)$ & $\mathscr{C}_7$ & $\GL_m(q) \wr \Sym(t)$ & $t > 2$; $(\star)$ holds & $4$\\
		
		$\ast$ & $\L_n(q)$ & $\mathscr{C}_7$ & $\GL_m(q) \wr \Sym(2)$ & $(H)_3$, $(D^{\mathsf{c}})_1$, $(F)_i$, $(\gamma)_j$ hold, with $i + j \geq 1$; & $3 + i + j$\\*
		&&&& $2 \nmid |\overline G \cap \overline \Delta|/|\overline \Omega|$; $2 \mid \frac{(q-1,m)^2}{(q-1,n)}$\\
		
		$\ast$ & $\L_n(q)$ & $\mathscr{C}_7$ & $\GL_m(q) \wr \Sym(2)$ & $(H)_2$, $(D^{\mathsf{c}})_1$, $(D)_1$, $(F)_i$, $(\gamma)_j$ hold, with $i + j \geq 1$ & $3 + i + j$\\
		
		$\ast$ & $\L_n(q)$ & $\mathscr{C}_7$ & $\GL_m(q) \wr \Sym(2)$ & $(H)_2$,  $(F)_1$, $(\gamma)_1$ hold; $2 \nmid |\overline G \cap \overline \Delta|/|\overline \Omega|$; $2 \mid \frac{(q-1,m)^2}{(q-1,n)}$ & $4$\\
		
		$\ast$ & $\L_n(q)$ & $\mathscr{C}_7$ & $\GL_m(q) \wr \Sym(2)$ & $(H)_1$, $(D)_1$, $(F)_1$, $(\gamma)_1$ hold & $4$\\
		
		$\ast$ & $\U_n(q)$ & $\mathscr{C}_2$ & $\GU_m(q) \wr \Sym(2)$ & $(H)_2$, $(D)_1$, $(D^{\mathsf{c}})_1$, $(F)_1$ hold & $4$\\
		
		$\ast$ & $\U_n(q)$ & $\mathscr{C}_7$ & $\operatorname{GU}_m(q) \wr \Sym(2)$ & $4 \mid m$; $2 \mid \frac{(q+1,m)^2}{(q+1,n)}$; $2 \nmid |\overline G \cap \overline \Delta|/|\overline \Omega|$; $(H)_3$, $(D^{\mathsf{c}})_1$, $(F)_1$ hold & $4$\\
		
		$\ast$ & $\U_n(q)$ & $\mathscr{C}_7$ & $\operatorname{GU}_m(q) \wr \Sym(2)$ & $4 \mid m$; $(H)_2$, $(D^{\mathsf{c}})_1$, $(D)_1$, $(F)_1$ hold & $4$\\
		
		& $\POmega^\pm_n(q)$ & $\mathscr{C}_1$ & $\O_m^{\epsilon_1}(q) \oplus \O_{n-m}^{\epsilon_2}(q)$ & $q$ is odd; both $Q_i$ have non-square discriminant; & $4$\\*
		& & & & $(\star)$ holds; $(\epsilon_1, \epsilon_2) \neq (-,-)$ & \\
		
		& $\POmega^\pm_n(q)$ & $\mathscr{C}_1$ & $\O_m^{\epsilon_1}(q) \oplus \O_{n-m}^{\epsilon_2}(q)$ & $q$ is odd; $Q_2$ has square discriminant; & $2 + i + j$\\*
		& & & & $(D)_i$ and $(F)_j$ hold with $i + j \geq 2$ & \\
		
		$\dagger$ & $\POmega^\pm_n(q)$ & $\mathscr{C}_2$ & $\O_m^{\epsilon}(q) \wr \Sym(t)$ & $q$ is odd; $t>2$ ; $m > 1$; $(\star)$ holds & $4$\\
		
		$\ast$ & $\POmega^\pm_n(q)$ & $\mathscr{C}_2$ & $\O_m^-(q) \wr \Sym(2)$ & $q$ is odd; $D_1 = \boxtimes$; $(\star)$ holds & $4$\\
		
		$\ast$ & $\POmega^\pm_n(q)$ & $\mathscr{C}_2$ & $\O_m^-(q) \wr \Sym(2)$ & $q$ is odd; $(\overline G \cap \overline \Delta)/\overline \Omega = \pi(\overline G) = \langle \ddot \rho^2, \ddot \delta \rangle$; $(F)_0$ holds & $4$\\
		
		$\ast$ & $\POmega^\pm_n(q)$ & $\mathscr{C}_2$ & $\O_m^+(q) \wr \Sym(2)$ & $q$ is odd; $(\overline G \cap \overline \Delta)/\overline \Omega = \pi(\overline G) = \langle \ddot \rho^2, \ddot \delta \rangle$;  $(F)_i$ holds & $4+i$\\
		
		$\ast$ & $\POmega^\pm_n(q)$ & $\mathscr{C}_2$ & $\O_m^+(q) \wr \Sym(2)$ & $q$ is odd; $(F)_1$ holds; $(\overline G \cap \overline \Delta)/\overline \Omega = \pi(\overline G) \in \{ \langle \ddot r_\Box, \ddot r_\boxtimes \rangle,  D_8 \}$ & $4$\\
		
		$\ast$ & $\POmega^\pm_n(q)$ & $\mathscr{C}_2$ & $\O_m^+(q) \wr \Sym(2)$ & $q$ is odd; $(\overline G \cap \overline \Delta)/\overline \Omega \in \{ \langle \ddot \rho^2 \rangle$, $\langle \ddot \delta \rangle, \langle \ddot \delta \ddot \rho^2 \rangle \}$; & $4$\\*
		& & & & $(F)_1$ and $(H)_2$ hold & \\
		
		$\ast$ & $\POmega^\pm_n(q)$ & $\mathscr{C}_2$ & $\O_m^+(q) \wr \Sym(2)$ & $q$ is odd; $(\overline G \cap \overline \Delta)/\overline \Omega = \pi(\overline G) \in \{  1, \langle \ddot r_\Box \rangle, \langle \ddot r_\boxtimes \rangle \}$; & $4$\\*
		& & & & $(F)_1$ holds; $D_1 = \Box$&  \\
		
		$\ast$ & $\POmega^\pm_n(q)$ & $\mathscr{C}_2$ & $\O_m^\epsilon(q) \wr \Sym(2)$ & $mq$ is odd; $(\overline G \cap \overline \Delta)/\overline \Omega \in \{ \langle \ddot r_\Box \rangle, \langle \ddot r_\boxtimes \rangle\}$;  & $4$\\*
		&&&& $(F)_1$ and $(H)_2$ hold&  \\
		
		$\ast$ & $\POmega^\pm_n(q)$ & $\mathscr{C}_2$ & $\O_m^\epsilon(q) \wr \Sym(2)$ & $mq$ is odd; $(F)_1$ and $(D)_2$ hold & $4$\\
		
		& $\POmega^\pm_n(q)$ & $\mathscr{C}_4$ & $\O_{n_1}^{\epsilon_1}(q) \otimes \O_{n_2}^{\epsilon_2}(q)$ & $q$ is odd; $n_1$ is even; $n_2$ is odd; $(\star)$ holds; & $4$\\*
		& & & & $Q_1$ has square discriminant if $\epsilon_1 = -$ & \\
		
		& $\POmega^\pm_n(q)$ & $\mathscr{C}_4$ & $\O_{n_1}^{\epsilon_1}(q) \otimes \O_{n_2}^{\epsilon_2}(q)$ & $q$ is odd; each $Q_i$ has square discriminant; & $4 + i$\\*
		& & & & $(D)_0$ and $(F)_i$ hold & \\
		
		& $\POmega^\pm_n(q)$ & $\mathscr{C}_4$ & $\O_{n_1}^{\epsilon_1}(q) \otimes \O_{n_2}^{\epsilon_2}(q)$ & $q$ is odd; each $Q_i$ has square discriminant; & $4$\\*
		& & & & $(D)_1$ and $(F)_1$ hold & \\
		
		& $\POmega^\pm_n(q)$ & $\mathscr{C}_4$ & $\O_{n_1}^{\epsilon_1}(q) \otimes \O_{n_2}^{\epsilon_2}(q)$ & $q$ is odd; $\epsilon_i \in \{\pm\}$; $\epsilon_i = +$ if $D(Q_i)$ is non-square; & $4$\\*
		& & & & the $Q_i$ have distinct discriminants; $(D)_1$ holds & \\
		
		& $\POmega^\pm_n(q)$ & $\mathscr{C}_4$ & $\O_{n_1}^{\epsilon_1}(q) \otimes \O_{n_2}^{\epsilon_2}(q)$ & $q$ is odd; each $Q_i$ has non-square discriminant; & $4$\\*
		& & & & $(D)_1$ and $(F)_1$ hold & \\
		
		& $\POmega^\pm_n(q)$ & $\mathscr{C}_4$ & $\O_{n_1}^{\epsilon_1}(q) \otimes \O_{n_2}^{\epsilon_2}(q)$ & $q$ is odd; each $Q_i$ has non-square discriminant; & $4$\\*
		& & & & $(D)_0$ and $(F)_1$ hold & \\
		
		& $\POmega^\pm_n(q)$ & $\mathscr{C}_4$ & $\O_{n_1}^{\epsilon_1}(q) \otimes \O_{n_2}^{\epsilon_2}(q)$ & $q$ is odd; each $Q_i$ has non-square discriminant; & $4 + i$\\*
		& & & & $(D)_1$ and $(F)_i$ hold & \\
		
		& $\POmega^\pm_n(q)$ & $\mathscr{C}_7$ & $\O_m^\epsilon(q) \wr \Sym(3)$ & $(\star)$ holds and either $m \not\equiv 2 \pmod 4$ or $q \not\equiv 3 \epsilon \pmod 4$ & $4$\\
		
		& $\POmega^\pm_n(q)$ & $\mathscr{C}_7$ & $\O_m^\epsilon(q) \wr \Sym(t)$ & $t > 3$; $(\star)$ holds & $4$\\
		
		$\ast$ & $\POmega^\pm_n(q)$ & $\mathscr{C}_7$ & $\O_m^-(q) \wr \Sym(2)$ & $m \equiv 2 \pmod 4$; $D(V_1) = \boxtimes$; $\frac{\overline G \cap \overline \Delta}{\overline \Omega} = \pi(\overline G) = \langle \ddot \rho^2, \ddot \delta \rangle$ & $4$\\
		
		$\ast$ & $\POmega^\pm_n(q)$ & $\mathscr{C}_7$ & $\O_m^-(q) \wr \Sym(2)$ & $m \equiv 2 \pmod 4$; $D(V_1) = \boxtimes$; $\frac{\overline G \cap \overline \Delta}{\overline \Omega} = \pi(\overline G) = \langle \ddot r_\Box, \ddot r_\boxtimes \rangle$; $(F)_1$ holds & $4$\\
		
		$\ast$ & $\POmega^\pm_n(q)$ & $\mathscr{C}_7$ & $\O_m^+(q) \wr \Sym(2)$ & $m \equiv 2 \pmod 4$; $D(V_1) = \boxtimes$; $\frac{\overline G \cap \overline \Delta}{\overline \Omega} = \pi(\overline G) = \langle \ddot \rho^2, \ddot \delta \rangle$; $(F)_i$ holds & $4 + i$\\
		
		$\ast$ & $\POmega^\pm_n(q)$ & $\mathscr{C}_7$ & $\O_m^+(q) \wr \Sym(2)$ & $m \equiv 2 \pmod 4$; $D(V_1) = \boxtimes$; $\frac{\overline G \cap \overline \Delta}{\overline \Omega} < \pi(\overline G) =  \langle \ddot \rho^2, \ddot \delta \rangle$; $(F)_1$ holds & $4$\\
		
		$\ast$ & $\POmega^\pm_n(q)$ & $\mathscr{C}_7$ & $\O_m^+(q) \wr \Sym(2)$ & $m \equiv 2 \pmod 4$; $D(V_1) = \boxtimes$; $(F)_1$ holds; & $4$\\*
		&&&& $\frac{\overline G \cap \overline \Delta}{\overline \Omega} = \pi(\overline G) \in  \{ \langle \ddot \rho^2 \rangle, \langle \ddot \delta \rangle, \langle \ddot \rho^2 \ddot \delta \rangle, \langle \ddot r_\Box, \ddot r_\boxtimes \rangle, D_8 \}$\\
		
		$\ast$ & $\POmega^\pm_n(q)$ & $\mathscr{C}_7$ & $\O_m^\epsilon(q) \wr \Sym(2)$ & $m \equiv 2 \pmod 4$; $D(V_i) = \Box$; $(F)_i$ holds; & $4 + i$\\*
		&&&& $\frac{\overline G \cap \overline \Delta}{\overline \Omega} = \pi(\overline G) \in \{ 1, \langle \ddot \delta \rangle, \langle \ddot \rho^2 \ddot \delta \rangle, \langle \ddot \rho^2, \ddot \delta \rangle \} $\\
		
		$\ast$ & $\POmega^\pm_n(q)$ & $\mathscr{C}_7$ & $\O_m^\epsilon(q) \wr \Sym(2)$ & $m \equiv 2 \pmod 4$; $D(V_i) = \Box$; $(F)_1$ holds; & $4$\\*
		&&&& $\pi(\overline G) < \frac{\overline G \cap \overline \Delta}{\overline \Omega} \in \{ \langle \ddot \delta \rangle, \langle \ddot \rho^2 \ddot \delta \rangle, \langle \ddot \rho^2, \ddot \delta \rangle \} $ \\
		
		$\ast$ & $\POmega^\pm_n(q)$ & $\mathscr{C}_7$ & $\O_m^\epsilon(q) \wr \Sym(2)$ & $m \equiv 2 \pmod 4$; $D(V_1) = \Box$; $(F)_1$ holds; & $4$\\*
		&&&& $\frac{\overline G \cap \overline \Delta}{\overline \Omega} = \pi(\overline G) \in  \{ \langle \ddot \rho^2 \rangle, \langle \ddot r_\Box \rangle, \langle \ddot r_\boxtimes \rangle, \langle \ddot r_\Box, \ddot r_\boxtimes \rangle, D_8 \}$\\
		
		$\ast$ & $\POmega^\pm_n(q)$ & $\mathscr{C}_7$ & $\O_m^+(q) \wr \Sym(2)$ & $4 \mid m$; $\frac{\overline G \cap \langle \overline \Omega, \overline \delta \rangle}{\overline \Omega} = \pi(\overline G \cap \langle \overline \Omega, \overline \delta, \overline \phi \rangle) \in \{ 1, \langle \ddot \delta \rangle\} $; $(F)_1$ holds & $4$\\
		
		$\ast$ & $\POmega^\pm_n(q)$ & $\mathscr{C}_7$ & $\O_m^-(q) \wr \Sym(2)$ & $4 \mid m$; $\frac{\overline G \cap \langle \overline \Omega, \overline \rho^2, \overline \delta \rangle}{\overline \Omega} = \pi(\overline G \cap \langle \overline \Omega, \overline \rho^2, \overline \delta, \overline \phi \rangle) \in \{ \langle \ddot \rho^2 \rangle, \langle \ddot \rho^2, \ddot \delta \rangle\} $; $(F)_1$ holds; & $4$\\
		
		\bottomrule
	\end{longtable}
\end{landscape}

In the following remark, we comment on the difficulty of determining in general whether $(D)_k$ and $(H)_l$ hold, and how this affects the entries in Table \ref{table:LMTtable}. This remark is rather technical and is not needed for the proof of Theorem \ref{theorem:LMTgenerators}, nor for the remainder of this paper. We include it for the interested reader, as it provides context for the formulation of Table \ref{table:LMTtable}.

\begin{remark}\label{remark:Hl}
	Suppose that $G$ is of type $\boldL$ with $n > 2$. The group $\overline \Sigma / \overline \Delta \cong C_f \times C_2$ is not cyclic, making it more difficult to describe explicitly when $(D)_1$ holds, as compared to the cases in Lemma \ref{lemma:GammaiFiDi}. Whether $(D)_1$ holds depends in particular on the structure of $\overline G / (\overline G \cap \overline \Delta)$ and on the preimages of its elements in $\ddot G$. This information is not readily captured by numerical invariants, as is seen in Example \ref{example:FDcounterexample}, and consequently we do not give explicit conditions for when $(D)_1$ holds in this case.
	
	Determining whether $(H)_l$ holds is more difficult still, and is the main focus of the proof of Theorem \ref{theorem:LMTcrowns}. Generally, this problem is very sensitive to the precise structure of $\ddot G$. This presents two main difficulties, both in proving Theorem \ref{theorem:LMTcrowns} and in presenting the resulting conditions in Table \ref{table:LMTtable}. First, there are a very large number of cases to consider, and small variations in the structure of $\ddot G$ can result in substantially different values of $(H)_l$. Second, even after carrying out these computations, incorporating all of the resulting cases into Table \ref{table:LMTtable} would make the table substantially more complicated. 
	
	When $H$ is in $\mathscr C_7$ and is of type $\O_m^\pm(q)\wr\Sym(2)$, we give explicit conditions on $\overline G$ and $\overline H$ for $(H)_l$ to hold for a particular $l$. In this case, $\ddot G$ is a subgroup of $D_8\times C_f$ or $C_2^2\times C_{2f}$, so the possibilities for $\ddot G$ can be described relatively succinctly. Despite this, the first issue mentioned above is a significant hindrance. We therefore reduce the problem of determining the chief factors of $\overline H$ in $H_{\overline \Omega}$ to determining the number of chief factors for a finite collection of groups. This reduces the remaining analysis to a finite computational problem, which can be handled by Magma \cite{magma}. When $H$ is instead of type $\GL_m^\pm(q)\wr\Sym(2)$ or is in $\mathscr C_2$ and is of type $\O_m^\pm(q)\wr\Sym(2)$, the number of possibilities for the structure of $\ddot G$ is much larger, so any attempt to give precise conditions on the structure of $\ddot G$ for which $(H)_l$ holds would require a substantially more complicated table. Moreover, we do not know of a comparable reduction to a computational problem in these cases. For this reason, in some rows of Table \ref{table:LMTtable}, one of the conditions for $d(\overline H) > 3$ is simply that $(H)_l$ holds for some $l$.
	
	Nevertheless, we have aimed to include as much information as possible in each row of Table \ref{table:LMTtable}, so that the table remains as useful as possible to the reader. For instance, consider the first entry of Table \ref{table:LMTtable} where $G$ is of type $\boldU$. Here, the condition for $d(\overline H) = 4$ is that $(H)_2$, $(D)_1$, $(D^{\mathsf{c}})_1$, and $(F)_1$ hold. It would be sufficient for this entry to state only that $(H)_2$, $(D)_1$, and $(F)_1$ hold, since $\delta_{\overline H}(C_2) =4$ if and only if this is the case. However, in the proof of Theorem \ref{theorem:LMTcrowns}, we show that if $(H)_2$ holds, then $(D^{\mathsf{c}})_1$ holds. Since determining whether $(D^{\mathsf{c}})_i$ holds is simpler than determining whether $(H)_l$ holds, we include this condition to help the reader discard pairs $(\overline G, \overline H)$ for which $d(\overline H)$ cannot exceed 3.
\end{remark}

Combining Theorem \ref{theorem:LMTcrowns} with \cite{LMT} allows us to restate a corrected version of \cite[Theorem 1]{LMT}.

\begin{theorem}\label{theorem:LMTgenerators}
	Let $G$ be an almost simple group with socle $G_0$, and let $H$ be a maximal subgroup of $G$. The following assertions hold:
	\begin{enumerate}[\upshape(1)]
		\item $d(H) \leq 5$;
		\item $d(H) \geq 4$ if and only if $\delta_{H}(C_2) = d(H)$ and one of the following holds, with $H$ given up to isomorphism in $\Aut(G_0)$:
		\begin{enumerate}[\upshape(i)]
			\item $G_0$ is an alternating group; $H = (T^k.(\Out(T) \times \Sym(k)) \cap G$ is of diagonal type; $\Sym(k) \leq H$; $d(\Aut(T) \cap H) = 3$. In this case, $d(H) = 4$;
			\item $G_0$ is a classical group, and $(G,H)$ is one of the pairs listed in Table \ref{table:LMTtable}.
		\end{enumerate}
	\end{enumerate}
\end{theorem}

\begin{remark}
	We note that, in the case that the socle of $G$ is a classical group, Theorems \ref{theorem:LMTcrowns} and \ref{theorem:LMTgenerators} are proven for any group $H$ which is a maximal or novelty maximal subgroup of $G$. Indeed, in \cite{LMT} they consider all groups in Aschbacher's classification. The only property of $\overline H$ which is used, other than the group structure described by \cite{KL}, is the fact that $\overline H \overline \Omega/\overline \Omega \cong \ddot G$. This property still holds when $\overline H$ is a novelty maximal subgroup of $\overline G$, and so Theorems \ref{theorem:LMTcrowns} and \ref{theorem:LMTgenerators} hold for such a group. 
\end{remark}

\subsection{Proof of Theorem \ref{theorem:LMTcrowns}}\label{section:LMTcorrectionproof}

Before proving Theorem \ref{theorem:LMTcrowns}, we require the following simple lemma.

\begin{lemma}\label{lemma:diagnonfrattini}
	Let $K = B \times L$ for some groups $B, L$ with $B \cong C_2$. Let $G$ be a subgroup of $K \wr \Sym(2)$ which projects onto $\Sym(2)$, with $G \cap K^2 = \{(k, k^{\phi}) : k \in K\}$ for some isomorphism $\phi$ of $K$ with $B^\phi = B$ and $L^\phi = L$. Let $\sigma$ be the non-trivial element of $\Sym(2)$ in $K \wr \Sym(2)$. If $G$ contains $\sigma$, then the chief factor $G \cap B^2$ is non-Frattini. If $G$ contains $(b,1) \sigma$ where $b$ is the non-identity element of $B$, then $G \cap B^2$ is Frattini.
\end{lemma}

\begin{proof}
	The group $G/(G \cap L^2)$ has order 4. It is isomorphic to $C_2^2$ if $G$ contains $\sigma$, and it is isomorphic to $C_4$ if $G$ contains $(b, 1) \sigma$. The image of $G \cap B^2$ in $C_2^2$ is non-Frattini, while its image in $C_4$ is Frattini, proving the desired result.
\end{proof}

We split the proof of Theorem \ref{theorem:LMTcrowns} into the following results.

\begin{prop}\label{prop:LMTcorrectionLUSC2}
	Let $H \in \mathscr{C}_2$ be of type $\GL_m(q) \wr \Sym(t)$, $\GU_m(q) \wr \Sym(t)$ or $\Sp_m(q) \wr \Sym(t)$. Then the pair $(\overline G, \overline H)$ satisfies Theorem \ref{theorem:LMTcrowns}.
\end{prop}

\begin{proof}
	Note we may assume that $G$ contains the scalars of $\Delta$, and we do so throughout the following proof. By Proposition \ref{prop:C2wreathstructure},
	\[
	H_\Omega = \Omega_1^t . \del((I_1/\Omega_1)^t) . \Sym(t) \leq I_1 \wr \Sym(t).
	\]
	Hence, $H_{\Omega}$ satisfies the conditions of Corollary \ref{cor:CFsinwreathproduct}.
	
	Assume that $\Omega_1$ is quasisimple, and that $t > 2$. Since $G$ is in case $\boldL$, $\boldU$ or $\boldS$, the group $I_1/\Omega_1$ is cyclic, and as such $\delta_{H_{\overline G}, H_{\overline \Omega}} (A)$ is at most 2, with equality only if $A \cong C_2^2$ and $t=4$. Indeed, we have $\delta_{H, H_{\Omega} \cap I_1^t}(A) \leq 1$ by Corollary \ref{cor:CFsinwreathproduct}, with equality only if $A \cong C_r^{t-1}$ (if $r \nmid t$) or $A \cong C_r^{t-2}$ (if $r \mid t$) for some prime $r$. Thus, it is clear that $\delta_{H,H_{\Omega}}(A) \leq 1$ unless $A \in \{C_2^2, C_3 \}$. Additionally, if $A \cong C_3$ then $\overline H$ can have two non-central chief factors isomorphic to $C_3$, but they have different centralisers, so they are non-isomorphic. 
	
	By \cite[Proposition 3.12]{LMT}, we can compute the chief factors in $\delta_{H_{\overline G}/H_{\overline \Omega}}(A)$, since $H_{\overline G}/H_{\overline \Omega}$ is isomorphic to $\overline G / \overline \Omega$. This proves Theorem \ref{theorem:LMTcrowns} when $\Omega_1$ is quasisimple and $t > 2$. 
	
	Next we assume that $t=2$, while $\Omega_1$ remains quasisimple. For this case we require more detailed information about the structure of $H$, which we describe now. By Proposition \ref{prop:C2wreathstructure}, we have
	\[
	H_\Sigma = I_1^t . \diag(( \Sigma_1^*/I_1)^t) . \Sym(t).
	\]
	Note that $H_\Sigma \leq \Sigma^*_1 \wr \Sym(t)$. Additionally, the scalars of $G$ in $H$ are $\diag(Z(\Delta_1)^t) \cap H$. 
	
	By the classification of subgroups of wreath products given in Section \ref{section:wreathproducts}, we must have that
	\[
	H = N_1^t . \del((I_1/N_1)^t) . \diag((P_1/I_1)^t) . \Sym(t) \leq P_1 \wr \Sym(t),
	\]
	where $N_1$ is the largest subgroup of $I_1$ such that $N_1^t$ is contained in $H$, and $P_1$ is the projection of $H \cap (\Sigma_1^*)^t$ to its first coordinate. Note that $N_1 \geq \Omega_1$, and that $H$ satisfies the conditions of Corollary \ref{cor:CFsinwreathproduct}.
	
	We have that $\delta_{H_{\overline G}, H_{\overline \Omega}} (A)$ is at most 2, with equality only if $A$ is $C_2$, and $| I_1/ \Omega_1|$ is even. Hence, applying \cite[Proposition 3.12]{LMT} again, we can prove (1) and (2) of Theorem \ref{theorem:LMTcrowns}. It remains to prove that $\delta_{\overline H}(C_2) \leq 5$, and that $\delta_{\overline H}(C_2) \geq 4$ if and only if $(\overline G, \overline H)$ is in Table \ref{table:LMTtable}.
	
	Firstly, if $G$ is in case $\boldS$, then $I_1/\Omega_1$ is trivial, and so $\delta_{H_{\overline G}, H_{\overline \Omega}} (C_2) = 1$. Thus, $\delta_{\overline H}(C_2) \leq 3$. Similarly, if $G$ is in case $\boldO^\circ$, then  $\delta_{\overline H}(C_2) \leq 3$.
	
	Next, if $G$ is in case $\boldU$, then $\delta_{H_{\overline G}/H_{\overline \Omega}}(C_2)$ is at most $2$, proving that $\delta_{\overline H}(C_2) \leq 4$. Additionally, $\delta_{\overline H}(C_2) = 4$ if and only if $(H)_2$, $(D)_1$ and $(F)_1$ holds. As previously mentioned, it is unfeasible to give numerical conditions on $\ddot G$ which ensure that $(H)_2$ hold. Despite this, it is feasible to determine conditions for $\delta_{H_{\overline G \cap \overline \Delta} , H_{\overline \Omega}}(C_2) = 2$. We have $l \leq \delta_{H_{\overline G \cap \overline \Delta} , H_{\overline \Omega}}(C_2)$, by Corollary \ref{cor:upperbounddeltaGN}, so we use this inequality to remove cases where $(H)_2$ cannot hold. Hence, we consider below under which conditions $\delta_{H_{\overline G \cap \overline \Delta} , H_{\overline \Omega}}(C_2) = 2$, given that $(D)_1$ and $(F)_1$ hold.
	
	Since $(D)_1$ holds, $(\overline G \cap \overline I) / \overline \Omega$ has even order, implying that $N_1 / \Omega_1$ has even order. Thus, $N_1$ contains an element with determinant $-1$, and so $H_{G \cap \Delta}$ contains the natural copy of $\Sym(2)$. Indeed, since $\det(\sigma) = \pm 1$, there exists an element $a \in I_1$ with determinant $\pm 1$ such that $(a, 1) \sigma \in H_{\Omega}$. Since $N_1$ contains $\Omega_1$ and an element with determinant $-1$, we have $a \in N_1$, and hence $\sigma \in H_{G \cap \Delta}$.
	
	Now $\delta_{H_{ G \cap \Delta}, H_{ \Omega}} (C_2) = 2$ if and only if $ H$ contains a non-Frattini chief factor isomorphic to $C_2$ in $\del((I_1/\Omega_1)^2)$. By Lemma \ref{lemma:diagnonfrattini}, this occurs if and only if 2 divides $|I_1/N_1|$, which is equal to $\frac{|\overline I|}{|\overline G \cap \overline I|}$. Indeed, let $Z$ denote the scalars of $\Delta$, which we have assumed are contained in $G$. Then, since $H_{\overline I}/H_{\overline G \cap \overline I} \cong \overline I / (\overline G \cap \overline I)$, we have that $|I_1/N_1| = |H_{IZ}/H_{G \cap IZ}| = |\overline I / (\overline G \cap \overline I)|$. Hence, we have shown that $\delta_{H_{G \cap \Delta}, H_{\Omega}}(C_2) = 2$ if and only if $(D^{\mathsf{c}})_1$ holds. 
	
	Finally, if $\delta_{H_{G \cap \Delta}, H_{\Omega}}(C_2) = 2$ then $\delta_{H_{\overline G \cap \overline \Delta}, H_{\overline \Omega}}(C_2) = 2$ if and only if the chief factor isomorphic to $C_2$ in $\del((I_1/\Omega_1)^2)$ is not contained in $\diag((Z(I_1)\Omega_1/\Omega_1)^2)$. It can be checked that this is the case if $|I_1/\Omega_1|_2 > 2$, where $|I_1/\Omega_1| = q+1$. However, since $(D^{\mathsf{c}})_1$ and $(D)_1$ both hold, we have that $4 \mid (q+1, n)$, and so in particular $|I_1/\Omega_1|_2 \geq 4$.
	
	To summarise, we have shown that, given $(D)_1$ and $(F)_1$ hold, $\delta_{H_{\overline G \cap \overline \Delta} , H_{\overline \Omega}}(C_2) = 2$ if and only if $(D^{\mathsf{c}})_1$ holds. Hence, $\delta_{\overline H}(C_2) = 4$ only if $(D)_1$, $(D^{\mathsf{c}})_1$, and $(F)_1$ hold. 
	
	Assume now that $G$ is in case $\boldL$. We proceed similarly to the $\boldU$ case. Since $I_1/\Omega_1$ is cyclic, and $\delta_{H_{\overline G}/H_{\overline \Omega}}(C_2)$ is at most 3, we have that $\delta_{\overline H}(C_2) \leq 5$. We now consider when $\delta_{\overline H}(C_2) \geq 4$. When $m$ is even, $\sigma$ is contained in $H_\Omega$. If $m$ is odd then, for some element $g \in I_1$ with determinant $-1$, we have $\sigma(g, 1) \in H_\Omega$. For such a $g$, the element $g \Omega_1$ is the unique element of $I_1/\Omega_1$ of order 2. Thus, the non-Frattini chief factor isomorphic to $C_2$ in $\del((\overline I_1 / \overline \Omega_1)^2)$ is non-Frattini in $H_{\overline G \cap \overline \Delta}$ if and only if: $2 \mid |\overline I|/|\overline G \cap \overline I|$; $|I_1/\Omega_1|_2 > 2$; and when $m$ is odd we have $4 \mid | I_1/ \Omega_1|= q-1$. 
	
	If $(F)_1$ holds, then $q$ is a square, and so in particular $4 \mid q-1$. On the other hand, if both $(D)_1$ and $(D^{\mathsf{c}})_1$ hold, then $4 \mid q-1$. Hence, $\delta_{\overline H}(C_2) \geq 4$ implies that $4 \mid q-1$, and if this is the case then $\delta_{H_{\overline G \cap \overline \Delta}, H_{\overline \Omega}}(C_2) = 2$ if and only if $(D^{\mathsf{c}})_1$ holds. This proves Theorem \ref{theorem:LMTcrowns} in this case.
	
	If $\Omega_1$ is instead not quasisimple, we use Section \ref{section:wreathproducts} to compute the non-Frattini chief factors of $\overline H$ contained in $\overline \Omega_1^t$ in order to obtain the same bounds as previously. In particular, we apply Corollary \ref{cor:CFsinwreathproduct} to $H$. Note that, if $t > 2$ and $B$ is a chief factor of $\overline I_1$ contained in $\overline \Omega_1$ which is non-central, then $B^t$ is a chief factor of $\overline H$. In Table \ref{table:CFsofIinOmega} we record the information we require about the groups in which $\overline \Omega_1$ is not simple.
	
	\begin{table}
		\centering
        \caption{The chief factors of $\overline I_1$ in $\overline \Omega_1$ which are non-Frattini in $\overline \Omega_1$, in the case that $\overline \Omega_1$ is not simple, and $\Omega_1 \neq \Omega_1(q)$.}
		\begin{tabular}{lllll}
			&& Chief factors of $\overline I_1$ in $\overline \Omega_1$, \\ 
			Case & $(m, q)$ & non-Frattini in $\overline \Omega_1$ & $I_1 / \Omega_1$ & $\overline \Sigma_1^*/\overline I_1$ \\ \toprule
			$\boldL$ & $(1,q)$ & - & $C_{q-1}$ & $C_f \times C_2$ \\
			& $(2,2)$ & non-central $C_3$, $C_2$ & 1 & $C_2$ \\
			& $(2,3)$ & $C_2^2$, non-central $C_3$ & $C_2$ & $C_2$ \\ \midrule
			$\boldU$ & $(1,q)$ & - & $C_{q+1}$ & $C_{2f}$ \\
			& $(2,2)$ & non-central $C_3$, $C_2$ & $C_3$ & $C_2$ \\
			& $(2,3)$ & $C_2^2$, non-central $C_3$ &  $C_4$ & $C_2$ \\
			& $(3,2)$ & $C_3^2$, $C_2^2$ & $C_3$ & $C_2$ \\ \midrule
			$\boldS$ & $(2,2)$ & non-central $C_3$, $C_2$ & 1 & $C_2$ \\
			& $(2,3)$ & $C_2^2$, central $C_3$ & 1 & $C_2$ \\
			& $(4,2)$ & $\Alt(6)$, $C_2$ & $1$ & $1$ \\ \midrule
			$\boldO$ & $(3,3)$ & $C_2^2$, non-central $C_3$ & $C_2^2$ & 1 \\
			& $(2,q,\pm)$ & $C_r$ for each $r \mid (q \mp 1)/(2,q-1)$ & $C_2^2$ or $C_2$ & $C_f$ or $C_{2f}$ \\
			& $(4,2,+)$ & $C_3^2$, $C_2$, $C_2$ & $C_2$ & $1$ \\
			& $(4,3,+)$ & $C_2^4$, central $C_3$, non-central $C_3$ & $C_2^2$ & $C_2$ \\
			& $(4,q,+)$, $q \geq 4$ & $\PSL_2(q)^2$ & $C_2^2$ or $C_2$ & $C_f$ or $C_{2f}$\\
			\bottomrule
		\end{tabular}
		\label{table:CFsofIinOmega}
	\end{table}
	
	Then the bounds in Theorem \ref{theorem:LMTcrowns} are clear. For instance, consider $\Omega_1 = \PSL_2(2)$. If $t > 2$, the set of non-Frattini chief factors of $H$ in $H_{\Omega}$ is a subset of $\{C_3^{t}, C_2^{t-1}, C_2 \}$. Hence, we obtain that, for any chief factor $A$,
	\begin{align*}
		\delta_{H, H_{\Omega}}(A) & \leq \begin{cases*}
			2 \qquad & if $A \cong C_2$,\\
			1 & otherwise,
		\end{cases*}\\
		\delta_{\overline H/ H_{\overline \Omega}}(A) & \leq 1,
	\end{align*}
	as required. By \cite[Table 8.8]{BHRD}, when $t=2$ the group $H_{\overline \Omega}$ is not a maximal or novelty maximal subgroup of $\PSL_4(2)$, so this case does not arise.
\end{proof}

\begin{prop}\label{prop:LMTcorrectionOC2}
	Let $H \in \mathscr{C}_2$ be of type $\O_m^\epsilon(q) \wr \Sym(t)$. Then the pair $(\overline G, \overline H)$ satisfies Theorem \ref{theorem:LMTcrowns}.
\end{prop}

\begin{proof}
	As in the proof of the previous proposition, we assume that $G$ contains the scalars of $\Delta$. Suppose that $m \neq 1$. Then, by Proposition \ref{prop:C2wreathstructure},
	\[
	H_\Omega = \Omega_1^t . \del((I_1/\Omega_1)^t) . \Sym(t) \leq I_1 \wr \Sym(t),
	\]
	and $H_{\Omega}$ satisfies the conditions of Corollary \ref{cor:CFsinwreathproduct}.
	
	Assume that $\Omega_1$ is quasisimple, and that $t > 2$. The section $I_1/\Omega_1$ is isomorphic to $C_2^2$ or $C_2$, and so $\delta_{H_{\overline G}, H_{\overline \Omega}} (A)$ is at most 3, with equality only if $A \cong C_2^2$ and $t=4$. Note additionally that $\delta_{H_{\overline G}, H_{\overline \Omega}} (C_2) = 1$, since the $C_2$ chief factor in $\Sym(t)$ is non-Frattini in $H_{\overline G}$. 
	
	By \cite[Proposition 3.12]{LMT}, we can compute the chief factors in $\delta_{H_{\overline G}/H_{\overline \Omega}}(A)$, since $H_{\overline G}/H_{\overline \Omega}$ is isomorphic to $\overline G / \overline \Omega$. This proves Theorem \ref{theorem:LMTcrowns} when $\Omega_1$ is quasisimple and $t > 2$. 
	
	Before moving on, we show that if $A$ is a non-central chief factor such that $\delta_{\overline H}(A) = 3$, then $d((L_A)_{\delta_{\overline H}(A)}) \leq 3$. By above, $A \cong C_2^2$, and $\overline H$ acts on $A$ as $\Sym(4)$ acts on its normal subgroup isomorphic to $C_2^2$. Using Magma \cite{magma}, we can compute that $s(A) = 0$ and $r(A) = 2$, and so $h(A) = 3$, see Theorem \ref{theorem:crowns}. Thus, $d((L_A)_{\delta_{\overline H}(A)}) = 3$. 
	
	Next we assume that $t=2$, while $\Omega_1$ remains quasisimple. For this case we use the detailed information about the structure of $H$ given in the proof of Proposition \ref{prop:LMTcorrectionLUSC2}. In particular, we have that $H$ is a subgroup of $\Sigma_1^* \wr \Sym(t)$. Additionally,
	\[
	H = N_1^t . \del((I_1/N_1)^t) . \diag((P_1/I_1)^t) . \Sym(t) \leq P_1 \wr \Sym(t),
	\]
	where $N_1$ is the largest subgroup of $I_1$ such that $N_1^t$ is contained in $H$, and $P_1$ is the projection of $H \cap (\Sigma_1^*)^t$ to its first coordinate. Note that $N_1 \geq \Omega_1$, and that $H$ satisfies the conditions of Corollary \ref{cor:CFsinwreathproduct}.
	
	As above, $\delta_{H_{\overline G}, H_{\overline \Omega}} (A)$ is at most 3 when $A$ is $C_2$, and is at most 1 otherwise. Hence, applying \cite[Proposition 3.12]{LMT} again, we can prove (1) and (2) of Theorem \ref{theorem:LMTcrowns}. It remains to prove that $\delta_{\overline H}(C_2) \leq 5$, and that $\delta_{\overline H}(C_2) \geq 4$ if and only if $(\overline G, \overline H)$ is in Table \ref{table:LMTtable}.
	
	If $G$ is in case $\boldO^\circ$, then  $\delta_{\overline H}(C_2) \leq 3$, so we assume that $G$ is in case $\boldO^\pm$. If $q$ is even, then $\delta_{\overline H}(C_2) \leq 3$. So we assume that $H$ is of type $\O_m^\epsilon(q) \wr \Sym(2)$ with $m > 2$ and $q$ odd. In this case we have that $I_1/\Omega_1$ is isomorphic to $C_2^2$. Suppose that $\delta_{H_{\overline G}/H_{\overline \Omega}}(C_2) = 3$. Then $(\overline G \cap \overline I)/\overline \Omega$ is not trivial, and so $\delta_{H_{\overline G}, H_{\overline \Omega}} (C_2) \leq 2$. Thus, for any $\overline H$ of this type we have that $\delta_{\overline H}(C_2) \leq 5$. 
	
	We next consider when $\delta_{\overline H}(C_2) \geq 4$. Note that, if $m$ is even, then either $Z(I_1)$ is contained in $\Omega_1$, if $D_1 = \Box$, or $S_1 = \Omega_1 \times Z(I_1)$, if $D_1 = \boxtimes$. On the other hand, if $m$ is odd, then $Z(I_1)$ is not contained in $\Omega_1$, and $I_1 = S_1 \times Z(I_1)$. We prove that $\del((S_1/\Omega_1)^2)$ is a non-Frattini chief factor of $\overline H$ only if all of the following hold:
	\begin{enumerate}[\upshape(1)]
		\item $D_1 = \Box$, or $m$ is odd,
		\item $\overline P_1/\overline I_1 \leq C_2 \times C_f$ does not project onto $C_2$,
		\item $N_1 \ngeq S_1$,
		\item if $H$ contains $\sigma (1,a)$, then $a \in \Omega_1$ or $a \not\in S_1$,
		\item when $m$ is even, $\pi(\overline G) \ngeq \pi(\overline S)$.
	\end{enumerate}
	Condition (1) holds if and only if $\del((S_1/\Omega_1)^2)$ is not contained in the scalars of $H$. If condition (2) does not hold then $\del((S_1/\Omega_1)^2)$ is always Frattini by Lemma \ref{lemma:jumpingbabyresultdel}. Next, by Corollary \ref{cor:CFsinwreathproduct}, if (3) does not hold then $\del((S_1/\Omega_1)^2)$ is Frattini. If $H$ contains $\sigma(1,a)$ where $a \in S_1 \setminus \Omega_1$, then $\del((S_1/\Omega_1)^2)$ is Frattini, by Lemma \ref{lemma:diagnonfrattini}. Finally, we assume that (5) does not hold. We may assume that (4) does not hold also, as otherwise the result follows. Let $\ddot \rho^2 \ddot \phi^i$ be in $\ddot G$ for some $i$. Hence, $(\phi_1, \phi_2)^i(1, b)$ is in $H$ for some $b \in S_1 \setminus \Omega_1$, where $\phi_i \in \Sigma_i$ are as defined in \cite[Sections 2.7 and 2.8]{KL}. Since $m$ is even, $\det(\sigma) = (-1)^m = 1$, so $\sigma \in H$. Note that $[(\phi_1, \phi_2)^i (1, b), \sigma]\Omega_1^2 = (b, b)\Omega_1^2$. Therefore, the central chief factor $\del((S_1/\Omega_1)^2)$ of $H$ is contained in $[H/\Omega_1^2, H/\Omega_1^2]$, and so it is Frattini. This implies that its image in $\overline H$ is Frattini also.

	Additionally, if $m$ is even, $\ddot G \leq D_8 \times C_f$ is a direct product, and conditions (1)-(5) hold, then $\del((S_1/\Omega_1)^2)$ is a non-Frattini chief factor of $\overline H$. Indeed, suppose that this holds. By (2), any complement of $S_1/\Omega_1$ in $I_1/\Omega_1$ is a $P_1$-normal subgroup. Choose a complement of $S_1/\Omega_1$ in $I_1/\Omega_1$ which contains $a \Omega_1$ for any $a \in I_1$ with $\sigma (1,a) \in H$. This exists by (3) and (4). Let $L$ be the preimage of this complement in $I_1$. Note that $H/(L^2 \cap H)$ is a subgroup of the wreath product $(P_1/L)\wr \Sym(2)$ which contains the natural copy of $\Sym(2)$. Since (3) holds and $\ddot G$ is a direct product, we can apply Lemma \ref{lemma:diagnonfrattini}, which states that the chief factor $\del((S_1/\Omega_1)^2)$ is non-Frattini in $H/(L^2 \cap H)$. In fact, by the proof of this lemma, since (1) holds and $m$ is even, there exists a maximal subgroup of $H/(L^2 \cap H)$ which contains the image of the scalars of $H$, and does not contain $\del((S_1/\Omega_1)^2)$. Hence, $\del((S_1/\Omega_1)^2)$ is non-Frattini in $\overline H$.
	
	Similarly, $\del((I_1/S_1)^2)$ is a non-Frattini chief factor of $\overline H$ only if:
	\begin{enumerate}[\upshape(1)]
		\item $m$ is even,
		\item $N_1 \leq S_1$,
		\item $(\ddot G \cap \langle \ddot I, \ddot \phi \rangle)/\langle \ddot \phi \rangle$ is contained in $\langle \ddot \rho^2, \ddot \phi \rangle/\langle \ddot \phi \rangle$,
		\item if $\sgn(Q_1) = -$ and $D_1 = \boxtimes$ then $(F)_0$ holds,
		\item if $H$ contains $\sigma(1,a)$ then $a \in S_1$.
	\end{enumerate}
	Moreover, if $\ddot G$ is a direct product and (1)-(5) hold, then $\del((I_1/S_1)^2)$ is a non-Frattini chief factor of $\overline H$. This can be shown by considering $H/(H \cap S_1^2)$ and applying Lemma \ref{lemma:diagnonfrattini}, as above.

	We use these results to determine the bounds in Theorem \ref{theorem:LMTcrowns}. Assume first that $m$ is even. By \cite[Proposition 2.5.13]{KL}, $D$ is always square, and so we use the notation introduced above the statement of Theorem \ref{theorem:LMTcrowns} for the elements of $(\overline G \cap \overline \Delta)/\overline \Omega \cong D_8$. First note that $\det(\sigma) = (-1)^m$, and so $\det(\sigma) = 1$ in this case, and $\sigma \in S$. On the other hand, the spinor norm of $\sigma$ is equal to
	\[
	\theta(\sigma) = \begin{cases*}
		\mu & if $D_1 = \boxtimes$,\\
		1 & if $D_1 = \Box$,
	\end{cases*}
	\]
	where $\mu$ is a generator of $\mathbb{F}^*$. As such, $\sigma \in \Omega$ if and only if $D_1 = \Box$. Otherwise, if $D_1 = \boxtimes$, then $\sigma \in S \setminus \Omega$, so $(1, a) \sigma \in H$ for some $a \in S_1 \setminus \Omega_1$. Thus the first set of conditions is equivalent to: $D_1 = \Box$; $\overline G/\overline I \leq C_2 \times C_f$ does not project onto $C_2$; and $\overline G \ngeq \overline S$. However, $\overline G/\overline I \leq C_2 \times C_f$ does not project onto $C_2$ if and only if $\overline G/\overline I$ is contained in $C_f$. Hence, $\del((S_1/\Omega_1)^2)$ is a non-Frattini chief factor of $\overline H$ only if: $D_1 = \Box$; $\overline G  \leq \langle \overline I, \overline \phi \rangle$; and $\overline G \ngeq \overline S$. The second set of conditions is similarly equivalent to: $\pi(\overline G) \cap \pi(\overline I) \leq \pi(\overline S$); and $(F)_0$ holds if $\sgn(Q_1) = -$ and $D_1 = \boxtimes$. Additionally, note that $\sgn(Q_1) = -$ and $(F)_1$ holds implies that $D_1 = \boxtimes$, so the final condition is equivalent to: $(F)_0$ holds if $\sgn(Q_1) = -$. Considering each of the possible subgroups of $D_8$ and their number of non-Frattini chief factors isomorphic to $C_2$ gives Theorem \ref{theorem:LMTcrowns}. The case where $m$ is odd is similar, using the fact that $\det(\sigma)=-1$.
	
	Suppose finally that $\Omega_1$ is not quasisimple. We continue to assume that $m \neq 1$, so that the structure of $H$ and $H_{\Omega}$ is as indicated previously. We hence treat this case as in the proof of Proposition \ref{prop:LMTcorrectionLUSC2}, by applying Corollary \ref{cor:CFsinwreathproduct} and using Table \ref{table:CFsofIinOmega}. Recall that, if $t > 2$ and $B$ is a chief factor of $\overline I_1$ contained in $\overline \Omega_1$ which is non-central, then $B^t$ is a chief factor of $\overline H$.
	
	If $\Omega_1$ is not $\Omega_2^\pm(q)$, $\Omega_4^+(2)$ or $\Omega_4^+(3)$, then the bounds in Theorem \ref{theorem:LMTcrowns} are clear, as previously. 
	
	If $\Omega_1$ is equal to $\Omega_2^\pm(q)$, then the bounds are clear other than the bound on $\delta_{\overline H}(C_2)$. We must have that $t>2$ in this case, since $\overline G_0$ is simple. If $q$ is even, then $\delta_{ H, \Omega_1^t}(C_2) = 0$, so we can assume that $q$ is odd. Recall that $\delta_{H, \del((I_1/\Omega_1)^t)}(C_2) = 0$. Note that $I_1$ is equal to a dihedral group with group of rotations $S_1$, which strictly contains $\Omega_1$. Since $t > 2$ and $G$ contains $\Omega_1^t. \del((I_1/\Omega_1)^t)$, we have that $\delta_{H, \Omega_1^t}(C_2) = 0$.
	
	In the case where $\Omega_1$ is equal to $\Omega_4^+(2)$, we treat $t=2$ by computer to show that $\delta_{\overline H}(C_2) \leq 3$. On the other hand, if $t > 2$, we treat $H$ as in the $\Omega_1 = \Omega_2^\pm(q)$ case in order to prove that $\delta_{ H_{\overline \Omega}}(C_2) \leq 2$ and so $\delta_{\overline H}(C_2) \leq 3$.
	
	Finally, assume that $\Omega_1 = \Omega_1(q)$. In this case, $H_{\Omega}$ is one of $\del(2^t). \Alt(t)$ or $\del(2^t).\Sym(t)$, and $q = p \geq 3$, by \cite[Proposition 4.2.15]{KL}. Note that, since $\overline \Omega$ is simple, $t \geq 4$. In fact, by \cite[Table 8.17]{BHRD} we have $t \geq 5$. Hence, by the proof of Corollary \ref{cor:CFsinwreathproduct}, the non-Frattini chief factors of $H$ in $H_\Omega$ are equal to: $C_2^{t-1}$ or $C_2^{t-2}$, according as $t$ is odd or even; $\Alt(t)$; and $C_2$ if $H_{\Omega}$ projects onto $\Sym(t)$. Thus, $\delta_{\overline H}(A) \leq 1 + \delta_{\overline H/H_{\overline \Omega}}(A) \leq 3$, with equality only if $A \cong C_2$, as required.
\end{proof}

\begin{prop}
	Let $H \in \mathscr{C}_7$ be of type: $\O_m(q) \wr \Sym(t)$ with $m$ odd; $\GL_m^\pm(q) \wr \Sym(t)$; or $\Sp_m(q) \wr \Sym(t)$. Then the pair $(\overline G, \overline H)$ satisfies Theorem \ref{theorem:LMTcrowns}.
\end{prop}

\begin{proof}
	We have
	\[
	\overline H = N_1^t . \del((\overline \Delta_1/N_1)^t) . \diag((P_1/\overline \Delta_1)^t) . J,
	\]
	where
	\begin{align*}
		\Alt(t) \leq J & \leq \Sym(t)\\
		\overline \Omega_1 \leq N_1 & \leq \overline \Delta_1\\
		\overline \Delta_1 \leq P_1 & \leq \overline \Sigma_1^*
	\end{align*}
	and $\overline H$ contains the natural copy of $\Alt(t)$. Additionally, if $t > 2$ then $J = \Sym(t)$. Since for cases $\boldL$ and $\boldU$ we have $\overline \Delta_1 = \overline I_1$, this proof follows in the same way as for $\mathscr{C}_2$. We describe some of the details below.
	
	Suppose that $G$ is in case $\boldL$, and $H$ is of type $\GL_m(q) \wr \Sym(2)$. Let $K_1$ be a subgroup of $\overline \Delta_1$ such that $H_{\overline \Omega} \cap \overline \Delta_1^2 = K_1^2 . \del((\overline \Delta_1 / K_1)^2)$. Then, $|K_1/\overline \Omega_1| = (q-1,m)^2/(q-1,n)$. This differs from the $\mathscr{C}_2$ case, where $H_\Omega \cap I_1^t = \Omega_1^t . \del((I_1/\Omega_1)^t)$. If $q$ is even, then $\delta_{\overline H, \overline \Delta_1^2}(C_2) = 0$ and $(D)_0$ holds, so $\delta_{\overline H}(C_2) \leq 3$. Thus, we assume in the following that $q$ is odd.
	
	Suppose first that $m \equiv 2 \pmod 4$, and $q \equiv -1 \pmod 4$, so $H_{\overline \Omega}$ does not project onto $\Sym(2)$, by \cite[Proposition 4.7.3]{KL}. Since $q \equiv -1 \pmod 4$, $q$ is not a square, and so $(F)_1$ does not hold in this case. Additionally $4 \nmid (q-1,m)$ and $2 \mid |K_1/\overline \Omega_1|$, so $\delta_{\overline H, \del((\overline \Delta_1/N_1)^2)}(C_2) = 0$. Moreover, $\delta_{\overline H, N_1^2}(C_2) \leq 1$ if $\overline H$ projects onto $\Sym(2)$, and $\delta_{\overline H, N_1^2}(C_2) \leq 2$ otherwise. Finally, if $2 \mid |H_{\overline G \cap \overline I}|/|H_{\overline \Omega}|$ then $\overline H$ projects onto $\Sym(2)$. 
	
	Suppose that $\delta_{\overline H/ H_{\overline \Omega}}(C_2) = 2$, so we have that $(D)_1$ and $(\gamma)_1$ hold, since $(F)_0$ holds. Hence, $2 \mid |H_{\overline G \cap \overline I}|/|H_{\overline \Omega}|$, and so $\delta_{\overline H, N_1^2}(C_2) \leq 1$, implying that $\delta_{\overline H}(C_2) \leq 3$. On the other hand, if $\delta_{\overline H/ H_{\overline \Omega}}(C_2) \leq 1$, then we also obtain $\delta_{\overline H}(C_2) \leq 3$ from the bounds above.
	
	Next assume that $m \not\equiv 2 \pmod 4$, or $q \not\equiv -1 \pmod 4$ so $H_{\overline \Omega}$ projects onto $\Sym(2)$. Moreover, assume that $H_{\overline \Omega}$ does not contain $\overline \sigma$. Note that $\det(\sigma) = (-1)^{\frac{1}{2} m(m-1)}$, and hence $\overline \sigma \not\in \overline \Omega$ if and only if: $4 \nmid m$ and $4 \nmid (m-1)$, so that $\det(\sigma) = -1$; and $2 \nmid (q-1)/(q-1,n)$, so that $-1$ is not in $\det(Z(\Delta))$. If $m$ is odd, then $n$ is also odd, and $\delta_{\overline H, H_{\overline \Omega}}(C_2) = 1$, implying that $\delta_{\overline H}(C_2) \leq 3$. Hence we can assume that $m \equiv 2 \pmod 4$, and in particular that $q \equiv 1 \pmod 4$. Therefore, $2 \nmid (q-1)/(q-1,n)$ implies that in fact $q \equiv 5 \pmod 8$. As such, $q$ is not a square, and so again $(F)_1$ cannot hold in this case. Additionally, since $m \equiv 2 \pmod 4$, we have the inequality $\delta_{\overline H, \del((\overline \Delta_1/N_1)^2)}(C_2) + \delta_{\overline H, N_1^2}(C_2) \leq 1$, and thus $\delta_{\overline H}(C_2) \leq 3$.
	
	As such, we can assume that $\overline H$ contains $\overline \sigma$. By Lemma \ref{lemma:diagnonfrattini}, $\delta_{ H_{\overline G \cap \overline \Delta}, \del((\overline \Delta_1/\overline \Omega_1)^2)}(C_2)$ is 1 if and only if $2 \mid |\overline \Delta_1/N_1|$ and is 0 otherwise. Similarly, $\delta_{H_{\overline G \cap \overline \Delta}, K_1^2}(C_2)$ is 1 if and only if $2 \mid |K_1/\overline \Omega_1|$ and is 0 otherwise. Using $|\overline \Delta_1 / N_1| = |\overline \Delta \cap H_{\overline \Delta} \overline \Omega|/|\overline G \cap H_{\overline \Delta} \overline \Omega|$ and $|K_1/\overline \Omega_1| = (q-1,m)^2/(q-1,n)$ gives the results in Table \ref{table:LMTtable}.
	
	If $\overline H$ is of type $\Sp_m(q) \wr \Sym(t)$ or $\O_m(q) \wr \Sym(t)$ the proofs that $\delta_{\overline H}(C_2) \leq 3$ are similar to the $\mathscr{C}_2$ case, so we do not describe the details here. 
\end{proof}

In order to prove the result for $\overline H$ of type $\O_m(q) \wr \Sym(t)$ with $m$ even, we first prove the following lemma. This result enables us to compute $\delta_{\overline H}(C_2)$ when $H$ is of type $\O_m(q) \wr \Sym(2)$ by determining the chief factors of a finite collection of groups.

\begin{lemma}\label{lemma:reducingLMTtofinitegp}
	Let $H \in \mathscr{C}_7$ be of type $\O_m^\pm(q) \wr \Sym(2)$, where either $\sgn(Q_1) = +$ or $D(V_1) = \Box$. Let $D \cong \overline \Delta_1 / \overline \Omega_1$. Then, for each group $\overline L$ with $\overline \Omega \leq \overline L \leq \overline \Sigma$, there exists a homomorphism $\pi_{\overline L} \colon H_{\overline L} \to (D \wr C_2) \times C_4$ such that $\delta_{H_{\overline L}}(C_2) = \delta_{\pi_{\overline L}(H_{\overline L})}(C_2)$. Additionally, the map 
	\begin{align*}
		\{ \overline L :  \overline \Omega \leq \overline L \leq \overline \Sigma \} & \to \{ K : \pi_{ \overline \Omega}(H_{\overline \Omega}) \leq K \leq \pi_{\overline \Gamma}(H_{\overline \Gamma})\}\\
		\overline L & \mapsto \pi_{\overline L}(H_{\overline L})
	\end{align*}
	is surjective.
\end{lemma}

\begin{proof}
	Let $r_1, r_2$ be two reflections of $V_1, V_2$ respectively, so that $\langle S_i, r_i \rangle = I_i$. Similarly, let $\delta_1, \delta_2$ be elements as defined in \cite[Sections 2.7 and 2.8]{KL} for $V_1$, $V_2$, so that $\langle I_i, \delta_i \rangle = \Delta_i$. Finally, let $\phi_i$ be defined as in \cite[Sections 2.7 and 2.8]{KL} so that $\Gamma_i = \langle \Delta_i, \phi_i \rangle$, and let $\phi = (\phi_1, \phi_2)$.
	
	Let $K = H_{\overline \Gamma} \overline \Omega$, and note that $H_{\overline G} = H_{\overline G \cap K}$ for all $\overline \Omega \leq \overline G \leq \overline \Gamma$. Additionally, $H_{\overline G}/H_{\overline \Omega} \cong (\overline G \cap K)/\overline \Omega$. Note that, when $\overline H \mleq \overline G$, we have that $\overline G \leq K$. However, we do not assume that this is the case.
	
	We have
	\[
	\frac{H_{\overline \Gamma}}{\overline \Omega_1^2} = \frac{H_{\overline \Delta}}{\overline \Omega_1^2} \times C_f \leq (D \wr \Sym(2)) \times C_f.
	\]
	Let $a$ be the order of $(\overline G \cap K)/(\overline G \cap K \cap \overline \Delta) \leq C_f$. Hence, $H_{\overline G}/\overline \Omega_1^2$ is contained in $(H_{\overline \Delta}/\overline \Omega_1^2) \times C_a$, and is equal to $(L_1 \times L_2) . C$ where $L_1 = H_{\overline G \cap \overline \Delta}/\overline \Omega_1^2$, $L_2.C \cong C_a$, and $C$ is a subquotient of $\overline \Delta/\overline \Omega \cong D_8$. As such, $C \cong C_b$ where $b = 1, 2$ or 4, and $L_2 \cong C_{a/b}$.
	
	Let $g$ be an element of order $f$ in $H_{\overline \Gamma}$ such that 
	\[
	\frac{H_{\overline \Gamma}}{\overline \Omega_1^2} = \frac{H_{\overline \Delta}}{\overline \Omega_1^2} \times \langle g \overline \Omega_1^2 \rangle.
	\]
	Note that we can choose $g = \overline \phi$ when $\frac{1}{4}m(p-1)$ is even or $D(V_i) = \boxtimes$, and can choose $g = (\overline r_1, \overline r_2) \overline \phi$ otherwise. We define the map $\pi_{\overline G}$ as a quotient map by a subgroup $N_{\overline G}$ of $\langle \overline \Omega_1^2, g \rangle$ which contains no non-Frattini chief factors of $H_{\overline G}$ isomorphic to $C_2$. This implies that $\delta_{H_{\overline G}}(C_2) = \delta_{\pi_{\overline G}(H_{\overline G})}(C_2)$.
	
	If $a$ is not divisible by 2, then there are no chief factors isomorphic to $C_2$ in $N_{\overline G} := \langle \overline \Omega_1^2, g^{f/a} \rangle \unlhd H_{\overline G}$. Therefore, we can take 
	\[
		\pi_{\overline G} \colon H_{\overline G} \to H_{\overline G}/N_{\overline G} \leq D \wr \Sym(2) \leq (D \wr \Sym(2)) \times C_4.
	\]
	If instead $2 \mid a$ but $4 \nmid a$, then there are no chief factors isomorphic to $C_2$ in $N_{\overline G} := \langle \overline \Omega_1^2, g^{2f/a} \rangle \unlhd H_{\overline G}$, so we can proceed as above. Finally, if $4 \mid a$, we show that any chief factor isomorphic to $C_2$ in $N_{\overline G} := \langle \overline \Omega_1^2, g^{4f/a} \rangle \unlhd H_{\overline G}$ is Frattini. 
	
	Let $h$ be an element of $H_{\overline \Delta}$ such that $h g^{f/a} \in H_{\overline G}$. Note that $h \overline \Omega_1^2$ and $g^{f/a} \overline \Omega_1^2$ commute in $H_{\overline \Gamma}/\overline \Omega_1^2$. If the order of $h \overline \Omega_1^2$ is less than 4, then $g^{4f/a}\overline \Omega_1^2$ is a square in $H_{\overline G}/\overline \Omega_1^2$, proving that there are no non-Frattini chief factors of $H_{\overline G}$ in $\langle \overline \Omega_1^2, g^{4f/a} \rangle$ isomorphic to $C_2$, as required. Thus, we suppose that $h \overline \Omega_1^2$ has order 8. This implies that $\overline \Delta_1/\overline \Omega_1 \cong D_8$. However, it can be shown that $h \overline \Omega_1^2 = (a_1, a_2)\sigma$ where $a_i \in D_8$ have $|a_1a_2| = 4$ and $\sigma$ is the non-identity element of $\Sym(2)$. However, by the proof of \cite[Propositions 4.7.6, 4.7.7]{KL} (see also the proof of Proposition \ref{prop:LMTC7calcs}), $H_{\overline \Omega}/\overline \Omega_1^2$ contains an element $(g_1, g_2) \in D_8^2$ such that $(g_1 a_1, g_2 a_2) \sigma$ has order less than 4. Hence, $H_{\overline G}/\overline \Omega_1^2$ contains $(g_1 a_1, g_2 a_2) \sigma (g^{f/a} \overline \Omega_1^2)$, and so $g^{4f/a}$ is a square in $H_{\overline G}/\overline \Omega_1^2$.
	
	Let $\pi_{\overline G}$ be the quotient map $\overline G \to \overline G/N_{\overline G}$. We have shown that $\delta_{H_{\overline G}}(C_2) = \delta_{\pi_{\overline G}(H_{\overline G})}(C_2)$. Additionally,
	\[
	\pi_{\overline \Omega}(H_{\overline \Omega}) = \frac{H_{\overline \Omega}}{\overline \Omega_1^2} \times 1,
	\]
	and
	\[
	\pi_{\overline \Gamma}(H_{\overline \Gamma}) = \frac{H_{\overline \Delta}}{\overline \Omega_1^2} \times C_{(f,4)}.
	\]
	Hence $\pi_{\overline \Omega}(H_{\Omega}) \leq \pi_{\overline G}(H_{\overline G}) \leq \pi_{\overline \Gamma}(H_{\overline \Gamma})$.
	
	We finally show that the map $\overline G \mapsto \pi_{\overline G}(H_{\overline G})$ given in the statement of the lemma is surjective. Suppose that $L$ is a group with $\pi_{\overline \Omega}(H_{\overline \Omega}) \leq L \leq \pi_{\overline \Gamma}(H_{\overline \Gamma})$. Therefore, $L / \pi_{\overline \Omega}(H_{\overline \Omega})$ is a subgroup of $\pi_{\overline \Gamma}(H_{\overline \Gamma})/\pi_{\overline \Omega}(H_{\overline \Omega})$, which is isomorphic to $(\ddot K \cap D_8) \times C_4$. Let the projection of $L / \pi_{\overline \Omega}(H_{\overline \Omega})$ to $C_4$ be $C_b$. Note that $b \mid f$. Thus, there exists a group $\overline G $ with $\overline \Omega \leq \overline G \leq \overline \Gamma$ such that $\overline G / \overline \Omega$, as a subgroup of $D_8 \times C_f$, is equal to $L / \pi_{\overline \Omega}(H_{\overline \Omega}) \leq D_8 \times C_b$. Note that $N_{\overline G} = \overline \Omega_1^2$ in this case. This implies that 
	\[
	\frac{\pi_{\overline G}(H_{\overline G})}{\pi_{\overline \Omega}(H_{\overline \Omega})} = \frac{L}{\pi_{\overline \Omega}(H_{\overline \Omega})},
	\]
	and hence $\pi_{\overline G}(H_{\overline G}) = L$.
\end{proof}

\begin{prop}\label{prop:LMTC7calcs}
	Let $H \in \mathscr{C}_7$ be of type $\O_m(q) \wr \Sym(t)$ with $m$ even. Then the pair $(\overline G, \overline H)$ satisfies Theorem \ref{theorem:LMTcrowns}.
\end{prop}

\begin{proof} 
	In this case, $\overline \Omega_1$ is simple, and, by \cite[Propositions 4.7.6, 4.7.7]{KL}, $N_1$ contains $\overline S_1$. If $t > 2$ then the proof follows as in the $\mathscr{C}_2$ case, so we additionally assume that $t=2$. Then, as in the $\mathscr{C}_2$ case, we can show that $\delta_{\overline H}(A) \leq 2$ for any chief factor $A$ not isomorphic to $C_2$. Hence, we show that $\delta_{\overline H}(C_2) \leq 5$, and we consider when $\delta_{\overline H}(C_2) \geq 4$. In order to do so, we apply Lemma \ref{lemma:reducingLMTtofinitegp} to reduce to computing the chief factors of subgroups of $(D \wr \Sym(2)) \times C_4$. We split into cases, depending on $H_{\overline \Omega}$. 
	
	Firstly, if $m \equiv 2 \pmod 4$ and $D(V_i) = \boxtimes$, then $H_{\overline \Omega}/\overline \Omega_1^2$ is equal to $\langle (\overline r_1, \overline r_2), (\overline r_1 \overline\delta_1, \overline\delta_2) \rangle$. Additionally, $\overline \Delta_1 / \overline \Omega_1 \cong C_2^2$. Let $\widetilde r_i := \overline r_i \overline \Omega_i$, $\widetilde \delta_i := \overline \delta_i \overline \Omega_i$ and let $\widetilde \phi$ be an element of order 4 in $(D \wr \Sym(2)) \times C_4$ so that this group is equal to $(D \wr \Sym(2)) \times \langle \widetilde \phi \rangle$. Using Magma \cite{magma}, we find that the subgroups of $((C_2^2) \wr \Sym(2)) \times C_4$ containing $L := \langle (\widetilde r_1, \widetilde r_2), (\widetilde r_1 \widetilde \delta_1, \widetilde \delta_2) \rangle$ with $\delta_K(C_2) \geq 4$ are:
	\begin{enumerate}[\upshape(1)]
		\item $\langle L, (\widetilde r_1, 1), (\widetilde \delta_1, 1) \widetilde \phi^2 \rangle$ with $\delta_K(C_2) = 4$;
		\item $\langle L, (\widetilde r_1, 1), (\widetilde \delta_1, 1) \rangle$ with $\delta_K(C_2) = 4$;
		\item $\langle L, (\widetilde r_1, 1) \widetilde \phi ^2, (\widetilde \delta_1, 1) \rangle$ and its conjugate $\langle L, (\widetilde r_1, 1) \widetilde \phi ^2, (\widetilde r_1 \widetilde \delta_1, 1) \rangle$ with $\delta_K(C_2) = 4$;
		\item $\langle L, (\widetilde r_1, 1), \widetilde \phi^2 \rangle$ with $\delta_K(C_2) = 4$;
		\item $\langle L, (\widetilde \delta_1, 1), \widetilde \phi^2 \rangle$ and its conjugate $\langle L, (\widetilde r_1 \widetilde \delta_1, 1), \widetilde \phi^2 \rangle$ with $\delta_K(C_2) = 4$;
		\item $\langle L, (\widetilde r_1, 1), \widetilde \phi^2, (\widetilde \delta_1, 1) \widetilde \phi \rangle$ with $\delta_K(C_2) = 4$;
		\item $\langle L, (\widetilde r_1, 1), \widetilde \phi \rangle$ with $\delta_K(C_2) = 4$;
		\item $\langle L, (\widetilde \delta_1, 1), \widetilde \phi \rangle$ and its conjugate $\langle L, (\widetilde r_1 \widetilde \delta_1, 1), \widetilde \phi \rangle$ with $\delta_K(C_2) = 4$;
		\item $\langle L, (\widetilde r_1, 1), \widetilde \phi^2, \overline \sigma \rangle$ with $\delta_K(C_2) = 4$;
		\item $\langle L, (\widetilde r_1, 1) \widetilde \phi, (\widetilde \delta_1, 1), \widetilde \phi^2 \rangle$ and its conjugate $\langle L, (\widetilde r_1, 1) \widetilde \phi, (\widetilde r_1 \widetilde \delta_1, 1), \widetilde \phi^2 \rangle$ with $\delta_K(C_2) = 4$;
		\item $\langle L, (\widetilde r_1, 1), (\widetilde \delta_1, 1), \widetilde \phi^2, \overline \sigma \rangle$ with $\delta_K(C_2) = 4$;
		\item $\langle L, (\widetilde r_1, 1), \widetilde \phi, \overline \sigma \rangle$ with $\delta_K(C_2) = 4$;
		\item $\langle L, (\widetilde r_1, 1), (\widetilde \delta_1, 1), \widetilde \phi, \overline \sigma \rangle$ with $\delta_K(C_2) = 4$;
		\item $\langle L, (\widetilde r_1, 1), (\widetilde \delta_1, 1), \widetilde \phi^2 \rangle$ with $\delta_K(C_2) = 5$;
		\item $\langle L, (\widetilde r_1, 1), (\widetilde \delta_1, 1), \widetilde \phi \rangle$ with $\delta_K(C_2) = 5$.
	\end{enumerate}
	
	For each subgroup above, we find all almost simple groups $\overline G$ such that $H_{\overline G}$ corresponds to a subgroup above. In order to do so, we note that $c=1$ in this case, and so $H_{\overline \Gamma}/H_{\overline \Omega} \cong \overline \Gamma / \overline \Omega$. For each element of $H_{\overline \Gamma}/H_{\overline \Omega}$ we compute the spinor norm, the determinant, and what each element scales the form by in order to determine that this isomorphism maps
	\begin{align*}
		(\widetilde r_1, 1) & \mapsto \ddot \rho^2\\
		(\widetilde \delta_1, 1) & \mapsto \ddot \delta\\
		\widetilde \phi & \mapsto \ddot \phi\\
		\overline \sigma & \mapsto \ddot r_{\Box} \text{ or } \ddot r_{\boxtimes}.
	\end{align*}
	Suppose that $\sgn(Q_1) = +$, so we may apply Lemma \ref{lemma:reducingLMTtofinitegp}. In Table \ref{table:LMTcorrectionC71}, we list which groups $\overline G$ correspond to the groups listed above. Note, in the third column of this table, the map $\pi$ is the quotient map of $\overline \Gamma$ by $\langle \overline \Omega, \overline \phi \rangle$. Additionally, in the fourth column a dash indicates that no conditions on $|\overline G \cap \langle \overline \Omega, \overline \phi \rangle|/|\overline \Omega|$ are required. However note that in all such cases, since $\overline G / \overline \Omega \leq D_8 \times C_f$ is not a direct product, we must have that $(F)_1$ holds.
	\begin{table}
		\centering
        \caption{Characterisation of the cases for which $\delta_{H_{\overline G}}(C_2) \geq 4$, given $m \equiv 2 \pmod 4$ and $D(V_i) = \boxtimes$.}
		\begin{tabular}{lllcl}
			Subgroup & $(\overline G \cap \overline \Delta)/\overline \Omega$ & $\pi(\overline G)$ & $2 \mid |\overline G \cap \langle \overline \Omega, \overline \phi \rangle|/|\overline \Omega|$ & $\delta_{H_{\overline G}}(C_2)$ \\ \toprule
			1, 6 & $\langle \ddot \rho^2 \rangle$ & $\langle \ddot \rho^2, \ddot \delta \rangle$ & - & 4\\
			2 & $\langle \ddot \rho^2, \ddot \delta \rangle$ & $\langle \ddot \rho^2, \ddot \delta \rangle$ & No & 4\\
			3, 10 & $\langle \ddot \delta \rangle$, $\langle \ddot  \rho^2 \ddot \delta \rangle$ & $\langle \ddot \rho^2, \ddot \delta \rangle$ & - & 4\\
			4, 7 & $\langle \ddot \rho^2 \rangle$ & $\langle \ddot \rho^2 \rangle$ & Yes & 4\\
			5, 8 & $\langle \ddot \delta \rangle$, $\langle \ddot  \rho^2 \ddot \delta \rangle$ & $\langle \ddot \delta \rangle$, $\langle \ddot  \rho^2 \ddot \delta \rangle$ & Yes & 4\\
			9, 12 & $\langle \ddot r_{\Box}, \ddot r_{\boxtimes} \rangle$ & $\langle \ddot r_{\Box}, \ddot r_{\boxtimes} \rangle$ & Yes & 4\\
			11, 13 & $D_8$ & $D_8$ & Yes & 4\\
			14, 15 & $\langle \ddot \rho^2, \ddot \delta \rangle$ & $\langle \ddot \rho^2, \ddot \delta \rangle$ & Yes & 5
			\\
			\bottomrule
		\end{tabular}
		\label{table:LMTcorrectionC71}
	\end{table}
	
	If instead $\sgn(Q_1) = -$ then we note that $\overline  \phi^f \overline \Omega_1^2 = (\overline r_1, \overline r_2)\overline \Omega_1^2$, so if $(F)_1$ holds then $\langle (\overline r_1, \overline r_2) \rangle \overline \Omega_1^2$ is a Frattini chief factor of $\overline H$. Hence, in this case we have $\delta_{\overline H, H_{\overline \Omega}}(C_2) \leq 1$, and so  $\delta_{\overline H}(C_2) \geq 4$ only if $(\star)$ holds. This implies that $\ddot G = L \times C_a$ with $L \in \{\langle \ddot \rho^2, \ddot \delta \rangle, \langle \ddot r_\Box, \ddot r_\boxtimes \rangle, D_8\}$ and $a$ an even divisor of $f$. However, if $L$ is $D_8$ then $\delta_{\overline H, H_{\overline \Omega}}(C_2) = 0$, and if $L$ is one of $\langle \ddot \rho^2, \ddot \delta \rangle$ or $\langle \ddot r_\Box, \ddot r_\boxtimes \rangle$ then $\delta_{\overline H, H_{\overline \Omega}}(C_2) = 1$. So $\delta_{\overline H}(C_2) = 4$ if and only if $\pi(\overline G) = \frac{\overline G \cap \overline \Delta}{\overline \Omega} \in \{\langle \ddot \rho^2, \ddot \delta \rangle, \langle \ddot r_\Box, \ddot r_\boxtimes \rangle\}$ and $(F)_1$ holds. If instead $(F)_0$ holds then the same analysis as above shows that $\delta_{\overline H}(C_2) = 4$ if and only if $\pi(\overline G) = \frac{\overline G \cap \overline \Delta}{\overline \Omega} = \langle \ddot \rho^2, \ddot \delta \rangle$.
	
	Next assume that $m \equiv 2 \pmod 4$ and $D(V_i) = \Box$, so we may apply Lemma \ref{lemma:reducingLMTtofinitegp}. In this case $H_{\overline \Omega}$ is equal to $\overline I_1^2$, and $c=1$. The isomorphism $H_{\overline \Gamma}/H_{\overline \Omega}$ to $\overline \Gamma/\overline \Omega$ maps
	\begin{align*}
		(\widetilde \delta_1, \widetilde \delta_2) & \mapsto \ddot \rho^2\\
		(\widetilde \delta_1, 1) & \mapsto \ddot \delta\\
		\widetilde \phi & \mapsto \ddot \phi\\
		\overline \sigma & \mapsto \ddot r_{\Box} \text{ or } \ddot r_{\boxtimes}.
	\end{align*}
	
	Using the same method as before, we can determine that $H_{\overline G}$ has $\delta_{H_{\overline G}}(C_2) \geq 4$ if and only if it is contained in Table \ref{table:LMTcorrectionC72}. As previously, a dash in the third column indicates that there are no restrictions on $|\overline G \cap \langle \overline \Omega, \overline \phi \rangle|/|\overline \Omega|$.
	
	\begin{table}
		\centering
        \caption{Characterisation of the cases for which $\delta_{H_{\overline G}}(C_2) \geq 4$, given $m \equiv 2 \pmod 4$ and $D(V_i) = \Box$.}
		\begin{tabular}{llcl}
			$(\overline G \cap \overline \Delta)/\overline \Omega$ & $\pi(\overline G)$ & $2 \mid |\overline G \cap \langle \overline \Omega, \overline \phi \rangle|/|\overline \Omega|$ & $\delta_{H_{\overline G}}(C_2)$ \\ \toprule
			1 & 1 & No & 4\\
			1 & $\langle \ddot \delta \rangle$, $\langle \ddot \rho^2 \ddot \delta \rangle$ & - & 4\\
			$\langle \ddot \delta \rangle$, $\langle \ddot \rho^2 \ddot \delta \rangle$ & $\langle \ddot \delta \rangle$, $\langle \ddot \rho^2 \ddot \delta \rangle$ & No & 4\\
			$\langle \ddot \rho^2 \rangle$ & $\langle \ddot \rho^2, \ddot \delta \rangle$ & - & 4\\
			$\langle \ddot \rho^2 \rangle$ & $\langle \ddot \rho^2 \rangle$ & Yes & 4\\
			$\langle \ddot \rho^2, \ddot \delta \rangle$ & $\langle \ddot \rho^2, \ddot \delta  \rangle$ & No & 4\\
			$\langle \ddot \delta \rangle$, $\langle \ddot \rho^2 \ddot \delta \rangle$ & $\langle \ddot \rho^2, \ddot \delta \rangle$ & - & 4\\
			$\langle \ddot r_{\Box} \rangle$, $\langle \ddot r_{\boxtimes} \rangle$ & $\langle \ddot r_{\Box} \rangle$, $\langle \ddot r_{\boxtimes} \rangle$ & Yes & 4\\
			$\langle \ddot r_{\Box}, \ddot r_{\boxtimes} \rangle$ & $\langle \ddot r_{\Box}, \ddot r_{\boxtimes} \rangle$ & Yes & 4\\
			$D_8$ & $D_8$ & Yes & 4
			\\
			1 & 1 & Yes & 5\\
			$\langle \ddot \delta \rangle$, $\langle \ddot \rho^2 \ddot \delta \rangle$ & $\langle \ddot \delta \rangle$, $\langle \ddot \rho^2 \ddot \delta \rangle$ & Yes & 5\\
			$\langle \ddot \rho^2, \ddot \delta \rangle$ & $\langle \ddot \rho^2, \ddot \delta  \rangle$ & Yes & 5\\
			\bottomrule
		\end{tabular}
		\label{table:LMTcorrectionC72}
	\end{table}
	
	Next assume that $m \equiv 0 \pmod 4$, and that $\sgn(Q_i) = +$. As such, $c=4$, and $H_{\overline \Omega} = \langle \overline I_1^2, (\overline \delta_1, \overline \delta_2), \overline \sigma \rangle$. In this case, since $c=4$, we have that $H_{\overline \Gamma}/H_{\overline \Omega} \cong (H_{\overline \Gamma} \overline \Omega)/\overline \Omega < \overline \Gamma / \overline \Omega$. In particular, by the same method as previously, we have $(H_{\overline \Gamma} \overline \Omega)/\overline \Omega = \langle \ddot \delta, \ddot \phi \rangle$. The isomorphism from $H_{\overline \Gamma}/H_{\overline \Omega}$ to $\langle \ddot \delta, \ddot \phi \rangle$ maps $(\overline \delta_1, \overline \delta_2)$ to $\ddot \delta$ and $(\overline \phi_1, \overline \phi_2)$ to $\ddot \phi$. Thus, we can determine that $H_{\overline G}$ has $\delta_{H_{\overline G}}(C_2) \geq 4$ if and only if it is contained in Table \ref{table:LMTcorrectionC73}.
	
	\begin{table}
		\centering
        \caption{Characterisation of the cases for which $\delta_{H_{\overline G}}(C_2) \geq 4$, given $m \equiv 0 \pmod 4$ and $\sgn(Q_i) = +$.}
		\begin{tabular}{llcl}
			$(\overline G \cap \langle \overline \Omega, \overline \delta \rangle)/\overline \Omega$ & $\pi(\overline G \cap \langle \overline \Omega, \overline \delta, \overline \phi \rangle)$ & $2 \mid |\overline G \cap \langle \overline \Omega, \overline \phi \rangle|/|\overline \Omega|$ & $\delta_{H_{\overline G}}(C_2)$ \\ \toprule
			1 & 1 & Yes & 4\\
			$\langle \ddot \delta \rangle$ & $\langle \ddot \delta \rangle$ & Yes & 4
			\\
			\bottomrule
		\end{tabular}
		\label{table:LMTcorrectionC73}
	\end{table}
	
	Finally, suppose that $m \equiv 0 \pmod 4$ and that $\sgn(Q_i) = -$. In this case, $c=2$ and $H_{\overline \Omega} = \langle \overline \Omega_1^2, (\overline r_1, \overline r_2), (\overline \delta_1, \overline \delta_2), (\overline r_1, 1) \overline \sigma \rangle$. Indeed, $\theta(\sigma)$ is equal to $2^{\binom{m}{2}} \mu^{n-1} (\mathbb{F}^\times)^2$, where $\mu$ is a generator of $\mathbb{F}^\times$. Since $4 \mid m$, $2^{\binom{m}{2}}$ is a square while $\mu^{n-1}$ is not a square, so we have that $\theta(\sigma) = \boxtimes$. Additionally, we have that $(H_{\overline \Gamma} \overline \Omega)/\overline \Omega$ is equal to $\langle \ddot \rho^2, \ddot \delta, \ddot \phi \rangle$ and
	\begin{align*}
		(\overline r_1, 1) & \mapsto \ddot \rho^2\\
		(\overline \delta_1, 1) & \mapsto \ddot \delta\\
		(\overline \phi_1, \overline \phi_2) & \mapsto \ddot \phi.
	\end{align*}
	
	Note that $\langle (\overline r_1, \overline r_2) \rangle \overline \Omega_1^2$ is Frattini, since it is a square. By the same method as previously, $H_{\overline G}$ has $\delta_{H_{\overline G}}(C_2) \geq 4$ if and only if it is contained in Table \ref{table:LMTcorrectionC74}.
	
	\begin{table}
		\centering
        \caption{Characterisation of the cases for which $\delta_{H_{\overline G}}(C_2) \geq 4$, given $m \equiv 0 \pmod 4$ and $\sgn(Q_i) = -$.}
		\begin{tabular}{llcl}
			$(\overline G \cap \langle \overline \Omega, \overline \rho^2, \overline \delta \rangle)/\overline \Omega$ & $\pi(\overline G \cap \langle \overline \Omega, \overline \rho^2, \overline \delta, \overline \phi \rangle)$ & $2 \mid |\overline G \cap \langle \overline \Omega, \overline \phi \rangle|/|\overline \Omega|$ & $\delta_{H_{\overline G}}(C_2)$ \\ \toprule
			$\langle \ddot \rho^2 \rangle$ & $\langle \ddot \rho^2 \rangle$ & Yes & 4\\
			$\langle \ddot \rho^2, \ddot \delta \rangle$ & $\langle \ddot \rho^2, \ddot \delta \rangle$ & Yes & 4
			\\
			\bottomrule
		\end{tabular}
		\label{table:LMTcorrectionC74}
	\end{table}
\end{proof}

\subsection{Non-Frattini chief factors of maximal subgroups of classical groups}

If $G$ is any classical group with $\Omega \leq G \leq \Sigma$ for which $\overline G$ is almost simple, we can use Theorem \ref{theorem:LMTcrowns} to determine a bound on the non-Frattini chief factors of any maximal subgroup $H \mleq G$, using the inequality $\delta_H(A) \leq \delta_{\overline H}(A) + 1$. In the lemma below we determine a bound on $\delta_H(A)$ in the case that $\overline G$ is not almost simple.

\begin{lemma} \label{lemma:non_simple_gen_bound}
	Suppose that $G$ is a classical group with $\Omega \leq G \leq \Sigma$, such that $\Omega$ is not quasisimple. Let $H$ be a maximal subgroup of $G$. Then a bound on the non-Frattini chief factors of $\Omega, I, G, \Omega \cap H$ and $H$ is found in Tables \ref{table:chieffactorsnonqs1} and \ref{table:chieffactorsnonqs2}, which, by \cite[Proposition 2.9.2]{KL} list all of the possibilities for $\Omega$.
\end{lemma}

\begin{table} 
	\begin{center}
		\caption{Bounds on the number of non-Frattini chief factors isomorphic to $C_2$ in certain subgroups of $\Sigma$, when $\Omega$ is not quasisimple, for some $H \mleq G$.} \label{table:chieffactorsnonqs1}
		\begin{tabular}{llrrrrr}
			Case & $(m, q)$ & $\delta_{\Omega}(C_2)$ & $\delta_{I \cap G}(C_2)$ & $\delta_{G}(C_2)$ & $\delta_{\Omega \cap H}(C_2)$ & $\delta_{H}(C_2)$ \\ \toprule
			$\boldL$ & $(1,q)$ & 0 & 1 & 2 & 0 & 2 \\
			& $(2,2)$ & 1 & 1 & 1 & 1 & 1 \\
			& $(2,3)$ & 0 & 1 & 1 & 2 & 2 \\ \midrule
			$\boldU$ & $(1,q)$ & 0 & 1 & 2 & 0 & 2 \\
			& $(2,2)$ & 1 & 1 & 2 & 1 & 2 \\
			& $(2,3)$ & 0 & 1 & 2 & 2 & 4 \\
			& $(3,2)$ & 2 & 2 & 3 & 2 & 3 \\ \midrule
			$\boldS$ & $(2,2)$ & 1 & 1 & 1 & 1 & 1 \\
			& $(2,3)$ & 0 & 0 & 1 & 2 & 3 \\
			& $(4,2)$ & 1 & 1 & 1 & 2 & 3 \\ \midrule
			$\boldO$ & $(1,q)$ & 0 & 2 & 3 & 0 & 3 \\
			& $(3,3)$ & 0 & 2 & 2 & 2 & 4 \\
			& $(2,q,\pm)$ & 1 & 2 & 4 & 1 & 4 \\
			& $(4,2,+)$ & 2 & 2 & 2 & 2 & 2 \\
			& $(4,3,+)$ & 0 & 2 & 2 & 4 & 4 \\
			& $(4,q,+)$, $q \geq 4$ & 0 & 2 & 3 & 4 & 6 \\ \bottomrule
		\end{tabular}
	\end{center}
\end{table}

\begin{table}
	\begin{center}
		\begin{threeparttable}
			\caption{Bounds on the number of non-Frattini chief factors equivalent to any $A \not\cong C_2$ in certain subgroups of $\Sigma$, when $\Sigma$ is not almost quasisimple.}\label{table:chieffactorsnonqs2}
			\begin{tabular}{llrrrrr}
				Case & $(m, q)$ & $\delta_{\Omega}(A)$ & $\delta_{I \cap G}(A)$ & $\delta_{G}(A)$ & $\delta_{\Omega \cap H}(A)$ & $\delta_{H}(A)$ \\ \toprule
				$\boldL$ & $(1,q)$ & 0 & 1 & 2 & 0 & 2 \\
				& $(2,2)$ & 1 & 1 & 1 & 1 & 1 \\
				& $(2,3)$ & 1 & 1 & 1 & 1 & 1 \\ \midrule
				$\boldU$ & $(1,q)$ & 0 & 1 & 2 & 0 & 2 \\
				& $(2,2)$ & 1 & 2 & 2 & 1 & 2 \\
				& $(2,3)$ & 1 & 1 & 1 & 1 & 1 \\
				& $(3,2)$ & 1 & 1 & 1 & 3\tnote{c} & 3\tnote{c} \\ \midrule
				$\boldS$ & $(2,2)$ & 1 & 1 & 1 & 1 & 1 \\
				& $(2,3)$ & 1 & 1 & 1 & 1 & 1 \\
				& $(4,2)$ & 1 & 1 & 1 & 1 & 1 \\ \midrule
				$\boldO$ & $(1,q)$ & 0 & 0 & 2 & 0 & 2 \\
				& $(3,3)$ & 1 & 1 & 1 & 1 & 1 \\
				& $(2,q,\pm)$ & 1 & 1 & 3\tnote{c} & 1 & 3\tnote{c} \\
				& $(4,2,+)$ & 2 & 2 & 2 & 2 & 2 \\
				& $(4,3,+)$ & 2 & 2 & 1 & 2 & 4 \\
				& $(4,q,+)$, $q \geq 4$ & 2\tnote{na} & 2\tnote{na} & 2\tnote{na} & 2 & 3 \\ \bottomrule
			\end{tabular}
			\begin{tablenotes}
				\item[c] Only achieved if $A$ is a central chief factor.
				\item[na] Only achieved if $A$ is a non-abelian chief factor.
			\end{tablenotes}
		\end{threeparttable}
	\end{center}
\end{table}

\begin{proof}
	If $q$ is equal to $2$ or $3$, then we can prove the result computationally. If $m=1$ then $\Omega = 1$, and we can use the structure results for $\Sigma/\Omega$ given in \cite[Chapter 2]{KL} to prove the desired result. Hence, we assume that $q > 3$ and $m > 1$, and in particular that $\Omega$ is of type $\boldO$. If $m=2$, then $\Omega$ is cyclic, and $I$ is a dihedral group. Thus the result is clear by \cite[Sections 2.7 and 2.8]{KL}.
	
	So we assume that $m=4$, and that $q \geq 4$. Then,
	\[
	\Omega = \SL_2(q) \circ \SL_2(q),
	\]
	with $\SL_2(q)$ quasisimple. So the entries for $\delta_{\Omega}(A)$, $\delta_I(A)$, and $\delta_G(A)$ given in the table are clear. Let $H$ be a maximal subgroup of $G$ such that $H \cap \Omega < \Omega$, and hence $H = (H \cap \Omega) . \left( \frac{G}{\Omega} \right)$.
	
	Suppose first that $q$ is odd. Following the proof of \cite[Proposition 2.9.1]{KL}, we now construct $\Sigma$ by adding generators to $\Omega$. Let $W_1, W_2$ be two $2$-dimensional vector spaces over $\mathbb{F}_q$, with symplectic respective bilinear forms $\kappa_1, \kappa_2$. Let $X_i := X(W_i, \kappa_i)$ for any $X$ in the chain of subgroups (\ref{eq:chainofsubgps}). Then, 
	\[
	\Sigma = \langle \Delta_1 \circ \Delta_2, \phi, \sigma \rangle,
	\]
	where $\phi$ is a field automorphism which acts the same on both $W_1$ and $W_2$, and $\sigma$ swaps the copies of $\Delta_i$.
	
	If the projection of $G$ to $\Sym(2)$ is trivial, then by Proposition \ref{prop:maxsubgpdirectproduct} and \cite[Table 8.1]{BHRD} we can prove the desired result.
	
	Suppose instead that $G$ projects onto $\Sym(2)$. If $H \cap (\SL_2(q) \circ \SL_2(q))$ is central in $\SL_2(q) \circ \SL_2(q)$, or is isomorphic to $Z(\SL_2(q)).\PSL_2(q)$, then $\delta_{H \cap \Omega}(A) \leq 1$ for any $A$, thus proving the desired result. Hence, by Proposition \ref{prop:XSimpleMaxSubgrpXkC}, we may assume that $H \cap (\SL_2(q) \circ \SL_2(q))$ is equal to $R_1 \circ R_2$ for some isomorphic maximal or novelty maximal subgroups $R_i$ of $\SL_2(q)$. By \cite[Table 8.1]{BHRD}, for any chief factor $A$ we have $\delta_{R_i}(A) \leq 2$, with equality only if $A \cong C_2$. If $B \not \cong C_2$ then this implies by Corollary \ref{cor:upperbounddeltaGN} that $\delta_{H, H \cap \Omega}(B) \leq 2$ and so $\delta_H(B) \leq 3$ as required. So we may assume that $B \cong C_2$. However, $d(H) \leq d(R_1) + 4 \leq 6$, so $\delta_H(C_2) \leq 6$ by Theorem \ref{theorem:crowns}.
\end{proof}

Additionally, we show it is often possible to obtain a better bound on $\delta_H(A)$ than $\delta_{\overline H}(A) + 1$ with the following result.

\begin{prop}\label{prop:atmost4inI}
	Let $H$ be a maximal subgroup of a classical group $G$ with $\Omega \leq G \leq \Sigma$. Then for any group $N$ with $\Omega \leq N \leq I$, we have $\delta_{H \cap N}(A) \leq 4$ for any chief factor $A$, with equality only if $A$ is central.
\end{prop}

\begin{proof}
	If $H$ contains $\Omega$ then the result is immediate by \cite[Proposition 3.12]{LMT} and Lemma \ref{lemma:non_simple_gen_bound}. Similarly, if $H$ does not contain $Z(\Delta) \cap G$ then $H = Z.\overline G$ for some cyclic group $Z$, so the result follows by \cite[Proposition 3.12]{LMT} and Lemma \ref{lemma:non_simple_gen_bound}. Hence, we assume that $H$ contains $Z(\Delta) \cap G$ and does not contain $\Omega$. In particular, $\overline H$ is a maximal subgroup of $\overline G$ which does not contain $\overline \Omega$.
	
	If $\overline \Omega$ is not simple, then Tables \ref{table:chieffactorsnonqs1} and \ref{table:chieffactorsnonqs2} provide the desired bound on $\delta_{H \cap N}(A)$ unless $\Omega = \Omega_4^+(q)$. In this case the same argument as in the proof of Lemma \ref{lemma:non_simple_gen_bound} provides the desired result, using the fact that $N/\Omega \leq C_2^2$.
	
	We hence assume that $\overline \Omega$ is simple. If $A \not\cong C_2$ then the result is proven by Theorem \ref{theorem:LMTcrowns}, so it suffices to prove that $\delta_{H \cap N}(C_2) \leq 4$. We may assume that $\delta_{\overline H \cap \overline N}(C_2) \geq 4$, as otherwise the result follows. So, by Table \ref{table:LMTtable}, $G$ is in case $\boldO$, and $H$ is either of type $\O_{n_1}^{\epsilon_1} \otimes \O_{n_2}^{\epsilon_2}$ with $n_i$ even, or is of type $\O_m^\epsilon(q) \wr \Sym(2)$ in $\mathscr C_7$ with $m$ even. 
	
	Assume first that $H$ is of type $\O_{n_1}^{\epsilon_1} \otimes \O_{n_2}^{\epsilon_2}$ with $n_i$ even. Then $\delta_{\overline H \cap \overline N}(C_2) = 4$ if and only if $\overline H \cap \overline N = \overline I_1 \times \overline I_2$, with $\overline I_j/\overline \Omega_j \cong C_2^2$ for each $j$. Hence, $Z(I_j)$ is contained in $\Omega_j$ and so is Frattini in $I_j$ for each $j$. This implies that the scalars in $H \cap N = I_1 \circ I_2$ are Frattini also, so $\delta_{H \cap N}(C_2) = 4$.
	
	Assume that $H$ is of type $\O_m^\epsilon(q) \wr \Sym(2)$ in $\mathscr C_7$ with $m$ even. By Table \ref{table:LMTtable}, we have $\overline N = \overline \Omega$ and $D(V) = D(V_i) = \Box$. As previously, the scalars of each $I_j$ are Frattini in $\Omega_j$, and so the scalars of $H \cap N$ are Frattini, proving that $\delta_{H \cap N}(C_2) = 4$.
\end{proof}

Let $\Sigma^*(V, \mathbb F, \kappa)$ be defined as in Definition \ref{def:sigma*}. Recall that $\Sigma^*$ is a classical group, unless $G$ is of type $\boldL$ and $\dim(V) \leq 2$. In this exceptional case we have that $\Sigma^* = \Sigma : C_2$. We prove the following result for $\Sigma^*$, extending the two results above.

\begin{lemma}\label{lemma:maxsubgpsSigmai}
	Consider a group $G$ with $\Omega \leq G \leq \Sigma^*$, and let $H$ be a maximal subgroup of $G$. Then, for any group $N$ such that $\Omega \leq N \leq I$, and any chief factor $A$, we have:
	\[
	\delta_{H \cap N}(A) \leq \begin{cases*}
		4 \qquad & if $A$ is central,\\
		3 & otherwise,
	\end{cases*}
	\]
	and
	\[
	\delta_{\overline H}(A) \leq \begin{cases*}
		5 \qquad & if $A \cong C_2$,\\
		3 & otherwise.
	\end{cases*}
	\]
	Additionally, 
	\begin{enumerate}[\upshape(1)]
		\item $\delta_{G}(A) \leq 4$, with equality only if $A \cong C_2$ and $G$ is of type $\boldO$;
		\item $\delta_{\overline G}(A) \leq 3$ with equality only if $A \cong C_2$ and $G$ is of type $\boldL$ or $\boldO$;
		\item $\delta_{G \cap N}(A) \leq 2$; and
		\item $\delta_{\overline G \cap \overline N}(A) \leq 2$.
	\end{enumerate}
\end{lemma}

\begin{proof}
	If $\dim(V) > 2$ or $G$ is not of type $\boldL$, the proof is complete by Theorem \ref{theorem:LMTgenerators} and Proposition \ref{prop:atmost4inI}. Therefore, we may assume that $\dim(V) \leq 2$, and that $G$ is a subgroup of $\Sigma : C_2$. 
	
	Suppose that $\overline \Omega$ is simple. By Corollary \ref{cor:maximalsubgpextension}, $H \cap \Omega$ is either central or a maximal or novelty maximal subgroup of $\SL_2(q)$. Hence, for any chief factor $A$,
	\[
	\delta_{H \cap \Omega}(A) \leq \begin{cases*}
		2 \qquad & if $A \cong C_2$,\\
		1 & otherwise,
	\end{cases*}
	\]
	by \cite[Tables 8.1, 8.2]{BHRD}. Additionally, by \cite[Section 2.2]{KL},
	\[
	\delta_{\frac{G \cap \Sigma}{\Omega}}(A) \leq \begin{cases*}
		2 \qquad & if $A$ is central,\\
		1 & otherwise.
	\end{cases*}
	\]
	Thus, $\delta_{H \cap N}(A) \leq 4$ and $\delta_{\overline H}(A) \leq 5$ for any chief factor $A$, with equality only if $A \cong C_2$.
	
	Suppose that $\overline \Omega$ is not simple. If $\dim(V) = 1$, then $\overline \Omega = 1$, so $\delta_{H \cap N}(A) \leq 1$ and $\delta_{\overline H}(A) \leq 3$ by \cite[Section 2.2]{KL}. Otherwise, $\dim(V) = 2$ and $q=2,3$, so $\delta_L(A) \leq 2$ for any subgroup $L$ of $I$, and any chief factor $A$. Hence, $\delta_{H \cap N}(A) \leq 2$ and $\delta_{\overline H}(A) \leq 3$ since $I = \Sigma$ in this case.
	
	The results for the chief factors of $P$, $\overline P$, $P \cap N$ and $\overline P \cap \overline N$ are clear by Lemma \ref{lemma:non_simple_gen_bound}, \cite[Proposition 3.12]{LMT} and \cite[Chapter 2]{KL}.
\end{proof}

\subsection{Wreath products of maximal subgroups of classical groups}

By Sections \ref{section:StC} and \ref{section:Aschbacher}, in future sections we often bound the number of chief factors of a group $M = (R \wr \Sym(t)) \cap H$ where: $R$ is a maximal subgroup of a classical group $K$, and $H$ is a specific subgroup of $K \wr \Sym(t)$. Our goal is to prove the following result regarding the non-Frattini chief factors of such a group.

\begin{prop} \label{prop:jumpingappliedtoclassmaxs}
	Consider two groups $N_1, K$ such that $\Omega \leq N_1 \leq I \leq K \leq \Sigma$. Assume that $d(\overline K/\overline I) = 2$, and $d(K/N_1) \geq 3$. Let $H$ be a subgroup of $K \wr \Sym(t)$ which contains $N_1^t$ and satisfies the conditions of Corollary \ref{cor:CFsinwreathproduct}. Additionally, for some $R \mleq K$ not containing $N_1$, let
	\[
	M := (R \wr \Sym(t)) \cap H.
	\]
	
	Then 
	\[
	\delta_{M, (R \cap N_1)^t}(A) \leq 3,
	\]
	for any chief factor $A$.
\end{prop}

In order to prove this result, we require the following result regarding the non-Frattini, central chief factors of $R \cap N_1$.

\begin{lemma}\label{lemma:prereqjumpingclassicals}
	Let $N_1, K, R$ be as defined in Proposition \ref{prop:jumpingappliedtoclassmaxs}. Then, 
	\[
		\delta_{(R \cap N_1)/[R, R \cap N_1]}(C_2) \leq 3,
	\]
	and $\delta_{R \cap N_1}(A) \leq 3$ for any central chief factor $A \not\cong C_2$.
\end{lemma}

Assuming the statement of this lemma, we provide the following proof of Proposition \ref{prop:jumpingappliedtoclassmaxs}.

\begin{proof}[Proof of Proposition \ref{prop:jumpingappliedtoclassmaxs}, assuming Lemma \ref{lemma:prereqjumpingclassicals}]
	Let $W$ be a non-Frattini chief factor of $M$ in $(R \cap N_1)^t$, and let $A$ be a chief factor of $R$ in $R \cap N_1$ with $W$ in $A^t$. Let $A_1$ be a chief factor of $R \cap N_1$ in $A$. By Theorem \ref{theorem:mainjumping},
	\[
	\delta_{M, (R \cap N_1)^t}(W) \leq \delta_{R \cap N_1}(A_1).
	\]
	Hence, if $A_1 \not\cong C_2$, the proof is complete by Lemmas \ref{lemma:maxsubgpsSigmai} and \ref{lemma:prereqjumpingclassicals}, so we assume that $A_1 \cong C_2$. If $A = A_1$, then $A$ is a central chief factor of $R$, and so, by Corollary \ref{cor:CFsinwreathproduct},
	\[
	\delta_{M, (R \cap N_1)^t}(W) \leq \delta_{(R \cap N_1)/[R, R \cap N_1]}(A_1) \leq 3.
	\]
	Thus, we assume that $A \cong C_2^k$ for some $k > 1$. Note that, by the proof of Corollary \ref{cor:CFsinwreathproduct}, $\delta_{M, (R \cap N_1)^t}(W)$ is at most the number of chief factors of $R$ in $R \cap N_1$ which are isomorphic to $C_2^k$ and are non-Frattini and central in $R \cap N_1$. By Propositions \ref{prop:CFsFromNormalSubgp} and \ref{prop:CFequivalenceinN}, the number of such chief factors of $R$ in $R \cap N_1$ is at most $\frac{1}{k}\delta_{R \cap N_1}(C_2)$. Since $\delta_{R \cap N_1}(C_2) \leq 4$ by Proposition \ref{prop:atmost4inI}, we obtain $\delta_{M, (R \cap N_1)^t}(A) \leq 2$, as required.
\end{proof}

\begin{proof}[Proof of Lemma \ref{lemma:prereqjumpingclassicals}]
	Since $d(K/N_1) \geq 3$ and $d(\overline K/\overline I) = 2$, $\Omega$ is of type $\boldL$ or $\boldO^{\pm}$ with $f$ even. Note that, if $\delta_{R \cap N_1}(C_2) \leq 3$, then $\delta_{(R \cap N_1)/[R, R \cap N_1]}(C_2) \leq 3$.
	
	We begin by considering the case where $\Omega$ is not quasisimple. Thus, $\Omega$ is one of $\SL_1(q)$, $\Omega_2(q)$ or $\Omega_4(q)$. In the first two cases, it is clear that $\delta_{R \cap N_1}(A) \leq 2$ for any chief factor $A$, since $N_1$ is either cyclic, or a dihedral group. In the latter case, we have 
	\[
	\Sigma = \langle \Delta_1 \circ \Delta_2, \phi, \sigma \rangle \leq \Sigma_1 \wr \Sym(2).
	\]
	If $K$ does not project onto $\Sym(2)$, then $K \cap N_1$ is equal to $(\SL_2(q) \circ \SL_2(q)).C$ where $C$ is either trivial or $C_2$. By Proposition \ref{prop:maxsubgpdirectproduct}, $R \cap N_1$ is isomorphic to $\SL_2(q) \circ \SL_2(q)$, $\SL_2(q) . C$ or $(R_1 \circ \SL_2(q)).C$ for some maximal or novelty maximal subgroup $R_1$ of $\SL_2(q)$. Since $\delta_{R_1}(A) \leq 2$ for any chief factor $A$, in all of the above cases we obtain that $\delta_{R \cap N_1}(A) \leq 3$ for any chief factor $A$. 
	
	Suppose instead that $K$ projects onto $\Sym(2)$. Thus, by Proposition \ref{prop:XSimpleMaxSubgrpXkC}, $R \cap \Omega$ is equal to one of $Z(\SL_2(q))$, $\SL_2(q)$, $\Omega$, or $R_1 \circ R_2$ for some isomorphic, maximal or novelty maximal subgroups $R_i$ of $\SL_2(q)$. In the first three cases we obtain $\delta_{R \cap N_1}(A) \leq 3$ for any chief factor $A$, so we assume that $R \cap \Omega = R_1 \circ R_2$. If $A \not\cong C_2$ is central then $\delta_{R_1}(A) \leq 1$, and so, by Corollary \ref{cor:upperbounddeltaGN}, $\delta_{R \cap \Omega}(A) \leq 2$. Since $\delta_{N_1/\Omega}(A) = 0$, this implies that $\delta_{R \cap N_1}(A) \leq 2$. We now consider $A \cong C_2$. We have $\diag(R_1 \times R_2) \leq [R, R \cap \Omega]$, so $\delta_{R \cap N_1, \frac{R \cap \Omega}{[R, R \cap \Omega]}}(C_2) \leq \delta_{R_i}(C_2) \leq 2$. Additionally, since $d(\overline K/\overline I) = 2$, if $\frac{R \cap N_1}{R \cap \Omega}$ is isomorphic to $C_2^2$, then $[R, R \cap N_1](R \cap \Omega)$ contains a subgroup isomorphic to $C_2$. Thus, $\delta_{\frac{R \cap N_1}{[R, R \cap N_1] (R \cap \Omega)}}(C_2) \leq 1$. Hence, we have shown that $\delta_{\frac{R \cap N_1}{[R, R \cap N_1]}}(C_2) \leq 3$.
	
	We assume for the remainder of this proof that $\Omega$ is quasisimple. If $R$ contains $\Omega$, then $\delta_{R \cap N_1}(A) \leq 2$ for any chief factor $A$, by Lemma \ref{lemma:maxsubgpsSigmai}. If $R$ does not contain $Z(N_1)$, then $R  \cap N_1 = Z . \overline N_1$ for some $Z < Z(N_1)$, so again the result is clear by Lemma \ref{lemma:maxsubgpsSigmai}. Hence, we may assume that $R$ contains $Z(N_1)$ and does not contain $\Omega$, so it is classified by Aschbacher's Theorem. 
	
	Suppose first that $R \in \mathscr{S}$, and so $\overline R$ is almost simple. If $\delta_{R_{\overline N_1}}(A) \leq 2$ for some chief factor $A$, then $\delta_{R_{N_1}}(A) \leq 3$. So we assume that $\delta_{R_{\overline N_1}}(A) = 3$, which implies that $\overline R$ is a classical almost simple group of type $\boldL$ or $\boldO^\pm$ and $A \cong C_2$, by \cite[Proposition 3.12]{LMT}. Let $T$ denote the socle of $\overline R$. If $T$ is of type $\boldL$, let $\Aut(T) = \langle T, \overline \delta_R, \overline \phi_R, \overline \iota_R \rangle$ as defined in \cite[Section 2.2]{KL}. Then $\delta_{R_{\overline N_1}}(C_2) = 3$ implies that $R_{\overline N_1}$ contains $\overline \delta_R^k \overline \iota_R$ for some $k$. Additionally, $R_{\overline N_1} \cap \langle T, \overline \delta_R \rangle = \langle T, \overline \delta_R^i\rangle$ and $\overline R \cap \langle T, \overline \delta_R \rangle = \langle T, \overline \delta_R^j\rangle$ for some $i, j$ with $j \mid i$ and $|i|_2 > |j|_2$. This implies that $[\overline \delta_R^k \overline \iota_R, \overline\delta_R^j] = \overline \delta_R^{2j} \in [R_{\overline N_1}, \overline R]$, and so   $R_{\overline N_1} \cap \langle T, \overline \delta_R \rangle \leq [\overline R, R_{\overline N_1}]$. Thus, $\delta_{R_{\overline N_1}/[\overline R, R_{\overline N_1}]}(C_2) \leq 2$ and hence $\delta_{R_{N_1}/[R, R_{N_1}]}(C_2) \leq 3$. The case for $T$ of type $\boldO$ is similar.
	
	We now consider $R$ in each class $\mathscr{C}_i$, using the details provided in Section \ref{section:Aschbacher} regarding the structure of these groups.
	
	If $R \in \mathscr{C}_1$ and is parabolic, then $R_{\Omega} = Q : L$ where $Q$ is the unipotent radical and $L$ is the Levi subgroup. Since $f>1$, the chief factors of $ R$ in $Q$ are all products of $\mathbb{F}_q$ and as such are non-central. Next, 
	\[
	\Omega_1 \circ \cdots \circ \Omega_d \leq R_{N_1}/Q \leq I_1 \circ \cdots \circ I_d
	\]
	for some classical groups $\Omega_i$ of type $\boldL$ or $\boldO^\pm$, over $\mathbb{F}_q$. Since $f > 1$, each $\Omega_i$ is either quasisimple, a product of quasisimple groups, or a cyclic group. If $\Omega$ is in case $\boldL$, then $d \leq 3$ and $I_i$ are all of type $\boldL$, so $\delta_{R_i}(A) \leq 1$ for any chief factor $A$ and any group $R_i$ such that $\Omega_i \leq R_i \leq I_i$. Hence, $\delta_{R_{N_1}}(A) \leq 3$ for any chief factor $A$. If instead $\Omega$ is in case $\boldO^\pm$, then we note we have $d \leq 2$. Additionally, $I_1$ is of type $\boldL$ while $I_2$ is of type $\boldO$. Thus, for any groups $R_i$ with $\Omega_i \leq R_i \leq I_i$, and any chief factor $A$, we have that $\delta_{R_1}(A) \leq 1$ and $\delta_{R_2}(A) \leq 2$, so again the result is clear.
	
	If $R \in \mathscr{C}_1$ is of type $I_1 \oplus I_2$, then by Section \ref{section:C1explanation} this implies that
	\[
	\Omega_1 \times \Omega_2 \leq R_{N_1} \leq I_1 \times I_2,
	\]
	where once again $\Omega_i$ are classical groups of type $\boldL$ or $\boldO$, over $\mathbb{F}_q$. Therefore, when $\Omega$ is in type $\boldL$ we have that $\delta_{R_{N_1}}(A) \leq 2$ for any chief factor $A$. If $\Omega$ is in type $\boldO$ then $\delta_{R_{N_1}}(A) \leq 3$ unless all of the following hold: $A \cong C_2$, $R_{N_1} = I_1 \times I_2$, and $I_j/\Omega_j \cong C_2^2$. Note that $[R_{N_1}, R_{N_1}]$ contains $\Omega_1 \times \Omega_2$. This is clear if $\Omega_i$ is trivial, or perfect. If $\Omega_i$ is instead cyclic, then $I_i$ is a dihedral group and $[I_i, I_i] = \Omega_i$ by definition. Additionally, $R_K$ contains $\delta_1 \delta_2$, since $\overline K$ contains $\overline \delta$. Hence, since $[I_i/\Omega_i, \delta_i \Omega_i] = S_i/\Omega_i$, we have that $[R, R_{N_1}] \geq S_1 \times S_2$. Thus, $\delta_{\frac{R_{N_1}}{[R, R_{N_1}]}}(C_2) \leq 2$, as required.
	
	By \cite[Tables 3.5.E, 3.5.F]{KL}, if $R \in \mathscr{C}_1$ is of type $\Sp_{n-2}(q)$, then $q$ is even, which is a contradiction.
	
	Next assume that $R$ is in $\mathscr{C}_2$, so $R$ is the stabiliser of a decomposition $V_1 \oplus \cdots \oplus V_t$ with $\dim(V_i) = m$ for all $i$. First we consider $R$ of type $I_1 \wr \Sym(t)$. If $\Omega_1 = \Omega_1(q)$, then 
	\[
	2^{t-1} . \Alt(t) \leq R_{\Omega} \leq 2^{t-1} . \Sym(t),
	\]
	with $t \geq 5$, so $\delta_{R \cap N_1}(A) \leq 3$ for any chief factor $A$. Otherwise, there exists a group $L_1$ with $\Omega_1 \leq L_1 \leq I_1$ such that
	\[
	R_{N_1} = L_1^t . \del((I_1/L_1)^t) . \Sym(t),
	\]
	and $R_{N_1}$ satisfies the conditions of Corollary \ref{cor:CFsinwreathproduct}. Hence, $\delta_{R \cap N_1}(A) \leq 3$ when $\Omega_1$ is quasisimple. When $\Omega_1$ is not quasisimple, the same result follows using the fact that $f > 1$. For instance, if $\Omega_1 = \Omega_2^\pm(q)$, then $\Omega = \Omega_{2t}^{\epsilon}(q)$ where $\epsilon = (\pm 1)^t$. Thus, $t > 2$, implying that $\delta_{R \cap N_1}(A) \leq 3$.
	
	We next assume that $R$ is of type $\GL_{n/2}(q).2$. This implies that
	\[
	\SL_{n/2}(q) \leq R_{N_1} \leq \GL_{n/2}(q).2.
	\]
	As such, since $f > 1$, we have that $\delta_{R_{N_1}}(C_2) \leq 2$, by Lemma \ref{lemma:maxsubgpsSigmai}.
	
	Assume now that $R$ is of type $\O_{n/2}(q)^2$ where $n/2 > 1$ is odd. Then
	\[
	S_1 \times S_2 \leq R_{N_1} \leq I_1 \times I_2,
	\]
	and so $\delta_{R_{N_1}}(A) = 0 $ if $A \not \cong C_2$ is a central chief factor. Additionally, $\Omega_1$ is isomorphic to $\Omega_2$, both are quasisimple, and $R$ contains an element swapping $I_1$ and $I_2$. Hence, $[R, R_{N_1}] \geq \Omega_1 \times \Omega_2 . \del(I_1/\Omega_1 \times I_2 / \Omega_2)$, and so
	\[
	\delta_{\frac{R_{N_1}}{[R, R_{N_1}]}}(C_2) \leq 2.
	\]
	
	Next, we consider $R \in \mathscr{C}_3$. Then, we have
	\[
	\Omega_\sharp \leq R_{N_1} \leq R_I = I_\sharp.Z_r
	\]
	by \cite[(4.3.11)]{KL}. Thus, $\delta_{R_{N_1}}(A) \leq 3$ by Lemma \ref{lemma:maxsubgpsSigmai}.
	
	Assume next that $R \in \mathscr{C}_4$, with $R$ of type $I_1 \otimes I_2$. As such, we have that either
	\[
	R_{I} = \frac{( I_1 \times  I_2).\del( \Delta_1/ I_1 \times  \Delta_2/ I_2)}{\del(Z(\Delta_1) \times Z(\Delta_2))},
	\]
	or
	\[
	R_{I} = \frac{I_1 \times I_2}{\del(Z(I_1) \times Z(I_2))}.
	\]
	For ease of notation, given any groups $G_1, G_2$ we let $G_1 \otimes G_2$ denote $\frac{G_1 \times G_2}{\del(Z_1 \times Z_2)}$ where $Z_i$ denotes the scalars of $G_i$.
	
	If $\Omega_1$ is in case $\boldL$ or $\boldS$, then $\delta_{R_{N_1}}(A) \leq 3$. Hence, we assume that $\Omega_1$ is in case $\boldO$. Then, by \cite[Table 3.5.E]{KL}, $R$ is maximal in $K$ if and only if $R$ is of type $\O_{m_1}^{\epsilon_1} \otimes \O_{m_2}^{\epsilon_2}$ with $m_2$ odd. By \cite[Proposition 4.4.17]{KL}, this implies that $R_{I} = I_1 \otimes S_2$, so $\delta_{R_{N_1}}(A) \leq 3$.
	
	Let $R \in \mathscr{C}_5$, so $\overline \Omega_\sharp \leq R_{\overline N_1} \leq \overline \Delta_\sharp$. Then, $\delta_{R_{\overline N_1}}(A) \leq 2$ and $\delta_{R_{N_1}}(A) \leq 3$ for any chief factor $A$.
	
	We do not consider $R \in \mathscr{C}_6$, since this implies that $f=1$.
	
	Suppose that $R \in \mathscr{C}_7$ is of type $I_1 \wr \Sym(t)$. Thus, either
	\[
	R_{I} = \frac{(I_1 \times \cdots \times I_t) . \del(\Delta_1/I_1 \times \cdots \times \Delta_t/I_t)}{\del(Z(\Delta_1) \times \cdots \times Z(\Delta_t))} .\Sym(t),
	\]
	or
	\[
	R_{I} = \frac{I_1 \times \cdots \times I_t}{\del(Z(I_1) \times \cdots \times Z(I_t))}.\Sym(t).
	\]
	Hence, we let $G_1 \otimes \cdots \otimes G_t$ denote $\frac{G_1 \times \cdots \times G_t}{\del(Z(G_1) \times \cdots \times Z(G_t))}$ for any groups $G_i$ where $Z_i$ denotes the scalars in $G_i$. By \cite[Section 4.7]{LMT}, we have that 
	\[
		R_{\Omega} \geq (S_1 \otimes \cdots \otimes S_t).\del(I_1/S_1 \otimes \cdots \otimes I_t/S_t) . \Alt(t).
	\]
	
	Suppose that $I_1$ is of type $\boldL$ or $\boldS$, so that $\Delta_1 / \Omega_1$ is cyclic. If $t > 2$ then $R_\Omega$ satisfies Corollary \ref{cor:CFsinwreathproduct}, so $\delta_{R_{N_1}}(A) \leq 3$. If $t=2$ then $\delta_{R_{N_1}}(A) \leq 3$ for any central chief factor $A$ not isomorphic to $C_2$, by Corollary \ref{cor:upperbounddeltaGN}. Thus, we let $A \cong C_2$ be a chief factor of $R$. If $R_{N_1}$ projects onto $\Sym(2)$ then the same argument as in the proof of Proposition \ref{prop:t2CFsinwreathproduct} applies, proving that $\delta_{R_{N_1}}(C_2) \leq 3$. Finally, if $R_{N_1}$ does not project onto $\Sym(2)$, then $R_{N_1} \leq \Delta_1 \otimes \Delta_2$, so $\delta_{R_{N_1}}(C_2) \leq 2$.
	
	Hence, we may assume that $R$ is of type $\O^\epsilon_m(q) \wr \Sym(t)$. Then, by \cite[Table 3.5.E]{KL}, we have $t=2$ and $m \equiv 2 \pmod 4$, in order for $R$ to be maximal in $K$. Note that, if $A \not\cong C_2$ is central, then $\delta_{R_{N_1}}(A) \leq 1$, so we assume that $A \cong C_2$.
	
	If $R_{N_1}$ projects onto $\Sym(2)$, then $N_1 = I$, by the proofs of \cite[Proposition 4.7.6, 4.7.7]{KL}. We can now apply Lemma \ref{lemma:jumpingbabyresultfull} to determine that $\delta_{R_{N_1}}(A) \leq 3$ for all central chief factors $A$. Therefore, we assume that $R_{N_1}$ does not project onto $\Sym(2)$, and so is contained in
	\[
	R_I \cap (\Delta_1 \otimes \Delta_2) = (I_1 \otimes I_2).\del(\Delta_1/I_1 \otimes \Delta_2/I_2).
	\]
	Since $d(\overline K / \overline I) = 2$, as in the proof of Proposition \ref{prop:LMTC7calcs}, we can determine that $R$ projects onto $\Sym(2)$. If $L_1$ is equal to the intersection of $R_{N_1}$ with $\Delta_1 \otimes Z_2 \cong \Delta_1$ and $L_2$ is defined similarly, we have that $(\Omega_1 \otimes \Omega_2).\del(L_1/\Omega_1 \otimes L_2/\Omega_2)$ is contained in $[R, R_{N_1}]$. As a result, $\delta_{R_{N_1}/[R, R_{N_1}]}(C_2) \leq 3$.
	
	We finally consider $R \in \mathscr{C}_8$. Thus, $\Omega$ is in case $\boldL$, and $R_\Delta = \Delta_\sharp$, where $X_\sharp = X(V, \mathbb{F}, \kappa_\sharp)$ for some form $\kappa_\sharp$ as given in \cite[Table 4.8.A]{KL}. Hence, $\delta_{R_G}(A) \leq 3$.
\end{proof}

\section{Proof of Theorem \ref{theorem:bigresult}: classical groups}\label{section:classicals}

We continue to use the notation of Section \ref{section:classicalprelims}. Let $(V, \mathbb F, \kappa)$ be a classical geometry of type $\boldL, \boldU, \boldS$, or $\boldO$, and let $G_0 = \Omega = \Omega(V, \mathbb F, \kappa)$, which we assume to be a quasisimple group. Let $\overline G$ be an almost simple group with socle $\overline G_0$, and let $\overline M$, $\overline H$ be subgroups of $\overline G$ such that
\[
	\overline M \mleq \overline H \mleq \overline G.
\]
In this section, we prove that $d(\overline M) \leq 7$, and we show this bound is tight (see Example \ref{example:boundtightC2}).

Recall that $\overline \Sigma$ is a subgroup of $\Aut(\overline \Omega)$. Moreover, $\overline \Sigma = \Aut(\overline \Omega)$, unless $\Omega$ is equal to $\Sp_4(q)$ with $q$ even, or is equal to $\Omega_8^+(q)$ for any $q$. In these cases, $\Aut(\overline \Omega)$ is equal to $\langle \overline \Sigma, \gamma \rangle$, for some graph automorphism $\gamma$. We defer the treatment of the cases where $\overline G \nleq \overline \Sigma$ to the end of this section. 

Hence, we begin by assuming that $\overline \Omega \leq \overline G \leq \overline \Sigma$. We may therefore let $G$, $H$, and $M$ denote the preimages of $\overline G$, $\overline H$, and $\overline M$ in $\Sigma$, respectively. By the correspondence theorem, $M \mleq H \mleq G$.

We classify the second maximal subgroups of $G$ by applying Aschbacher's Theorem. In order to do so we use the descriptions of the maximal subgroups of $G$ given in \cite{BHRD}, \cite[Chapter 4]{KL}, and Section \ref{section:classicalprelims} when applying Aschbacher's Theorem.

\begin{theorem}[{\cite[Main Theorem]{KL}}]
	Suppose that $\overline \Omega \leq \overline G \leq \overline \Sigma$ is a classical group as above, and $\overline H \mleq \overline G$. Then $\overline H$ is either one of the groups in \cite[Tables 3.5A-F]{KL}, or $\overline H$ is almost simple.
\end{theorem}

When $\overline H$ is almost simple, we have that $d(\overline M) \leq 5$ by Theorem \ref{theorem:LMTgenerators}. So we assume in the remainder of this section that $\overline H$ is one of the groups in \cite[Tables 3.5A-F]{KL}.

\subsection{Class \texorpdfstring{$\mathscr{C}_1$}{C1}, not parabolic}

\begin{prop}\label{prop:C1secondmaxls}
	Suppose that $H \in \mathscr{C}_1$ is not parabolic. Then $d(\overline M) \leq 7$.
\end{prop}

\begin{proof}
	If $H$ is of type $\Sp_{n-2}(q)$ then $\overline H$ is almost simple by \cite[Proposition 4.1.7]{KL}, so $d(\overline M) \leq 5$ by Theorem \ref{theorem:LMTgenerators}.
	
	Hence, we may assume that $H$ is of type $\GL^\epsilon_m(q) \perp \GL^\epsilon_{n-m}(q)$, $\Sp_m(q) \perp \Sp_{n-m}(q)$ or $\O_m^{\epsilon_1}(q) \perp \O_{n-m}^{\epsilon_2}(q)$. Thus, $H$ stabilises $V_1 \oplus V_2 = V$ for some vector subspaces $V_i$ of $V$, and
	\begin{align*}
		H_\Omega & = (\Omega_1 \times \Omega_2) . \diag(I_1/\Omega_1 \times I_2 / \Omega_2),\\
		H_\Sigma & = (I_1 \times I_2).\diag(\Sigma^*_1/I_1 \times \Sigma^*_2/I_2) \leq \Sigma_1^* \times \Sigma_2^*.
	\end{align*}
	
	As in Proposition \ref{prop:maxsubgpdirectproduct}, let $\pi_i \colon H \to \Sigma_i$ denote the natural projection maps, and then define
	\[
	N_i = H \cap \Sigma^*_i, \,\, P_i = \pi_i(H)
	\]
	for each $i$. Thus, by above, for each $i$ we have that
	\begin{align*}
		\Omega_i \leq N_i & \leq I_i\\
		I_i \leq P_i & \leq \Sigma^*_i.
	\end{align*}
	Additionally, the scalars of $G$ in $H_{\Sigma}$ are equal to $\diag(Z(\Delta_1) \times Z(\Delta_2)) $.
	
	We assume first that $M$ is not a subdirect subgroup of $P_1 \times P_2$, so, by Proposition \ref{prop:maxsubgpdirectproduct}, $M$ contains either $N_1$ or $N_2$. Suppose without loss of generality that $M$ contains $N_1$. Then there exists some maximal subgroup $J$ of $P_2$ such that $M = (P_1 \times J) \cap H$. Hence, $\delta_{\overline M}(A) \leq \delta_{M, N_1}(A) + \delta_{\overline J}(A)$, so we bound the two summands separately. Firstly, by Lemma \ref{lemma:maxsubgpsSigmai}, we have that $\delta_{\overline J}(A) \leq 5$, with equality only if $A \cong C_2$. Additionally, by Lemma \ref{lemma:maxsubgpsSigmai}, $\delta_{N_1}(A) \leq 2$. Hence, $\delta_{M, N_1}(A) \leq 2$ by Corollary \ref{cor:upperbounddeltaGN}, and so
	\begin{align*}
		\delta_{\overline M}(A) & \leq \delta_{M, N_1}(A) + \delta_{\overline J}(A)\\
		& \leq \begin{cases*}
			2 + 5 \qquad & if $A \cong C_2$,\\
			2 + 4 & otherwise,
		\end{cases*}
	\end{align*}
	implying that $d(\overline M) \leq 7$ by Theorem \ref{theorem:crowns}.
	
	As such, we can assume that $M$ is subdirect, and does not contain $N_i$ for any $i$. Hence, $M = (M_1 \times M_2) . C$ where each $M_i$ is a maximal $P_i$-normal subgroup of $N_i$, with $C \cong P_i / M_i$. Thus,
	\[
	\delta_{\overline M}(A) \leq \delta_{M, M_1}(A) + \delta_{\overline P_2}(A).
	\]
	When $\Omega_1$ is quasisimple, since $\overline \Omega_1 \not\cong \overline \Omega_2$, we have that $M_1$ contains $\Omega_1$. Hence, by Corollary \ref{cor:upperbounddeltaGN} and Lemma \ref{lemma:maxsubgpsSigmai}, we have that $\delta_{M, M_1}(A) \leq 2$. Additionally, by Lemma \ref{lemma:maxsubgpsSigmai}, $\delta_{\overline P_2}(A) \leq 3$ with equality only if $A \cong C_2$. Hence,
	\[
	\delta_{\overline M}(A) \leq \begin{cases*}
		5 \qquad & if $A \cong C_2$,\\
		4 & otherwise,
	\end{cases*}
	\]
	which implies that $d(\overline M) \leq 5$, by Theorem \ref{theorem:crowns}.
	
	On the other hand, when $\Omega_1$ is not quasisimple, it can be shown that for any normal subgroup $M_1$ of $I_1$ we have $d(M_1) \leq 4$, and hence by Lemma \ref{lemma:maxsubgpsSigmai},
	\[
	d(\overline M) \leq d(M_1) + d(\overline P_2) \leq 4 + 3 = 7.
	\]
\end{proof}

\subsection{Class \texorpdfstring{$\mathscr{C}_2$}{C2}}\label{section:classicalsC2}

We begin this section by assuming that $H$ is of type $\GL^\epsilon_m(q) \wr \Sym(t)$, $\Sp_m(q) \wr \Sym(t)$ or $\O_m^\epsilon(q) \wr \Sym(t)$, and that $m > 1$ in this final case. As such, $H$ is the stabiliser of a decomposition $V_1 \oplus \cdots \oplus V_t$ of $V$ for some $V_i \leq V$ which are pairwise isometric. By Section \ref{section:C2explanation},
\[
H_{\Omega} = \Omega_1^t . \del((I_1 / \Omega_1)^t) . \Sym(t)
\]
and
\[
H_\Sigma = I_1^t . \diag((\Sigma^*_1/I_1)^t) : \Sym(t) \leq \Sigma^*_1 \wr \Sym(t),
\]
unless $\Omega_1 = \GL_1(q)$ and $t=2$, in which case
\[
	H_\Sigma = I_1^t . \diag((\Sigma_1/I_1)^t) : \Sym(t) \leq \Sigma_1 \wr \Sym(t) \leq \Sigma_1^* \wr \Sym(t).
\]

We now define some notation, given the structure of $H$ indicated above. Let $\Omega_1 \leq N_1 \leq I_1$ be the largest group such that $N_1^t \leq H$. Additionally, let $I_1 \leq P_1 \leq \Sigma^*_1$ be the projection of $H \cap (\Sigma^*_1)^t$ to the first coordinate. Then, by the discussion in Section \ref{section:wreathproducts}, we have that
\[
H = N_1^t . \del((I_1/N_1)^t) . \diag((P_1/I_1)^t).\Sym(t).
\]
To be more precise, $H \cap I_1^t = \{(a_1, \dots, a_t) \in I_1^t : \prod a_i \in N_1 \}$, and $H/(H \cap I_1^t)$ contains a subgroup of the form $\{(x, \dots, x) \in (\Sigma^*_1/I_1)^t : x \in P_1/I_1 \}$. For any $\sigma \in \Sym(t)$ and $1 \leq i \leq t$, there exists some $a \in I_i$ with determinant $\pm 1$ such that $(1, \dots, 1, a, 1, \dots, 1) \sigma \in H$. Finally, $H$ contains the natural copy of $\Alt(t)$. In particular, this implies that $H$ satisfies the conditions of Corollary \ref{cor:CFsinwreathproduct}. For convenience, we record the structure of $I_1/\Omega_1$ and $\overline \Sigma^*_1 / \overline I_1$ in Table \ref{table:C2basicstructure}.
\begin{table}
	\centering
    \caption{The structure of $I_1/\Omega_1$ and $\overline \Sigma_1^*/\overline I_1$.}
	\begin{tabular}{lll}
		Type & $I_1/\Omega_1$ & $\overline \Sigma^*_1/ \overline I_1$ \\  \toprule
		$\boldL$ & $C_{q-1}$ & $C_f \times C_2$\\
		$\boldU$ & $C_{q+1}$ & $C_{2f}$\\
		$\boldS$ & $1$ & $C_f$, $C_{2f}$, or $C_f \times C_2$\\
		$\boldO$ & $C_2^2$, or $C_2$ & $C_f$, $C_{2f}$, or $C_f \times C_2$
	\end{tabular}
	\label{table:C2basicstructure}
\end{table}

Following the method described in Section \ref{section:prelimsmethod}, we classify the maximal subgroups of $H$ by considering the normal series
\[
1 < \Omega_1^t < N_1^t < I_1^t \cap H < P_1^t \cap H < H.
\]
For each pair of successive normal subgroups $X < Y$ in the series above, we classify which maximal subgroups $M$ of $H$ can contain $X$ while not containing $Y$. We begin by considering the cases with $Y = \Omega_1^t$ or $Y = N_1^t$. In particular, $M$ projects onto $H/N_1^t$. Note that
\[
\overline H = \frac{N_1^t}{\diag(Z(N_1)^t)} . \frac{\del((I_1/N_1)^t)}{\diag((Z(I_1)/Z(N_1))^t \cap \del((I_1/N_1)^t)} . ((\overline P_1/\overline I_1) \times \Sym(t)).
\]
Let $\pi \colon H \to \overline H$ denote the natural projection map. We use this notation to distinguish, for instance, between $\overline N_1^t = (N_1/Z(N_1))^t$, and $\pi(N_1^t)$ which is equal to $N_1^t/\diag(Z(N_1)^t)$ by above. Since $M$ projects onto $H/N_1^t$, we have $\overline M = (\overline M \cap \pi(N_1^t)).(\overline H/\pi(N_1^t))$. Thus, to bound $d(\overline M)$ in this case, we first wish to bound the number of chief factors of $\overline H/\pi(N_1^t)$.

\begin{prop} \label{prop:boundtopC2}
	Let $H$ be as above. We have $d(\overline H/\pi(N_1^t)) \leq 4$ with equality only if $t = 2$, and $d(\overline H/\pi(H \cap I_1^t)) \leq 3$ . Additionally, if $t > 2$, then
	\[
	\delta_{\frac{\overline H}{\pi(N_1^t)}}(A) \leq \begin{cases*}
		3 \qquad & if $A \cong C_2^2$ and $t=4$, or $A \cong C_2$,\\
		2 & otherwise.
	\end{cases*}
	\]
	If instead $t=2$, then we have
	\[
	\delta_{\frac{\overline H}{\pi(N_1^t)}}(A) \leq \begin{cases*}
		4 \qquad & if $A \cong C_2$,\\
		2 & otherwise.
	\end{cases*}
	\]
\end{prop}

\begin{proof}
	Firstly, by Table \ref{table:C2basicstructure}, we have
	\[
	\delta_{(\overline P_1/\overline I_1) \times \Sym(t)}(A) \leq \begin{cases*}
		3 \qquad & if $A \cong C_2$,\\
		2 & if $A$ is otherwise central,\\
		1 & otherwise.
	\end{cases*}
	\]
	Indeed, this table implies that $\delta_{\overline P_1/\overline I_1}(A)$ is bounded by 2 when $A \cong C_2$, by 1 when $A$ is a central chief factor, and is 0 otherwise. Hence, we have $d(\overline H/\pi(H \cap I_1^t)) \leq 3$ by Theorem \ref{theorem:crowns}.
	
	We next prove the bounds on $\delta_{\frac{\overline H}{\pi(N_1^t)}}(A)$. Assume that $I_1/\Omega_1$ is cyclic. Then $\delta_{I_1/N_1}(A) \leq 1$, so, by the proof of Corollary \ref{cor:CFsinwreathproduct}, $\delta_{ H, \del((I_1/N_1)^t)}(A) \leq 1$ for any chief factor $A$. Additionally, if $A = C_2$ then we have equality only if $t=2$. Hence, we have proven the desired bounds.
	
	Assume that $I_1/N_1$ is not cyclic, and thus is isomorphic to $C_2^2$. Hence, the non-Frattini chief factors of $H$ in $\del((I_1/N_1)^t)$ are isomorphic to $C_2^{t-1}$ (if $t=2$ or $t$ is odd), or $C_2^{t-2}$ (if $t \neq 2$ is even), and there are at most 2 such chief factors. Additionally, if $\delta_{\overline P_1/\overline I_1}(C_2) = 2$ then there is only one non-Frattini chief factor of $H$ in $\del((I_1/N_1)^t)$, by Lemma \ref{lemma:jumpingbabyresultdel}. Thus, we can again prove the result by combining this with the bound on the chief factors of $(\overline P_1/\overline I_1) \times \Sym(t)$.
	
	By Theorem \ref{theorem:crowns}, to prove the bound on $d(\overline H/\pi(N_1^t))$ it is sufficient to show that $d((L_A)_{\delta_{\overline H/\pi(N_1^t)}(A)}) \leq 3$ when $A \cong C_2^2$. If $\delta_{\overline H/\pi(N_1^t)}(A) \leq 2$ then this is clear, so we assume that $\delta_{\overline H/\pi(N_1^t)}(A) = 3$. Hence, $\overline H/\pi(N_1^t)$ acts on $A$ as $\Sym(4)$ acts on its unique minimal normal subgroup isomorphic to $C_2^2$. Thus, as in the proof of Proposition \ref{prop:LMTcorrectionOC2}, we obtain $d((L_A)_{3})  = 3$. 
\end{proof}

\begin{prop}\label{prop:C2subdirectcase}
	Let $H$ be as above and suppose that $M \mleq H$ does not contain $\Omega_1^t$. If $M\cap \Omega_1^t$ is subdirect then $M$ contains $K^t$, where $K$ is a maximal $P_1$-normal subgroup of $\Omega_1$. 
\end{prop}

\begin{proof}
	Note that, since $M$ projects onto $H/\Omega_1^t$ and $M \cap \Omega_1^t$ is subdirect, $M \cap P_1^t$ is a subdirect subgroup of $P_1^t$.
	
	Let $K$ be the intersection of $M$ with $\Omega_1 \times 1 \times \cdots \times 1$, so $K^t \leq M$. Since $M \cap P_1^t$ is a subdirect subgroup of $P_1^t$, $K$ is a normal subgroup of $P_1$. If $\Omega_1$ is quasisimple, then $M$ contains $Z(\Omega_1^t) = \Phi(\Omega_1)^t$, so $K = Z(\Omega_1)$ which is a maximal $P_1$-normal subgroup of $\Omega_1$.
	
	Now assume that $\Omega_1$ is not quasisimple. Thus, we apply Proposition \ref{prop:subdirectwreathcase}, with $N_1$ as $X$ in the statement of the proposition. Let $N$ be the intersection of the maximal $P_1$-normal subgroups of $N_1$. Then, $M \cap \Omega_1^t$ contains $(N \cap \Omega_1)^t$. For most choices of $\Omega_1$, $N \cap \Omega_1$ is itself a maximal $P_1$-normal subgroup of $\Omega_1$, and hence the result is proven. Thus, we only need to consider the few $\Omega_1$ where this is not the case, these being $\Omega_1 = \Omega_2^\pm(q)$ for certain $q$, and $\Omega_1 = \Omega_4^+(3)$.
	
	If $\Omega_1 = \Omega_2^\pm(q)$, then $\Omega_1 = C_{\frac{q \mp 1}{(q-1,2)}}$. It follows that $M$ contains $\Phi \! \left( C_{\frac{q \mp 1}{(q-1,2)}} \right)^t$, and hence we consider $(M \cap \Omega_1^t)/\Phi(\Omega_1)^t$ which is a subgroup of $C_l^t$ for some square-free integer $l$. We claim that, for each prime $r \mid \ell$, the group $M/\Phi(\Omega_1)^t$ contains either $C_r^t$ or $C_{l/r}^t$. Indeed, suppose that $M/\Phi(\Omega_1)^t$ does not contain $C_r^t$. Note that $C_{l/r}$ is a characteristic subgroup of $C_l$, so $C_{l/r}^t$ is a normal subgroup of $H/\Phi(\Omega_1)^t$. Hence, $\frac{M \cap \Omega_1^t}{\Phi(\Omega_1)^t} C_{l/r}^t$ is an $M$-normal subgroup of $C_l^t$, implying that it is equal to $(M \cap \Omega_1^t)/\Phi(\Omega_1)^t$ or $C_l^t$. In the former case we have $C_{l/r}^t \leq M/\Phi(\Omega_1)^t$, so assume that $\frac{M \cap \Omega_1^t}{\Phi(\Omega_1)^t} C_{l/r}^t = C_l^t$. Hence, $\frac{M \cap \Omega_1^t}{\Phi(\Omega_1)^t} \leq C_r^t \times C_{l/r}^t$ projects onto $C_r^t$, and has $M/\Phi(\Omega_1)^t \cap C_r^t < C_r^t$. Since $r$ and $l/r$ are coprime, this is a contradiction by Goursat's Lemma. Thus, we have shown that $M/\Phi(\Omega_1)^t$ contains $C_r^t$ or $C_{l/r}^t$, proving the statement of the proposition in this case.
	
	Assume now that $\Omega_1 = \Omega_4^+(3)$. Then, if the intersection of a maximal $P_1$-normal subgroup of $N_1$ with $\Omega_1$ is proper, it is equal to one of two index 3 subgroups of $\Omega_1$. We denote these subgroups by $K_1, K_2$. Hence, $N \cap \Omega_1$ contains $K_1 \cap K_2$, which is a subgroup of index 9 in $\Omega_1$. In fact, $\Omega_1/(K_1 \cap K_2)$ is isomorphic to $C_3^2$, and, for each $i$, $K_i/(K_1 \cap K_2)$ is a $P_1$-normal subgroup which is a non-central section of $I_1$.
	
	Suppose that $t > 2$. Then the projection of the group $(M \cap I_1^t) / (K_1 \cap K_2)^t$ onto $(I_1/\Omega_1)^t$ contains $\del((I_1/\Omega_1)^t)$. Additionally, $(M \cap P_1^t) / (K_1 \cap K_2)^t$ is subdirect. Thus one can show that $M / (K_1 \cap K_2)^t$ contains $K_i^t/(K_1 \cap K_2)^t$ for both $i$, leading to a contradiction in this case. If instead $t=2$, then we can check computationally that $M$ contains either $K_1^2$ or $K_2^2$.
\end{proof}

For this next result, we introduce the following notation. Given a group $G$ and a $G$-group $A$, we write $d_G(A)$ for the smallest size set $X \subseteq A$ such that $A = \langle x^G : x \in X \rangle$. We say such a set $X$ generates $A$ under the action of $G$.

\begin{prop}\label{prop:C2McontainingOmega1}
	Let $H$ be as above. Suppose that $M \mleq H$ does not contain $\Omega_1^t$. Then $d(\overline M) \leq 7$.
\end{prop}

\begin{proof}
	By Propositions \ref{prop:XNonsimpleMaxSubgpXkC}, \ref{prop:nonsubdirectwreathcase} and \ref{prop:C2subdirectcase}, one of the following holds: $M$ contains $K^t$ for some maximal $P_1$-normal subgroup of $\Omega_1$, or $M$ is conjugate to $(R \wr \Sym(t)) \cap H$ for some non-normal, maximal subgroup $R$ of $P_1$. Indeed, if $M \cap \Omega_1^t$ is subdirect, then, by Proposition \ref{prop:C2subdirectcase}, it is clear that the former holds. If $M \cap \Omega_1^t$ is not subdirect in $\Omega_1^t$ and is normal in $P_1^t$, then it is equal to $K^t$ where $K$ is a maximal $P_1$-normal subgroup of $\Omega_1$, by maximality of $M$. If $M \cap \Omega_1^t$ is not subdirect in $\Omega_1^t$ and is not normal in $P_1^t$, then $M$ is conjugate to $( R \wr \Sym(t)) \cap H$, for some $ R \mleq P_1$, by Proposition \ref{prop:nonsubdirectwreathcase}.
	
	Hence, we treat each of these cases separately below. We assume first that $M$ is conjugate to $( R \wr \Sym(t)) \cap H$ for some $ R \mleq P_1$. As we wish to bound $d(\overline M)$, we may assume simply that $M$ is equal to $( R \wr \Sym(t)) \cap H$. If $\Sigma_1^* = \Sigma_1:C_2$, then we have that $d(R \cap N_1) \leq 3$. Indeed, $\Sigma_1$ is of type $\boldL$, and $\dim(V_1) \leq 2$, so either $N_1$ is cyclic, or $\SL_2(q) \leq N_1 \leq \GL_2(q)$, so a similar argument to that of Lemma \ref{lemma:maxsubgpsSigmai} proves the desired result. Hence, by Proposition \ref{prop:boundtopC2},
	\[
	d(\overline M) \leq d(R \cap N_1) + d(\overline H / \pi(N_1^t)) \leq 7.
	\]
	Thus, we assume that $\Sigma^*_1 = \Sigma_1$, and so $P_1$ is a classical group. If $\delta_{\overline H, \frac{\pi(P_1^t \cap H)}{\pi(N_1^t)}}(C_2) = 3$, then by Proposition \ref{prop:jumpingappliedtoclassmaxs} we have that $\delta_{M, M \cap N_1^t}(A) \leq 3$ for any chief factor $A$. Hence, by Proposition \ref{prop:boundtopC2},
	\[
	\delta_{\overline M}(A) \leq \delta_{M, M \cap N_1^t}(A) + \delta_{\overline H, \frac{\overline H}{\pi(N_1^t)}}(A) \leq  \begin{cases*}
		3+4 \qquad \qquad & if $A \cong C_2$,\\
		3+3 & otherwise.
	\end{cases*}
	\]
	If instead $\delta_{\overline H, \frac{\pi(P_1^t \cap H)}{\pi(N_1^t)}}(C_2) \neq 3$, then by the proof of Proposition \ref{prop:boundtopC2} we have that $\delta_{\overline H/\pi(N_1^t)}(A) \leq 3$ for any chief factor $A$, with equality only if $A $ is $C_2^2$ and $t=4$, or $A \cong C_2$. Additionally, by Proposition \ref{prop:atmost4inI}, $\delta_{R \cap N_1}(A) \leq 4$ for any chief factor $A$, with equality only if $A$ is central. Thus, by Corollary \ref{cor:CFsinwreathproduct}, $\delta_{M, M \cap N_1^t}(A) \leq 4$ for any chief factor $A$. Additionally, if $A \cong C_2^2$ and $t=4$ then $\delta_{M, M\cap N_1^t}(A) \leq \delta_{R, R_{N_1}}(A) \leq 3$, by Lemma \ref{lemma:maxsubgpsSigmai}. Hence,
	\[
	\delta_{\overline M}(A) \leq \delta_{M, M \cap N_1^t}(A) + \delta_{\overline H/\pi(N_1^t)}(A) \leq \begin{cases*}
		4 + 3 & if $A \cong C_2$,\\
		3 + 3 \qquad & if $A \cong C_2^2$ and $t=4$,\\
		4 + 2 & otherwise.
	\end{cases*}
	\]
	This implies that $d(\overline M) \leq 7$, by Theorem \ref{theorem:crowns}.
	
	We next assume that $M$ contains $K^t$ for some maximal $P_1$-normal subgroup $K$ of $\Omega_1$. Firstly, if $\Omega_1$ is quasisimple, then $M \cap \Omega_1^t$ is equal to $Z(\Omega_1^t).\diag(\Omega_1^t)$, by Proposition \ref{prop:XSimpleMaxSubgrpXkC}. Thus, by applying Theorem \ref{theorem:crowns}, we can determine that
	\begin{align*}
		d(\overline M) & \leq d_{\overline M}(\pi(Z(\Omega_1)^t)) + d(\overline M/\pi(Z(\Omega_1)^t)) \\
		& \leq d(Z(\Omega_1)) + d(\overline H/\pi(\Omega_1^t))\\
		& \leq 1 + 5 = 6.
	\end{align*}
	Indeed, $Z(\Omega_1)$ is cyclic, and $\overline H/\pi(\Omega_1^t)$ is a quotient of a maximal subgroup of $\overline G$, so $d(\overline H/\pi(\Omega_1^t)) \leq 5$ by Theorem \ref{theorem:LMTgenerators}.
	
	Next, if $\Omega_1$ is not quasisimple, we use the following bound
	\begin{align*}
		d(\overline M) & \leq d(K) + d(\overline M / \pi(K^t)).
	\end{align*}
	If $q \leq 3$ and $t = 2$, then we can check computationally that $d(\overline M) \leq 6$ in this case. So we assume that either $q > 3$ or $t > 2$.
	
	If $q \leq 3$, we can compute directly that $d(K) \leq 3$, with equality only if $\Omega_1$ is one of $\SU_3(2)$ or $\Omega_4^+(2)$. On the other hand, if $q > 3$, then $I_1$ is one of: a cyclic group, a dihedral group, or a subgroup of $(\GL_2(q) \circ \GL_2(q)) : \Sym(2)$, where this final case arises when $\Omega_1 = \Omega_4^+(q)$. Thus in these cases we have that $d(K) \leq 2$ for any normal subgroup $K$ of $I_1$.
	
	In order to bound $d(\overline M / \pi(K^t))$, we note that, by Corollary \ref{cor:CFsinwreathproduct}, we have $\delta_{M, \frac{M \cap N_1^t}{K^t}}(A) \leq 1$ for any chief factor $A$, since $\Omega_1/K$ is a chief factor of $P_1$. Hence,
	\[
	\delta_{\overline M/\pi(K^t)}(A) \leq 1 + \delta_{\overline H/\pi(\Omega_1^t)}(A).
	\]
	Additionally, by Theorem \ref{theorem:LMTcrowns},
	\[
	\delta_{\overline H/\pi(\Omega_1^t)}(A) \leq \begin{cases*}
		5 \qquad & if $A \cong C_2$ and $t = 2$,\\
		4 & if $A \cong C_2$ and $t > 2$,\\
		3 & otherwise.
	\end{cases*}
	\]
	Suppose that $q > 3$. When $d(K) \leq 1$, we obtain $d(\overline M) \leq 7$, as required. So assume that $d(K) = 2$, which implies that $\Omega_1 = \Omega_4^+(q) = \SL_2(q) \circ \SL_2(q)$. Hence, $\delta_{M,  (M \cap \Omega_1^t)/Z(\Omega_1^t)}(A) \leq 2$ when $A$ is a non-abelian chief factor, and is 0 otherwise. Hence, by Theorem \ref{theorem:crowns},
	\[
	d(\overline M) \leq d(Z(\Omega_1)) + d(\overline H / \pi( \Omega_1^t)) \leq 6.
	\]
	Next let $q \leq 3$, and $t > 2$. If $d(K) = 2$ then we can conclude by Theorem \ref{theorem:crowns} that $d(\overline M) \leq 7$. So assume that $\Omega_1$ is one of $\SU_3(2)$ or $\Omega_4^+(2)$. Let $Z := Z(\Omega_1)$, and note that $Z$ is either $C_3$ or 1, respectively. It can be checked that $Z \leq \Phi(K)$, so $d(\overline M) = d(\overline M/\pi(Z^t))$. In both cases, $P_1$ has two chief factors in $K/Z$, one isomorphic to $C_3^2$ and one isomorphic to $C_2$. Let these chief factors be denoted by $B_1$ and $B_2$ respectively. Since $M$ projects onto $\Sym(t)$, we have $d_M(B_1^t) = d_M(B_2^t) = 1$. Hence,
	\[
	d(\overline M) = d(\overline M/\pi(Z^t)) \leq 2 + d(\overline M / \pi(K^t)) \leq 7,
	\]
	using the fact that $t > 2$.
\end{proof}

\begin{prop}
	Let $H$ be as above. Suppose that $M$ contains $\Omega_1^t$ and does not contain $N_1^t$. Then $d(\overline M) \leq 7$.
\end{prop}

\begin{proof}
	If $M$ contains $K^t$ where $K$ is a maximal $P_1$-normal subgroup of $N_1$, then the same method as in Proposition \ref{prop:C2McontainingOmega1} proves that $d(\overline M) \leq 7$, since $d(K) \leq 2$ for any normal subgroup $K$ of $I_1$ containing $\Omega_1$ by Lemma \ref{lemma:maxsubgpsSigmai}.
	
	We now apply Proposition \ref{prop:XNonsimpleMaxSubgpXkC}. Assume first that $M \cap N_1^t$ is subdirect. If $N_1/\Omega_1$ is cyclic, then, as previously, $M \cap N_1^t$ contains $K^t$ for some maximal $P_1$-normal subgroup $K$ of $N_1$ containing $\Omega_1$. Hence, $d(\overline M) \leq 7$. If instead $N_1/\Omega_1 = C_2^2$, then $I_1 = N_1$. Additionally, for some $\Sym(t)$-invariant subgroup $X$ of $C_2^t$, $M$ has shape
	\[
	M = \Omega_1^t . X . C_2^t . \frac{H}{H \cap I_1^t}.
	\]
	As such, since $X$ is either trivial, $\del(C_2^t)$ or $\diag(C_2^t)$, we have 
	\[
	d(\overline M) \leq d(\Omega_1) + 2d(C_2) + d \left( \frac{\overline H}{\overline H \cap \pi(I_1^t)} \right) \leq 7,
	\]
	by Proposition \ref{prop:boundtopC2}.
	
	If instead $M \cap N_1^t$ is not subdirect then $M \cap N_1^t$ is conjugate to $R^t$ for some subgroup $\Omega_1 \leq R < N_1$. As such,
	\[
	d(\overline M) \leq d(R) + d \left( \frac{\overline H}{\pi(N_1^t)} \right) \leq 2 + 4 = 6.\qedhere
	\]
\end{proof}

\begin{prop}
	Let $H$ be as above. Suppose that $M$ contains $N_1^t$ and does not contain $H \cap I_1^t$. Then, $d(\overline M) \leq 7$.
\end{prop}

\begin{proof}
	Computing the structure of $M$ follows as previously. Firstly, if $(M \cap I_1^t)/N_1^t$ is not subdirect in $(I_1/N_1)^t$, then there exists some $N_1 \leq R < I_1$ such that
	\[
	M = N_1^t.\del((R/N_1)^t).(H/(H \cap I_1^t)).
	\]
	As such, by Proposition \ref{prop:boundtopC2} and Lemma \ref{lemma:maxsubgpsSigmai},
	\[
	d(\overline M) \leq d(N_1) + d(R/N_1) + d(\overline H/\pi(H \cap I_1^t)) \leq 2+2+3 = 7.
	\]
	
	Hence, we assume that $(M \cap I_1^t)/N_1^t$ is subdirect. If $I_1/N_1$ is cyclic, then there exists some maximal subgroup $K$ of $I_1$ containing $N_1$ such that
	\[
	M = N_1^t . \del((K/N_1)^t) . X . \frac{H}{H \cap I_1^t},
	\]
	where $X=\diag((I_1/K)^t)$. Hence,
	\[
	d(\overline M) \leq d(N_1) + d(K/N_1) + d(I_1/K) + d \left( \frac{\overline H}{\pi(H \cap I_1^t)} \right) \leq 2+1+1+3 = 7.
	\]
	
	If instead $I_1/N_1$ is isomorphic to $C_2^2$, then $N_1 = \Omega_1$. Hence, $M$ has shape
	\[
	M = \Omega_1^t . X . \del(C_2^t) . \frac{H}{H \cap I_1^t},
	\]
	for some subgroup $X $ of $\del(C_2^t)$ which is $\Sym(t)$-invariant. Hence, as before, we obtain $d(\overline M) \leq 7$.
\end{proof}

\begin{prop}
	Let $H$ be as above. Suppose that $M$ contains $H \cap I_1^t$ and does not contain $H \cap P_1^t$. Then $d(\overline M) \leq 7$.
\end{prop}

\begin{proof}
	There exists some subgroup $K < \overline P_1/\overline I_1$ such that 
	$$\overline M = \pi(N_1^t . \del((I_1/N_1)^t)) . K . \Sym(t).$$ 
	By Table \ref{table:C2basicstructure}, $\delta_K(A) \leq 2$ for any chief factor $A$, with equality only if $A \cong C_2$. Hence $d(K.\Sym(t)) \leq 3$, and so
	\[
	d(\overline M) \leq d(N_1) + d(I_1/N_1) + d(K.\Sym(t)) \leq 2 + 2 + 3 = 7.
	\]
\end{proof}

\begin{prop} \label{prop:lastcaseC2}
	Let $H$ be as above, and suppose that $M$ contains $H \cap P_1^t$. Then $d(\overline M) \leq 7$. 
\end{prop}

In order to prove this proposition, we require some preliminary lemmas.

\begin{lemma} \label{lemma:permmodulemults}
	Let $K$ be a group, $r$ a prime, $V$ a cyclic $\mathbb F_r[K]$-module, and $W$ an irreducible $\mathbb F_r[K]$-module. Then $W$ has multiplicity at most $\dim_E(W)$ in $V/\Rad(V)$.
\end{lemma}

\begin{proof}
	Let $n$ be the multiplicity of $W$ in $V/\Rad(V)$. We show first that 
	\[
		\dim_E(\End_{\mathbb{F}_r[K]}(W)) = n.
	\]
	We may write $V/\Rad(V)=W^n \oplus W_1 \oplus \cdots \oplus W_m$ for some irreducible $\mathbb F_r[K]$-modules $W_i$ with $W_i\not\cong W$. Since $W$ is irreducible, Schur's Lemma implies that the kernel of any map in $\Hom_{\mathbb F_r[K]}(V, W)$ contains the preimage of $W_1\oplus\hdots\oplus W_m$ in $V$. Thus,
	\begin{align*}
		\Hom_{\mathbb F_r[K]}(V, W) \cong \Hom_{\mathbb F_r[K]}(W^n, W)
		\cong \left( \Hom_{\mathbb F_r[K]}(W, W) \right)^n
		\cong E^n.
	\end{align*}
	
	Since $V$ is cyclic, we can write $V = \langle a \rangle_{\mathbb{F}_r[K]}$ for some $a \in V$. Let $\{f_1, \dots, f_\ell\}$ be a basis of $\Hom_{\mathbb{F}_r[K]}(V,W)$ over $E$, and let $f_i(a) = w_i \in W$ for each $i$. We claim that $\{w_1, \dots, w_{\ell}\}$ is a linearly independent set of $W$ over $E$, and hence $\dim_E(\Hom_{\mathbb{F}_r[K]}(V,W)) \leq \dim_E(W)$.
	
	Suppose that $\alpha_1 w_1 + \cdots + \alpha_\ell w_\ell = 0$ for some $\alpha_i \in E$. Then, by definition,
	\[
	\alpha_1 f_1(a) + \cdots \alpha_\ell f_\ell(a) = 0,
	\]
	and hence, for any $k \in \mathbb{F}_r[K]$, we have
	\[
	\alpha_1 f_{w_1}(ka) + \cdots \alpha_\ell f_{w_\ell}(ka) = 0,
	\]
	since $\alpha_i \in \End_{\mathbb{F}_r[K]}(W)$ and $f_{w_i} \in \Hom_{\mathbb{F}_r[K]}(V,W)$ commute with $k$. However, $V$ is cyclic, so this implies that $\sum_i \alpha_i f_i = 0$. Thus, $\alpha_i = 0$ for all $i$.
\end{proof}

\begin{lemma} \label{lemma:delmodulemults}
	Let $N$ be a transitive permutation group of degree $d$, and let $V = \mathbb{F}_r^d$ be the natural permutation module of $N$ over $\mathbb{F}_r$, for some prime $r$. Suppose that $N$ has no quotients isomorphic to a non-trivial subgroup of $C_r:C_{r-1}^2$. Then $\del(V)$ has no quotients isomorphic to $C_r$.
\end{lemma}

\begin{proof}
	Firstly, any 1-dimensional module of $N$ over $\mathbb{F}_r$ is equal to the trivial module, as otherwise $N$ maps non-trivially into $\Aut(C_r) = C_{r-1}$, contradicting our assumption. So we can suppose for a contradiction that there exists a submodule $U$ of $\del(V)$ such that $\del(V)/U$ is isomorphic to $C_r$, the trivial $\mathbb{F}_r[N]$-module.
	
	By Lemma \ref{lemma:permmodulemults}, $V$ has at most one trivial quotient. However, $V/\del(V)$ is a trivial quotient, and thus $V/U = \mathbb{F}_r^2$ is an indecomposable 2-dimensional module of $N$. Thus, there exists a map from $N$ into $\GL_2(r)$, with the image of $N$ lying in the parabolic subgroup $P_1 \cong C_r : C_{r-1}^2$. Additionally, in order for this module to be indecomposable, the image of $N$ in $P_1$ cannot be trivial. Thus, we have a contradiction.
\end{proof}

\begin{lemma}\label{lemma:delgens}
	Let $L$ be an abelian group, and define injective maps $\iota_i \colon L \to L^t$ as in Section \ref{section:StC}. Let $\Sym(t)$ act on $L^t $ in such a way that $\iota_i(L)^\sigma = \iota_{i^\sigma}(L)$ for any $i$ and any $\sigma \in \Sym(t)$. If $K$ is a maximal subgroup of $\Sym(t)$ such that $\del(L^t)$ is a $K$-invariant subgroup of $L^t$, then $d_K(\del(L^t)) \leq d(L)$.
\end{lemma}

\begin{proof}
	Fix a generating set $X$ of $L$. Assume first that $K$ is intransitive, so $K = \Sym(m) \times \Sym(t-m)$ for some $m$. Then, the set $\{\iota_1(x)\iota_{m+1}(x^{-1}) : x \in X \}$ generates $\del(L^t)$ under the action of $K$. If instead $K$ is transitive but imprimitive, then $K = \Sym(k) \wr \Sym(t)$ and the set $\{\iota_1(x)\iota_{k+1}(x^{-1}) : x \in X \}$ generates $\del(L^t)$ under the action of $K$.
	
	Hence, we assume that $K$ is a primitive subgroup of $\Sym(t)$. Consider the complete graph on $t$ points, on which $\Sym(t)$, and thus $K$, acts naturally. Let $\Gamma$ be the subgraph given by the orbit of the edge $\{1,2\}$ under the action of $K$. This graph is connected, since any connected component is a block of $K$. Indeed, any element $g \in K$ maps a path from $i$ to $j$ to a path from $i^g$ to $j^g$, so it is clear that $g$ maps each connected component to itself, or to a different connected component.
	
	We show that this implies that $Y = \{\iota_1(x)\iota_2(x^{-1}) : x \in X \}$ generates $\del(L^t)$ under the action of $K$. Let $\widetilde L = \langle y^K : y \in Y \rangle$. For any $g \in K$, $\widetilde L$ contains $\del(\iota_{1^g}(L) \times \iota_{2^g}(L))$. Hence, for any edge $\{i,j\}$ in $\Gamma$, $\widetilde L$ contains $\del(\iota_i(L) \times \iota_j(L))$. Since $\Gamma$ is connected, for any $i, j$ there exists a path from $i$ to $j$, given by $\{i_0, i_1\}, \dots, \{i_{r-1}, i_r\}$ where $i_0 = i$ and $i_r = j$. Then $\langle \del(\iota_{i_s}(L) \times \iota_{i_{s+1}}(L)) \rangle_s$ contains $\del(\iota_i(L) \times \iota_j(L))$, and hence $\widetilde L$ contains $\del(L^t)$.
\end{proof}

\begin{proof}[Proof of Proposition \ref{prop:lastcaseC2}]
	Let $J_1$ be the projection of $M$ to $\Sym(t)$, so that
	\[
	M = N_1^t . \del((I_1/N_1)^t) . (P_1/I_1) . J_1,
	\]
	where $J_1 \mleq \Sym(t)$.
	
	To begin with, if $J_1$ is intransitive, then $J_1 = \Sym(m) \times \Sym(t-m)$ with $m \neq t-m$, unless $t=2$, in which case we have $m =t-m = 1$ and $J_1 = 1$. By Corollary \ref{cor:CFsinwreathproduct}, for each chief factor $A$ of $P_1$ in $N_1$, the non-Frattini chief factors of $M$ in $A^m$ are isomorphic to $A^m$, $A^{m-1}$, or $A$. If $A, B$ are two chief factors of $P_1$ in $N_1$, then the chief factors of $M$ in $A^m$ and $B^m$ are $M$-equivalent only if $A$ and $B$ are $N_1$-equivalent as sections. Finally, $A^m$ is non-Frattini in $M$ only if $A$ is non-Frattini in $N_1$. Similar results can be obtained for the chief factors of $M$ in $N_1^{t-m}$. Hence, by Corollary \ref{cor:upperbounddeltaGN}, we obtain that, for any chief factor $A \cong S^b$,
	\[
	\delta_{M, N_1^t}(A) \leq 2 \max_{\{N_1\text{-group } B \cong S^c, \, c \mid b \}} \delta_{N_1}(B).
	\]
	If $B$ is a chief factor of $P_1$ in $I_1/N_1$, then $B$ has prime order, and so $d_{M}(\del(B^m)) = 1$ by Lemma \ref{lemma:delgens}. Thus, by Lemma \ref{lemma:permmodulemults},
	\[
	\delta_{\overline M, \pi(M \cap I_1^t)/\pi(N_1^t)}(A) \leq \dim_E(A) \max_{\{N_1\text{-group } B \cong S^c, \, c \mid b\}} \delta_{I_1/N_1}(B),
	\]
	where $E = \End_{\mathbb{F}_r[\overline M/C_{\overline M}(A)]}(A)$. We have the following bounds, by Lemma \ref{lemma:maxsubgpsSigmai} and Table \ref{table:C2basicstructure}:
	\begin{align*}
		\delta_{N_1}(B) & \leq \begin{cases*}
			2 \qquad & if $B \cong C_2$,\\
			1 & otherwise,
		\end{cases*}\\
		\delta_{I_1/N_1}(B) & \leq \begin{cases*}
			2 \qquad & if $B \cong C_2$,\\
			1 & otherwise,
		\end{cases*}\\
		\delta_{\overline P_1/\overline I_1 \times J_1}(A) & \leq \begin{cases*}
			4 \qquad & if $A \cong C_2$,\\
			2 & if $A$ is otherwise central, or $A \cong C_3$,\\
			1 & otherwise.
		\end{cases*}
	\end{align*}
	Additionally, by Lemma \ref{lemma:non_simple_gen_bound}, $\delta_{N_1}(C_2) + \delta_{I_1/N_1}(C_2) \leq 3$. For any chief factor $A$, let 
	\[
	\delta_1 := 2 \max_{N_1\text{-group } B \cong S^c, c \mid b} \delta_{N_1}(B) + \delta_{\overline P_1/\overline I_1 \times J_1}(A).
	\]
	Then Theorem \ref{theorem:crowns} implies that
	\[
	d((L_A)_{\delta_{\overline M}(A)}) \leq \max_{N_1\text{-group } B \cong S^c, c \mid b} \delta_{I_1/N_1}(B) + \theta(A) +  \Bigg\lceil \frac{\delta_1 + s(A)}{\dim_E(A)} \Bigg\rceil.
	\]
	This is bounded above by 7 when $A$ is not isomorphic to $C_2$, by the bounds above. Thus, $d(\overline M) > 7$ if and only if $\delta_{\overline M}(C_2) > 7$, by Theorem \ref{theorem:crowns}. If $A$ is isomorphic to $C_2$, then $\dim_E(A) =1$, so we have
	\[
	\delta_{\overline M}(C_2) \leq 2 \delta_{N_1}(C_2) + \delta_{I_1/N_1}(C_2) + \delta_{\overline P_1/\overline I_1 \times J_1}(C_2).
	\] 
	Suppose for a contradiction that $\delta_{\overline M}(C_2) > 7$, and so by the bounds above we have
	\begin{align*}
		\delta_{N_1}(C_2) = 1, & \qquad \delta_{I_1/N_1}(C_2) = 2, \qquad  \delta_{\overline P_1/\overline I_1}(C_2) = 2, \text{ or }\\
		\delta_{N_1}(C_2) = 2, & \qquad \delta_{I_1/N_1}(C_2) + \delta_{\overline P_1/\overline I_1}(C_2) \geq 2.
	\end{align*}
	In either case, $P_1$ is of type $\boldO$. Suppose that $\overline M$ satisfies the conditions of the first case. Then $\delta_{I_1/N_1}(C_2) = 2$ implies that $I_1/N_1 \cong C_2^2$, and so $N_1 = \Omega_1$. However, $\delta_{\overline P_1/\overline I_1}(C_2) = 2$ implies that $P_1$ contains an element $\delta_1$ such that $C := [\langle \delta_1 \rangle \Omega_1, I_1/\Omega_1]$ has order 2. Hence, by Lemma \ref{lemma:jumpingbabyresultdel}, $\del(C^t)$ is Frattini in $M$. Thus, $\delta_{M, \del((I_1/N_1)^t)}(C_2) \leq 1$, and so $\delta_{\overline M}(C_2) \leq 7$ as required.
	
	In the second case, we assume first that $q$ is even. This implies that $I_1 / \Omega_1 \cong C_2$, so $\delta_{N_1}(C_2) = 2$ only if $\Omega_1 = \Omega_4^+(2)$. However, in this case we have $\overline P_1 / \overline I_1 = 1$, so $\delta_{I_1/N_1}(C_2) + \delta_{\overline P_1 / \overline I_1}(C_2) \leq 1$. Hence, we may assume that $q$ is odd.
	
	Suppose that $I_1 = N_1$, so $N_1 / \Omega_1 \cong C_2^2$. Hence, since $\delta_{I_1/N_1}(C_2) = 0$ we have $\delta_{\overline P_1/\overline I_1}(C_2) = 2$, and so $C := [P_1/\Omega_1, N_1/\Omega]$ has order 2. Thus, by Lemma \ref{lemma:jumpingbabyresultfull}, $C^t$ is Frattini in $M$. Therefore, $\delta_{M,N_1^t}(C_2) \leq 2$, implying that $\delta_{\overline M}(C_2) \leq 6$. Suppose instead that $N_1 < I_1$. Since $\delta_{N_1}(C_2) = 2$, this implies that $\Omega_1 = \Omega_2^\pm(q)$, and $N_1$ is isomorphic to a dihedral group of index 2 in $I_1$. Thus, $[I_1, N_1] \geq \Omega_1$, and so as previously we have $\delta_{M, N_1^t}(C_2) \leq 2$.
	
	We have shown that $d(\overline M) \leq 7$ when $J_1$ is intransitive, so we assume for the remainder of this proof that $J_1$ is transitive. We begin by assuming that $\Omega_1$ is quasisimple, and hence the non-Frattini chief factors of $\overline M$ are equal to those of $\overline M / \pi(\Omega_1^t)$, except for an additional $\overline \Omega_1^t$ in $\overline M$. Thus, by Theorem \ref{theorem:crowns}, we have that $d(\overline M) = d(\overline M/\pi(\Omega_1^t))$. By Lemma \ref{lemma:delgens},
	\[
	d(\overline M / \pi(\Omega_1^t)) \leq d(N_1/\Omega_1) + d(I_1/N_1) + d(\overline P_1 / \overline I_1 \times J_1).
	\]
	If $ \Omega_1$ is not in case $\boldL$ or $\boldO^\pm$, then 
	\[
	d(\overline M / \pi(\Omega_1^t)) \leq 1+1+5 = 7.
	\]
	So we assume that $\Omega_1$ is in case $\boldL$ or $\boldO^\pm$ with $q$ odd, and that the first inequality above gives us a bound of 8. Hence, $d(N_1/\Omega_1) + d(I_1/N_1) = 2$, $\delta_{\overline P_1 / \overline I_1}(C_2) = 2$ and $\delta_{J_1}(C_2) = 4$, by Lemma \ref{lemma:maxsubgpsSigmai}, Table \ref{table:C2basicstructure} and \cite[Theorem 2.7]{LMT}. In particular, $J_1$ is of simple diagonal type, so we write $J_1 = T^k . (\Out(T) \times \Sym(k)) \mleq \Sym(t)$ for some simple group $T$ with $t = |T|^{k-1}$.
	
	If $\Omega_1$ is of type $\boldO$ and $N_1/\Omega_1 = I_1/\Omega_1 \cong C_2^2$, then, by Lemma \ref{lemma:jumpingbabyresultfull}, $(S_1/\Omega_1)^t \cong C_2^t$ is Frattini in $M$. Hence, for each prime $r$, there is at most one chief factor $X/Y$ of $P_1$ contained in $N_1/\Omega_1$ isomorphic to $C_r$ such that the corresponding section $X^t/Y^t$ of $M$ is non-Frattini. Note the same is true in the $\boldL$ case, since $N_1/\Omega_1$ is cyclic. Let this section $(X/Y)^t$ of $M$ be denoted by $W \cong C_r^t$. By Corollary \ref{cor:CFsinwreathproduct}, any non-Frattini chief factor of $M$ in $(N_1/\Omega_1)^t$ which is isomorphic to $A \cong C_r^a$ is contained in $W$. Hence, $\delta_{M, (N_1/\Omega_1)^t}(A) = \delta_{M,W}(A)$. Similarly, there exists a chief factor $B$ of $P_1$ in $I_1/N_1$ such that $\delta_{M, \del((I_1/N_1)^t)}(A) = \delta_{M, B^t}(A)$.
	
	Let $A$ be a chief factor of $\overline M$. If $A$ is non-abelian, then $\delta_{\overline M}(A) \leq 1$, so we assume that $A \cong C_r^k$ for some prime $r$ and some $k \in \mathbb N$. Recall that $\delta_{\overline M, \pi(\Omega_1^t)}(A) = 0$. Let $\delta_1 = \delta_{\overline M, \pi(N_1^t)/\pi(\Omega_1^t)}(A)$, $\delta_2 = \delta_{\overline M, \pi(M \cap I_1^t)/\pi(N_1^t)}(A)$ and $\delta_3 = \delta_{\overline M/\pi(M \cap I_1^t)}(A)$. Note that, by \cite[Theorem 2.7]{LMT} and Table \ref{table:C2basicstructure},
	\[
	\delta_3 \leq \begin{cases*}
		6 \qquad & if $A \cong C_2$,\\
		3 & if $A$ is otherwise central,\\
		2 & otherwise.
	\end{cases*}
	\]
	Additionally, for $i = 1, 2$, by Lemmas \ref{lemma:permmodulemults} and \ref{lemma:delgens}, we have that $\delta_i \leq \dim_E(A)$, where $E = \End_{\mathbb{F}_r[\overline M/C_{\overline M}(A)]}(A)$. If $A \not\cong C_2$ then, by Theorem \ref{theorem:crowns}, 
	\[
	d((L_A)_{\delta_{\overline M}(A)}) \leq \theta(A) + 2 + \Bigg\lceil \frac{\delta_3 + s(A)}{\dim_E(A)} \Bigg\rceil \leq 5.
	\]
	If $A \cong C_2$ then $\dim_E(A) = 1$, $s(A) = 0$, and $\theta(A) = 0$. Additionally, $\delta_2 = 0$ by applying Lemma \ref{lemma:delmodulemults} and Corollary \ref{cor:upperbounddeltaGN} with $N = T^k \unlhd J_1$. Hence,
	\[
	d((L_A)_{\delta_{\overline M}(A)}) \leq 1 + \delta_3 \leq 7.
	\]
	So we have shown that $d(\overline M)\leq 7$, by Theorem \ref{theorem:crowns}.
	
	\begin{table}
		\centering
        \caption{Bounds on the number of generators of certain subquotients of $\Sigma_1^*$, and the corresponding bound obtained for $d(\overline M)$, in the case that $\overline \Omega_1$ is not simple.}
		\begin{tabular}{llrrrr}
			Case & $(m,q)$ & $d(N_1)$ & $d(I_1/N_1)$ & $d(\overline P_1/\overline I_1)$ & $d( \overline M)$ \\ \toprule
			$\boldL$ & $(2,2)$ & 2 & 0 & 1 & 7 \\ 
			& $(2,3)$ & 2 & 1 & 1 & 8 \\ \midrule
			$\boldU$ & $(2,2)$ & 2 & 1 & 1 & 8 \\ 
			& $(2,3)$ & 2 & 1 & 1 & 8 \\
			& $(3,2)$ & 2 & 1 & 1 & 8 \\ \midrule
			$\boldS$ & $(2,2)$ & 2 & 0 & 0 & 6 \\ 
			& $(2,3)$ & 2 & 0 & 0 & 6 \\
			& $(4,2)$ & 2 & 0 & 1 & 7 \\ \midrule
			$\boldO$ & $(3,3)$ & 2 & 2 & 0 & 8 \\
			& $(2,q,\pm)$ & 1 & 2 & 2 & 9 \\
			& $(4,2,+)$ & 2 & 1 & 0 & 7 \\
			& $(4,3,+)$ & 2 & 2 & 1 & 9 \\
			& $(4,q,+)$ & 2 & 2 & 2 & 10 \\ \bottomrule
		\end{tabular}
		\label{table:C2crudebound}
	\end{table}
	
	Finally, we assume that $\Omega_1$ is not quasisimple. If $\Omega_1 = 1$ then the same proof as above holds, so we additionally assume that $m > 1$. Hence, $\Omega_1$ is one of the groups in Table \ref{table:C2crudebound}. By Lemma \ref{lemma:delgens},
	\[
	d(\overline M) \leq d(N_1) + d(I_1/N_1) + d((\overline P_1/\overline I_1) \times J_1)
	\]
	which gives the desired bound of 7 if $\Omega_1$ is of type $\boldS$, or is equal to $\SL_2(2)$ or $\Omega_4^+(2)$. Otherwise, we indicate the non-Frattini chief factors of $I_1/\Omega_1$, and the chief factors of $I_1$ in $\Omega_1$ which are non-Frattini in $\Omega_1$, in Table \ref{table:C2CFsdetailed}. If $B$ is a chief factor of $I_1$ in $\Omega_1$ which is Frattini in $\Omega_1$, then $B^t$ is Frattini in $M$, so we do not consider these chief factors. Additionally, if $B$ is a non-central chief factor of $I_1$ in $\Omega_1$, then $B^t$ is a chief factor of $M$, since $J_1$ transitive implies that $t > 2$.
	
	\begin{table}
		\centering
        \caption{The chief factors of $I_1/\Omega_1$ and of $I_1$ in $\Omega_1$ which are non-Frattini in $\Omega_1$, for certain classical groups.}
		\begin{tabular}{llll}
			Case & $(m,q)$ & C.f.s of $I_1$ in $\Omega_1$, NF in $\Omega_1$ & NF C.f.s of $I_1/\Omega_1$ \\ \toprule
			$\boldL$ & $(2,3)$ & $C_2^2$, non-central $C_3$ & $C_2$\\\midrule
			$\boldU$ & $(2,2)$ & non-central $C_3$, $C_2$ & $C_3$\\
			& $(2,3)$ & $C_2^2$, non-central $C_3$ & $C_2$\\
			& $(3,2)$ & $C_2^2$, $C_3^2$ & $C_3$ \\ \midrule
			$\boldO$ & $(3,3)$ & $C_2^2$, non-central $C_3$ & $C_2$, $C_2$\\
			& $(2,q,\pm)$ & $C_r$ for each $r \mid \frac{q\mp 1}{(q-1,2)}$ & $C_2$, $C_2$ \\
			& $(4,3,+)$ & $C_2^4$, non-central $C_3$, non-central $C_3$ & $C_2$, $C_2$ \\
			& $(4,q,+)$ & $\PSL_2(q)^2$ & $C_2$, $C_2$ \\ \bottomrule
		\end{tabular}
		\label{table:C2CFsdetailed}
	\end{table}
	
	Suppose that the chief factors of $I_1$ in $\Omega_1$ which are non-Frattini in $\Omega_1$ are all non-central. Then Table \ref{table:C2CFsdetailed} implies that $\delta_{M, \Omega_1^t}(A) \leq 1$ for all $A$. Additionally, in all such cases the chief factors of $\overline M$ in $\pi(\Omega_1^t)$ are non-isomorphic to the chief factors of $\overline M / \pi(\Omega_1^t)$. Hence, by Theorem \ref{theorem:crowns}, 
	\[
	d(\overline M) \leq \max(2,d(\overline M/\pi(\Omega_1^t))) \leq 7.
	\]
	So we assume in the following that there exists a chief factor of $I_1$ in $\Omega_1$ which is non-Frattini in $\Omega_1$ and is central.
	
	Suppose that $\Omega_1 = \SU_2(2) \cong \Sym(3)$. Let $A$ be a chief factor of $\overline M$. If $A$ is non-abelian, then $\delta_{\overline M}(A) \leq 2$, so assume that $A$ is abelian. Let $\delta_0 = \delta_{\overline M, \pi(\Omega_1^t)}$, and define $\delta_1, \delta_2$ and $\delta_3$ as previously. Then, we have that: $\delta_0 = 1$ if $A \cong C_3^t$; $\delta_0 \leq \dim_E(A)$ if $A \cong C_2^k$ for some $k$; and $\delta_0 = 0$ otherwise. Similarly, $\delta_1$ and $\delta_2$ are bounded above by $\dim_E(A)$ if $A \cong C_3^k$ for some $k \leq t$, and are 0 otherwise. Moreover, $\delta_1 > 0$ implies that $\delta_2 = 0$, and vice versa. Finally, $\delta_3 \leq 5$ if $A \cong C_2$, $\delta_3 \leq 3$ if $A$ is otherwise central, and $\delta_3 \leq 2$ otherwise. Hence, as previously, we obtain that
	\[
	d((L_A)_{\delta_{\overline M}(A)}) \leq \begin{cases*}
		6 \qquad & if $A \cong C_2$,\\
		5 & otherwise,
	\end{cases*}
	\]
	so $d(\overline M) \leq 6$.
	
	Suppose that $\Omega_1 = \Omega_2^\pm(q)$, so $I_1$ is a dihedral group and $\Omega_1$ is a cyclic group. When $r$ is odd, the non-Frattini chief factor of $I_1$ in $\Omega_1$ isomorphic to $C_r$ is non-central, so we have that 
	\[
	\delta_{\overline M}(C_r^t) \leq 1 + \delta_{\overline M/\pi(\Omega_1^t)}(C_r^t) \leq 3.
	\]
	Additionally, $\delta_{\overline M}(C_r^a) = \delta_{\overline M/\pi(\Omega_1^t)}(C_r^a) \leq 3$ for any $a \neq t$. Hence, we let $A$ be a chief factor isomorphic to $C_2^a$ for some $a$. Note that $\Omega_1 = [I_1, I_1]$, so $\delta_{M, \del(\Omega_1^t)}(A) = 0$ by Lemma \ref{lemma:jumpingbabyresultdel}. If $|N_1/\Omega_1| \geq 2$, then $\Omega_1 = [I_1, N_1]$, so $\delta_{M,\Omega_1^t}(A) = 0$, implying that $d(\overline M) \leq \max(3, d(\overline M/\pi(\Omega_1^t))) \leq 7$ by Theorem \ref{theorem:crowns}. Therefore, we assume that $N_1 = \Omega_1$. We have $\delta_{\overline M, \pi(\Omega_1)^t}(A) \leq 1$ with equality only if $A \cong C_2$, so $\delta_{\overline M}(A) = \delta_{\overline M / \pi(\Omega_1^t)}(A) $ for any chief factor $A \cong C_2^a$ with $a > 1$. 
	
	We now show that $\delta_{\overline M}(C_2) \leq 7$. Firstly, if $J_1$ is of simple diagonal type or is almost simple, then $\soc(J_1)$ is a product of non-abelian simple groups, and is transitive. Hence, the same proof as previously shows that $\delta_1 = \delta_2 = 0$, so $\delta_{\overline M/\pi(\Omega_1^t)}(C_2) \leq 6$. Therefore, $\delta_{\overline M}(C_2) \leq 7$. On the other hand, for any other maximal subgroups $J_1$ of $\Sym(t)$, we have $\delta_{J_1}(C_2) \leq 2$ by the O'Nan-Scott Theorem. So, using the fact that $\delta_1 = 0$ and $\delta_2 \leq 2$, we obtain that $\delta_{\overline M / \pi(\Omega_1^t)}(C_2) \leq 6$. Thus, $\delta_{\overline M}(A)$ is bounded above by 7 when $A \cong C_2$, and otherwise is bounded above by $\delta_{\overline M/\pi(\Omega_1^t)}(A)$ or 3. As such, $d(\overline M) \leq \max(7, d(\overline M/\pi(\Omega_1^t)) ) \leq 7$, by Theorem \ref{theorem:crowns}.
\end{proof}

\begin{prop}
	Suppose that $H \in \mathscr C_2$ is of type $\O_1(q) \wr \Sym(t)$. Then $d(\overline M) \leq 7$.
\end{prop}

\begin{proof}
	Note that $q = p \geq 3$, $t \geq 5$, and, by \cite[Proposition 4.2.15]{KL} we have
	\[
	\Alt(t) \leq H_\Omega = \del(C_2^t).J \leq C_2 \wr \Sym(t),
	\]
	for some group $J$ such that $\Alt(t) \leq J \leq \Sym(t)$. Since $t \geq 5$, $H_\Omega$ satisfies the conditions of Corollary \ref{cor:CFsinwreathproduct}. If $M$ contains $H_\Omega$, then $d(\overline M) = d(H_{\overline L})$ for some $L$, and hence $d(\overline M) \leq 4$ by Theorem \ref{theorem:LMTcrowns}. If $M$ does not contain $\del(C_2^t)$ then $M \cap \del(C_2^t)$ is either trivial or $\diag(C_2^t)$, by Corollary \ref{cor:CFsinwreathproduct}. Hence, by Theorem \ref{theorem:crowns} and \cite[Proposition 3.12]{LMT},
	\[
	d(\overline M) \leq d((M \cap \del(C_2^t)).J) + d(\overline H/H_{\overline \Omega}) \leq 2+3 = 5.
	\]
	Thus, we may assume that
	\[
	M \cap H_\Omega = \del(2^t).J_1
	\]
	where $J_1$ is a maximal or novelty maximal subgroup of $J$, by Corollary \ref{cor:maximalsubgpextension}. Note that, since $t \geq 5$, the same proof as in Lemma \ref{lemma:delgens} applies to prove that $d_{J_1}(\del(C_2^t)) = 1$. Hence, if $d(J_1) \leq 3$, we have that
	\[
	d(\overline M) \leq 1 + d(J_1) + d(\overline H/H_{\overline \Omega}) \leq 7,
	\]
	by \cite[Proposition 3.12]{LMT}. So we can assume by \cite[Theorem 2.7]{LMT} that $\delta_{J_1}(C_2) = 4$ and $J_1$ is of simple diagonal type. Note that, since $J_1$ is a primitive subgroup of $\Sym(t)$, we may apply the same proof as in Lemma \ref{lemma:delgens} to determine that $\del(C_2^t)$ is a cyclic module of $J_1$.
	
	By Lemma \ref{lemma:delmodulemults} and Corollary \ref{cor:upperbounddeltaGN} applied with $N = \soc(J_1)$, we have that $\delta_{\overline H, \del(C_2^t)}(C_2) = 0$. Hence, 
	\[
	\delta_{\overline M}(C_2) \leq \delta_{J_1}(C_2) + \delta_{\overline H/H_{\overline \Omega}}(C_2) \leq 7,
	\]
	by \cite[Theorem 2.7, Proposition 3.12]{LMT}. Additionally, for any chief factor $A \not\cong C_2$ we have
	\[
	\delta_{\overline M/\pi(\del(2^{t-1}))}(A) \leq \delta_{J_1}(A) + \delta_{\overline H/H_{\overline \Omega}}(A) \leq 4
	\]
	by \cite[Theorem 2.7, Proposition 3.12]{LMT}. Therefore, by Lemma \ref{lemma:permmodulemults}, since $\del(C_2^t)$ is a cyclic module of $J_1$,
	\[
	d((L_A)_{\delta_{\overline M}(A)}) \leq \theta(A) + 1 + \Bigg\lceil \frac{\delta_{\overline M/\pi(\del(2^{t-1}))}(A) + s(A)}{\dim_E(A)} \Bigg\rceil \leq 6,
	\]
	by the same method as in Proposition \ref{prop:lastcaseC2}. Hence, by Theorem \ref{theorem:crowns}, $d(\overline M) \leq 7$.
\end{proof}

\begin{prop}
	Suppose that $H$ is in $\mathscr{C}_2$, and is of type $\GL_{n/2}(q^u).2$. Then $d(\overline M) \leq 7$.
\end{prop}

\begin{proof}
	Firstly, if $M$ contains $H_I$, then $d(\overline M) \leq 3$ by Theorem \ref{theorem:LMTgenerators}, so we assume that $M$ does not contain $H_I$.
	
	By \cite[Lemma 4.2.3]{KL}, $H_I = I_1 . 2$, and $H_\Omega \geq \Omega_1$. Assume that $\Omega_1$ is not quasisimple. This implies that either $n/2 = 1$, or $n/2=2$ and $q^u=2,3$. The latter we treat computationally to show that $d(\overline M) \leq 3$. Otherwise, if $n = 2$, then $H_I = C_{q-1}.C_2$, so $d(\overline M \cap \overline I) \leq 2$. Hence, $d(\overline M) \leq 5$ by \cite[Proposition 3.12]{LMT}.
	
	Thus, we assume that $\Omega_1$ is quasisimple. Let $N = H_G \cap I_1$, which contains $\Omega_1$. If $M$ does not contain $Z(N)$, then $M = Z . (H/Z(N))$ for some subgroup $Z$ of $Z(N)$, so $d(\overline M) \leq 1 + d(H) \leq 5$ by Theorem \ref{theorem:LMTgenerators}. Similarly, if $M$ contains $\Omega_1$ then $M \cap I = \Omega_1.J$ where $J$ is a subgroup of $ C_{q^u-1}.C_2$. Thus, since $\Omega_1$ is quasisimple, by Theorem \ref{theorem:crowns} and \cite[Proposition 3.12]{LMT} we have that
	\[
	d(\overline M) \leq d(J) + d(H_{\overline G}/H_{\overline G \cap \overline I}) \leq  4.
	\]
	So we assume that $M$ contains $Z(N)$, and does not contain $\Omega_1$. Thus, by Corollary \ref{cor:maximalsubgpextension}, $(M \cap N)/Z(N)$ is either a maximal or novelty maximal subgroup of $\overline N$. By Proposition \ref{prop:atmost4inI}, this implies that $d(M \cap N) \leq 4$, and hence 
	\[
		d(\overline M) \leq 4 + d(C_2) + d(H_{\overline G}/H_{\overline G \cap \overline I}) \leq 7,
	\]
	by \cite[Proposition 3.12]{LMT}.
\end{proof}

\begin{prop}
	Suppose that $H \in \mathscr{C}_2$ is of type $\O_{n/2}(q)^2$. Then $d(\overline M) \leq 7$.
\end{prop}

\begin{proof}
	If $M$ contains $H_\Omega$, then $M = H_L$ for some $L$, so $d(\overline M) \leq d(H_{\overline L}) \leq 3$ by Theorem \ref{theorem:LMTgenerators}. So we assume that $M$ does not contain $H_\Omega$.
	
	Let $m = n/2$, which is odd by definition. If $\Omega_m(q)$ is not quasisimple, then $(n, q) = (6, 3)$, so we can check computationally that $d(\overline M) \leq 4$. Hence, we assume that $\Omega_m(q)$ is quasisimple. Let $H_{\Gamma} = N_{\Gamma} \{V_1, V_2\}$ for some $m$-dimensional, orthogonal vector spaces $V_i$ which are similar but not isomorphic.
	
	By Proposition \ref{prop:prelimC2Omq2}, we have
	\[
	H_{\Omega} = S_1 \times S_2,
	\]
	and
	\[
	H_\Sigma = (I_1 \times I_2).\diag((\Sigma_1/I_1) \times (\Sigma_2/I_2)) : \Sym(2).
	\]
	If $H$ does not project onto $\Sym(2)$, then $H \leq \Sigma_1 \times \Sigma_2$. Thus, we can use the same method as for $\mathscr{C}_1$ to determine that $d(\overline M) \leq 5$.
	
	If $H$ projects onto $\Sym(2)$, we apply a similar argument to the case where $H$ of type $\O_m^{\epsilon}(q) \wr \Sym(t)$, noting that in this case $H$ may not satisfy the conditions of Corollary \ref{cor:CFsinwreathproduct}. Since $m$ is odd, $-1 \not\in S_1$, so $S_1, S_2$ are almost simple groups. If $M$ contains $\Omega_1 \times \Omega_2$, then $M \cap H_\Omega = (\Omega_1 \times \Omega_2).J$ where $J < C_2^2$. Thus, by Theorem \ref{theorem:crowns} and \cite[Proposition 3.12]{LMT},
	\[
	d(\overline M) \leq d(J) + d(\overline H/H_{\overline \Omega}) \leq 4.
	\]
	So we assume that $M$ does not contain $\Omega_1 \times \Omega_2$. Thus $M$ does not contain $\Omega_1$ or $\Omega_2$, as $M$ projects onto $\Sym(2)$. The group $M \cap H_\Omega$ is not subdirect in $S_1 \times S_2$, by Goursat's Lemma, since $M \cap S_1$ does not contain $\Omega_1$ and so is not normal in $S_1$. Hence, by Corollary \ref{cor:maximalsubgpextension} and Proposition \ref{prop:XNonsimpleMaxSubgpXkC}, $M \cap H_\Omega = R_1 \times R_2$ where each $R_i$ is a maximal or novelty maximal subgroup of $S_i$. Note that $d(R_i) \leq 3$, by Theorem \ref{theorem:LMTgenerators}. Thus, since $M$ projects onto $\Sym(2)$, $d(\overline M) \leq d(R_i) + d(\overline H/H_{\overline \Omega}) \leq 6$ by \cite[Proposition 3.12]{LMT}.
\end{proof}

We end this section with the following example, which exhibits an infinite family of maximal subgroups $H$ of type $\O_m^{\epsilon}(q) \wr \Sym(t)$ together with maximal subgroups $M$ of $H$ such that $d(\overline M) = 7$.

\begin{example}\label{example:boundtightC2}
	Let $\Omega = \Omega_{20m}^+(9)$ for any $m \geq 1$, and let $G = \langle I, \phi \rangle$. Let $H$ be a maximal subgroup of $G$ of type $\O_{4m}^{+}(9) \wr \Sym(5)$. Hence,
	\[
	H = I_1^5 . \diag((\langle I_1, \phi \rangle/I_1)^5) : \Sym(5),
	\]
	and
	\[
	H/\Omega_1^t \cong ((C_2^2) \wr \Sym(5)) \times C_2.
	\]
	Using Magma, we can show that there exists a maximal subgroup $L$ of $H/\Omega_1^t$ for which $\delta_L(C_2) = 7$. For instance, 
	\[
	L = ((C_2^2) \wr (\Sym(2) \times \Sym(3)) \times C_2
	\]
	is one such maximal subgroup. Let $M$ be the preimage of $L$ in $H$. Since the scalars of $I_1$ are contained in $\Omega_1$, the scalars of $H$ are contained in $\Omega_1^t$, so the non-Frattini chief factors of $M$ in $M/\Omega_1^t$ are the same as those of $\overline M$ in $\overline M / \pi(\Omega_1^t)$. Thus, $\delta_{\overline M/\pi(\Omega_1^t)}(C_2) = 7$, and hence $\delta_{\overline M}(C_2) \geq 7$. By Theorem \ref{theorem:crowns}, this implies that $d(\overline M) \geq 7$, so by the results above we have $d(\overline M) = 7$.
\end{example}

\subsection{Class \texorpdfstring{$\mathscr{C}_3$}{C3}}

\begin{prop}
	Suppose that $H \in \mathscr{C}_3$. Then $d(\overline M) \leq 7$.
\end{prop}

\begin{proof}
	Let $\mathbb{F}_{\sharp}$ be a finite extension of $\mathbb{F}$ of degree $r$, which is prime. Let $V_{\sharp}$ be the vector space $V \otimes_{\mathbb{F}} \mathbb{F}_{\sharp}$ over $\mathbb{F}_{\sharp}$, and $\kappa_{\sharp}$ a form on this vector space as defined in \cite[Section 4.3]{KL}. Then, by Section \ref{section:C3explanation}, if $\Omega \leq G \leq \Gamma$ we have
	\[
	\Omega_{\sharp} \leq H_G \leq \Gamma_{\sharp}.
	\]
	Thus, by Theorem \ref{theorem:LMTgenerators} and Lemma \ref{lemma:non_simple_gen_bound},
	\[
	d(M) \leq 6.
	\]
	
	So we can assume that $G$ is not contained in $\Gamma$. The following are immediate by the definition of $\mathscr C_3$: $G$ and $\Omega_\sharp$ are in case $\boldL$; $n > 2$;
	\[
	H_G = (H_{G \cap \Gamma}).C_2;
	\]
	and $\Omega_\sharp$ is quasisimple unless $n = r$. However, if $n=r$ then $\Gamma_\sharp$ is isomorphic to $C_{q-1} : C_{f}$, and so $d(\overline M) \leq 3$. Thus, we may assume that $\Omega_\sharp$ is quasisimple.
	
	Let $Z := Z(\Delta_\sharp \cap H)$. If $M$ does not contain $Z \unlhd H$, then $M$ is equal to $C . (H/Z)$, where $C$ is a subgroup of the cyclic group $Z$. Hence, by Theorem \ref{theorem:LMTgenerators},
	\[
	d(M) \leq 1 + d(H/Z) \leq 5.
	\]
	Hence, we assume that $M$ contains $Z$.
	
	If $M$ contains $\Omega_\sharp$, then $M = \Omega_\sharp . J$ where $J$ is a subgroup of $(\Sigma_\sharp/\Omega_\sharp).C_2$. Thus, by \cite[Section 2.2]{LMT} and Theorem \ref{theorem:crowns}, we have that $d(M) \leq 1 + d(M/\Omega_\sharp) \leq 5$. As such, we assume additionally that $M$ does not contain $\Omega_\sharp$. By Corollary \ref{cor:maximalsubgpextension}, this implies that $(M/Z) \cap \overline \Gamma_\sharp$ is a maximal or novelty maximal subgroup of the almost simple group $(H/Z) \cap \overline \Gamma_\sharp$. Thus, by Theorem \ref{theorem:LMTgenerators},
	\[
	d(M) \leq d(Z) + d((M/Z) \cap \overline \Gamma_\sharp) + d(C_2) \leq 7.\qedhere
	\]
\end{proof}

\subsection{Class \texorpdfstring{$\mathscr{C}_4$}{C4}}

\begin{prop}
	Suppose that $H \in \mathscr{C}_4$. Then $d(\overline M) \leq 7$.
\end{prop}

\begin{proof}
	If $H \in \mathscr{C}_4$, then $H$ stabilises a tensor decomposition $V = V_1 \otimes V_2$. Let $\kappa_i$ be a form on $V_i$ as defined in \cite[Section 4.4]{KL}. Then, by Section \ref{section:C4explanation}, we have that
	\[
	\overline \Omega_1 \times \overline \Omega_2 < \overline H \leq \overline \Sigma_1^* \times \overline \Sigma_2^*.
	\]
	As previously, let $P_i$ be the projection of $\overline H$ to each coordinate, and $N_i$ the intersection of $\overline H$ with $P_i \times 1$. Note that $\overline \Omega_i < N_i \leq \overline \Delta_i$.
	
	Hence, by Proposition \ref{prop:maxsubgpdirectproduct}, for any $\overline M \mleq H_{\overline G}$, either $\overline M$ contains $N_i$ for some $i$, or $\overline M$ is subdirect. In the first case, if we assume that $\overline M$ contains $N_1$ without loss of generality, then we have
	\[
	\overline M = N_1 . L_2
	\]
	where $L_2 \mleq P_2$. Thus, by Lemma \ref{lemma:maxsubgpsSigmai}, Theorem \ref{theorem:LMTgenerators} and Corollary \ref{cor:upperbounddeltaGN},
	\begin{align*}
		\delta_{\overline M}(A) \leq \delta_{\overline M, N_1}(A) + \delta_{L_2}(A) \leq \begin{cases*}
			2 + 5 \qquad & if $A \cong C_2$,\\
			2 + 3 & otherwise.
		\end{cases*}
	\end{align*}
	If instead $\overline M$ is subdirect, then it is equal to $(K_1 \times K_2).C$ for some maximal $P_i$-normal subgroups $K_i$ of $N_i$ with $K_i.C \cong P_i$. As in the proof of Proposition \ref{prop:C1secondmaxls}, we have that $d(K_i) \leq 4$ for each $i$. Additionally, $d(P_i) \leq 3$ by Lemma \ref{lemma:maxsubgpsSigmai}. Hence,
	\[
	d(\overline M) \leq d(K_1) + d(P_2) \leq  7. \qedhere
	\]
\end{proof}

\subsection{Class \texorpdfstring{$\mathscr{C}_5$}{C5}}

\begin{prop}
	Suppose that $\overline H \in \mathscr{C}_5$. Then $d(\overline M) \leq 7$.
\end{prop}

\begin{proof}
	If $\overline M$ contains $H_{\overline \Omega}$, then $d(\overline M) \leq 5$ by Theorem \ref{theorem:LMTgenerators}, so we assume that $\overline M$ does not contain $H_{\overline \Omega}$. Let $\mathbb{F}_\sharp$ be a subfield of $\mathbb{F}$ of index $r$ prime, and $V_\sharp$ an $n$-dimensional vector space over $\mathbb{F}_\sharp$, so that $V = V_\sharp \otimes \mathbb{F}$. Additionally, let $X_\sharp := X(V_\sharp)$ for $X \in \{\Omega, \Delta\}$. By \cite[(4.5.5)]{KL} we have that
	\[
	\overline \Omega_\sharp \unlhd H_{\overline \Omega} \unlhd H_{\overline \Delta} = \overline \Delta_\sharp.
	\]
	Thus, $H_{\overline G}$ is equal to $N.J$ for some groups $N, J$ with $\overline \Omega_\sharp \leq N \leq \overline \Delta_\sharp$ and $J \leq \overline \Sigma/\overline \Delta$.
	
	If $\overline M$ contains $\overline \Omega_\sharp$, then 
	\[
	d(\overline M) \leq d(\overline M \cap N) + d(J) \leq 2 + 2 = 4,
	\]
	by Lemma \ref{lemma:maxsubgpsSigmai} and \cite[Proposition 3.12]{LMT}. So we assume that $\overline M$ does not contain $\overline \Omega_\sharp$. If $\overline \Omega_\sharp$ is simple, $\overline M \cap N$ is a maximal or novelty maximal subgroup of $N$ by Corollary \ref{cor:maximalsubgpextension}, so, by Theorem \ref{theorem:LMTgenerators}, $d(\overline M) \leq 4 + d(J) \leq 6$.
	
	Hence, we assume that $\overline \Omega_\sharp$ is not simple. If $\overline \Omega_\sharp = \overline \Omega_4^+(q^{1/r})$ with $q^{1/r} \geq 4$, then $\overline \Omega_\sharp = \PSL_2(q^{1/r})^2$, where $\PSL_2(q^{1/r})$ is simple. In this case $G$ is not of type $\boldL$, so $d(J) \leq 1$ by \cite[Proposition 3.12]{LMT}. Applying a similar argument to that used in the proof of  Lemma \ref{lemma:non_simple_gen_bound}, we have $d(\overline M \cap N) \leq 5$ and so $d(\overline M) \leq 6$.
	
	Otherwise, any subgroup $K$ of $\overline \Omega_\sharp$ has $d(K) + d(J) \leq 5$. Indeed, if $q^{1/r} \leq 3$ then we can check computationally that this holds. Otherwise, $\overline \Omega_\sharp$ is a cyclic group, so the result follows. Hence,
	\[
	d(\overline M) \leq d(M \cap \overline \Omega_\sharp) + d(N/\overline \Omega_\sharp) + d(J) \leq 7,
	\]
	by Lemma \ref{lemma:maxsubgpsSigmai}.
\end{proof}

\subsection{Class \texorpdfstring{$\mathscr{C}_6$}{C6}}

\begin{prop}
	Suppose that $H \in \mathscr{C}_6$. Then $d(\overline M) \leq 7$.
\end{prop}

\begin{proof}
	Suppose that $G$ is in case $\boldL$ or $\boldU$, and that $H$ is of type $r^{2m}.\Sp_{2m}(r)$, with $n = r^m$ and $r$ an odd prime. By \cite[Lemma 5.8]{BLS}, when $n=3$ we have $d(\overline{M}) \leq 3$. So we assume that $n \geq 5$, and hence, by Proposition \ref{prop:C6prelims1}, 
	\[
	r^{2m} : \Sp_{2m}(r) \leq H_{\overline{G}} \leq r^{2m} : \GSp_{2m}(r).
	\]
	Let $N := r^{2m}$ be the unique minimal normal subgroup of $H_{\overline{G}}$. If $\overline{M}$ contains $N$, then $\overline{M} = N . B$ where $B$ is maximal in a classical group. Hence, by \cite[Lemma 2.5]{BLS} and Theorem \ref{theorem:LMTgenerators}, we have
	\[
	d(\overline{M}) \leq 1 + 4 = 5.
	\]
	So next we assume that $\overline{M}$ does not contain $N$, and so, since $N$ is an abelian minimal normal subgroup, we have
	\[
	\overline{M} \cong H_{\overline{G}} / N.
	\]
	Thus, by Theorem \ref{theorem:LMTgenerators},
	\[
	d(\overline{M}) \leq d(H_{\overline G}) \leq 3.
	\]
	
	Now assume that $H$ is of type $(4 \circ 2 ^{1+2m}). \Sp_{2m}(2)$, so $G$ is in case $\boldL^\epsilon$. We begin by assuming that $n \geq 8$. By Proposition \ref{prop:C6prelims2},
	\[
	H_{\overline{G}} = (2^{2m} . \Sp_{2m}(r)) \times C,
	\]
	where $C \leq C_2$. Additionally, $n = 2m \geq 8$, so $\Sp_{2m}(r)$ is a quasisimple group. If $\overline M$ does not contain $2^{2m}$, then $d(\overline M) \leq 2$ by Theorem \ref{theorem:crowns}. Otherwise, $\overline M = 2^{2m} . J$, where $J \mleq \Sp_{2m}(r) \times C$. As such, $J$ is equal to $J_1 \times C$ or $\Sp_{2m}(r) \times 1$, for some maximal subgroup $J_1$ of $\Sp_{2m}(r)$. Thus, $d(J) \leq 5$ by Theorem \ref{theorem:LMTgenerators}, and so $d(\overline M) \leq 6$ by \cite[Lemma 2.5]{BLS}.
	
	Assume next that $H$ is of type $(4 \circ 2 ^{1+2m}). \Sp_{2m}(2)$ with $n = 4$. By Proposition \ref{prop:C6prelims2},
	\[
	2^4 . \Alt(6) \leq H_{\overline{\Omega}} \leq  H_{\overline G} \leq (2^4 . \Sym(6)) \times C_2.
	\]
	Write $H_{\overline G} = 2^4.K$ where $K \leq \Sym(6) \times C_2$. If $\overline M$ does not contain $2^4$, then $d(\overline M) = d(K) \leq 2$ by Theorem \ref{theorem:crowns}. Hence, we assume that $\overline M = 2^4 . J$, where $J \mleq K$. If $J$ contains $\Alt(6)$ then $\delta_{M, 2^4}(A) \leq 1$ for any $A$, so $d(\overline M) \leq 1 + d(J) \leq 3$. Otherwise, $J \cap \Alt(6)$ is a maximal or novelty maximal subgroup of $\Alt(6)$. However, we can use Magma to show $\delta_J(C_2) \leq 4$ and $\delta_J(A) \leq 1$ for any other chief factor $A$. Additionally, for any such $J$, we have $\delta_{M, 2^4}(A) \leq 1$. Hence,
	\begin{align*}
		\delta_M(A) & = \delta_{\overline M, 2^4}(A) + \delta_J(A) \\
		& \leq  \begin{cases*}
			1 + 4 \qquad & if $A \cong C_2$\\
			1 + 1 \qquad & otherwise.
		\end{cases*}
	\end{align*}
	and so $d(\overline M) \leq 5$.
	
	Now assume that $G$ is in case $\boldL$, and $H$ is of type $2_-^{1+2}.\O_2^-(2)$ with $q=p \geq 3$. Then $H_{\overline{G}}/2_-^{1+2} = \Alt(4)$ or $\Sym(4)$ by \cite[Proposition 4.6.7]{KL}, and hence
	\[
	d(\overline M) \leq d(\overline M \cap 2_-^{1+2}) + d(\overline M/(\overline M \cap 2_-^{1+2})) \leq 3 + 2 =5.
	\]
	Assume next that $H$ is of type $2_+^{1+2m}.\O_{2m}^+(2)$. Then, by Proposition \ref{prop:C6prelims4},
	\[
	2^{2m} . \Omega^+_{2m}(2) \leq H_{\overline{G}} \leq 2^{2m} . \O^+_{2m}(2) .
	\]
	As such, using the same argument as above and Proposition \ref{prop:atmost4inI}, we obtain that $d(\overline M) \leq 5$.
\end{proof}

\subsection{Class \texorpdfstring{$\mathscr{C}_7$}{C7}}

The proof in this section closely resembles the proof for $H \in \mathscr{C}_2$. In this case, $H$ stabilises a tensor decomposition $V_1 \otimes \cdots \otimes V_t$ of $V$, where the vector spaces $V_i$ are pairwise similar. Let $\kappa_i$ be forms on each $V_i$, as defined in \cite[Section 4.7]{KL}. By Section \ref{section:C7explanation},
\[
H_{\overline \Sigma} = \overline \Delta_1^t . \diag((\overline \Sigma_1^*/\overline \Delta_1)^t).\Sym(t) \leq \overline \Sigma_1^* \wr \Sym(t),
\]
and
\[
H_{\overline \Omega} \geq \overline \Omega_1^t . \del((\overline I_1/\overline \Omega_1)^t) . \Alt(t).
\]
Additionally, if $H_{\overline \Omega}$ projects onto $\Sym(t)$, then
\[
\overline \Omega_1^t . \del((\overline \Delta_1/\overline \Omega_1)^t).\Sym(t) \leq H_{\overline \Omega} \leq \overline \Delta_1 \wr \Sym(t),
\]
and so $H_{\overline \Omega}$ satisfies the conditions of Corollary \ref{cor:CFsinwreathproduct}. We note that this is always the case if $t \geq 4$. If $t = 3$ then $H_{\overline \Omega}$ contains $\Alt(3)$ and $\overline \Omega_1^t.\del((\overline \Delta_1/\overline \Omega_1)^t)$. Hence, when $t \geq 3$, we have
\[
H_{\overline G} = N_1^t . \del((\overline \Delta_1/ N_1)^t) . \diag(( P_1/\overline  \Delta_1)^t) . J,
\]
for some groups $N_1, P_1, J$ with 
\begin{align*}
	\overline{\Omega}_1 \leq N_1 & \leq \overline{\Delta}_1,\\
	\overline \Delta_1 \leq P_1 & \leq \overline \Sigma_1^*,\\
	\Alt(t) \leq J \, \,\,& \leq \Sym(t),
\end{align*}
and $J = \Sym(t)$ if $t \geq 4$. If instead $t=2$, then, by Goursat's Lemma,
\[
H_{\overline G} = N_1^2 . \diag((P_1/N_1)^2).J,
\]
for some groups $N_1, P_1, J$ with
\begin{align*}
	\overline{\Omega}_1 \leq N_1 & \leq \overline{\Delta}_1,\\
	N_1 \leq P_1 & \leq \overline \Sigma_1^*,\\
	J & \leq \Sym(2).
\end{align*}

We begin by assuming that $ H_{\overline G}$ projects onto $\Sym(2)$ when $t = 2$. In the following, we divide the maximal subgroups of $H_{\overline{G}}$ into cases, depending on which section of the following normal series they do not contain:
\[
1 < (\overline \Omega_1')^t < N_1^t < H_{\overline{G}} \cap \overline \Delta_1^t < H_{\overline{G}} \cap P_1^t < H_{\overline{G}}.
\]
For more details regarding this method, see Sections \ref{section:prelimsmethod} and \ref{section:classicalsC2}.

We first prove the following preliminary result regarding the chief factors of $\overline H / N_1^t$.

\begin{prop}\label{prop:boundC7top}
	Suppose that $\overline H$ projects onto $\Sym(2)$ when $t=2$. Then $d(\overline H/N_1^t) \leq 4$, with equality only if $t = 2$. Additionally, if $t > 2$, then
	\[
	\delta_{\overline H/N_1^t}(A) \leq \begin{cases*}
		3 \qquad & if $A \cong C_2^2$ and $t=4$, or $A \cong C_2$,\\
		2 & otherwise.
	\end{cases*}
	\]
	If instead $t=2$, then
	\[
	\delta_{\overline H/N_1^t}(A) \leq \begin{cases*}
		4 \qquad & if $A \cong C_2$,\\
		2 & otherwise.
	\end{cases*}
	\]
\end{prop}

\begin{proof}
	If $t > 3$ then $H_{\overline \Omega}$ satisfies the conditions of Corollary \ref{cor:CFsinwreathproduct}, and hence so does $H_{\overline G}$. Thus the result follows as in Proposition \ref{prop:boundtopC2}. If $t=2$ then $\overline H/N_1^t$ has shape $(P_1/N_1).C_2$. Hence, the result follows, since $\delta_{P_1/N_1}(A) \leq 3$ for all $A$, with equality only if $A \cong C_2$, by Lemma \ref{lemma:maxsubgpsSigmai}.
	
	Hence, we assume that $t=3$ and that $H_{\overline \Omega}$ does not satisfy the conditions of Corollary \ref{cor:CFsinwreathproduct}. Thus $\Omega_1$ is of type $\boldO^\pm$, and $m \equiv 2 \pmod 4$. Additionally, $H_{\overline G}$ contains $\Alt(3)$. Note that $\overline \Sigma_1^*/\overline \Delta_1$ is isomorphic to $C_f$, and that $\overline \Delta_1 / \overline \Omega_1$ is isomorphic to $D_8$ or $C_2^2$. The non-Frattini chief factors of $H_{\overline G}$ in $\del((\overline \Delta_1/\overline \Omega_1)^t)$ are isomorphic to $C_2^2$, and there are 2 such chief factors. Hence, 
	\[
	\delta_{\overline H/N_1^t}(A) \leq \delta_{\overline H, \del((\overline \Delta_1/N_1)^2)}(A) + \delta_{P_1/\overline \Delta_1}(A) + \delta_J(A) \leq \begin{cases*}
		2 + 0 + 0 \qquad & if $A \cong C_2^2$,\\
		0 + 1 + 1 & otherwise.\\
	\end{cases*}
	\]
\end{proof}

\begin{prop}\label{prop:C7Omegatcase}
	Suppose that $\overline H$ projects onto $\Sym(2)$ when $t=2$. Let $\overline M$ be a maximal subgroup of $\overline H$ which does not contain $(\overline \Omega_1')^t$. Then $d(\overline M) \leq 7$.
\end{prop}

\begin{proof}
	By the definition of $H$ in \cite[Section 4.7]{KL}, $\overline \Omega_1'$ is a non-abelian simple group. This implies by \cite[Proposition 2.9.2]{KL} that $\overline \Omega_1$ is either a non-abelian simple group, or is equal to $\Sp_4(2) \cong \Sym(6)$. Note that, in the latter case, $\overline \Sigma_1 = \overline \Omega_1$. Hence, $\overline \Sigma_1$ is an almost simple group with socle $\overline \Omega_1'$, and so is $N_1$.
	
	If $M \cap (\overline \Omega_1')^t$ is subdirect, it is equal to $\diag((\overline \Omega_1')^t)$, by Lemma \ref{lemma:subdirectsubgpXk}. If $M \cap (\overline \Omega_1')^t$ is not subdirect in $(\overline \Omega_1')^t$, and is normal in $P_1^t$, then it is equal to the trivial group. In either case, by Theorem \ref{theorem:crowns},
	\[
	d(\overline M) \leq d \left( \frac{\overline M}{\overline M \cap (\overline \Omega_1')^t} \right) = d(\overline H/(\overline \Omega_1')^t) \leq d(\overline H) \leq 5,
	\]
	where the final inequality is by Theorem \ref{theorem:LMTgenerators}. Hence we assume that $\overline M \cap (\overline \Omega_1')^t$ is not subdirect in $(\overline \Omega_1')^t$ and is not normal in $P_1^t$. Thus, by Proposition \ref{prop:nonsubdirectwreathcase}, $\overline M$ is conjugate to $( R \wr \Sym(t)) \cap \overline H$ for some $R \mleq P_1$. As such,
	\[
	d(\overline M) \leq d( R \cap N_1) + d(\overline H/N_1^t).
	\]
	The group $R \cap N_1$ is a maximal or novelty maximal subgroup of $N_1$, by Corollary \ref{cor:maximalsubgpextension}. By Theorem \ref{theorem:LMTgenerators} and Proposition \ref{prop:boundC7top}, we have $d(R \cap N_1) \leq 4$ and $d(\overline H/N_1^t) \leq 4$. Thus, for a contradiction, we assume that $d(R \cap N_1) = 4$ and $d(\overline H/N_1^t) = 4$. This implies by Proposition \ref{prop:boundC7top} that $t=2$ and $\delta_{\overline H / N_1^t}(C_2) = 4$. In the following, we use the notation of Section \ref{section:LMTcorrections}, in particular of Table \ref{table:LMTtable}. Since $N_1 \leq \overline \Delta_1$, both $(F)_0$ and $(\gamma)_0$ hold for $R \cap N_1$, so $d(R \cap N_1) = 4$ implies that $\Omega_1$ is in case $\boldO^\pm$. Additionally, since $\delta_{\overline H / N_1^t}(C_2) = 4$, $(D)_2$ does not hold for $R \cap N_1$. Hence, one of the following holds:
	\begin{enumerate}[\upshape(1)]
		\item $R \in \mathscr{C}_4$, $q$ is odd, $R$ is of type $\O_{n_1}^{\epsilon_1} \otimes \O_{n_2}^{\epsilon_2}$ with $n_i$ both even; or
		\item $R \in \mathscr{C}_7$, $q$ is odd, $m \equiv 2 \pmod 4$, $D(V_i) = \Box$, and $N_1 \in \{\overline \Omega, \langle \overline \Omega, \overline \delta \rangle, \langle \overline \Omega, \overline \rho^2 \overline \delta \rangle\}$.
	\end{enumerate}
	By \cite[Tables 3.5.E, 3.5.G]{KL}, if $R \in \mathscr C_4$ is of type $\O_{n_1}^{\epsilon_1} \otimes \O_{n_2}^{\epsilon_2}$ with $n_i$ both even, then $R$ is not maximal in $P_1$, since $P_1$ contains $\overline I_1$. Thus, $R$ is not in case (1).
	
	Let $R$ be in case (2). Note that $N_1$ is a normal subgroup of $P_1$, and $P_1$ contains $\langle N_1, \overline I_1 \rangle$. Thus, $N_1$ is not equal to $\langle \overline \Omega, \overline \delta \rangle$ or $\langle \overline \Omega, \overline \rho^2 \overline \delta \rangle$, so $N_1 = \overline \Omega$. Hence, there exists a classical geometry $(V_*, \mathbb F, \kappa_*)$ of type $\boldO$ such that $R \cap N_1 = \overline I_*^2$, where $\overline \Omega_*$ is a simple group. Since $P_1 \geq \overline I_1$, we have 
	\[
	R \geq  \left(\overline I_*^2 . \del((\overline \Delta_*/\overline I_*)^2) \right) : \Sym(2).
	\]
	Hence, by Lemma \ref{lemma:jumpingbabyresultfull}, there are no abelian, non-Frattini chief factors of $\overline M$ in $(\overline S_*^2)^t$. Thus, by Theorem \ref{theorem:crowns} and Proposition \ref{prop:boundC7top}, we have
	\begin{align*}
		d(\overline M) & \leq d\left(\overline M \Big/ \left(\overline S_*^2 \right)^t \right) \\
		& \leq d\left((R \cap N_1)\big/\overline S_*^2\right) + d(\overline H / N_1^t)\\
		& \leq d((\overline I_*/\overline S_*)^2) + 4 = 6. \qedhere
	\end{align*}
\end{proof}

\begin{prop}\label{prop:C7casecontainsOmega1}
	Suppose that $\overline H$ projects onto $\Sym(2)$ when $t=2$. Suppose that $\overline M \mleq \overline H$ contains $(\overline \Omega_1')^t$ and does not contain $N_1^t$. Then $d(\overline M) \leq 7$.
\end{prop}

\begin{proof}
	Let $K$ be the largest subgroup of $N_1$ such that $K^t$ is contained in $M$. If $N_1/\overline \Omega_1'$ is cyclic, then, by the same argument as in Proposition \ref{prop:C2subdirectcase}, we have that $N_1/K \cong C_r$ for some prime $r$. If $N_1/\overline \Omega_1' \cong D_8$, then $\overline M/(\overline \Omega_1')^t$ contains $\Phi(D_8)^t$, so $N_1/K$ is either $C_2$ or $C_2^2$. Finally, if $N_1/\overline \Omega_1' \cong C_2^2$ then it is clear that $N_1/K$ is isomorphic to $C_2$ or $C_2^2$. If $N_1/K \cong C_r$ for some prime $r$, then there exists some $J$-invariant group $Y < C_r^t$, possibly trivial, such that
	\[
	\overline M = K^t . Y . (\overline H /N_1^t).
	\]
	Note that $d_J(Y) \leq 1$. Indeed, this is clear by Corollary \ref{cor:CFsinwreathproduct} when $t \geq 4$. When $t = 2$, $Y$ is cyclic, so $d(Y) = 1$. When $t=3$ we can use Goursat's Lemma to determine that there exists some isomorphism $\phi$ of $C_r$ such that $Y$ is equal to $\{(a, a^\phi, a^{\phi^2} : a \in C_r\}$ or $\langle (a, a^\phi, 1), (1, a, a^\phi) : a \in C_r \rangle$. In either case, $d_{\Alt(3)}(Y) = 1$. Thus, by Theorem \ref{theorem:crowns}, Lemma \ref{lemma:maxsubgpsSigmai} and Proposition \ref{prop:boundC7top},
	\[
	d(\overline M) \leq d(K/\Omega_1') + 1 + d(\overline H / N_1^t) \leq 2+1+4 = 7.
	\]
	Otherwise, if $N_1/K \cong C_2^2$, there exists some $J$-invariant group $Y < C_2^t$ such that
	\[
	\overline M = K^t . Y . C_2^t . (\overline H /N_1^t).
	\]
	As previously, $d_J(Y) = 1$. We have $N_1/K \cong C_2^2$, so $d(K/\Omega_1') \leq 1$ by \cite[Chapter 2]{KL}, and $d(\overline H / N_1^t) \leq 3$ by Proposition \ref{prop:boundC7top}. Thus,
	\[
	d(\overline M) \leq d(K/\Omega_1') + 1 + d(C_2) + d(\overline H/N_1^t) \leq 6. \qedhere
	\]
\end{proof}

When $\overline M$ contains $N_1^t$ and does not contain $\overline H \cap P_1^t$, we briefly split into cases depending on whether $t > 2$ or $t=2$.

\begin{prop}
	Let $t > 2$. Suppose that $\overline M \mleq \overline H$ contains $N_1^t$, and does not contain $\overline H \cap \overline \Delta_1^t$. Then $d(\overline M) \leq 7$.
\end{prop}

\begin{proof}
	As previously, there exists a group $K < \overline \Delta_1/N_1$ such that $\overline M / N_1^t$ contains $\del(K^t)$, and either $\overline \Delta_1/K \cong C_r$ for some prime $r$, or $\overline \Delta_1/K \cong C_2^2$. Hence we may use the same proof as in Proposition \ref{prop:C7casecontainsOmega1} to prove that $d(\overline M) \leq 6$.
\end{proof}

\begin{prop}
	Let $t > 2$. Suppose that $\overline M \mleq \overline H$ contains $\overline H \cap \overline \Delta_1^t$, and does not contain $\overline H \cap P_1^t$. Then $d(\overline M) \leq 7$.
\end{prop}

\begin{proof}
	There exists some $K < P_1/\overline \Delta_1$ such that
	\[
	\overline M = N_1^t . \del((\overline \Delta_1/N_1)^t) . K . J.
	\]
	Note that $P_1/\overline \Delta_1$ is either cyclic, or a direct product of two cyclic groups. As such, by Theorem \ref{theorem:crowns} and Lemma \ref{lemma:maxsubgpsSigmai},
	\[
	d(\overline M) \leq d(N_1/\overline \Omega_1') + d(\overline \Delta_1/N_1) + d(K . J) \leq 2 + 2 + 3 = 7.\qedhere
	\]
\end{proof}

\begin{prop}
	Suppose that $t=2$ and that $\overline H$ projects onto $\Sym(2)$. Additionally, assume that $\overline M \mleq \overline H$ contains $N_1^t$ and does not contain $\overline H \cap P_1^t$. Then $d(\overline M) \leq 7$.
\end{prop}

\begin{proof}
	Since $t=2$, we have $\overline H = N_1^2 . (P_1/N_1).C_2$. Hence, $\overline M = N_1^2 . K . C_2$ for some subgroup $K$ of $P_1/N_1$. Note that $d(N_1/\overline \Omega_1') \leq 2$ and $d(K) \leq 3$ by Lemma \ref{lemma:maxsubgpsSigmai} and \cite[Chapter 2]{KL}, so $d(\overline M) \leq 6$.
\end{proof}

\begin{prop}
	Suppose that $\overline H$ projects onto $\Sym(2)$ when $t = 2$. Let $\overline M \mleq \overline H$ contain $\overline H \cap P_1^t$. Then $d(\overline M) \leq 7$.
\end{prop}

\begin{proof}
	Firstly, when $t=2$, $\overline M$ is equal to $N_1^2.(P_1/N_1)$. Hence, by Theorem \ref{theorem:crowns} and Lemma \ref{lemma:maxsubgpsSigmai},
	\[
	d(\overline M) \leq d(\overline M/(\overline \Omega_1')^t) \leq 2d(N_1/\overline \Omega_1') + d(P_1/N_1) \leq 4 + 3 = 7.
	\]
	
	Next, assume that $t=3$ and $J = \Alt(3)$, so $\overline \Omega_1$ is of type $\boldO^\pm$. Additionally, by \cite[Propositions 4.7.6, 4.7.7]{KL}, we have that $D(Q_1) = \boxtimes$, implying that $\overline \Delta_1/\overline \Omega_1 \cong C_2^2$. Since $\overline M = N_1^3 . \del((\overline \Delta_1/N_1)^3).(P_1/\overline \Delta_1)$, by Theorem \ref{theorem:crowns},
	\[
	d(\overline M) \leq 3 d(N_1/\overline \Omega_1) + 2 d(\overline \Delta_1/N_1) + d(P_1/\overline \Delta_1).
	\]
	However, since $\overline \Delta_1/\overline \Omega_1 \cong C_2^2$, we have that $d(N_1/\overline \Omega_1) \leq 2$, $d(\overline \Delta_1/N_1) \leq 2$, and $d(N_1/\overline \Omega_1) + d(\overline \Delta_1/N_1) \leq 2$. So $3 d(N_1/\overline \Omega_1) + 2 d(\overline \Delta_1/N_1) \leq 6$. Since $P_1/\overline \Delta_1$ is cyclic, we obtain $d(\overline M) \leq 7$ as required.
	
	For the remainder of this proof, we assume that $t > 2$ and that $J = \Sym(t)$. There exists some $J_1 \mleq \Sym(t)$ such that
	\[
	\overline M = N_1^t . \del((\overline\Delta_1/N_1)^t) . \diag((P_1/\overline\Delta_1)^t) . J_1.
	\]
	Suppose that $J_1$ is intransitive, so $J_1 = \Sym(m) \times \Sym(t-m)$ for some $m < t/2$. Note that $\delta_{J_1}(A) \leq 2$ for any chief factor $A$, with equality only if $A$ is isomorphic to $C_2$ or $C_3$. Hence, by Theorem \ref{theorem:crowns}, $d((P_1/\overline \Delta_1).J_1) \leq 4$, with equality only if $\Omega_1$ is of type $\boldL$. By Lemma \ref{lemma:delgens},
	\[
	d(\overline M/N_1^t) \leq d(\overline \Delta_1/N_1) + d((P_1 / \overline \Delta_1).J_1).
	\]
	If $\overline \Omega_1$ is in case $\boldO$, with $d(N_1/\overline\Omega_1) = 2$ and $d(\overline\Delta_1/N_1) = 1$, then $\overline S_1^t/\overline \Omega_1^t$ is Frattini in $\overline M$ by Lemma \ref{lemma:jumpingbabyresultfull}. Hence, by Theorem \ref{theorem:crowns},
	\[
	d(\overline M) \leq 2d(N_1/\overline S_1) + d(\overline\Delta_1/N_1) + d((P_1/\overline\Delta_1).J_1) \leq 2+1+3 = 6.
	\]
	In all other cases, by \cite[Chapter 2]{KL}, we have that either $d(N_1/\overline \Omega_1') \leq 1$ and $d(\overline \Delta_1/N_1) \leq 2$, or $d(N_1/\overline \Omega_1') = 2$ and $d(\overline \Delta_1/N_1) = 0$, so
	\[
	d(\overline M) \leq 2d(N_1/\overline \Omega_1') + d(\overline\Delta_1/N_1) + d((P_1/\overline\Delta_1).J_1) \leq 7.
	\]
	Thus, we assume that $J_1$ is transitive for the remainder of this proof. By Lemma \ref{lemma:delgens} and Theorem \ref{theorem:crowns}, we have
	\[
	d(\overline M) \leq d(N_1/\overline\Omega_1') + d(\overline \Delta_1/N_1) + d((P_1/N_1).J_1).
	\]
	We bound each of the terms above using \cite[Chapter 2]{KL}. If $\overline \Omega_1$ is in case $\boldL$, we have the bounds
	\[
	d(N_1/\overline\Omega_1') + d(\overline \Delta_1/N_1) \leq 2, \qquad d((P_1/\overline \Delta_1).J_1) \leq 6.
	\]
	When $\overline \Omega_1$ is in case $\boldO^\pm$, we have
	\[
	d(N_1/\overline\Omega_1') + d(\overline \Delta_1/N_1) \leq 3, \qquad d((P_1/\overline \Delta_1).J_1) \leq 5.
	\]
	Finally, if $\overline \Omega_1$ is in case $\boldU$, $\boldS$ or $\boldO^\circ$, we have
	\[
	d(N_1/\overline\Omega_1') + d(\overline \Delta_1/N_1) \leq 2, \qquad d((P_1/\overline \Delta_1).J_1) \leq 5.
	\]
	In this last case, the desired result is proved. Hence, we only need to consider $\overline \Omega_1$ in cases $\boldL$ and $\boldO^\pm$, specifically when both upper bounds above are tight. Note, in these cases $\overline \Omega_1 = \overline \Omega_1'$.
	
	In the $\boldL$ case, the proof that $d(\overline M) \leq 7$ follows in the exact same way as for $H \in \mathscr{C}_2$, since $I_1 = \Delta_1$ in this case.
	
	Hence, we assume that $\overline \Omega_1$ is of type $\boldO^\pm$. Additionally, we can assume that $\overline \Delta_1 / \overline \Omega_1 \cong D_8$ and $\overline \Omega_1 < N_1 < D_8$, as otherwise $d(N_1/\overline\Omega_1') + d(\overline \Delta_1/N_1) \leq 2$. Since $d((P_1/N_1).J_1) = 5$, we have that $\delta_{J_1}(C_2) = 4$ by Theorem \ref{theorem:LMTgenerators}, and $J_1$ is of simple diagonal type. Let the socle of $J_1$ be $T^k$ for some non-abelian simple group $T$.
	
	Firstly, if $|N_1 / \overline \Omega_1| = 4$, then $X := [\overline \Delta_1/\overline\Omega_1, N_1/\overline\Omega_1]$ is isomorphic to $C_2$. Hence, by Lemma \ref{lemma:jumpingbabyresultfull}, the corresponding section $X^t$ in $\overline M$ is Frattini, and thus, 
	\[
	d(\overline M) \leq d((N_1/\overline\Omega_1)/X) + d(\overline \Delta_1/N_1) + d((P_1/\overline \Delta_1).J_1) \leq 1 + 1 + 5 = 7.
	\]
	So we can assume instead that $N_1/\overline\Omega_1 \cong C_2$, and $\overline \Delta_1/N_1 \cong C_2^2$. In particular, since there is a unique normal subgroup of $D_8$ of order 2, we have that $N_1/\overline \Omega_1 = (\overline \Delta_1/\overline \Omega_1)'$.
	
	By Lemma \ref{lemma:delmodulemults} and Corollary \ref{cor:upperbounddeltaGN} applied with $N = T^k \unlhd J_1$, there are no chief factors isomorphic to $C_2$ in $\del((\overline \Delta_1/N_1)^t)$. By Lemma \ref{lemma:permmodulemults}, there is at most one chief factor of $\overline M$ in $(N_1/\overline \Omega_1)^t$ isomorphic to $C_2$, so
	\[
	\delta_{\overline M}(C_2) \leq 1 + \delta_{(P_1/\overline \Delta_1).J_1}(C_2) \leq 6.
	\]
	We additionally note that $\del((N_1/\overline \Omega_1)^t)$ is Frattini in $\overline M$ by Lemma \ref{lemma:jumpingbabyresultdel}. Hence there are no other non-Frattini chief factors of $\overline M$ in $(N_1/\overline \Omega_1)^t$, other than the chief factor isomorphic to $C_2$ mentioned previously.
	
	Suppose now that $A \cong C_2^b$ with $b > 1$. Then, $\delta_{\overline M, N_1^t}(A) = 0$ and $\delta_{(P_1/\overline \Delta_1).J_1}(A) \leq 1$. Additionally, by Lemmas \ref{lemma:permmodulemults} and \ref{lemma:delgens}, $\delta_{\overline M, (\overline M \cap \overline \Delta_1^t)/N_1^t}(A) \leq 2\dim_E(A)$. Hence, by Theorem \ref{theorem:crowns},
	\[
	d((L_A)_{\delta_{\overline M}(A)}) \leq \theta(A) + 2 + \Bigg\lceil \frac{1 + s(A)}{\dim_E(A)} \Bigg\rceil \leq 4.
	\]
	For any other chief factor $A$, we have $\delta_{\overline M}(A) = \delta_{(P_1/\overline \Delta_1).J_1}(A) \leq 3$. Thus, $d(\overline M) \leq 6$ by Theorem \ref{theorem:crowns}.
\end{proof}

\begin{prop}
	Suppose that $t=2$ and that $\overline H$ does not project onto $\Sym(2)$. Then $d(\overline M) \leq 7$.
\end{prop}

\begin{proof}
	We have
	\[
	(\overline S_1 \times \overline S_2). \del((\overline I_1/\overline S_1)^2) \leq \overline H \leq  (\overline \Delta_1 \times \overline \Delta_2).(P_1/\overline \Delta_1) \leq \overline \Sigma_1 \times \overline \Sigma_2.
	\]
	As such, $\overline H$ has the same shape as in the $\mathscr{C}_4$ case, except $\overline \Delta_1 \cong \overline \Delta_2$. Hence, a similar proof can be used to show that $d(\overline M) \leq 7$.
\end{proof}

\subsection{Class \texorpdfstring{$\mathscr{C}_8$}{C8}}

\begin{prop}
	Suppose that $H \in \mathscr{C}_8$. Then $d(\overline M) \leq 6$.
\end{prop}

\begin{proof}
	By Section \ref{section:C8explanation},
	\[
	H_{\overline \Gamma} = \overline \Gamma_\sharp.
	\]
	Hence, if $H_{\overline G} \leq H_{\overline \Gamma}$, then $d(\overline M) \leq 6$ by Theorem \ref{theorem:LMTgenerators} and Lemma \ref{lemma:non_simple_gen_bound}. As such, we assume that $H_{\overline G} \nleq H_{ \overline \Gamma}$, and so $G$ is in case $\boldL$, with $H_{\overline \Sigma} = H_{\overline \Gamma} . C_2$. In particular, $H_{\overline G}$ is equal to $H_{\overline G \cap \overline\Gamma} . C_2$. 
	
	Suppose first that $\overline \Omega_\sharp$ is a simple group. If $\overline M$ contains $\overline \Omega_\sharp$, then $\overline M = \overline \Omega_\sharp . J$ where $J$ is a subgroup of $(\overline \Gamma_\sharp/\overline \Omega_\sharp).C_2$. Thus, by Theorem \ref{theorem:crowns} and \cite[Proposition 3.12]{LMT}, we have that $d(\overline M) \leq 1 + d(J) \leq 5$. So we assume that $\overline M$ does not contain $\overline \Omega_\sharp$.
	
	Hence, $\overline M \cap \overline \Gamma$ is either trivial, or is a maximal or novelty maximal subgroup of $H_{\overline G \cap \overline \Gamma}$, by Corollary \ref{cor:maximalsubgpextension}. Thus, $d(\overline M) \leq d(\overline M \cap \overline \Gamma) + 1 \leq 6$, by Theorem \ref{theorem:LMTgenerators}.
	
	So we suppose that $\overline \Omega_\sharp$ is not a simple group. If $n = 2$ and $q \leq 9$, or if $n \leq 4$ and $q \leq 4$, we check computationally that $d(\overline M) \leq 4$. Otherwise, by \cite[Table 3.5.A]{KL}, $n=2$ and $\overline \Omega_\sharp = \overline \Omega_2^\pm(q)$, which is cyclic. Hence, $\overline M$ is equal to $C.(\overline H/\overline \Omega_\sharp)$ for some cyclic group $C$, or is equal to $\overline \Omega_\sharp.J$ for some $J \leq (\overline \Gamma_\sharp/\overline \Omega_\sharp).C_2$. In the first case we have $d(\overline M) \leq 1 + d(\overline H) \leq 4$ by Theorem \ref{theorem:LMTgenerators}. In the second case we have that $d(\overline M) \leq 2 + d(J) \leq 6$, by \cite[Proposition 3.12]{LMT} and Lemma \ref{lemma:non_simple_gen_bound}.
\end{proof}

\subsection{Maximal subgroups with exceptional automorphisms}

In the preceding section we have shown that $d(\overline M) \leq 7$ when $\overline M$ is a second maximal subgroup of a classical group $\overline G$ with $\overline \Omega \leq \overline G \leq \overline \Sigma$. Hence, the only almost simple groups with classical socle left to consider are those with $\overline G \nleq \overline \Sigma$. As such, we may assume that $\overline G_0$ is one of $\PSp_4(q)$ or $\POmega_8^+(q)$ and so $\Aut(\overline G_0) = \langle \overline \Sigma, \gamma \rangle$ for some graph automorphism $\gamma$ of order 2 or 3 respectively. A maximal subgroup $\overline H$ of $\overline G$ falls into one of two families, those for which $\overline H \cap \overline \Sigma$ is maximal in $\overline G \cap \overline \Sigma$, and those for which $\overline H \cap \overline \Sigma$ is a novelty maximal subgroup of $\overline G \cap \overline \Sigma$.

The novelty maximal subgroups not contained in \cite[Chapter 3]{KL} are described in \cite[Table 2]{BLS}, and we record them in Table \ref{table:excpsnovelty}. The non-novelty maximal subgroups $\overline M \cap \overline \Sigma$ which can appear are also somewhat restricted, since they are normalised by $\Aut(\overline G_0) / \overline \Sigma$. We list the possibilities in Table \ref{table:excpsnonnovelty}, excluding any of type $\mathscr{S}$ since we have $d(\overline M) \leq 5$ for such an $H$. For the $\PSp_4(q)$ novelty cases, it is proved in \cite[Lemma 5.11]{BLS} that $d(\overline M) \leq 7$.

\begin{table}
	\centering
    \caption{The maximal subgroups $\overline H$ of $\overline G$ for which $\overline H \cap \overline \Sigma$ is a novelty maximal subgroup of $\overline G \cap \overline \Sigma$.}
	\begin{tabular}{lll}
		$\overline{G}_0$ & Type of $\overline H$ & Conditions\\ \toprule
		$\PSp_4(q)$ & $O_2^\epsilon(q) \wr \Sym(2)$ & $q > 2$ even \\
		& $O_2^-(q^2).2$ & $q > 2$ even\\
		$\POmega_8^+(q)$ & $\GL_3^\epsilon(q) \times \GL_1^\epsilon(q)$ & \\
		& $\O_2^-(q^2) \times \O_2^-(q^2)$ & \\
		& $[2^9] . \SL_3(2)$ & $q=p>2$ \\ \bottomrule
	\end{tabular}
	\label{table:excpsnovelty}
\end{table}

\begin{table}
	\centering
    \caption{The maximal subgroups $\overline H$ of $\overline G$ for which $\overline H \cap \overline \Sigma$ is a maximal subgroup of $\overline G \cap \overline \Sigma$.}
	\begin{tabular}{lll}
		$\overline{G}_0$ & Type of $\overline H$ & Conditions\\ \toprule
		$\PSp_4(q)$ & $\Sp_4(q_0)$ & $q = q_0^r$ \\
		$\POmega_8^+(q)$ & $\Omega^+_2 (q)^4.(2d)^3.\Sym(4)$ & \\
		& $\Omega^-_2 (q)^4.(2d)^3.\Sym(4)$ & \\ 
		& $\Omega_4^+ (q)^2.[2d].\Sym(2)$ & \\
		& $\Omega ^+_8 ( q_0 )$ & $q = q_0^r$ \\
		& $\SO^+_8 (q_0).2$ & $q = q_0^2$ \\ \bottomrule
	\end{tabular}
	\label{table:excpsnonnovelty}
\end{table}

We remark that the maximal subgroups $\overline H$ for which $\overline H \cap \overline \Sigma$ is a maximal subgroup of $\overline G \cap \overline \Sigma$ were omitted from the analysis in \cite{BLS}, and so the result below which bounds $d(\overline M)$ for these cases is new.

We begin with a useful lemma regarding the number of generators of $\Out(\overline \Omega)$ in the case that $\Omega = \Omega_8^+(q)$.

\begin{lemma}\label{lemma:triality}
	Let $G_0 = \Omega_8^+(q)$, and consider a group $\overline G$ with $\overline G_0 \leq \overline G \leq \Aut(\overline G_0)$. Assume that $\overline G$ contains a triality automorphism of $\overline G_0$. Then $d(\overline G/ \overline G_0) \leq 2$.
\end{lemma}

\begin{proof}
	Note, $\Out(\overline G_0)$ is isomorphic to $\Sym(l) \times C_f$, with $l = 3$ when $q$ is even, and $l = 4$ when $q$ is odd. Consider a subgroup $K$ of $\Sym(l) \times C_f$ such that $K/(K \cap C_f)$ contains an element of $\Sym(l)$ of order 3. This implies that $\delta_{K/(K \cap C_f)}(A) \leq 1$, and hence $\delta_K(A) \leq 2$, with equality only if $A$ is central. Thus, the result is proven by Theorem \ref{theorem:crowns}.
\end{proof}

Let $H_{\overline X}$ be defined as previously for any subgroup $X$ with $\Omega \leq X \leq \Sigma$, and let $H_X$ be the preimage of this group in $X$. Let $\pi$ denote the projection map $H_\Sigma \to H_{\overline \Sigma}$. Additionally, for a group $L \leq \overline \Sigma$, we let $L^\wedge$ be the full preimage of $L$ in $ \Sigma$.

\begin{prop}
	Let $\overline G_0 = \POmega_8^+(q)$. Suppose that $\overline H$ is a maximal subgroup of $\overline G \nleq \overline \Sigma$, and that $\overline H \cap \overline \Sigma$ is a novelty maximal subgroup of $\overline G \cap \overline \Sigma$. Then for any maximal subgroup $\overline M$ of $\overline H$, we have $d(\overline M) \leq 7$.
\end{prop}

\begin{proof}
	If $\overline M$ contains $H_{\overline \Omega}$, then $d(\overline M) \leq 5$ by Theorem \ref{theorem:LMTgenerators}, so we assume in the remainder of the proof that this is not the case.
	
	Suppose that $\overline H$ is of type $\GL_3^\epsilon(q) \times \GL_1^\epsilon(q)$, for $\epsilon \in \{\pm 1\}$. By \cite[Table 1]{K}, if we define $d := (q-1, 2)$, then
	\[
	H_\Omega = \left( C_{\frac{q - \epsilon}{d}} \times \frac{1}{d} \GL_3^\epsilon(q) \right) . C_2^d.
	\]
	In fact, by \cite[Section 3.2]{K}, we have that $H_\Gamma$ is the intersection of two maximal subgroups of $\Gamma$. The first is a maximal subgroup $R$ of type $\mathscr{C}_1$ which is the stabiliser of an $\epsilon 2$-space. The second is a maximal subgroup $K$ which is either of type $\mathscr{C}_2$ when $\epsilon = +$, or of type $\mathscr{C}_3$ if $\epsilon = -$. 
	
	For any $X \in \{\Omega, S, I, \Delta, \Gamma, \Sigma, \Sigma^*\}$, we let $X_i = X(V_i)$ where $V_1$ is a $\epsilon 2$-space and $V_2 = V_1^\perp$.
	
	When $\epsilon = +$, if we define $y, \phi$ as in the proof of \cite[Proposition 4.2.7]{KL}, then $H_\Gamma = R_\Gamma \cap K_\Gamma$ implies that
	\[
	H_\Gamma = \langle I_1 \times I_2, y, \phi \rangle \leq \Sigma_1 \times \Sigma_2 .
	\]
	In particular, $H_\Gamma = (I_1 \times I_2) . (C_2 \times C_f)$. On the other hand, if $\epsilon = -$, then we have
	\[
	H_\Gamma = \langle I_1 \times I_2, \delta, \phi_\sharp \rangle \leq \Sigma_1 \times \Sigma_2 
	\]
	where $\phi_\sharp$ is as defined in \cite[(4.3.8)]{KL}, and $\delta$ as in \cite[Section 2.7]{KL}. In particular, $H_\Gamma = (I_1 \times I_2) . C_{q-1} . C_{2f}$. Finally, $\overline H = H_{\overline G \cap \overline \Gamma } . C_3$.
	
	Now consider $\overline M \mleq \overline H$. Note that $\pi(\Omega_2)$ is a normal subgroup of $\overline H$, since it is equal to the perfect core of $H_{\overline \Omega}$, and hence it is a characteristic subgroup of $H_{\overline \Omega}$. If $\overline M$ contains $ \pi(\Omega_2)$, then
	\[
	\overline M = \pi(\Omega_2). J . C_3
	\]
	where $J$ is a subgroup of $(\pi(I_1 \times I_2)/\pi(\Omega_2)) . \overline \Sigma_2/ \overline I_2$. Since $I_1$ and $I_2/\Omega_2$ are cyclic, and any subgroup of $\overline \Sigma_2/\overline I_2$ is 2-generator, we have that $d(J) \leq 4$. Thus, since the only non-Frattini chief factor of $\overline M$ in $\pi(\Omega_2)$ is $\overline \Omega_2$, which is non-abelian, we obtain that $d(\overline M) \leq d(J) + d(C_3) \leq 5$ by Theorem \ref{theorem:crowns}.
	
	So we may assume that $\overline M$ does not contain $ \pi(\Omega_2) \cong \Omega_2$. By Proposition \ref{prop:maxsubgpsextensioncenter1}, the group $\overline M \cap \pi(\Omega_2)$ is equal to a maximal or novelty maximal subgroup of $\Omega_2$, which we denote by $J$. By Theorem \ref{theorem:LMTgenerators}, since $(F)_0$ and $(\gamma)_0$ hold we have that $d(J) \leq 4$. Thus, by \cite[Proposition 7.1]{LMT},
	\[
	d(\overline M) \leq d(J) + d(\overline H/\pi(\Omega_2)) \leq 4 + 3 = 7.
	\]
	
	Suppose next that $H$ is of type $\O_2^-(q^2) \times \O_2^-(q^2)$. Then we have that
	\[
	H_{\overline \Omega} = (D_{2l} \times D_{2l}).[2^2]
	\]
	where $l = (q^2+1)/2$, and as such is always odd. Thus, for any subgroup $K \leq D_{2l}^2$, $K$ has 2 non-Frattini chief factors isomorphic to $C_r$ for each $r \mid 2l$, and $K$ has no other non-Frattini chief factors. Thus, by Corollary \ref{cor:upperbounddeltaGN}, we have
	\begin{align*}
		\delta_{\overline M}(A) & \leq \delta_{\overline M, \overline M \cap \overline \Omega}(A) + \delta_{\overline H/H_{\overline \Omega}}(A)\\
		& \leq \begin{cases*}
			4 + 2 \qquad & if $A \cong C_2$,\\
			2 + 2 & otherwise.
		\end{cases*}
	\end{align*}
	Thus, we have shown that $d(\overline M) \leq 6$.
	
	We finally assume that $H$ is of type $[2^9].\SL_3(2)$. Then, $H_{\overline \Omega}$ embeds in $\frac{C_2 \wr \Sym(8)}{\diag(C_2^8)}$, with 
	\[
	H_{\overline{\Omega}} \cap \frac{C_2^8}{\diag(C_2^8)} = \frac{\del(C_2^8)}{\diag(C_2^8)} \cong C_2^6
	\]
	Thus, $C_2^3. \L_3(2) \leq \Sym(8)$. By the proof of \cite[Proposition 3.4.2]{K}, the action of $\L_3(2)$ on $C_2^3$ is the natural action. We can use Magma to confirm that, up to conjugation, the only group of this shape in $\Sym(8)$ is the affine group $C_2^3 : \L_3(2)$. As such, $H_{\overline{\Omega}} = C_2^3 . C_2^3 . C_2^3 . \L_3(2)$, with $\L_3(2)$ acting irreducibly on each $C_2^3$.
	
	Note, $[2^9] = O_2(H_{\overline \Omega})$ is a normal subgroup of $\overline H$. Let a normal series of $H_{\overline\Omega}$ through $O_2(H_{\overline \Omega})$ be denoted by
	\[
	1 = X_0 < X_1 < X_2 < X_3 = O_2(H_{\overline \Omega}).
	\]
	By above, $X_i/X_{i-1}$ is isomorphic to $C_2^3$, and $X_2/X_0 \cong C_2^6$. The groups $X_i$ are not necessarily normal in $\overline H$. However, we can compute the chief factors of $M_{\overline \Omega}$ contained in each $X_i/X_{i-1}$, and then use Corollary \ref{cor:upperbounddeltaGN} to bound the chief factors of $\overline M$.
	
	Suppose that $\overline M$ does not contain $O_2(H_{\overline \Omega})$, and thus
	\[
	\overline M \cap \overline \Omega = (\overline M \cap O_2(H_{\overline \Omega})) . \L_3(2).
	\]
	Since $\overline M \cap \overline \Omega$ projects onto $\L_3(2)$, any chief factor of $\overline M \cap \overline \Omega$ in $\overline M \cap O_2(H_{\overline \Omega})$ is isomorphic to $C_2^3$, and there are at most 2 such chief factors. Thus, 
	\begin{align*}
		\delta_{\overline M}(A) & = \delta_{\overline M, \overline M \cap O_2(H_{\overline \Omega})}(A) + \delta_{\L_3(2)}(A) + \delta_{\overline H/H_{\overline \Omega}}(A) \\
		& \leq \begin{cases*}
			2 + 0 + 0 \qquad & if $A \cong C_2^{3m}$ for some $m$,\\
			0 + 1 + 0  & if $A \cong \L_3(2)$,\\
			0 + 0 + 2 & otherwise.
		\end{cases*}
	\end{align*}
	This implies that $d(\overline M) \leq 3$ by Theorem \ref{theorem:crowns}.
	
	Next we assume that $\overline{M}$ contains $O_2(H_{\overline \Omega})$. Thus, by Corollary \ref{cor:maximalsubgpextension}, $\overline M \cap \overline \Omega$ is equal to $[2^9].J$, where $J$ is some maximal or novelty maximal subgroup of $\L_3(2)$. Since $\L_3(2)$ acts naturally on the three copies of $C_2^3$, we have $\delta_{\overline M \cap \overline \Omega, X_i/X_{i-1}}(A) \leq 1$ for any $i$ and any chief factor $A$. Additionally, $\delta_{\overline M \cap \overline \Omega, X_i/X_{i-1}}(A) = 0 $ for any chief factor $A$ not isomorphic to $C_2^a$ with $1 \leq a \leq 3$. Using Magma, we can determine that $\delta_{J}(A) \leq 1$ for any chief factor $A$. Thus, by Corollary \ref{cor:upperbounddeltaGN},
	\begin{align*}
		\delta_{\overline M}(A) & = \delta_{\overline M, [2^9]}(A) + \delta_{\overline M, J}(A) + \delta_{\overline G / \overline G_0}(A)\\
		& \leq \begin{cases*}
			3 + 1 + 2 \qquad & if $A \cong C_2$,\\
			3 + 1 + 1 \qquad & if $A \cong C_2^c$, for $c > 1$,\\
			0 + 1 + 2 \qquad & otherwise.
		\end{cases*}
	\end{align*}
	Hence, $d(\overline M) \leq 6$.
\end{proof}

We now treat the non-novelty maximal subgroups.

\begin{prop}
	Suppose that $\overline H$ is a maximal subgroup of $\overline G \nleq \overline \Sigma$, and that $\overline H \cap \overline \Sigma$ is a maximal subgroup of $\overline G \cap \overline \Sigma$. Then for any maximal subgroup $\overline M$ of $\overline H$, we have $d(\overline M) \leq 7$.
\end{prop}

\begin{proof}
	If $\overline M$ contains $H_{\overline \Omega}$, then $d(\overline M) \leq 5$, by Theorem \ref{theorem:LMTgenerators}. Thus, we assume that $\overline M$ does not contain $H_{\overline \Omega}$ throughout. 
	
	Assume that $\overline G_0 = \PSp_4(q)$ with $q > 2$ even. Then, $\overline H$ is of type $\Sp_4(q_0)$ for some $q_0$ such that $q_0^r = q$, where $r$ is a prime. Let $\mathbb{F}_\sharp := \mathbb{F}_{q_0}$ which is a subfield of $\mathbb{F}_q$. By Section \ref{section:C5explanation}, we have
	\[
	H_{\overline \Delta} = \overline \Delta_\sharp,
	\]
	and $H_{\overline \Omega}$ contains $\overline \Omega_\sharp$. Let $\overline H = N.J$ where $N = \overline H \cap \overline \Delta_\sharp$ and $J = \overline H/N$. Note that $N$ is an almost simple group, since $q_0=2$ implies that $N = \Sym(6)$. Any subgroup of $\Aut(\overline G_0)/\overline \Delta$ is 2-generated, by \cite[Section 2.4]{KL}. Hence, if $\overline M$ contains $N$ then $d(\overline M) \leq d(N/\overline \Omega_\sharp') + 2 \leq 3$ by \cite[Section 2.4]{KL} and Theorem \ref{theorem:crowns}. On the other hand, if $\overline M$ does not contain $N$, then $\overline M \cap N$ is either: trivial, a subgroup containing $\overline \Omega_\sharp$, a maximal subgroup of $N$, or a novelty maximal subgroup of $N$, by Corollary \ref{cor:maximalsubgpextension}. Thus, by Theorem \ref{theorem:LMTgenerators} and \cite[Proposition 3.12]{LMT},
	\[
	d(\overline M) \leq d(\overline M \cap N) + 2 \leq 5.
	\]
	
	Assume for the remainder of this proof that $\overline G_0 = \POmega_8^+(q)$. Additionally, we can treat $q = 2, 3$ computationally to show $d(\overline M) \leq 6$ in this case, and so we may assume that $q > 3$.
	
	Suppose first that $\overline H$ is in $\mathscr{C}_2$ and is of type $I_1 \wr \Sym(t)$, so $\Omega_1$ is either equal to $\Omega_2^\pm(q)$ or $\Omega_4^+(q)$. We treat this case similarly to the $\mathscr{C}_2$ case in Section \ref{section:classicalsC2}. However, this case requires more care, since the normal series given in this section is no longer guaranteed to be a normal series of $\overline H$. Let $N_1$ be the largest group such that $N_1^t \leq H_{(\overline G \cap \overline \Gamma)^\wedge}$, and $P_1$ the projection of $H_{(\overline G \cap \overline \Gamma)^\wedge} \cap \Sigma_1^t$ to the first coordinate. Hence,
	\[
	H_{(\overline G \cap \overline \Gamma)^\wedge} = N_1^t . \del((I_1/N_1)^t). \diag((P_1/I_1)^t) . \Sym(t).
	\]
	
	Assume that $\Omega_1 = \Omega_4^+(q)$ with $t=2$, and recall that $q > 3$. This implies that $ \pi(\Omega_1^t) = \pi((\SL_2(q) \circ \SL_2(q))^{t})$, is equal to the perfect core of $H_{\overline \Omega}$, and hence is characteristic. Additionally, $H_{\overline \Omega}$ contains $\pi(\Omega_1^t.\del((I_1/\Omega)^t))$, so $\overline \Omega_1^t$ is a chief factor of $\overline H$.
	
	Suppose that $\overline M$ does not contain $\pi(\Omega_1^t)$. Then, since $t=2$, we have
	\[
	d\left( \frac{\overline M \cap \overline \Gamma}{ \overline M \cap \pi(\Omega_1^t)} \right) \leq 5,
	\]
	by the same method as in Proposition \ref{prop:boundtopC2}. Additionally, by Corollary \ref{cor:upperbounddeltaGN}, there is at most one non-Frattini chief factor of $\overline M$ in $\frac{\overline M \cap \overline \Gamma}{ \pi(M \cap \Omega_1^t)}$ isomorphic to $C_3$, so $d(\overline M / \pi(M \cap \Omega_1^t)) \leq 5$ by Theorem \ref{theorem:crowns}.
	
	By Proposition \ref{prop:CFsFromNormalSubgp}, $(\overline M \cap \pi(\Omega_1^t))/\pi(Z(\Omega_1^t))$ is isomorphic to: $\PSL_2(q)^b$ for some $0 \leq b \leq 2$; or to $\overline R^{2t}$ where $\overline R$ is a maximal, or novelty maximal subgroup of $\PSL_2(q)$. If $(\overline M \cap \pi(\Omega_1^t))/\pi(Z(\Omega_1^t))$ is isomorphic to $\PSL_2(q)^b$, then $d(\overline M \cap \pi(\Omega_1^t)) \leq 2$ by Theorem \ref{theorem:crowns}, so $d(\overline M) \leq 7$. 
	
	Hence, we assume that $\overline M \cap \pi(\Omega_1^t) = \pi((R \circ R)^t)$ for some group $R$ which is either a maximal or novelty maximal subgroup of $\SL_2(q)$. Note that $d(R) \leq 2$ by \cite[Tables 8.1, 8.2]{BHRD}. Since $\overline M \cap \overline \Gamma$ projects onto $\Sym(t)$ and onto $\pi(\del(I_1^t))/\pi(\Omega_1^t)$, it acts transitively on the $2t$ copies of $\SL_2(q)$ in $\pi(\Omega_1^t)$. Hence,
	\[
	d(\overline M) \leq d(R) + d\left( \frac{\overline M}{ \overline M \cap \pi(\Omega_1^t)} \right) \leq 7,
	\]
	by above. On the other hand, if $\overline M$ contains $\pi(\Omega_1^t)$, then
	\[
	(\overline M \cap \overline \Omega)/\pi(\Omega_1^t) < [2d] . \Sym(2).
	\]
	Hence, by Corollary \ref{cor:upperbounddeltaGN}, 
	\[
	\delta_{\overline M, (\overline M \cap \overline \Omega)/\pi(\Omega_1^t)}(A) \leq \begin{cases*}
		3 \qquad & if $A \cong C_2$,\\
		2 & if $A \cong C_2^2, C_2^3$,\\
		0 & otherwise.
	\end{cases*}
	\]
	Additionally, $\delta_{\overline M,\pi(\Omega_1^t)}(A) \leq 4$ if $A \cong \PSL_2(q)^b$ for some $b \leq 4$, and is 0 for any other chief factor $A$. Finally, $\delta_{\overline M/(\overline M \cap \overline \Omega)}(A) \leq 2$ for any chief factor $A$, by Lemma \ref{lemma:triality}, since $\overline M/(\overline M \cap \overline \Omega) \cong \overline G/\overline \Omega$. Thus,
	\[
	\delta_{\overline M}(A) \leq \begin{cases*}
		5 \qquad & if $A \cong C_2$,\\
		4 \qquad & otherwise,
	\end{cases*}
	\]
	implying that $d(\overline M) \leq 5$ by Theorem \ref{theorem:crowns}.
	
	We next assume that $\Omega_1 = \Omega_2^\pm(q) = C_{\frac{q \mp 1}{(q-1,2)}}$ and that $t=4$. We begin by classifying the possible non-Frattini chief factors of $\overline H$ contained in $H_{ \overline \Omega}$. Note that, by Proposition \ref{prop:CFsFromNormalSubgp}, the non-Frattini chief factors of $\overline H$ are equal to products of non-Frattini chief factors of $H_{\overline G \cap \overline \Gamma}$ in $H_{\overline \Omega}$. Additionally, $\overline H$ permutes these chief factors of $H_{\overline G \cap \overline \Gamma}$ in $H_{\overline \Omega}$. Thus, a chief factor of $\overline H$ in $H_{\overline \Omega}$ is not equal to a chief factor of $H_{\overline G \cap \overline \Gamma}$ in $H_{\overline \Omega}$ only if $H_{\overline \Omega}$ contains 3 isomorphic non-Frattini chief factors of $H_{\overline G \cap \overline \Gamma}$.
	
	If $\frac{q \mp 1}{(q-1,2)}$ has prime factorisation equal to
	\[
	\frac{q \mp 1}{(q-1,2)} = p_1 ^{i_1} \cdots p_j^{i_j},
	\]
	let $l = p_1^{i_1-1} \cdots p_j^{i_j-1}$ and $m = p_1 \cdots p_j$, so that $\Phi(\Omega_1^t) = C_l^t$ and $\Omega_1^t/\Phi(\Omega_1^t) \cong C_m^t$. Since $H_{(\overline G \cap \overline \Gamma)^\wedge}$ satisfies the conditions of Corollary \ref{cor:CFsinwreathproduct}, for any $a$ and any $r \neq 2,3$ the group $H_{(\overline G \cap \overline \Gamma)^\wedge}$ has at most 1 non-Frattini chief factor in $H_{\Omega}$ which is isomorphic to $C_r^a$. This implies that $H_{\overline G \cap \overline \Gamma}$ has at most 1 non-Frattini chief factor in $H_{\overline \Omega}$ isomorphic to $C_r^a$. Thus, by Proposition \ref{prop:CFsFromNormalSubgp}, any non-Frattini chief factor of $\overline H$ in $H_{\overline \Omega}$ isomorphic to $C_r^a$ is a chief factor of $H_{\overline G \cap \overline \Gamma}$. Similarly, when $r=3$, $H_{\Omega}$ contains at most two non-Frattini chief factors of $H_{(\overline G \cap \overline \Gamma)^\wedge}$ isomorphic to $C_3^a$ for any $a$. Thus, the chief factors of $\overline H$ in $H_{\overline \Omega}$ which are isomorphic to $C_3^a$ are chief factors of $H_{\overline G \cap \overline \Gamma}$. Finally, the non-Frattini chief factors of $H_{(\overline G \cap \overline \Gamma)^\wedge}$ in $H_{\Omega}$ isomorphic to $C_2^a$ are either isomorphic to $C_2$, of which there are 2, or isomorphic to $C_2^2$, of which there are 3. Hence, a non-Frattini chief factor of $\overline H$ in $H_{\overline \Omega}$ isomorphic to $C_2^a$ is not a chief factor of $H_{\overline \Omega}$ only if $a = 6$, and this chief factor is unique.
	
	Since $\overline M$ contains $\pi(\Phi(\Omega_1^t)) = \pi(C_l^t)$, we classify $(\overline M \cap \overline \Omega)/\pi(\Phi(\Omega_1^t))$. Note that $H_{\Omega}/\Phi(\Omega_1^t)$ is isomorphic to
	\[
	C_m^t . \del((I_1/\Omega_1)^t) . \Sym(t),
	\]
	where $m = p_1 \cdots p_j$, and so is square-free. Without loss of generality, let $p_1, \dots, p_s$ be the primes in $p_1, \dots, p_j$ which are not equal to 2 or 3, for some $j-2 \leq s \leq j$. In $H_{\overline \Omega}$, the following are characteristic subgroups, and thus are normal subgroups of $\overline H$:
	\[
	\pi(C_{lp_1}^t) < \pi(C_{lp_1 p_2}^t) < \cdots < \pi(C_{lp_1 \cdots p_s}^t).
	\]
	Assume that $\overline M$ contains $\pi(C_{lp_1 \cdots p_{k-1}}^t)$ and does not contain $\pi(C_{lp_1 \cdots p_k}^t)$ for some $2 \leq k \leq i$. Hence, by Corollary \ref{cor:CFsinwreathproduct}, there exists some subgroup $X$ of $C_{p_k}^t$ equal to $\del((C_{p_k})^t)$, $\diag((C_{p_k})^t)$ or 1 such that
	\[
	\overline M = \pi(C_{lp_1 \cdots p_{k-1}}^t . X) . (\overline H/\pi(C_{lp_1 \cdots p_k}^t)).
	\]
	By Theorem \ref{theorem:LMTgenerators}, $d(\overline H/\pi(C_{lp_k}^t)) \leq 5$, and so $d(\overline M) \leq 7$, since $\overline M \cap \overline \Omega$ projects onto $\Sym(4)$.
	
	Thus, we assume that $\overline M$ contains $L := \pi(C_{lp_1 \cdots p_s}^t)$. The group $H_{\Omega}/C_{lp_1 \cdots p_s}^t$ is isomorphic to
	\[
	C_b^t . \del((I_1/\Omega_1)^t) . \Sym(4) \leq D_{4b} \wr \Sym(4),
	\]
	where $b$ is one of $2$, $3$, or $6$. Given the information in \cite[Section 4.2]{KL}, we can construct all possibilities for this group in Magma. For each possible $H_{\Omega}/L$, we determine its subgroups of index $3, 3^2, 2, 2^2$ and $2^6$, and prove that each of these groups is $4$-generated. Hence, $d((\overline M \cap \overline \Omega)/L) \leq 4$. Thus, if $\overline M \cap \overline \Omega$ projects onto a transitive subgroup of $\Sym(4)$, we have 
	\[
	d(\overline M) \leq d(C_{lp_1 \cdots p_s}) + 4 + d(\overline G / \overline G_0) \leq 7,
	\]
	by Lemma \ref{lemma:triality}. On the other hand, if $\overline M \cap \overline \Omega$ does not project onto a transitive subgroup of $\Sym(4)$, then it projects onto $\Sym(3)$, by our analysis above. Additionally, $\overline M$ contains $C_{2lp_1 \cdots p_s}^t$, and $(\overline M \cap \overline \Omega)/\pi(C_{2l}^t)$ is 3-generated. Thus,
	\[
	d(\overline M) \leq 2d(C_{2lp_1 \cdots p_s}) + d((\overline M \cap \overline \Omega)/\pi(C_{2l}^t)) + d(\overline G / \overline G_0) \leq 2 + 3 + 2 = 7.
	\]
	
	The last two cases with $\overline H$ of type $\Omega_8^+(q_0)$ and $\SO_8^+(q_0).2$ can be treated in the same way as $\Sp_4(q_0)$ above to obtain $d(\overline M) \leq 7$.
\end{proof}

\section{Proof of Theorem \ref{theorem:bigresult}: alternating groups}\label{section:alternating}

Let $G$ be an almost simple group with socle $G_0$, and let $M \mleq H \mleq G$. In this section, we assume that $G_0 := \soc(G)$ is $\Alt(n)$ for some $n \geq 5$. We show that $d(M) \leq 7$, and that this bound is tight (see Proposition \ref{prop:altdiagonaltype}). The case where $n \leq 8$ can be checked computationally to confirm that $d(M) \leq 3$. Hence, we assume that $n \geq 9$ and thus $\Alt(n) \leq G \leq \Sym(n)$.

Throughout, we use the O'Nan-Scott Theorem, which classifies the maximal subgroups $H$ of $G$, and states the following.

\begin{theorem}
	If $H$ is a maximal subgroup of $G \in \{\Alt(n), \Sym(n)\}$, then one of the following holds:
	\begin{enumerate}[\upshape(1)]
		\item $H$ is intransitive, and $H = (\Sym(m) \times \Sym(n-m)) \cap G$ for some $m < n/2$;
		\item $H$ is an affine group $\AGL_d(p) \cap G$, with $n = p^d$;
		\item $H = (\Sym(k) \wr \Sym(t)) \cap G$, for some $k \geq 2, t \geq 2$ such that $n = kt$ or $n = k^t$;
		\item $H = (T^k . (\Out(T) \times \Sym(k))) \cap G$, for some simple group $T$, and $k \geq 2$ such that $n = |T|^{k-1}$; or
		\item $H$ is an almost simple group.
	\end{enumerate}
\end{theorem}

If $H$ is an almost simple group, then $M \mleq H$ implies that $d(M) \leq 5$ by \cite[Theorem 1]{LMT}. We now treat the remaining cases for $H$.

\subsection{Intransitive subgroups} \label{section:altcase1}

\begin{prop}
	Let $\Alt(n) \leq G \leq \Sym(n)$ for some $n \geq 9$, and let $H := (\Sym(m) \times \Sym(n-m)) \cap G$ be an intransitive maximal subgroup of $G$. Then any maximal subgroup $M$ of $H$ has $d(M) \leq 5$.
\end{prop}

\begin{proof}
	Note that, by definition, $m \neq n/2$. Since $n \geq 9$, one of $\Alt(m)$ or $\Alt(n-m)$ is simple. Additionally, $H$ is a subdirect subgroup of $\Sym(m) \times \Sym(n-m)$. Let $N_1 := G \cap \Sym(m)$ and $N_2 := G \cap \Sym(n-m)$. By Proposition \ref{prop:maxsubgpdirectproduct}, either $M$ contains $N_i$ for some $i$, or $M$ is subdirect in $\Sym(m) \times \Sym(n-m)$. 

	If $M$ contains some $N_i$, then we can assume without loss of generality that $M \geq N_1$. Hence, 
	\[
	M = (\Sym(m) \times J) \cap G
	\]
	for some $J \mleq \Sym(n-m)$. As such, 
	\[
	\delta_M(A) = \delta_{M, N_1}(A) + \delta_{J}(A).
	\]
	When $n-m \geq 5$, \cite[Theorem 1]{LMT} states that
	\[
	\delta_J(A) \leq \begin{cases*}
		4 \qquad & if $A \cong C_2$,\\
		3 & otherwise.
	\end{cases*}
	\]
	If $n-m < 5$, we can check computationally that the same bound holds. To bound $\delta_{M, N_1}(A)$, note that $N_2$ is either $\Alt(m)$ or $\Sym(m)$. For any $m$, we have that $\delta_{N_1}(A) \leq 1$ for any chief factor $A$. Hence, $\delta_{M, N_1}(A) \leq 1$ for any chief factor $A$, by Corollary \ref{cor:upperbounddeltaGN}. Thus, 
	\[
	\delta_M(A) \leq \begin{cases*}
		5 \qquad & if $A \cong C_2$,\\
		4 \qquad & otherwise,
	\end{cases*}
	\]
	implying that $d(M) \leq 5$, by Theorem \ref{theorem:crowns}.

	Assume next that $M$ is subdirect, and let $P_1 = \Sym(m)$ and $P_2 = \Sym(n-m)$. By Proposition \ref{prop:maxsubgpdirectproduct}, $M = (M_1 \times M_2).C$ for some maximal $P_i$-normal subgroups $M_i$ of $N_i$ which satisfy $P_i/M_i \cong C$. If $G = \Sym(n)$, then $N_1 = \Sym(m)$ and $N_2 = \Sym(n-m)$. Hence, $M_1 = \Alt(m)$ and $M_2 = \Alt(n-m)$. Thus, by Corollary \ref{cor:upperbounddeltaGN} and the fact that $\delta_{M_i}(A) \leq 1$ for any chief factor $A$, we have that
	\[
	\delta_M(A) \leq \delta_{M, N_1}(A) + \delta_{P_2}(A) \leq 2,
	\]
	for any chief factor $A$. Therefore, $d(M) \leq 3$ by Theorem \ref{theorem:crowns}.

	If $G = \Alt(n)$, then $N_1 = \Alt(m), N_2 = \Alt(n-m)$. However, by above, $N_i$ is simple for some $i$, so we assume without loss of generality that $N_2$ is simple. Hence, $M_2 = 1$, and so $\Sym(m)/M_1 \cong \Sym(n-m)$. This is a contradiction, as $m \neq n-m$. As such, there cannot exist such a maximal subgroup $M$ when $G = \Alt(n)$.
\end{proof}

\subsection{Affine subgroups} \label{section:altcase2}

\begin{prop}
	Let $\Alt(n) \leq G \leq \Sym(n)$ for some $n \geq 9$, and let $H := \AGL_d(p) \cap G$ be a maximal subgroup of $G$ of affine type. Then any maximal subgroup $M$ of $H$ has $d(M) \leq 6$.
\end{prop}

\begin{proof}
	The group $H$ is of the form $V : L$, where $V = \mathbb{F}_p^d$ and $L = \GL(V) \cap G$. Let $M$ be a maximal subgroup of $H$. Let $Z := Z(\GL(V))$.

	To begin with, if $M$ contains $V$ then $M = V : J$ for some $J <_{\max} L$. If $\SL(V)$ is quasisimple then, by Aschbacher's Classification, this implies that $J$ acts either irreducibly on $V$, or $V \downarrow J = V_1 | V_2$ for some vector spaces $V_i$, on which $J$ acts irreducibly. Thus,
	\[
	\delta_M(A) \leq \begin{cases*}
		1 + \delta_J(A) \qquad & if $V$ is irreducible and $A \cong V$,\\
		2 + \delta_J(A) \qquad & if $V$ is reducible, and $A \cong V_i \leq V$,\\
		\delta_J(A) & otherwise.
	\end{cases*}
	\]
	Since $J$ is a maximal subgroup of $L$, by Proposition \ref{prop:atmost4inI}, we have that
	\[
	\delta_J(A) \leq \begin{cases*}
		4 \qquad & if $A$ is central,\\
		3 \qquad & otherwise.
	\end{cases*}
	\]
	Therefore, 
	\[
	\delta_M(A) \leq \begin{cases*}
		6 \qquad & if $A$ is central,\\
		5 \qquad & otherwise.
	\end{cases*}
	\]
	Hence $d(M) \leq 6$ by Theorem \ref{theorem:crowns}.

	If instead $\SL(V)$ is not quasisimple, then we have the following possibilities: $d=1$, or $(d,p) = (2,2), (2,3)$. If $(d,p)$ is one of $(2,2)$ or $(2,3)$ then $n \leq 8$, so we do not consider these cases. On the other hand, if $d=1$ then $V$ and $\GL(V)$ are cyclic groups, so $d(M) \leq 2$.

	So we may assume that $V \not\leq M$. Thus, since $V$ is an abelian minimal normal subgroup of $H$, this implies that $M \cong L$. Hence, $d(M) = d(L) \leq 2$. Indeed, when $\SL(V)$ is quasisimple this is clear, since $\SL(V) \leq L \leq \GL(V)$. On the other hand, if $\SL(V)$ is not quasisimple, then we again have that $d=1$, since $n \geq 8$, and hence $d(L) = 1$.
\end{proof}

\subsection{Imprimitive and wreath-type subgroups} \label{section:altwreath}

In this section, we let
\[
H = (\Sym(k) \wr \Sym(t)) \cap G,
\]
for some $k, t$ such that $n = kt$ or $n = k^t$, according as $H$ is of wreath or product-action type. Since $H \unlhd \Sym(k) \wr \Sym(t)$, we have $\Alt(k)^t \leq H$, and hence we use Proposition \ref{prop:XSimpleMaxSubgrpXkC} in the case that $k \geq 5$, as well as Proposition \ref{prop:XNonsimpleMaxSubgpXkC} in the case that $k \leq 4$.

To be more precise, if $H$ is strictly contained in $\Sym(k) \wr \Sym(t)$, then $H$ is equal to one of 
\[
\Alt(k)^t . (\del(C_2^t) . \Sym(t)), \text{ or } \,\,\,  \Sym(k)^t . \Alt(t).
\]
Note that, in the first case, we have
\[
H \cap \Sym(k)^t = \{ (g_1, \dots, g_t) \in \Sym(k)^t : \prod_i \sgn(g_i) = 1 \}.
\]
Let $J$ be the projection of $H$ to $\Sym(t)$, so $\Alt(k) \leq J \leq \Sym(k)$.

\begin{prop}
	Let $\Alt(n) \leq G \leq \Sym(n)$ for some $n \geq 9$, and let $H := (\Sym(k) \wr \Sym(t)) \cap G$. Then any maximal subgroup $M$ of $H$ has $d(M) \leq 7$.
\end{prop}

We prove this proposition using the method described in Section \ref{section:prelimsmethod}. In order to do so, we consider the following normal series of $H$:
\[
1 < \Alt(k)^t < \Sym(k)^t \cap H < H,
\]
and consider whether $M$ contains each normal subgroup of this series.

\begin{prop}
	Suppose that $J$ is transitive, and $M$ does not contain $\Alt(k)^t$. Then $d(M) \leq 7$.
\end{prop}

\begin{proof}
	Assume first that $k \geq 5$, so that $\Alt(k)$ is a simple group. Since $J$ is a primitive subgroup of $\Sym(t)$, Proposition \ref{prop:XSimpleMaxSubgrpXkC} states that $M \cap \Alt(k)^t$ is conjugate to one of:
	\begin{enumerate}[\upshape(1)]
		\item $M_1^t$ where $M_1$ is a maximal, novelty maximal or trivial subgroup of $\Alt(k)$; or
		\item $\diag(\Alt(k)^t)$.
	\end{enumerate}
	Next we consider the chief factors of $H/\Alt(k)^t$. This group is either equal to $\del(C_2^t).\Sym(t)$, $C_2^t.\Alt(t)$, or $C_2 \wr \Sym(t)$. In any case we have $\delta_{H/\Alt(k)^t}(A) \leq 2$, with equality only if $A \cong C_2$, or $A \cong C_2^2$ if $t=4$, by a similar argument to that of Corollary \ref{cor:CFsinwreathproduct}. Hence, if $M$ is of type (1) above, by \cite[Theorem 1]{LMT}, we have 
	\[
	d(M) \leq d(M_1) + d(H/\Alt(k)^t) \leq 4+3 = 7.
	\]
	If $M$ is of type (2), then
	\[
	\delta_M(A) \leq \begin{cases*}
		2 \qquad & if $A \cong C_2, C_2^2$,\\
		1 \qquad & otherwise,
	\end{cases*}
	\]
	so $d(M) \leq 3$.

	Next, we consider the case where $k \leq 4$. Let $N_1$ be the largest subgroup of $\Sym(k)$ such that $N_1^t \leq H$. Since we satisfy the conditions of Proposition \ref{prop:subdirectwreathcase}, and $M$ does not contain $\Alt(k)^t$, we have that either:
	\begin{enumerate}[\upshape(1)]
		\item $M \cap \Alt(k)^t = M_1^t$ for some $M_1 < \Alt(k)$; or
		\item $M\cap N_1^t$ is subdirect, and contains $K^t$ where $K$ is the intersection of all of the maximal normal subgroups of $\Sym(k)$ in $N_1$.
	\end{enumerate}
	The first case follows exactly as above to show that $d(M) \leq 5$.

	Assume that $M$ is in the second case. If $N_1 = \Sym(k)$ then $K = \Alt(k)$, which is a contradiction. We hence assume that $N_1 = \Alt(k)$. Then there is a unique maximal $\Sym(k)$-invariant subgroup of $\Alt(k)$, which is trivial if $k=2,3$, or equal to $C_2^2$ if $k=4$. If $k=2$, then $M$ contains $\Alt(k)^t = 1$, so we have a contradiction. 

	If $k=3$, then $\Alt(k) \cong C_3$, and $M \cap C_3^t$ is subdirect. If $t > 2$ then, since $M$ contains a subgroup which projects onto $\del((\Sym(k)/\Alt(k))^t)$, we have that $M$ contains $C_3^t$, which is a contradiction. On the other hand, if $t=2$, we have $n$ equal to 6 or 9, in which case we can check computationally that $d(M) \leq 4$.

	If $k=4$, then $K \cong C_2^2$ and $\Sym(4)/K = \Sym(3)$. Hence, by the same reasoning, when $t > 2$ the quotient $M/K^t$ contains $\Alt(3)^t$, implying that $M$ contains $\Alt(k)^t$, which is a contradiction. If $t=2$ then $n$ is one of $8$ or $16$, and we can check computationally that $d(M) \leq 4$.
\end{proof}

\begin{prop}
	Suppose that $J$ is transitive. Assume that $M$ contains $\Alt(k)^t$ and does not contain $\Sym(k)^t \cap H$. Then $d(M) \leq 5$.
\end{prop}

\begin{proof}
	If $t \geq 5$ then $H$ satisfies the conditions of Corollary \ref{cor:CFsinwreathproduct}, and hence we can determine that the proper, non-trivial $J$-invariant subgroups of $(\Sym(k)/\Alt(k))^t$ are $\diag(C_2^t)$ and $\del(C_2^t)$. On the other hand, if $t \leq 4$, we can determine computationally that this result still holds. This implies that
	\[
	M = \Alt(k)^t . K . J
	\]
	where $K$ is equal to $1$, $\diag(C_2^t)$ or $\del(C_2^t)$. Hence, since $d_J(K) = 1$, we have
	\[
	d(M) \leq d(\Alt(k)) + 1 + d(J) \leq 5. \qedhere
	\]
\end{proof}

\begin{prop}
	Suppose that $J$ is transitive and that $M$ contains $H \cap \Sym(k)^t$. Then $d(M) \leq 5$.
\end{prop}

\begin{proof}
	Let $M = (H \cap \Sym(k)^t) . J_1$ for some maximal subgroup $J_1$ of $J$. If $t \leq 3$ then $d(J_1) \leq 1$ and $d(M \cap \Sym(k)^t) \leq 3$. Indeed, this latter bound is clear by Theorem \ref{theorem:crowns} in the case that $k \geq 5$, and can be proven computationally otherwise. So we obtain that $d(M) \leq 4$.

	Hence, we assume that $t \geq 4$. We begin by bounding $d(M/\Alt(k)^t)$. The group $J_1$ has at most two orbits on $\{1, \dots, t\}$, implying that $d_{J_1}((M \cap \Sym(k)^t)/\Alt(k)^t) \leq 2$, with equality only if $J_1$ is imprimitive . By \cite[Theorem 1]{LMT}, we have $d(J_1) \leq 2$ if $J_1$ is imprimitive, and $d(J_1) \leq 4$ otherwise. Thus, $d(M/\Alt(k)^t) \leq 5$. 

	If $k \geq 5$ or $k=2$ then, by Theorem \ref{theorem:crowns}, we have $d(M) = d(M/\Alt(k^t)) \leq 5$. Hence, we assume that $3 \leq k \leq 4$. The chief factors of $\Sym(k)$ in $\Alt(k)$ are either $\{C_3\}$ or $\{C_3, C_2^2\}$ according as $k$ is 3 or 4, and are all non-central. Since $t \geq 4$, the chief factors of $M$ in $\Alt(k)^t$ are hence isomorphic to $C_3^m$ or $C_2^{2m}$ for some $m \leq t$ equal to the size of an orbit of $J_1$. Since $t \geq 4$, $J_1$ has at most two orbits, and their sizes are distinct. Hence, $\delta_{M, \Alt(k)^t}(A) \leq 1$. 

	If $J_1$ is intransitive, then $d(J_1) \leq 2$ and $\delta_M(A) \leq 1 + \delta_{M/\Alt(k)^t}(A)$. Hence, by Theorem \ref{theorem:crowns}, we have that $d(M) \leq 1 + d(M/\Alt(k)^t) \leq 5$. If $J_1$ is transitive, then the chief factors of $M$ in $\Alt(k)^t$ are isomorphic to $C_3^t$ or $C_2^{2t}$. By \cite[Theorem 2.7]{LMT}, $\delta_{M/\Alt(k)^t}(A) \leq 2$ for any such chief factor. So we obtain that $d(M) \leq \max(3, d(M/\Alt(k)^t)) \leq 5$, by Theorem \ref{theorem:crowns}.
\end{proof}

\begin{prop}
	Suppose that $J$ is intransitive. Then $d(M) \leq 5$.
\end{prop}

\begin{proof}
	Since $J \geq \Alt(t)$, this implies $t=2$ and $H = \Sym(k)^2$. Hence, the same proof as in Section \ref{section:altcase1} proves the desired result. Indeed, if $k < 5$ then we can prove computationally that $d(M) \leq 4$. Otherwise, if $k \geq 5$, then the proof in Section \ref{section:altcase1} implies that $M$ is either $\Alt(k).C_2$, $\Alt(k)^2 . C_2$, or $J \times \Sym(k)$ for some $J \mleq \Sym(k)$. In each case we can obtain $d(M) \leq 5$ as previously.
\end{proof}

\subsection{Diagonal type subgroups} \label{subsection:case4}

In this section, we consider $H$ of diagonal type, which implies
\[
H = (T^k. (\Out(T) \times \Sym(k))) \cap G,
\]
for some non-abelian simple group $T$. More precisely,
\[
T^k. (\Out(T) \times \Sym(k)) = \langle T^k, \diag(\Aut(T)^k), \Sym(k) \rangle \leq \Aut(T) \wr \Sym(k).
\]
Since $T$ is a non-abelian simple group, if $r = [T^k. (\Out(T) \times \Sym(k)):H] \leq 2$, then
\[
H = T^k . \frac{1}{r} (\Out(T) \times \Sym(k)).
\]
Hence, we let $L$ denote $H/T^k = \frac{1}{r} (\Out(T) \times \Sym(k))$.

\begin{prop}\label{prop:altdiagonaltype}
	Let $\Alt(n) \leq G \leq \Sym(n)$ for some $n \geq 9$. If
	\[
	H = (\langle T^k, \diag(\Aut(T)^k), \Sym(k) \rangle) \cap G \mleq G,
	\]
	and $M \mleq H$, we have that $d(M) \leq 7$. Additionally, there exists such a second maximal subgroup $M$ of this type such that $d(M) = 7$.
\end{prop}

\begin{proof}
	We assume first that $k = 2$, and $L = \Out(T)$. Then, $H = T^2 . \Out(T) \leq \Aut(T)^2$. If $M < \Aut(T)^2$ is not subdirect then, by Proposition \ref{prop:maxsubgpdirectproduct}, $M$ is equal to $(M_1 \times \Aut(T)) \cap H$ for some $M_1$ maximal in $\Aut(T)$. Thus, by Theorem \ref{theorem:LMTcrowns} and \cite[Theorem 1]{LMT}, we have
	\[
	\delta_M(A) \leq \delta_{M_1}(A) + \delta_{M, T}(A) \leq \begin{cases*}
		5 + 0 \qquad & if $A \cong C_2$,\\
		3 + 0 \qquad & if $A \not\cong C_2$ is abelian,\\
		2 + 1 & otherwise.
	\end{cases*}
	\]
	As such, $d(M) \leq 5$. On the other hand, if $M$ is subdirect, then $M \cong \Aut(T)$, so $d(M) \leq 3$ by \cite[Proposition 3.12]{LMT}. Hence, we may assume that the projection of $L$ to $\Sym(k)$ is equal to $\Sym(2)$ in the $k=2$ case, and contains $\Alt(k)$ otherwise. By Proposition \ref{prop:XSimpleMaxSubgrpXkC}, $M$ is one of the following:
	\begin{enumerate}[\upshape(1)]
		\item $T^k . J$ for some $J \mleq L$;
		\item $(\prod_{i=1}^k M_i) . L$ where $M_i$ are maximal, novelty maximal, or trivial subgroups of $T$, all conjugate in $\Aut(T)$; or
		\item $\diag(T^k) . L$.
	\end{enumerate}
	Note, for the final case we have used the fact that the projection of $L$ to $\Sym(k)$ is primitive. 

	Assume first that $M = T^k . J$ for some $J \mleq L$. A maximal subgroup of $\Alt(k)$ or $\Sym(k)$ has at most three orbits on $\{1, \dots, k\}$, so
	\[
	\delta_{M,T^k}(A) \leq \begin{cases*}
		3 \qquad & if $A \cong T^i$ for some $i$,\\
		0 & otherwise.
	\end{cases*}
	\]
	Let the maps $\pi_1, \pi_2$ be the projections of $L \leq \Out(T) \times \Sym(k)$ to each of its coordinates. We define $N_i := \ker(\pi_{3-i})$, and $P_i := \pi_i(L)$ for $i=1,2$. Note, since $r = [\Out(T) \times \Sym(k) : L] \leq 2$, we have two possibilities. Either $L = P_1 \times P_2$, or $P_1=\Out(T)$, $P_2=\Sym(k)$, and $L = (N_1 \times N_2).2$. Thus $L$ contains $1 \times \Alt(k)$, and contains a subgroup of index 2 of $\Out(T) \times 1$.

	Assume that $J$ is not a subdirect subgroup of $P_1 \times P_2$. Then, by Proposition \ref{prop:maxsubgpdirectproduct}, $J$ contains $N_1$ or $N_2$. If $J \geq N_1$, then
	\[
	J = (P_1 \times J_2) \cap G
	\]
	for some $J_2 \mleq P_2$. Note that $P_2$ is one of $\Alt(k)$ or $\Sym(k)$. Hence, we apply \cite[Propositions 3.12, 5.1]{LMT} and Corollary \ref{cor:upperbounddeltaGN} to obtain that
	\begin{align*}
		\delta_M(A) & = \delta_{M,T^k}(A) + \delta_{J, N_1}(A) + \delta_{J_2}(A)\\
		& \leq \begin{cases*}
			0 + 3 + 4 \qquad & if $A \cong C_2$,\\
			0 + 2 + 3 & if $A$ is otherwise central,\\
			0 + 2 + 2 & if $A$ is abelian and non-central,\\
			3 + 0 + 2 & otherwise.
		\end{cases*}
	\end{align*}
	Thus, $d(M) \leq 7$. We show that this bound is tight. Let $G = \Sym(n)$, so $L = \Out(T) \times \Sym(k)$ is a direct product. Hence,
	\[
	\delta_M(A) = \delta_{M,T^k}(A) + \delta_{\Out(T)}(A) + \delta_{J_2}(A)
	\]
	Letting $T = \PSL_4(9)$, we have $\delta_{\Out(T)}(C_2) = 3$. For $k$ large enough, there exists a maximal subgroup $J_2$ of $\Sym(k)$ such that $\delta_{J_2}(C_2) = 4$, by \cite[Theorem 1]{LMT}. Hence, there exists a maximal subgroup $M$ of $H$ such that $\delta_M(C_2) = 7$, and so $d(M) = 7$ by Theorem \ref{theorem:crowns}.

	If instead $J$ contains $N_2$, then we can similarly obtain that $J = (J_1 \times P_2) \cap G$ for some $J_1 \mleq P_1$, and so, by \cite[Proposition 3.12]{LMT} and Corollary \ref{cor:upperbounddeltaGN},
	\begin{align*}
		\delta_M(A) & \leq \delta_{M,T^k}(A) + \delta_{J_1}(A) + \delta_{J, N_2}(A)\\
		& \leq \begin{cases*}
			0 + 3 + 1 \qquad & if $A \cong C_2$,\\
			0 + 2 + 1 & if $A$ is otherwise abelian,\\
			3 + 0 + 1 & otherwise.
		\end{cases*}
	\end{align*}
	Thus, $d(M) \leq 4$.

	If $J$ is subdirect, then, by Proposition \ref{prop:maxsubgpdirectproduct}, $J = (J_1 \times J_2) . C$ for some maximal $P_i$-normal subgroups $J_i$ of $N_i$, with $P_i/J_i \cong C$ for each $i$. Thus, $\delta_M(A) \leq \delta_{M, T^k}(A) + \delta_J(A) \leq 3 + 1$ when $A$ is non-abelian. Hence, $d(M) = d(M/T^k)$ by Theorem \ref{theorem:crowns}, so
	\begin{align*}
		d(M) & \leq d(J_1) + d(P_2) \leq 5,
	\end{align*}
	by \cite[Proposition 3.12]{LMT}.

	We next assume that $M = (\prod_{i=1}^k M_i) . L$ for some subgroups $M_i$ of $T$. If $M_i=1$ for all $i$ then $d(M) = d(H/T^k) \leq 4$ by Theorem \ref{theorem:LMTgenerators}. Hence we assume that each $M_i$ is either a maximal or novelty maximal subgroup of $T$. We have
	\[
	d(M) \leq d(M_1) + d(L).
	\]
	By \cite{BLS2013}, $d(M_1) \leq 4$. On the other hand, since $\delta_{N_2}(A) \leq 1$ for any chief factor $A$, by Corollary \ref{cor:upperbounddeltaGN}, we have
	\begin{align*}
		\delta_L(A) & = \delta_{P_1}(A) + \delta_{L, N_2}(A)\\
		& \leq \begin{cases*}
			3 + 1 \qquad & if $A \cong C_2$,\\
			1 + 1 & otherwise.
		\end{cases*}
	\end{align*}
	Thus, $d(L) \leq 4$. If $d(L) \leq 3$ or $d(M_1) \leq 3$ we obtain $d(M) \leq 7$ by above, so we assume that $d(L) = d(M_1) = 4$. Hence, we have $\delta_{L, N_2}(C_2) = 1$, which implies that $N_2 = P_2 = \Sym(k)$. Thus, by Goursat's Lemma, $N_1 = P_1$ also. In particular, the conditions of Corollary \ref{cor:CFsinwreathproduct} hold. Let $A_1$ be the preimage of $P_1$ in $\Aut(T)$. The intersection of $H \leq A_1 \wr \Sym(k)$ with $A_1^k$ is subdirect. Additionally, $M \cap T^k$ is not normal in $\Aut(T)^k$. Thus, by Proposition \ref{prop:nonsubdirectwreathcase}, we can assume that $M$ is equal to $(\widetilde M_1 \wr \Sym(k)) \cap H$ for some maximal subgroup $\widetilde M_1$ of $A_1$ with $\widetilde M_1 \cap T = M_1$.

	Since $\delta_{P_1}(C_2) = 3$, $T$ is equal to $\PSL_n(q)$ or $\POmega_n^\pm(q)$ for some $q$ a square prime power, and some $n$. Let $T = \overline \Omega(V, \mathbb F, \kappa)$  for some classical geometry $(V, \mathbb F, \kappa)$ of type $\boldL$ or $\boldO^\pm$. By Theorem \ref{theorem:LMTgenerators}, the following hold: $T$ is of type $\boldO^\pm$, $M_1$ is in $\mathscr{C}_4$ or $\mathscr C_7$, and $\delta_{M_1}(C_2) = 4$. If $A_1$ contains $\overline I(V, \mathbb F, \kappa)$, then the same proof as in Proposition \ref{prop:C7Omegatcase} shows that $d(M) \leq 7$. So we assume that this is not the case. Hence, $\delta_{P_1}(C_2) = 3$ implies that $\Out(T) \cong D_8 \times C_f$, $P_1 \cap D_8 = C_2^2$, and $P_1 = (P_1 \cap D_8) \times (P_1 \cap C_f)$. In particular, $A_1$ contains $\overline S(V, \mathbb F, \kappa)$.

	Suppose that $M_1 \in \mathscr C_4$. Then, as $d(M_1) = 4$, there exist some classical geometries $(V_i, \mathbb F_i, \mathbb \kappa_i)$ for $i = 1,2$ such that $M_1 = \overline I(V_1, \mathbb F_1, \mathbb \kappa_1) \times \overline I(V_2, \mathbb F_2, \mathbb \kappa_2)$. Additionally, both classical geometries are of type $\boldO^\pm$ and have $D(V_i) = \Box$. Let $X_i := X(V_i, \mathbb F_i, \mathbb \kappa_i)$ for $X \in \{\Omega, I, \Delta, \Sigma\}$. Note that $\widetilde M_1 \leq \overline \Sigma_1 \times \overline \Sigma_2$ and $\widetilde M_1 \cap (\overline \Delta_1 \times \overline \Delta_2)$ is subdirect, by \cite[Propositions 4.4.14, 4.4.15]{KL}, since $A_1$ contains $\overline S(V, \mathbb F, \kappa)$. Hence, $\overline S_1 \times \overline S_2$ is contained in $[\widetilde M_1, M_1]$. Thus, by Lemma \ref{lemma:jumpingbabyresultfull}, if $A \cong C_2$ is a chief factor of $\widetilde M_1$ in $\overline S_1 \times \overline S_2$ then $A^t$ is Frattini in $M$. Note that $\delta_{M_1}(A) \leq 2$ for any chief factor $A$ not isomorphic to $C_2$, by \cite[Chapter 2]{KL} and Lemma \ref{lemma:non_simple_gen_bound}, as $q$ is a square prime power. By Corollary \ref{cor:CFsinwreathproduct}, for any chief factor $A$ of $M$,
	\begin{align*}
		\delta_M(A) & \leq \delta_{M,M_1^k}(A) + \delta_L(A) \\
		& \leq \begin{cases*}
			2 + 4 \qquad & if $A \cong C_2$,\\
			2 + 2 \qquad & otherwise. 
		\end{cases*}
	\end{align*}
	Hence, by Theorem \ref{theorem:crowns}, $d(M) \leq 6$.

	Suppose that $M_1 \in \mathscr C_7$. Then, as in the proof of Proposition \ref{prop:LMTC7calcs}, we have that $M_1 = \overline I(V_*, \mathbb F_*, \kappa_*)^2$ and $\widetilde M_1$ contains $\overline \Delta(V_*, \mathbb F_*, \kappa_*)^2$ for some classical geometry $(V_*, \mathbb F_*, \kappa_*)$ of type $\boldO^\pm$ with $D(V_*) = \Box$. Let $X_* := X(V_*, \mathbb F_*, \kappa_*)$ for $X \in \{\Omega, S\}$. Then, since $\overline \Omega_*$ is a simple group, $[\widetilde M_1, M_1]$ contains $\overline S_*^2$. Thus, by Lemma \ref{lemma:jumpingbabyresultfull} and Theorem \ref{theorem:crowns},
	\[
	d(M) \leq d(M/(\overline S_*^2)^k) \leq d(\overline I_*^2/\overline \Omega_*^2) + d(H/T^k) \leq 6
	\]
	where the final inequality is by \cite[Theorem 1]{LMT}.
\end{proof}

\section{Proof of Theorem \ref{theorem:bigresult}: sporadic groups}\label{section:sporadics}

In this last section, we finish the proof of Theorem \ref{theorem:bigresult} by considering the case where $G$ has sporadic socle. In particular, we prove the following.

\begin{theorem} \label{theorem:sporadics}
	Let $G_0$ be a sporadic group, and $G$ an almost simple group with socle $G_0$. Let $H$ be a maximal subgroup of $G$, and let $M$ be a maximal subgroup of $H$. Then $d(M) \leq 5$. Moreover, this bound is tight.
\end{theorem}

In order to prove this result, we rely primarily on computer-aided computations of the generator number of the second maximal subgroup $M$. This is done using the following general framework:
\begin{enumerate}[(1)]
	\item We construct the almost simple group $G$ or the maximal subgroup $H$ explicitly in Magma.
	\item We enumerate the second maximal subgroups $M$ contained in $G$ or $H$ respectively.
	\item We compute the generator number of $M$.
\end{enumerate}
While this framework allows us to prove Theorem \ref{theorem:sporadics} for most of the second maximal subgroups $M$, the specific implementation of the framework for each $M$ can vary greatly, and depends on a number of factors such as the data available in the $\Atlas$ and the degree of the representations of $G$ or $H$. As such, we will dedicate the majority of this section to explaining the different techniques used in the proof of Theorem \ref{theorem:sporadics}. The implementation of these techniques for each second maximal subgroup $M$ can be found in the GitHub repository \url{https://github.com/patriciamedinacapilla/SecondMaximalSubgroupsSporadics}. The output of these files gives the proof of Theorem \ref{theorem:sporadics}, which we collate at the end of the section in Tables \ref{table:smallsporadics}, \ref{table:BabyMonster}, and \ref{table:Monster}.

\subsection{Small sporadic groups}

The most straightforward application of the paradigm given above occurs when the almost simple group $G$ has a permutation representation stored in Magma. In this case, step (1) can be achieved by simply loading in this representation into Magma. We may then use the \texttt{MaximalSubgroups} function to enumerate $H$ and $M$, completing step (2). For each such second maximal subgroup $M$, its generator number can then be computed using the \texttt{SmallestGeneratingSet} function, completing step (3) and giving a proof of Theorem \ref{theorem:sporadics} in this case.

While quite straightforward, this technique can be applied to most of the sporadic groups, namely those with socle equal to one of $\M_{11}$, $\M_{12}$, $\J_1$, $\M_{22}$, $\J_2$, $\M_{23}$, $\HS$, $\J_3$, $\M_{24}$, $\McL$, $\He$, $\Ru$, $\Suz$, $\ON$, $\Co_3$, $\Co_2$, $\Fi_{22}$, $\Fi_{23}$, or $\Co_1$.

For the groups $G_0 = \HN$ and $G_0 = \Fi_{24}'$, Magma stores a permutation representation of degree 1140000 and 306936 respectively. While these representations can still be easily loaded into Magma, we found that it was unfeasible to use the \texttt{MaximalSubgroups} function to enumerate the maximal subgroups $H$ of $G$. Instead, we utilise the data in the Web $\Atlas$ (see \url{https://brauer.maths.qmul.ac.uk/Atlas/}) to construct the maximal subgroups of $G$. Indeed, generators for each maximal subgroup of $\Aut(G_0)$ are stored in the Web $\Atlas$ as words in the standard generators of $\Aut(G_0)$. We can therefore construct each maximal subgroup of $\Aut(G_0)$ as a permutation group by evaluating these words on the generators of our permutation representation of $\Aut(G_0)$. The second maximal subgroups $M$ and their generator number are then enumerated and computed as previously.

We note however that the Web $\Atlas$ does not contain generators for the maximal subgroups of $G_0$ in these two cases. To remedy this, we note that most of the maximal subgroups of $G_0$ are equal to the intersection of a maximal subgroup of $\Aut(G_0)$ with $G_0$. We can thus construct these maximal subgroups as previously. The remaining maximal subgroups are all almost simple, and hence $d(M)$ can be bounded by Theorem \ref{theorem:LMTgenerators}.

Finally, when $G_0 = \Ly$, $G_0 = \Th$ or $G_0 = \J_4$, Magma does not store a permutation representation of $G$. However, Magma does store moderately large degree matrix representations of these groups. To enumerate the maximal subgroups $H$, we again use the generators for these groups stored in the Web $\Atlas$ as above. However, to enumerate the second maximal subgroups $M$, it is unfeasible to use the \texttt{MaximalSubgroups} function, as the degree of the matrix representation is too large. Instead, we utilise the \texttt{LMGMaximalSubgroups} function, which is specifically tailored towards computations with large degree matrix representations.

When computing the generator number of $M$, we found that using \texttt{SmallestGeneratingSet} does not always work as smoothly as with permutation groups. This is due to the random behaviour of \texttt{AbelianQuotient}, which attempts to construct a presentation for $M$. In the file \texttt{core/mingenset\_lmg.m} in our GitHub repository, we provide an alternative version of the function \texttt{SmallestGeneratingSet} which uses LMG functions instead. This seems to eliminate the worst case behaviour and may be of independent interest when computing generator numbers of matrix groups.

The remaining two sporadic groups are the Baby Monster and the Monster group. Both of these groups do not have small permutation or matrix representations, and as such cannot be treated using the above methods. Instead, we treat these groups separately in the following two sections.

\subsection{The Monster group}

When $G = \mathbb M$, most of the maximal subgroups $H$ of $G$ are provided in the Web $\Atlas$ as small degree permutation or matrix groups. As such, we may proceed using the techniques described in the previous section to enumerate the second maximal subgroups $M$ contained in $H$, and to bound $d(M)$.

When $H = 2 . \mathbb{B}$, $3.\Fi_{24}$, $2^2.{}^2E_6(2):\Sym(3)$, $\Sym(3) \times \Th$ or $(D_{10} \times \HN).2$, the provided representations are quite large, and the above computations very resource intensive. Instead, we note that the structure of these groups is quite straightforward, and the desired bounds easily obtainable from the material developed in previous sections. Notably, if $H$ is $2 . \mathbb{B}$, $3.\Fi_{24}$, or $2^2.{}^2E_6(2):\Sym(3)$, then $H$ is almost quasisimple. As such, we may use Theorem \ref{theorem:LMTgenerators} in order to bound $d(M)$ by 4, 3, and 5 respectively. If $H$ is $\Sym(3) \times \Th$ or $(D_{10} \times \HN).2$, we may use Proposition \ref{prop:maxsubgpdirectproduct} to compute the possibilities for $M$. We may then use Theorem \ref{theorem:LMTgenerators} and Theorem \ref{theorem:crowns} to bound $d(M)$ by 4 and 5 respectively.

This leaves us with the three $2$-local subgroups $H = 2^{1+24} . \Co_1$, $2^{10+16}.\O_{10}^+(2)$ and $2^{5+10+20}.(\Sym(3) \times L_5(2))$. These groups do not have representations or generators in the Web $\Atlas$, and as such we cannot apply the method described previously. Instead, we make use of computer-aided calculations to determine the possible chief factors of the maximal subgroups $M$ of $H$ and apply the theory of crowns as in previous sections.

The general method for these three groups proceeds as follows. Let $N = O_2(H)$ and $L = H/N$. Note that since an irreducible subgroup of $\GL_m(2)$ has trivial $2$-core, each chief factor of $H$ in $N$ is an irreducible $\mathbb{F}_2[L]$-module. Firstly, we compute the chief factors of $H$ using the list of irreducible $\mathbb{F}_2[L]$-modules in conjunction with the character table of $\mathbb M$. If $M$ does not contain $N$, then we show that the chief factors of $M$ are a subset of those of $H$, since $N$ centralises any chief factor of $H$ in $N$. Hence we can use Theorem \ref{theorem:crowns} to bound $d(M)$ in this case. If $M$ contains $N$ then $M = N.J$ for some maximal subgroup $J$ of $L$. We then use Magma to compute the restriction of the chief factors of $H$ contained in $N$ to $J$, thus computing the chief factors of $M$. This again allows us to bound $d(M)$ by Theorem \ref{theorem:crowns}. To do this effectively, we provide the function \texttt{BoundGensFromChiefFactors} in the file \texttt{core/generation.m} of our GitHub repository which computes the bound in Theorem \ref{theorem:crowns} from a list of chief factors.

To illustrate this procedure, we give a detailed proof in the case that $H = 2^{5+10+20}.(\Sym(3) \times L_5(2))$.

\begin{prop}\label{prop:maxMonster}
	Suppose $H = 2^{5+10+20}.(\Sym(3) \times L_5(2)) \mleq \mathbb{M}$. Then $d(M) \leq 5$.
\end{prop}

\begin{proof}
	For reference, we remark that the order of $\L_5(2)$ is $2^{10} \times 3^2 \times 5 \times 7 \times 31$, and that the only irreducible $\mathbb{F}_2[\L_5(2)]$-modules of dimension at most 20 have dimensions 1, 5, and 10, with the only $1$-dimensional module being the trivial module. The only irreducible $\mathbb{F}_2[\Sym(3)]$-modules are the trivial module, and a module of dimension 2.
	
	Given any chief factor $X/Y \cong 2^m$ of $H$ contained in $[2^{35}]$, the quotient $H/C_H(X/Y)$ is an irreducible subgroup of $\GL_m(2)$. As such, the $2$-core of $H/C_H(X/Y)$ is trivial, and so $H/C_H(X/Y)$ is a quotient of $\Sym(3) \times \L_5(2)$. Hence, each chief factor of $H$ in $[2^{35}]$ is an irreducible $(\Sym(3) \times \L_5(2))$-module. 
	
	We now determine the possibilities for the chief factors of $H$ contained in $[2^{35}]$, as $(\Sym(3) \times \L_5(2))$-modules. Firstly, $2^5$ is a chief factor of $H$, with $\L_5(2)$ acting naturally on it and $\Sym(3)$ acting trivially on it, by \cite[Lemma 4.14]{DLP}. 
	
	We show that $\L_5(2)$ cannot act trivially on any chief factor of $H$ in $[2^{35}]$. Let $a$ be an element of $[2^{35}].\L_5(2)$ of order $31$, and let $A = \langle a \rangle$. Additionally, write $L_1 := [2^{35}].\Sym(3)$. Then, the element $a[2^{35}]$ centralises the subgroup $L_1/[2^{35}] \cong \Sym(3)$.  
	By \cite[Corollary 3.28]{Isaacs}, this implies that
	\[
	\Sym(3) \cong C_{L_1/[2^{35}]}(A) = \frac{C_{L_1}(A)[2^{35}]}{[2^{35}]},
	\]
	since $(2^{35}, |A|) = 1$. Hence, $C_H(A)$ contains an element of even order whose image in $H/[2^{35}]$ has order 2. However, from the character table of $H$ we can determine that $|C_H(a)|_2 = 2$, and so $C_H(A) \cap [2^{35}]$ is the trivial group. As such, given any chief factor $X/Y$ of $H$ in $[2^{35}]$, the element $a$ cannot act trivially on $X/Y$, by \cite[Corollary 3.28]{Isaacs}. Hence the same is true for $\L_5(2)$.
	
	Similarly, taking $a$ in $[2^{35}].\Sym(3)$ of order 3 and letting $L_2 := [2^{35}].\L_5(2)$, we note that $a[2^{35}]$ centralises $L_2/[2^{35}]$. Hence, as above, the image of $C_H(a)$ in $H/[2^{35}]$ has order divisible by $|\L_5(2)| = 2^{10} \times 3^2 \times 5 \times 7 \times 31$. Additionally, $a$ centralises a minimal normal subgroup of $H$ in $[2^{35}]$ of order $2^5$. There is a unique conjugacy class of $H$ of elements of order 3 whose centralisers have order divisible by 31, and so we can compute that $|C_H(a)|_2 =2^{15}$. Thus, we can apply the same method as for $\L_5(2)$ to determine that $\Sym(3)$ does not act trivially on any chief factor of $H$ in $[2^{35}]/2^5$.
	
	Hence, we have the following possibilities for the chief factors of $H$ in $[2^{35}]$. The bottom $2^5$ is the natural module for $\L_5(2)$, tensored with the trivial module for $\Sym(3)$. The $2^{10}$ factor is irreducible, and isomorphic to $2^5 \otimes 2^2$. Finally, the $2^{20}$ factor is either irreducible, and isomorphic to $2^{10} \otimes 2^2$, or contains two chief factors, both isomorphic to $2^5 \otimes 2^2$.
	
	If $M$ does not contain $[2^{35}]$, then the set of chief factors of $M$ in $[2^{35}] \cap M$ is a subset of one of the multisets $\{2^5, 2^{10}, 2^{10}, 2^{10}\}$ or $\{2^5, 2^{10}, 2^{20}\}$. On the other hand, the chief factors of $M$ not in $[2^{35}]$ are $\{\L_5(2), C_2, C_3\}$. Thus, by Theorem \ref{theorem:crowns},
	\[
	d(M) \leq 4.
	\]
	
	Hence, we may assume that $M$ contains $[2^{35}]$, and so $M$ is of the form $[2^{35}].J$ for some maximal subgroup $J$ of $\Sym(3) \times \L_5(2)$. In the file \texttt{monster/2\textasciicircum (5+10+20).(S3 x L5(2)).m} of our GitHub repository, we determine how the modules above restrict to $J$. We find that:
	\begin{enumerate}[\upshape(1)]
		\item $2^5 \downarrow J$ has only one non-Frattini chief factor;
		\item $(2^5 \otimes 2^2) \downarrow J$ has only one non-Frattini chief factor, which is never isomorphic to the non-Frattini chief factor in $2^5 \downarrow J$; and
		\item $2^{20} \downarrow J$ has either one or two non-Frattini chief factors. If it has two, they are isomorphic to $2^{10}$.
	\end{enumerate}
	Additionally, $\delta_J(A) \leq 2$ if $A \cong C_2, C_3$, and $\delta_J(A) \leq 1$ for any other chief factor $A$. Thus, we have $\delta_M(A) \leq 4$ for any chief factor $A$ of $M$, and so $d(M) \leq 5$ by Theorem \ref{theorem:crowns}.
\end{proof}

\subsection{The Baby Monster group}

When $G = \mathbb B$, generators of most of the maximal subgroups $H$ of $G$ are given in the Web $\Atlas$ as words in the standard generators of $\mathbb B$. However, the smallest degree representation of $\mathbb B$ is a $4370$-dimensional representation over $\mathbb F_2$. In practice, computing the maximal subgroups of these maximal subgroups is unfeasible in such a large dimensional matrix group. To alleviate this issue, one can attempt to construct a smaller dimensional faithful representation of the maximal subgroup, from the original $4370$-dimensional representation. We provide a function that uses the \texttt{Meataxe} and other LMG functions to achieve this. One can then use the previously described techniques to enumerate $M$ and bound $d(M)$. While this is theoretically possible for most of the maximal subgroups of $\mathbb{B}$, it is incredibly slow to check whether a representation is faithful, and so we try to avoid this method where possible. In fact, we only use this method when $H$ is $[2^{30}].\L_5(2)$.

Another way of constructing a representation of $H$ is by using the double cover of $\mathbb B$ contained in $\mathbb M$. Let $t$ be an involution of $\mathbb M$ in class $2A$. Then, the centraliser $C_{\mathbb M}(t)$ is isomorphic to the double cover $2.\mathbb B$ of the Baby Monster. As such, any maximal subgroup $H$ of $\mathbb B$ has a preimage $2.H$ contained in the Monster. If $2.H$ is contained in another maximal subgroup of $\mathbb M$, say $K$, then $2.H$ is equal to $C_K(t)$. As such, when this is the case, we can obtain maximal subgroups of $\mathbb B$ as centralisers in maximal subgroups of the Monster.

This is particularly of interest when $H$ is a $p$-local subgroup. In Table \ref{table:plocalinclusion} we indicate some $p$-local subgroups $H$ of $\mathbb B$ for which we have a corresponding $p$-local subgroup $K$ of $\mathbb M$ with $2.H \leq K$. These results can be deduced from \cite[Theorem 2]{MS} in the case that $p=2$, and obtained from \cite[Table 1]{AW} in the case that $p$ is odd.

\begin{table}
	\caption{Some $p$-local maximal subgroups of $\mathbb{B}$, and their overgroups in $\mathbb{M}$}
	\begin{center}
		\begin{tabular}{cc}\label{table:plocalinclusion}
			$H \mleq \mathbb B$ & $2.H \leq K \mleq \mathbb M$ \\\toprule
			$2^{1+22}.\Co_2$ & $2^{1+24}.\Co_1$\\
			$2^{2+10+20}.(\M_{22}:2 \times S_3)$ & $2^{2 + 11 + 22}.(S_3 \times \M_{24})$\\
			$[2^{35}].(S_5 \times \L_3(2))$ & $2^{3+6+12+18}.((3S_6 \times \L_3(2))$\\
			$5^{1+4}:2^{1+4}.A_5.4$ & $5^{1+6}:2.J2.4$\\
			$5^2:4\Sym(4) \times \Sym(5)$ & $(5^2:[2^4] \times U_3(5)).\Sym(3)$\\
			$3^{1+8}.2^{1+6}.U_4(2).2$ & $3^{1+12}.2\Suz.2$\\
			$(3^2:D_8 \times U_4(3).2.2).2$ & $(3^2:2 \times O_8^+(3)).\Sym(4)$\\
			$[3^{11}].(\Sym(4) \times 2\Sym(4))$ & $3^{2+5+10}.(M_{11} \times 2\Sym(4))$\\\bottomrule
		\end{tabular}
	\end{center}
\end{table}

In practice, we construct $K$ as a small degree permutation or matrix group, using the Web $\Atlas$. We then generate random involutions $t$ in $K$ and compute their centralisers, until $|C_K(t)| = |2.H|$. Using the character table of $K$ provided in GAP's \cite{GAP4} \texttt{CTblLib} package, we can confirm that there is a unique centraliser of an involution of this order, and hence $C_K(t) = 2.H$. We can now apply the previous techniques in order to bound $d(M)$.

The $2$-local subgroups $2^{1+22}.\Co_2$ and $2^{1+8+16}.\S_8(2)$ we treat using the same module analysis method as for the $2$-local maximal subgroups $2^{1+24}.\Co_1$, $2^{10+16}.\O_{10}^+(2)$ and $2^{5+10+20}.(\Sym(3) \times \L_5(2))$ of $\mathbb M$.

For the remaining maximal subgroups $H$, we analyse their structure and use the theory of crowns in order to bound $d(M)$, following the same approach used for certain maximal subgroups of $\mathbb M$. If $H$ is almost simple or almost quasisimple, then $d(M)$ can be bounded using Theorem \ref{theorem:LMTgenerators}. If $H$ is a direct product, we can use Proposition \ref{prop:maxsubgpdirectproduct} to enumerate $M$, and then use Theorem \ref{theorem:LMTgenerators} to bound $d(M)$. However, there are a few groups for which the structure is harder to analyse, so we provide a complete proof for these three groups here.

\begin{prop}\label{prop:maxBabyMonster6}
	Suppose $H = (2^2 \times F_4(2)):2$. Then $d(M) \leq 5$.
\end{prop}

\begin{proof}
	Using the character table of $H$ provided in GAP, it can be shown that the centre of $H$ has order $2$. Hence, if $g$ is an element of order 2 such that $H = (2^2 \times F_4(2)) : \langle g \rangle$, and $h_1, h_2$ are generators of $2^2$, we have that
	\[
	(h_1, 1)^g = (h_2, 1),
	\]
	with the subgroup $\langle (h_1 h_2^{-1}, 1) \rangle$ being equal to the center of $H$. This implies that 
	\[
	Z(H) = Z(H) \cap [H, H] \leq \Phi(H).
	\]
	Additionally, any subgroup of $H$ which contains $2^2$ and projects onto $H/(2^2 \times F_4(2))$ also contains $Z(H)$ in its Frattini subgroup, by the same reasoning.
	
	We now consider the different possibilities for $M$. If $M$ contains $F_4(2)$, then $M$ is equal to $F_4(2).J$ for some group $J$ of order dividing 4. As such, by Theorem \ref{theorem:crowns}, $d(M) = 2$.
	
	If $M$ does not contain $F_4(2)$, then $M$ contains $2^2$. Indeed, by the same argument as in Proposition \ref{prop:maxsubgpdirectproduct}, the subgroup $M 2^2$ contains $M$ and is proper, so is equal to $M$.  Thus, by Corollary \ref{cor:maximalsubgpextension}, $M$ is equal to $(2^2 \times J).2$ where $J$ is a maximal or novelty maximal subgroup of $F_4(2)$. Additionally, by above, $\Phi(M)$ contains $\langle (h_1 h_2^{-1}, 1) \rangle$. Thus, by Theorem \ref{theorem:crowns},
	\[
	d(M) \leq d(J) + 2 \leq 5.\qedhere
	\]
\end{proof}

\begin{prop} \label{prop:maxBabyMonster19}
	Suppose $H = (\Sym(6) \times \L_3(4) : 2 ) :2$. Then $d(M) \leq 4$.
\end{prop}

\begin{proof}
	Consider the quotient map 
	\[
	\pi \colon H \to \Aut(\Alt(6) \times \L_3(4)) \cong \Aut(\Alt(6)) \times \Aut(\L_3(4)).
	\]
	We show that $\pi$ is injective. Note that the restriction of $\pi$ to the subgroup $\Sym(6) \times \L_3(4) : 2$ is injective, since $\Sym(6)$ and $\L_3(4):2$ are almost simple groups. Hence, if $\pi$ is not injective, its kernel has order 2. For a contradiction, assume that $h$ is an involution of $H$ in $\ker(\pi)$, and so its centraliser $C_H(h)$ contains $\Alt(6) \times \L_3(4)$. However, using the character table of $H$, we can find that there are no centralisers of involutions of order divisible by $|{\Alt(6)} \times \L_3(4)|$. 
	
	Hence, we can view $H$ as a subgroup of index $2$ of the direct product $\Aut(\Alt(6)) \times \Aut(\L_3(4))$. It is now straightforward to find all possibilities for $H$ in Magma, and compute that $d(M) \leq 4$ for each one.
\end{proof}

\begin{prop} \label{prop:maxBabyMonster22}
	Suppose $H = (\Sym(6) \times \Sym(6)).4$. Then $d(M) \leq 5$.
\end{prop}

\begin{proof}
	If $M$ contains the normal subgroup $\Alt(6) \times \Alt(6)$ of $H$, this implies that $M = (\Alt(6) \times \Alt(6)).J$ where $J$ is a maximal subgroup of a group of order 16. Thus, $d(M) \leq 3$ by Theorem \ref{theorem:crowns}.
	
	Hence, we can assume that $M$ does not contain $\Alt(6) \times \Alt(6)$. If $\Alt(6) \times 1$ is a normal subgroup of $H$, we have that $M$ is equal to $J.\Alt(6).2^2.4$, for some maximal or novelty maximal subgroup $J$ of $\Alt(6)$. Using Magma, we can prove that $d(J) \leq 2$, and hence $d(M)\leq 5$.
	
	So we may assume that $\Alt(6) \times \Alt(6)$ is a minimal normal subgroup of $H$. Hence, by Proposition \ref{prop:XSimpleMaxSubgrpXkC}, the group $M \cap (\Alt(6) \times \Alt(6))$ is isomorphic to 1, $\Alt(6)$, or $J^2$ for some maximal or novelty maximal subgroup $J$ of $\Alt(6)$. In the first and second cases we have $d(M) \leq 3$. In the latter case, we have
	\[
	d(M) \leq d(J) + 3 \leq 5. \qedhere
	\]
\end{proof}

\subsection{Tables}

The computations outlined in the previous sections and implemented in our GitHub repository provide the proof of Theorem \ref{theorem:sporadics}. We summarise the output of these computations in the following three tables. These tables indicate an upper bound on $d(M)$, where $M$ is a second maximal subgroup contained in $G$, in the first table, or $H$, in the second and third tables. The third column of each table states which method has been used to compute this bound, while the fourth column indicates the file containing the code for this computation. A value marked with a dagger signifies that the bound is known to be sharp. 

\begin{table}
	\caption{A bound on the number of generators of the second maximal subgroups of $G$, when $\soc(G)$ is a sporadic group not equal to the Baby Monster group or the Monster group. All files are found in the \texttt{maximals} folder of our GitHub repository.}
	\label{table:smallsporadics}
	\begin{center}
    \addtolength{\leftskip} {-2cm}
    \addtolength{\rightskip}{-2cm}
		\begin{tabular}{>{\raggedright}p{3.5cm}S[table-format=1]ll}
			$G$ & {$d(M)$} & Method & File\\ \toprule
			$\M_{11}$, $\J_1$, $\M_{24}$ & 2\textsuperscript{\textdagger} & Permutation representation & \texttt{small\_sporadics.m}\\
			$\M_{12}$, $\M_{22}$, $\J_2$, $\M_{23}$, $\J_3$, $\McL$, $\He$, $\Ru$, $\ON$, $\Co_3$, $\Co_2$, $\Fi_{23}$, $\Co_1$ & 3\textsuperscript{\textdagger}\\
			$\HS$, $\Suz$ & 4\textsuperscript{\textdagger}\\
			$\Fi_{22}$ & 5\textsuperscript{\textdagger}\\ \midrule
			$\Fi_{24}$ & 4\textsuperscript{\textdagger} & Permutation representation \& $\Atlas$ & \texttt{Fi24.m}\\
			$\HN$ & 4\textsuperscript{\textdagger} &  & \texttt{HN.m}\\ \midrule
			$\J_4$ & 2\textsuperscript{\textdagger} &  Matrix representation \& $\Atlas$ & \texttt{J4.m}\\ 
			$\Ly$ & 3\textsuperscript{\textdagger} & & \texttt{Ly.m}\\
			$\Th$ & 3\textsuperscript{\textdagger} &  & \texttt{Th.m}\\\bottomrule
		\end{tabular}
	\end{center}
\end{table}

\begin{table}
	\caption{A bound on the number of generators of the maximal subgroups of each maximal subgroup of the Baby Monster. All files are found in the \texttt{maximals/baby\_monster} folder of our GitHub repository.}
	\label{table:BabyMonster}
	\begin{center}
    \addtolength{\leftskip} {-2cm}
    \addtolength{\rightskip}{-2cm}
		\begin{tabular}{>{\raggedright}p{5cm}S[table-format=1]ll}
			$H$ & {$d(M)$} & Method & File\\ \toprule
            $2^{2+10+20}.(M_{22}:2 \times \Sym(3))$ & 3 & Double cover in $\mathbb M$ & \texttt{2\textasciicircum(2+10+20).(M22 2 x S3).m}\\
            $[2^{35}].(\Sym(5) \times L_3(2))$ & 3 & & \texttt{[2\textasciicircum35].(S5 x L3(2)).m}\\
			$3^{1+8}.2^{1+6}.U_4(2).2$ & 3 &  & \texttt{3\textasciicircum(1+8).2\textasciicircum(1+6).U4(2).2.m}\\
            $[3^{11}].(\Sym(4) \times 2\Sym(4))$ & 3 & & \texttt{[3\textasciicircum11].(S4 x 2S4).m}\\
            $5^{1+4}.2^{1+4}.\Alt(5).4$ & 3 & & \texttt{5\textasciicircum(1+4).2\textasciicircum(1+4).A5.4.m}\\
            $5^2:4\Sym(4) \times \Sym(5)$ & 3 & & \texttt{5\textasciicircum2 4S4 x S5.m}\\
            $(3^2:D_8 \times U_4(3).2.2).2$ & 4 & & \texttt{(3\textasciicircum2 D8 x U4(3).2.2).2.m}\\ \midrule
			$2^{1+8+16}.\S_8(2)$ & 3 & Module analysis & \texttt{2\textasciicircum(9+16).Sp8(2).m}\\
			$2^{1+22}.\Co_2$ & 4 & & \texttt{2\textasciicircum(1+22).Co2.m}\\
			$5^3.\L_3(5)$ & 4 & & \texttt{5\textasciicircum3.L3(5).m}\\ \midrule
			$[2^{30}].L_5(2)$ & 2\textsuperscript{\textdagger} & Meataxe & \texttt{[2\textasciicircum30].L5(2)}\\ \midrule
            $\Fi_{23}$, $\Th$, $\O_8^+(3):\Sym(4)$, $\L_2(31)$, $M_{11}$, $\L_3(3)$, $\L_2(17):2$, $\L_2(11):2$ & 2\textsuperscript{\textdagger} & Almost simple & -\\
			$\HN : 2$, $\L_2(49).2$ & 3 & & -\\ \midrule
			$2.{}^2E_6(2):2$ & 4 & Almost quasisimple & - \\ \midrule
			$47:23$ & 1\textsuperscript{\textdagger} & Structure & -\\
			$\Sym(3) \times \Fi_{22}:2$, $5:4 \times \HS:2$, $\Sym(4) \times {}^2F_4(2)$, $\Sym(5) \times M_{22}:2$ & 4 &  & -\\ \midrule
			$(\Sym(6) \times \L_3(4):2).2$ & 4 & Proposition \ref{prop:maxBabyMonster19} & -\\
			$(2^2 \times F_4(2)):2$ & 5 & Proposition \ref{prop:maxBabyMonster6} & -\\
			$(\Sym(6) \times \Sym(6)).4$ & 5 & Proposition \ref{prop:maxBabyMonster22} & -\\
			\bottomrule
		\end{tabular}
	\end{center}
\end{table}

\begin{table}
	\caption{A bound on the number of generators of the maximal subgroups of each maximal subgroup of the Monster. All files are found in the \texttt{maximals/monster} folder of our GitHub repository.}
	\label{table:Monster}
	\begin{center}
    \addtolength{\leftskip} {-2cm}
    \addtolength{\rightskip}{-2cm}
		\begin{tabular}{>{\raggedright}p{6cm}S[table-format=1]ll}
			$H$ & {$d(M)$} & Method & File \\ \toprule
			$2 . \mathbb{B}$ & 4 & Quasisimple & - \\ \midrule
			
			$3.\Fi_{24}$ & 3 & Almost quasisimple & - \\
			$2^2.{}^2E_6(2):\Sym(3)$ & 5 &  & - \\ \midrule
			
			$2^{1+24} . \Co_1$ & 3 & Module analysis & \texttt{2\textasciicircum(1+24).Co1.m}\\
			$2^{10+16}.\O_{10}^+(2)$ & 3 &  & \texttt{2\textasciicircum(10+16).O10p(2).m}\\
			$2^{5+10+20}.(\Sym(3) \times L_5(2))$ & 5 &  & \texttt{2\textasciicircum(5+10+20).(S3 x L5(2)).m}\\ \midrule
			
			$\Sym(3) \times \Th$ & 3 & Structure & - \\
			$(D_{10} \times \HN).2$ & 5 & & - \\ \midrule
			
			$2^{2+11+22}.(M_{24} \times \Sym(3))$, $3^{1+12}.2\Suz.2$, $2^{3+6+12+18}.(3\Sym(6) \times L_3(2))$, $3^8.O_8^-(3).2_3$, $(3^2:2 \times O_8^+(3)).\Sym(4)$, $(7:3 \times \He):2$, $(\Alt(5) \times \Alt(12)):2$, $5^{3+3}.(2 \times L_3(5))$, $(\Alt(5) \times U_3(8):3_1):2$, $(L_3(2) \times \Sym(4):2).2$, $7^{1+4}:(3 \times 2\Sym(7))$, $(5^2:[2^4] \times U_3(5)).\Sym(3)$, $(L_2(11) \times M_{12}):2$, $7^{2+1+2}:\GL_2(7)$, $(L_2(11) \times L_2(11)):4$, $13^2:2L_2(13).4$, $(7^2:(3 \times 2\Alt(4)) \times L_2(7)).2$, $(13:6 \times L_3(3)).2$, $13^{1+2}:(3 \times 4\Sym(4))$, $\L_2(71)$, $\L_2(59)$, $11^2:(5 \times 2\Alt(5))$, $\L_2(29):2$, $7^2:\SL_2(7)$, $\L_2(19):2$, $41:40$ & 2\textsuperscript{\textdagger} & Direct & \texttt{small\_maximals.m}\\

			$3^{2+5+10}.(M_{11} \times 2\Sym(4))$, $3^{3+2+6+6}:(L_3(3) \times SD_{16})$, $5^{1+6}:2\J_2:4$, $(\Alt(6) \times \Alt(6) \times \Alt(6)).(2 \times \Sym(4))$, $5^{2+2+4}:(\Sym(3) \times \GL_2(5))$, $(\Alt(7) \times (\Alt(5) \times \Alt(5)):2^2):2$, $5^4:(3 \times 2L_2(25)):2_2$, $\Sym(5)^3:\Sym(3)$ & 3\textsuperscript{\textdagger} &  & \\
			
			$M_{11} \times \Alt(6).2^2$ & 4\textsuperscript{\textdagger} &  & \\
			
			\bottomrule
		\end{tabular}
	\end{center}
\end{table}

\clearpage
\bibliography{bib}{}
\bibliographystyle{plain}

\end{document}